\documentclass[aos,preprint,reqno]{imsart}
\setattribute{journal}{name}{}
\RequirePackage[OT1]{fontenc}
\RequirePackage{amsthm,amsmath}
\RequirePackage[numbers]{natbib}
\RequirePackage[colorlinks,linkcolor=blue,citecolor=blue,urlcolor=blue]{hyperref}
\usepackage{amssymb}
\usepackage{hyperref}
\usepackage[overload]{textcase}
\usepackage{extarrows}
\usepackage{multirow}
\usepackage{natbib}
\usepackage{stmaryrd}
\usepackage{color}
\usepackage{dsfont}
\usepackage{amsmath}
\usepackage{enumitem}
\usepackage{mathrsfs}
\usepackage[dvipsnames]{xcolor}
\usepackage{bbm}
\usepackage{bm}
\usepackage{comment}
\usepackage{caption}
\usepackage{booktabs}
\usepackage{graphicx}
\usepackage{subcaption}
\usepackage{float}
\usepackage{epstopdf}

\numberwithin{equation}{section}

\theoremstyle{plain} 
\newtheorem{thm}{Theorem}[section]
\newtheorem{lem}[thm]{Lemma}
\newtheorem{cor}[thm]{Corollary}
\newtheorem{prop}[thm]{Proposition}
\newtheorem{asm}[thm]{Assumption}
\newtheorem{de}[thm]{Definition}
\newtheorem*{thm*}{Theorem}

\theoremstyle{remark}
\newtheorem{rmk}[thm]{Remark}
\newtheorem{ex}[thm]{Example}

\newcommand{\ii}{\mathrm{i}}

\newcommand{\cf}{\emph{c.f., }}

\newcommand{\wt}{\widetilde}
\newcommand{\diag}{\mathrm{diag}}

\newcommand{\bs}{\boldsymbol}

\newcommand{\la}{\langle}
\newcommand{\ra}{\rangle}

\newcommand{\rf}[1]{(\ref{#1})}
\newcommand{\mb}[1]{\mathbf{#1}}
\newcommand{\sn}{\sqrt{N}}
\newcommand{\mG}{\mathcal{G}}

\newcommand{\mP}{\mathcal{P}}

\newcommand {\mbE}{\mathbb{E}}

\newcommand{\onh}{O_\prec(N^{-\frac12})}
\newcommand {\bu}{{\mb{u}}}
\newcommand {\bv}{{\mb{v}}}
\newcommand {\bw}{{\mb{w}}}

\newcommand {\lb}{\llbracket}
\newcommand {\rb}{\rrbracket}

\renewcommand{\mathbf}[1]{\bs{#1}}

\begin{document}

\begin{frontmatter}

\title{Statistical inference for principal components of spiked covariance matrices}
\runtitle{Statistical inference for principal components}

\begin{aug}
\author{\fnms{Zhigang} \snm{Bao}\thanksref{t1}\ead[label=e1]{mazgbao@ust.hk}},
\author{\fnms{Xiucai} \snm{Ding}\thanksref{t4}\ead[label=e4]{xiucai.ding@duke.edu}},
\author{\fnms{Jingming} \snm{Wang}\thanksref{t3}\ead[label=e3]{jwangdm@connect.ust.hk}},
\author{\fnms{Ke} \snm{Wang}\thanksref{t2}\ead[label=e2]{kewang@ust.hk}}

\affiliation{Hong Kong University of Science and Technology \thanksmark{t1} \thanksmark{t3} \thanksmark{t2}\\ Duke University \thanksmark{t4} }

\thankstext{t1}{Z.G. Bao was  is partially supported by Hong Kong RGC grant  GRF 16301519 and NSFC 11871425.}
\thankstext{t3}{J.M. Wang  was partially supported by Hong Kong RGC grant ECS 26301517 and GRF 16300618.}
\thankstext{t2}{K. Wang was partially supported by Hong Kong RGC grant GRF 16301618 and GRF 16308219.}
\runauthor{Z.G. Bao, X.C.Ding, J.M. Wang, and K. Wang.}

\address{Z.Bao\\J.Wang\\K.Wang\\
Department of Mathematics,  \\ Hong Kong University of Science and Technology, \\ Hong Kong\\
\printead{e1}\\
\phantom{E-mail:\ } \hspace{-1ex}\printead*{e3} \\ 
~~~~~~~~~~~\printead*{e2}}

\address{Department of Mathematics \\ Duke University \\USA\\
\printead{e4}\\
\phantom{E-mail:\ }}

\end{aug}

\runauthor{Zhigang Bao, Xiucai Ding, Jingming Wang, Ke Wang}

\begin{abstract}
In this paper, we study the asymptotic behavior of the  extreme eigenvalues and  eigenvectors of  the high dimensional spiked sample covariance matrices, in the supercritical case when a reliable detection of spikes is possible. Especially, we derive the joint distribution of the extreme eigenvalues and the generalized components of the associated eigenvectors, i.e., the projections of the eigenvectors onto arbitrary given direction,  assuming that the dimension and sample size are comparably large. In general, the joint distribution is given in terms of linear combinations of finitely many  Gaussian and Chi-square variables, with parameters depending on the projection direction and the spikes. Our assumption on the spikes is fully general. First, the strengths of spikes are only required to be slightly above  the critical threshold  and   no upper bound on the strengths is needed. Second,  multiple  spikes,  i.e., spikes with the same strength, are allowed. Third,  no structural assumption is imposed on the spikes. 
Thanks to the  general setting, we can  then apply the results to various high dimensional statistical hypothesis testing problems  involving both the eigenvalues and eigenvectors. Specifically, we propose accurate and powerful statistics to conduct hypothesis testing on the principal components. These statistics are data-dependent and  adaptive to the underlying true spikes.  Numerical simulations also  confirm the accuracy and powerfulness of our proposed statistics and illustrate significantly better performance compared to the existing methods in the literature. Especially, our methods are accurate and powerful even when either the spikes are small or the dimension is large.
\end{abstract}



\end{frontmatter}

\section{Introduction}

Covariance matrices play  important role in multivariate analysis and high dimensional statistics, and find applications in many scientific fields.  Moreover,  many statistical methodologies and techniques rely on the knowledge of the structure of the covariance matrix, to name but a few, Principal Component Analysis, Discriminant Analysis and Cluster Analysis. For detailed discussions of the applications and methodologies, we refer the readers to the monographs \cite{TA,JWHT, JI, MP} for a review. It is well-known in the high dimensional setting when the dimension is comparable with or much larger than the sample size, a direct application of the sample covariance matrix for hypothesis testing may result in untrustful conclusions \cite{YZB}. Consequently, a thorough understanding of the distributions of the eigenvalues and eigenvectors of sample covariance matrices is in demand for high dimensional statistical inference.  

In the literature of high dimensional statistics, a popular and sophisticated model is the spiked covariance matrix model proposed by Johnstone in \cite{Johnstone}, where a finite number of spikes (i.e., eigenvalues detached from the bulk of the spectrum) are added to the spectrum of the population covariance matrix; see (\ref{19071901}) and (\ref{19071902}) below. Throughout the paper, with certain abuse of terminology, we use the word ``spike" to represent either a detached eigenvalue $1+d_i$ (c.f. (\ref{19071901}), (\ref{19071902})) or the whole rank one matrix corresponding to a detached eigenvalue $(1+d_i)\bv_i\bv_i^*$ (c.f. (\ref{19071901}), (\ref{19071902})). These spikes can have various practical meanings in different fields. For instance, they correspond to the first few important factors in  factor models arising from financial economics \cite{FLM,ONA}, the number of patterns in  genetic variation across the globe \cite{DO}, the number of clusters in gene expression data \cite{KML} and the number of signals in single detection \cite{NJ, ONA1}. In this paper, we investigate the distributions of the principal components of the spiked sample covariance matrix, i.e, the sample counterparts of the extreme eigenvalues and eigenvectors (especially  those spikes) of the population covariance matrices. 
The principal components of spiked sample covariance matrices play important roles in Principal Component Analysis for high dimensional data. A lot of work has been devoted to estimating the principal components in various settings.  For instance, sparse principal component analysis \cite{BJNP,sparse1} is proposed to estimate the spiked eigenvalues and eigenvectors assuming some sparsity structure in the population eigenvectors; factor-model based estimators \cite{BN, BN2, ONA2} for the eigenvectors are constructed if the population covariance matrix is of approximate factor-model type; and some regularization-based methods \cite{sparse2, SZ, ZHT} have been proposed under various structural assumptions. 

Despite the wide applications of the principal components in high dimensional statistics, most of the literature focus on the estimation part. Much less is known about their distributions, especially for the leading eigenvectors. As a consequence, a thorough study of the statistical inference for the population covariance matrix in the high dimensional setting is still missing, especially for hypothesis testing problems involving both eigenvalues and eigenvectors. For instance, eigenvectors and eigenspaces play an important role in statistical learning. However, the existing literature has only been able to test whether the eigenvectors or eigenspaces of the population covariance matrix are equal to some given ones under the assumption that the dimension is much smaller than the sample size \cite{LeCamtest, KLsplit, bootsfr, SF, SS}. For another example, in Principal Component Analysis, 
the loadings are transformations of the original variables to the eigenvectors. They describe how much each variable contributes to a particular eigenvector and researchers are interested in hypothesis testing and constructing confidence intervals for them \cite{loadinginterval1,loadinginterval2,loadingtest}. The loadings are scaled eigenvectors using their corresponding eigenvalues and therefore, the joint distribution of the extreme eigenvalues and eigenvectors of the sample covariance matrices will be needed to conduct the hypothesis testing on them.  

Driven by these challenges, we study the joint distributions of the extreme eigenvalues and the generalized components of their associated eigenvectors for the spiked sample covariance matrices,  in the high dimensional setting. 
Based on these results, we will be able to perform hypothesis testings with statistics constructed from both eigenvalues and eigenvectors.

Specifically, in this paper, we consider the sample covariance matrices of the form
\begin{align}
Q=TXX^*T^*, \label{matrix model}
\end{align}
where $T$ is a $M\times M$ deterministic matrix and $X$ is a $M\times N$ random matrix with independent entries and $\mathbb{E}XX^*=I_M$. Further, we assume that the population covariance matrix $\Sigma:=TT^*$ admits the following form
\begin{align}
\Sigma=I_M+S, \label{19071901}   
\end{align}
where $S$ is a fixed-rank deterministic positive semi-definite matrix.  We first impose the following assumptions.

\begin{asm}\label{assumption}  Throughout the paper, we suppose the following assumptions hold.

\noindent (i)({\it  On dimensionality }):  We assume that $M\equiv M(N)$ and $N$ are comparable and there exist constants $\tau_2>\tau_1>0$ such that 
\begin{align} \label{eq_highdimensionalregime}
y\equiv y_N=M/N\in (\tau_1,\tau_2).
\end{align}

\noindent (ii)({\it On $S$}): We assume that $S$ admits the following spectral decomposition
\begin{align}
S=\sum_{i=1}^r d_i \bv_i\bv_i^*, \label{19071902}
\end{align}
where $r\geq 1$ is a fixed integer. Here  $d_1\geq \cdots \geq d_r >0$ are the ordered eigenvalues of $S$, and $\bv_i=(v_{i1},\ldots, v_{iM})^*$'s are the associated unit eigenvectors. All $d_i\equiv d_i(N)$ may be $N$-dependent.

\noindent (iii)({\it On $X$}): For the matrix $X=(x_{ij})$, we assume that the entries $x_{ij}\equiv x_{ij}(N)$ are real random variables satisfying
\begin{align*}
\mathbb{E}x_{ij}=0,\qquad \mathbb{E}x_{ij}^2=1/N.
\end{align*}
Moreover, we assume the existence of  large moments, {\it i.e.,} for any integer $p\geq3$, there exists a constant $C_p>0$, such that
\begin{align*}
\mathbb{E}\vert\sn x_{ij}\vert^p\leq C_p<\infty.
\end{align*}
We further assume that all $\sqrt{N}x_{ij}$'s possess the same $3$rd and $4$th cumulants, which are denoted by $\kappa_3$ and $\kappa_4$ respectively. 
\end{asm}

Throughout the paper, for simplicity,  we will mainly work with the setting 
\begin{align}
T= \Sigma^{\frac12}. \label{19071802}
\end{align} 
We remark that our results hold for much more general $T$ satisfying $\Sigma=TT^*$.  We refer to Remark \ref{rmk.extension} for more discussions on the extension along this direction. 
\subsection{Summary of previous related theoretical results}
In this section, we summarize the results related to the spiked sample covariance matrix  from the Random Matrix Theory literature.

We denote by $\mu_1\geq \cdots\geq \mu_{M \wedge N},  M \wedge N:=\min\{M,N\},$  the nontrivial eigenvalues of $Q$ and  $\mb{\xi}_i$ the unit eigenvector associated with $\mu_i$.  
The primary interest of the sample covariance matrix $Q$ lies in the asymptotic behavior of a few largest  $\mu_i$'s and the associated $\mb{\xi}_i$'s when $N$ is large, under various assumptions of $d_i$'s and $\mathbf{v}_i$'s. Significant progress has been made on this topic in the last few years.  It has been well-known since the seminal work of Baik, Ben Arous and P\'{e}ch\'{e} \cite{BBP}  that the largest eigenvalues $\mu_i$'s undergo a phase transition (BBP transition) w.r.t. the size of $d_i$'s. On the level  of the first order limit,  when $d_i>\sqrt{y}$, the eigenvalue $\mu_i$  jumps out of the support of the Marchenko-Pastur law (MP law) and converges to a limit determined by $d_i$, while  in the case of  $d_i\leq \sqrt{y}$, it sticks to the right end of the   Marchenko-Pastur (MP) law  $(1+\sqrt{y})^2$. In the former case, we call $\mu_i$ an {\it outlier} or {\it outlying eigenvalue}, while in the latter case we call $\mu_i$ a {\it sticking eigenvalue}. On the level of the second order fluctuation, it was revealed in \cite{BBP} that a phase transition  for  $\mu_i$ takes place in the regime $d_i-\sqrt{y}\sim N^{-\frac13}$. Specifically, if $d_i-\sqrt{y}\ll N^{-\frac13}$ (subcritical regime), the eigenvalue $\mu_i$ still admits the Tracy-Widom type distribution; if $d_i-\sqrt{y}\gg N^{-\frac13}$ (supercritical regime), the eigenvalue  $\mu_i$ is asymptotically Gaussian; while if $d_i-\sqrt{y}\sim N^{-\frac13}$ (critical regime), the limiting distribution of the eigenvalue $\mu_i$ is some interpolation between  Tracy-Widom and Gaussian.  The works \cite{Johnstone} and \cite{BBP} are on real and complex spiked Gaussian covariance matrices respectively. On extreme eigenvalues, further study for more generally distributed  covariance matrices can be found in \cite{BaikS, CDM16, BGM11, Paul, BY08, BY12, bloemendal2016principal, DY19, LHY}. The limiting behavior of the extreme eigenvalues  has also been studied for various related models, such as  the finite-rank deformation of Wigner matrices \cite{CDM16, BGM11, CDF09,  CDF12,   FP07, KY13, KY14,  Peche06, PRS13},  the signal-plus-noise model \cite{BGN12, LV, Ding}, the general spiked $\beta$ ensemble \cite{BVirag13, BVirag16}, and also the finite-rank deformation of general unitary/orthogonal invariant matrices \cite{BGN11, BBCF, BBCFAOP}.  

In contrast, the study on the limiting behavior of the eigenvectors associated with the extreme eigenvalues is much less.   On the level of the first order limit, it is known that the $\mathbf{\xi}_i$'s  are delocalized and purely noisy in the subcritical regime, but has a bias on the direction of $\mathbf{v}_i$ in the supercritical regime.  We refer to \cite{BGN11, BGN12, Capitaine17, Ding, Paul, bloemendal2016principal, DY19} for more details of such a phenomenon. It was recently noticed in \cite{ bloemendal2016principal} that a $d_i$ close to the critical point can cause a bias even for the non-outlier eigenvectors. On the level of the second order fluctuation, it was proved in \cite{ bloemendal2016principal} that the eigenvectors are asymptotically Gaussian in the subcritical regime, for the spiked covariance matrices. In the supercritical regime, a non-universality phenomenon was shown in \cite{CDM18} and \cite{BDW} for the eigenvector distribution for the finite-rank deformation of Wigner matrices and the signal-plus-noise model, respectively. The non-universality phenomenon in the supercritical regime has been previously observed in \cite{CDF09, KY13, KY14} for the extreme eigenvalues of the  finite-rank deformation of Wigner matrices.  Here we also refer to \cite{MJMY, JY, FFHL} for related study on the extreme eigenstructures of various finite-rank deformed models from more statistical perspective.   

\subsection{An overview of our results}
 In the theoretical part of this paper,  we will primarily  focus on the distribution of the eigenvectors $\mb{\xi}_i$'s associated with the outlying  eigenvalues $\mu_i$'s. That means, we will focus on the supercritical regime, in contrast to the work \cite{ bloemendal2016principal} where the eigenvector distribution in the subcritical regime was obtained. The results in the supercritical regime are particularly important for the statistical applications, since it is well-known that a reliable detection of spikes based on eigenvalues is only possible in this regime in general; see \cite{NE2008, NS2010, Nadler2008,  PWBM} for instance. Our assumption on the spikes is fully general (c.f. Assumption \ref{supercritical}). Especially, we do allow $d_i$'s to be divergent and   multiple (i.e. some $d_i$'s are identical). In case that the spikes are simple (i.e. $d_i$'s are all distinct), we also establish the joint distribution of the outlying eigenvalues and the associated  eigenvectors for the spiked covariance matrices.  This is the primary goal of the Principal Component Analysis from statistics point of view. More specifically, in this paper, we are interested in the  distribution of the largest $\mu_i$'s and the \emph{generalized component} of the top eigenvectors, i.e., the projections of those eigenvectors onto a general direction. More specifically, let $\mb{w}\in S^{M-1}_\mathbb{R}$ be any deterministic unit vector. We will study the limiting distribution of $ |\langle\mb{w},\mb{\xi}_i\rangle|^2$ in the supercritical regime under general assumption of the spikes, and also state the joint distribution of $ |\langle\mb{w},\mb{\xi}_i\rangle|^2$ and $\mu_i$'s in case that the spikes are simple.  
 We emphasize here that in case that a spike is multiple, one can also describe the joint distribution of eigenvalues and eigenvectors using the approach in this paper. But the result does not have a succinct form so we omit it from the statements of our main theorems; see Remark \ref{rmk. multiple eigenvalue} for more details. Nevertheless, we will describe (in certain equivalent form) and prove an extension of the joint eigenvalue-eigenvector distribution to the multiple case in  the application part, Section \ref{s.application}, and present applications of this result.

In the application part of this paper, we construct statistics to infer the principal components. We mainly focus on two hypothesis testing problems regarding the eigenspaces, (\ref{eq_spectralprojection}) and (\ref{eq_orthogobalteset}). To our best knowledge, it is the first time that these problems are tackled for spiked covariance matrices  in the high dimensional regime (\ref{eq_highdimensionalregime}) without imposing any structural assumptions on the spikes.  Our proposed statistics make use of some plug-in estimators and are adaptive to the information of unknown spikes, for instance, their values and multiplicity. Thanks to the joint distribution of the eigenvalues and eigenvectors, we can easily establish the asymptotic distributions of our testing statistics; see Section \ref{sec.statanddist} for more details.
Our methodology is simple, computationally cheap and easy to be implemented. Extensive numerical simulations lend strong support to our testing statistics. Especially, our proposed statistics are accurate and powerful regardless of the value of $y$ and magnitude of the spikes. Moreover, for testing (\ref{eq_spectralprojection}), our statistics show better performance compared to the existing methods in the literature both in terms of accuracy and power. We point out that our methodology can be used to study other hypothesis testing problems regarding Principal Component Analysis and this will be discussed in Section \ref{s.application}.

 Finally, we remark that our theory and applications are highly compatible. The theoretical results are not only interesting and natural on their own, they are also highly motivated by the  applications, which are all fundamental problems in the statistics literature. Especially, the generality of our assumptions in the theoretical part is not pursued only for technical interest; it is indeed indispensable in the application part. We leave more details about our contribution and novelties to Section \ref{s.pfst},   since they can be better illustrated only after necessary notations and main results are stated.

\vspace{2ex} 
{\bf Notation: } Throughout the paper, the sample size $N$ will be the fundamental parameter which goes to $\infty$.  The symbol $o_N(\cdot)$ stands for any quantity going to $0$ as $N$ goes to $\infty$. We use~$c$ and~$C$ to denote positive finite constants that do not depend on the ~$N$. Their values may change from line to line.  For two positive quantities $A_N, B_N$ depending on $N$ we use the notation $A_N \asymp B_N$ to denote $C^{-1}A_N\leq B_N \leq C A_N$ for some constant $C>1$. Further, we write $A_N\doteq B_N$ if $A_N=B_N(1+o_N(1))$. 

 For vectors $\mb{v},\mb{w}\in \mathbb{C}^N$, we write $\mb{v}^*\mb{w}=\la \mb{v},\mb{w}\ra$ for their scalar product. We emphasize here, unless otherwise specified,  the vectors in this paper are real vectors and thus $\mb{v}^*\mb{w}=\mb{v}^{\top}\mb{w}$.   Further, for a matrix $A$, we denote by $\|A\|_{\text{op}}$ its operator norm, while we use $\|\mathbf{v}\|$ to represent  the $\ell^2$ norm for a vector $\mathbf{v}$.  
 
 We use double brackets to denote index sets, i.e.\ for $n_1, n_2\in\mathbb{R}$, $\llbracket n_1,n_2\rrbracket:= [n_1, n_2] \cap\mathbb{Z}$. In addition, we use $\mathbf{1}_n=\frac{1}{\sqrt{n}}(1,\ldots, 1)^*$ to denote the  
 $n$-dimensional normalized all-1 vector.  Further, we denote by $\mathds{1}(E)$ or $\mathds{1}_E$ the indicator function of an event $E$. 

\vspace{2ex}
{\bf Organization:} The paper is organized as the following: In Section \ref{sec.mainresult}, we state our main results and proof strategy. In Section \ref{s.application}, we discuss several applications of our results and present the simulation results.  In Section \ref{s. pre}, we introduce some basic notions and preliminary results for later discussion. Section \ref{secgfr} is devoted to the Green function representations of our eigenvalue and eigenvector statistics. Then in Section \ref{s.proof of thm}, we prove our main result, Theorem \ref{mainthm}, based on a key technical recursive moment estimate for  some Green function statistics; see Proposition \ref{recursivemP}.  The proof of Theorem  \ref{thm.simple case} is also stated in Section \ref{s.proof of thm}.  The proof of Proposition \ref{recursivemP} is  postponed to Appendix \ref{sec:prop}.  In addition, in Appendix \ref{s.derivative of G}, we collect some basic formulas concerning the derivatives of Green function, for the convenience of the reader. In Appendix \ref{appendix.proof}, we provide the technical proofs for the statistics used in Section \ref{s.application}. Finally, Appendix \ref{sec_additionalsimulation} collects some additional simulation results regarding non-Gaussian data.   

\section{Main results, proof strategy and novelties}\label{sec.mainresult}
In this section, we state our main results and explain our proof strategy and novelties. 
We will need the following notions of a high-probability bound and convergence in distribution. The notion of {\it stochastic domination} was introduced in \cite{EKY2013}, which provides a precise statement of the form ``$X_N$ is bounded by $Y_N$ up to a small power of $N$ with high probability''.
\begin{de}(Stochastic domination) \label{def.sd} Let
\begin{align*}
X=\big(X_N(u): N\in \mathbb{N}, u\in U_N\big), \quad Y=\big(Y_N(u): N\in \mathbb{N}, u\in U_N\big)
\end{align*}
 be two families of  random variables, where $Y$ is nonnegative, and $U_N$ is a possibly $N$-dependent parameter set.  We say that $X$ is bounded by $Y$, uniformly in $u$, if for all small $\varrho>0$ and large $\phi>0$, we have
 \begin{align*}
 \sup_{u\in U_N}\mathbb{P}\big(|X_N(u)|>N^{\varrho}Y_N(u)\big)\leq N^{-\phi}
 \end{align*}
 for large $N\geq N_0(\varrho, \phi)$. Throughout the paper, we use the notation $X=O_\prec(Y)$ or $X\prec Y$ when $X$ is stochastically bounded by $Y$ uniformly in $u$.   Note that in the special case when $X$ and $Y$ are deterministic, $X\prec Y$
	means for any given $\varrho>0$,  $|X_{N}(u)|\leq N^{\varrho}Y_{N}(u)$ uniformly in $u$, for all sufficiently large $N\geq N_0(\varrho)$.

 In addition,  we also say that an $N$-dependent event $\mathcal{E}\equiv \mathcal{E}(N)$ holds with high probability if, for any large $\varphi>0$,  
\begin{align*}
\mathbb{P}(\mathcal{E}) \geq 1- N^{-\varphi},
\end{align*}
for sufficiently large $N \geq N_0( \varphi)$.
\end{de}

\begin{de} \label{defn_asymptotic}
Two sequences of random vectors, $\mathsf X_N \in \mathbb{R}^k$ and $\mathsf Y_N \in \mathbb{R}^k$, $N\geq 1$,  are \emph{asymptotically equal in distribution}, denoted as $\mathsf X_N \simeq \mathsf Y_N,$ if they are tight and satisfy
\begin{equation*}
\lim_{N \rightarrow \infty} \big( \mathbb{E}f(\mathsf X_N)-\mathbb{E}f(\mathsf Y_N) \big)=0
\end{equation*}   
for any bounded continuous function $f:\mathbb{R}^k\to \mathbb{R}$. 
\end{de}

\subsection{Main results}
Since we will focus on the supercritical regime, we  further make the following assumption. First, we recall the notation $\lb 1, m\rb:=\{1,\cdots,m\}$.
\begin{asm} \label{supercritical} 
Let $\epsilon>0$ be any small but fixed constant. Let $d_i\equiv d_i(N), i\in\lb 1, r\rb$ be the eigenvalues of $S$ in (\ref{19071902}). 
There exists a maximum  integer $ r_0\equiv r_0(\epsilon)\in \lb1, r\rb$, such that  for any $i\in \lb 1, r_0\rb$,
\begin{align}
   d_i- y^{1/2} > N^{-\frac 13 +\epsilon }  \label{asd}
\end{align}
for all sufficiently large $N\geq N_0(\epsilon)$. 
Moreover,  for a fixed $i\in \lb 1,  r_0\rb $, there exists a (unique) index set $\mathsf{I}\equiv \mathsf{I}(i) \subset \lb1, r_0\rb$ such  that $i\in \mathsf{I}$ and  for any $t\in \mathsf{I}$, 
\begin{align}
d_t= d_{i}, \quad \delta_i:=\min_{j\in \mathsf{I}^c} | d_{i}-d_j| > d_i^{3/2}(d_i-y^{1/2})^{-\frac 12}N^{-\frac 12+ \epsilon}, \label{asd2} 
\end{align}
where we denote $\mathsf{I}^c:=\lb 1,r\rb \setminus \mathsf{I}$. By definition, $\delta_t$ (or $\mathsf{I}(t)$) is the same for all $t\in \mathsf{I}(i)$.

\end{asm}

\begin{rmk}
It is known that the BBP phase transition takes place on the regime $d_i-y^{1/2}\sim N^{-\frac13}$; see for instance \cite{BBP,KY13,bloemendal2016principal}. Hence, (\ref{asd}) ensures that we are in the supercritical regime. Further, note that we do not assume the spikes $d_i\equiv d_i(N)$ to be bounded in $N$. That means, we do allow $d_i\sim N^c$ for any $c>0$, say. 
In \rf{asd2}, the first identity means that we allow  $d_i$ to be multiple.  And the second inequality is the so-called  {\it non-overlapping condition} which guarantees that the distinct (possibly multiple) $d_i$'s are well-separated such that the eigenvalues $\mu_i$'s corresponding to  distinct $d_i$'s do not have essential overlap on the scale of fluctuation; see detailed explanation  in \cite{bloemendal2016principal} for instance. Note that the prefactor $d_i^{3/2}$ is not included in the non-overlapping condition in \cite{bloemendal2016principal}. But this factor is needed to cover the case when the $N$-dependent $d_i$ is large. More precisely, when $d_i$ is  large, as stated in Theorem 2.3 of  \cite{bloemendal2016principal}, the fluctuation of  $\mu_i$ around its limiting location is of order $O_\prec(d_iN^{-1/2})$, and thus  the lower bound on the RHS of the second inequality in \rf{asd2} which separates the limiting locations of  $\mu_i$'s corresponding to distinct $d_i$'s is slightly larger than the fluctuation of  $\mu_i$'s. We emphasize here that in reality, it can certainly happen that two distinct $d_i$'s are close enough to violate the non-overlapping condition. However, in this case, since the fluctuation of their sample counterparts, $\mu_i$'s, have essential overlap,  effective inference of $d_i$'s based on $\mu_i$'s is believed to be impossible in general. 
Also, since eigenvectors are sensitive to the eigenvalue gap, in this case, inference of $\bv_i$'s based on $\mb{\xi}_i$'s will also be unreliable. Therefore, the non-overlapping condition together with (\ref{asd}) can be regarded as a nearly minimal condition for a reliable detection of spikes.  
\end{rmk}

We will state our main results under Assumptions \ref{assumption} and \ref{supercritical}, after necessary notations are introduced. Rewrite  the spectral decomposition of $\Sigma$ as $\Sigma= \sum_{i=1}^M \sigma_i \mathbf v_i \mathbf v_i^* = I_M + \sum_{i=1}^M d_i \mathbf v_i \mathbf v_i^*,$ where by Assumption \ref{assumption}, $d_{r+1}=\cdots=d_M=0$. Further, we emphasize here that the specific choice of the orthonormal $\mathbf{v}_i$'s for $i=r+1,\ldots, M$ is irrelevant to our discussion since only $\sum_{j=r+1}^M \mathbf{v}_i\mathbf{v}_i^*$ will be involved. In the sequel, we fix an $i$ and consider a (possibly) multiple $d_i$. Let $\mathsf{I}\equiv \mathsf{I}(i)$ be the index set of this multiple $d_i$ in Assumption \ref{supercritical}.  In order to study the generalized components of the eigenvectors of the $|\mathsf{I}|$-fold multiple  $d_i$, we introduce 
\begin{align}Z_{\mathsf{I}}:=\sum_{t\in \mathsf{I}} \mathbf v_t \mathbf v_t^*, \label{def of ZI}
\end{align} 
 the orthogonal projection onto $\mathrm{Span}\{\mathbf v_t\}_{t\in \mathsf{I}}$ 
and 
the corresponding random projection 
\begin{align}
{\rm P}_{\mathsf{I}}:=\sum_{t\in \mathsf{I}} \mb{\xi}_t\mb{\xi}_t^*, \label{def of PI}
\end{align}  
which is the sample counterpart of $Z_{\mathsf{I}}$. Note that in case $|\mathsf{I}|>1$, it is meaningless to do statistical inference for an individual  $d_i\mathbf{v}_i\mathbf{v}_i^*$, since in the multiple case, there is an arbitrariness in the  choice  of $\{\mathbf{v}_t\}_{t\in \mathsf{I}}$ as a basis for certain subspace. Hence, it is more natural to study $Z_{\mathsf{I}}$ and its sample counterpart. 
For any unit $\mathbf{w}\in \mathbb{R}^{M}$, denote its projection onto $\mathrm{Span}\{\mathbf v_t \}_{t\in \mathsf{I}}$ by 
\begin{align}
\bw_{\mathsf{I}}:=Z_{\mathsf{I}} \bw, \label{081510}
\end{align}
and its weighted projection onto $\mathrm{Span}\{\mathbf v_j \}_{j\in \lb 1,M\rb \setminus \mathsf{I}}$ by 
 \begin{align} \label{def:v_I_varsigma_I}
{\mb{\varsigma}}_\mathsf{I} := \sum_{j\in \lb 1,M\rb \setminus \mathsf{I}} \frac{d_i\sqrt{d_j +1}}{d_i -d_j} \langle \bw, \bv_j \rangle \bv_j
 \end{align}
 with its normalized version 
 \begin{align}\label{def.nor.Varsigma}
\mb{\varsigma}_{\mathsf{I}}^0: = \begin{cases}
\mb{\varsigma}_{\mathsf{I}}/\Vert \mb{\varsigma}_{\mathsf{I}}\Vert, &\text{if }\mb{\varsigma}_{\mathsf{I}}\neq \mb{0};\\
\mb{0}, &\text{otherwise}.
\end{cases}
\end{align}

  For any vectors $\mathbf{a}_l=(a_l(i))\in \mathbb{R}^M, l\in \mathbb{Z}^+$, we set the notations 
 \begin{align}
 \mathbf{s}_{k_1,\cdots, k_t}(\mathbf{a}_1, \cdots, \mathbf{a}_t) = 
  \sum_{j=1}^M a_1(j)^{k_1}\cdots a_t(j)^{k_t}.
 \end{align}
 For instance, $\mathbf{s}_{1,3}(\mathbf{a}_1, \mathbf{a}_2) = 
  \sum_{j=1}^M a_1(j)  a_2(j)^{3}$. 
 
For brevity, we introduce the following auxiliary functions for $d>0$
 \begin{align}
&\mathtt f(d):=\frac{y(1+d)}{d(d+y)} \Big( 1+\frac{d(1+d)}{d+y} \Big), 
\qquad\mathtt g(d):= \frac{2\sqrt{(d+1)(d+\sqrt y)}}{d+y
}, \label{eq_keyeqone}\\
& \mathtt h(d):=\frac{d+1}{d+y} ,\qquad \mathtt l(d):={\frac{1+d}{\sqrt{d(d+y)}}}. \label{eq_keyeqtwo}
\end{align}
With the above notations, we further define two symmetric matrices $A_{\mathsf{I}}^\bw$ and $B_{\mathsf{I}}^\bw$ of size $(r+2)\times (r+2)$, indexed by the vectors $\bw_{\mathsf{I}}, {\mb{\varsigma}}_\mathsf{I}, \{\bv_t\}_{t\in \mathsf{I}}$ and $\{\bv_j\}_{j\in \mathsf{I}^c}$. 
Here we recall the definition $\mathsf{I}^c=\llbracket 1, r\rrbracket\setminus \mathsf{I}$ from Assumption \ref{supercritical}.  The only non-zero entries  of $A_{\mathsf{I}}^\bw$ are given by the ones below and those followed by symmetry

\begin{align}\label{def:covA}
&A_{\mathsf{I}}^\bw(\bw_{\mathsf{I}}, \bw_{\mathsf{I}}) =2y \mathtt h(d_i)^2(1+y \mathtt h(d_i)^2) \|\bw_{\mathsf{I}}\|^4, \quad { A_{\mathsf{I}}^\bw({\mb{\varsigma}}_\mathsf{I}, {\mb{\varsigma}}_\mathsf{I}) = \mathtt g(d_i)^2 \|\bw_{\mathsf{I}}\|^2}, \nonumber\\
&{ A_{\mathsf{I}}^\bw(\bv_t, \bv_t) = \mathtt h(d_i) },
 \quad A_{\mathsf{I}}^\bw(\bv_j, \bv_j) = \mathtt l(d_i)^2 \|\bw_{\mathsf{I}}\|^2 , 
\nonumber\\
& { A_{\mathsf{I}}^\bw({\mb{\varsigma}}_\mathsf{I}, \bv_t) = \mathtt g(d_i)\sqrt{\mathtt h(d_i)}  \langle \bw_{\mathsf{I}}, \bv_t \rangle,\quad  A_{\mathsf{I}}^\bw({\mb{\varsigma}}_\mathsf{I}, \bv_j) = \mathtt g(d_i) \mathtt l(d_i) \langle {\mb{\varsigma}}_\mathsf{I}^0, \bv_j \rangle \|\bw_{\mathsf{I}}\|^2,  } 
\nonumber\\
& { A_{\mathsf{I}}^\bw(\bv_t,\bv_j)=\sqrt{\mathtt h(d_i)} \mathtt l(d_i) \langle \bw_{\mathsf{I}}, \bv_t \rangle \langle {\mb{\varsigma}}_\mathsf{I}^0, \bv_j \rangle, \quad \text{for } t\in \mathsf{I}, j\in \mathsf{I}^c.  } 
\end{align}

The only non-zero entries  of $B_{\mathsf{I}}^\bw$ are given by the ones below and those followed by symmetry 
\begin{align}\label{def:covB}
&B_{\mathsf{I}}^\bw(\bw_{\mathsf{I}}, \bw_{\mathsf{I}}) =\mathtt f(d_i)^2 s_4(\bw_{\mathsf{I}}), \quad B_{\mathsf{I}}^\bw({\mb{\varsigma}}_\mathsf{I}, {\mb{\varsigma}}_\mathsf{I}) = \mathtt g(d_i)^2 s_{2,2}({\mb{\varsigma}}_\mathsf{I}^0,\bw_{\mathsf{I}}),  \nonumber\\
&B_{\mathsf{I}}^\bw(\bv_{t_1}, \bv_{t_2}) = \mathtt h(d_i)s_{1,1,2}(\bv_{t_1},\bv_{t_2},\mb{\varsigma}_{\mathsf{I}}^0), \quad B_{\mathsf{I}}^\bw(\bv_{j_1},\bv_{j_2})= \mathtt l(d_i)^2 s_{1,1,2}(\bv_{j_1},\bv_{j_2}, \bw_{\mathsf{I}}),\nonumber\\  
&B_{\mathsf{I}}^\bw(\bw_{\mathsf{I}},\mb{\varsigma}_{\mathsf{I}} ) = \mathtt f(d_i) \mathtt g(d_i) s_{1,3}(\mb{\varsigma}_{\mathsf{I}}^0,\bw_{\mathsf{I}} ), \quad B_{\mathsf{I}}^\bw(\bw_{\mathsf{I}}, \bv_t) = \mathtt f(d_i) \sqrt{\mathtt h(d_i)}s_{1,1,2}(\bv_t, \mb{\varsigma}_{\mathsf{I}}^0, \bw_{\mathsf{I}}),\nonumber\\  
&B_{\mathsf{I}}^\bw(\bw_{\mathsf{I}}, \bv_j) = \mathtt f(d_i) \mathtt l(d_i) s_{1,3}( \bv_j, \bw_{\mathsf{I}}), \quad B_{\mathsf{I}}^\bw({\mb{\varsigma}}_\mathsf{I}, \bv_t) = \mathtt g(d_i)\sqrt{\mathtt h(d_i)} s_{1,1,2}(\bv_t, \bw_{\mathsf{I}}, {\mb{\varsigma}}_\mathsf{I}^0),\nonumber\\
& B_{\mathsf{I}}^\bw({\mb{\varsigma}}_\mathsf{I}, \bv_j) = \mathtt g(d_i) \mathtt l(d_i) s_{1,1,2}(\bv_j, {\mb{\varsigma}}_\mathsf{I}^0,\bw_{\mathsf{I}}), \nonumber\\
& B_{\mathsf{I}}^\bw(\bv_t, \bv_j) = \sqrt{\mathtt h(d_i)} \mathtt l(d_i) s_{1,1,1,1}(\bv_t,\bv_j,\bw_{\mathsf{I}}, \mb{\varsigma}_{\mathsf{I}}^0),   \qquad \text{for  }t,t_1,t_2\in \mathsf{I}, j, j_1, j_2 \in \mathsf{I}^c. 
\end{align}

To avoid the ambiguity on the realizations of the matrices $A_{\mathsf{I}}^\bw$ and $B_{\mathsf{I}}^\bw$ due to the possible permutations of the vector-indices $\bw_{\mathsf{I}}, {\mb{\varsigma}}_\mathsf{I}, \{\bv_t\}_{t\in \mathsf{I}}, \{\bv_j\}_{j\in \mathsf{I}^c}$, we fix the ordering of these indices as follows:  $\bw_{\mathsf{I}}$, ${\mb{\varsigma}}_\mathsf{I}$, followed by 
$\bv_t$'s in  ascending order of $t\in \mathsf{I}$, and then followed by $\bv_j$'s in  ascending order of $j\in \mathsf{I}^c$. Then, $A_{\mathsf{I}}^\bw(\bw_{\mathsf{I}},\mb{\varsigma}_{\mathsf{I}})$ will be the $(1,2)$-entry of  $A_{\mathsf{I}}^\bw$, for instance. Correspondingly, in the sequel, we use the shorthand notation such as $(\alpha,\beta, \{\gamma_t\}_{t\in \mathsf{I}}, \{\delta_j\}_{j\in \mathsf{I}^c})$ to represent the vector with components $\alpha,\beta, \gamma_t, \delta_j\in \mathbb{R}, t\in \mathsf{I},j\in \mathsf{I}^c$. Further, the ordering of the components in $(\alpha,\beta, \{\gamma_t\}_{t\in \mathsf{I}}, \{\delta_j\}_{j\in \mathsf{I}^c})$ is analogous to the ordering of the vector-indices of the matrices $A_{\mathsf{I}}^{\bw}$ and $B_{\mathsf{I}}^{\bw}$. Similar notations are used throughout the paper without further explanation.

With the above notations, we now state the first main theorem regarding the general components. 

\begin{thm}\label{mainthm}
Suppose that Assumptions \ref{assumption},  \ref{supercritical}
 and the setting (\ref{19071802}) hold. Fix an  $i\in \lb 1, r_0\rb$ and let $\bw\in S_{\mathbb{R}}^{M-1}$ be any deterministic unit vector. In case $d_i\equiv d_i(N)\to \infty$ as $N\to\infty$, we additionally assume that $|y-1|\geq \tau_0$ for some small but fixed $\tau_0>0$. Then there exist random  variables $\Theta_{\bw_{\mathsf{I}}}^{\bw}, \Lambda_{\mb{\varsigma}_{\mathsf{I}}}^{\bw}, \{\Delta_{\bv_t}^{\bw}\}_{t\in \mathsf{I}}, \{\Pi_{\bv_j}^{\bw}\}_{j\in \mathsf{I}^c}$
such that $\la \bw, {\rm{P}}_{\mathsf{I}}\bw \ra$ admits the following expansion
 \begin{align} \label{weak_gf_decomp}
\la \bw, {\rm{P}}_{\mathsf{I}}\bw \ra&=\frac{d_i^2-y}{d_i(d_i+y)}\la \bw, {\rm{Z}}_{\mathsf{I}}\bw \ra+ 
\frac{1}{\sqrt{N(d_i^2-y)}}  \,\Theta_{\bw_{\mathsf{I}}}^{\bw}
+\frac{\| \mb{\varsigma}_{\mathsf{I}} \| \sqrt{d_i- y^{1/2}}}{\sn d_i}\, \Lambda_{\mb{\varsigma}_{\mathsf{I}}}^{\bw}   \nonumber \\
&
+\frac{\| \mb{\varsigma}_{\mathsf{I}} \|^2 }{Nd_i} \sum_{t\in \mathsf{I}} (\Delta_{\bv_t}^{\bw})^2-\frac{1 }{N}\sum_{j\in \mathsf{I}^c}\frac{d_id_j}{(d_i-d_j)^2} (\Pi_{\bv_j}^{\bw})^2 \nonumber\\
&+ O_\prec \bigg(\frac{1}{N^{\frac 12+\varepsilon}}\Big( \frac{\Vert \bw_{\mathsf{I}}\Vert^2}{\sqrt{d_i^2-y}} \,+ \Vert \bw_{\mathsf{I}}\Vert \Vert \mb{\varsigma}_{\mathsf{I}}\Vert \frac{\sqrt{ d_i- y^{1/2}}}{d_i}\,\Big)\bigg) 
\nonumber\\
&
+O_\prec \bigg(\frac{1}{N^{1+\varepsilon}}\Big(\frac{ \Vert \mb{\varsigma}_{\mathsf{I}}\Vert^2}{d_i}\,+ \Vert \bw_{\mathsf{I}}\Vert^2 \sum_{j\in \mathsf{I}^c} \frac{d_id_j}{(d_i-d_j)^2}\,\Big)\bigg) 
\end{align}
for some small constant $\varepsilon>0$, and 
\begin{align}
\Big( \Theta_{\bw_{\mathsf{I}}}^{\bw}, \Lambda_{\mb{\varsigma}_{\mathsf{I}}}^{\bw}, \{\Delta_{\bv_t}^{\bw}\}_{t\in \mathsf{I}}, \{\Pi_{\bv_j}^{\bw}\}_{j\in \mathsf{I}^c} \Big) \simeq \mathcal N \left(\mb{0}, A_{\mathsf{I}}^{\bw} + \kappa_4  \frac{d_i^2-y}{d_i^2} B_{\mathsf{I}}^{\bw}\right). \label{asymptotic normality}
\end{align} 
 Here  $\mathcal N(\mb{0}, A_{\mathsf{I}}^{\bw} + \kappa_4  \frac{d_i^2-y}{d_i^2} B_{\mathsf{I}}^{\bw})$ represents a Gaussian vector with mean $\mb{0}$ and covariance matrix $A_{\mathsf{I}}^{\bw} + \kappa_4  \frac{d_i^2-y}{d_i^2} B_{\mathsf{I}}^{\bw}$ with $A_{\mathsf{I}}^{\bw}$ and $B_{\mathsf{I}}^{\bw}$ defined in \eqref{def:covA} and \eqref{def:covB} respectively. 
 \end{thm}
\begin{rmk} Here we further explain how to read off the information of  the limiting behaviour of $\la \bw, {\rm{P}}_{\mathsf{I}}\bw \ra$ from the expansion in (\ref{weak_gf_decomp}). The first term in the RHS of (\ref{weak_gf_decomp}) is the first order deterministic estimator of  $\la \bw, {\rm{P}}_{\mathsf{I}}\bw \ra$ which can be biased in the high-dimensional case, especially when $d_i$ is a fixed constant independent of $N$. The second and the third terms in the RHS of (\ref{weak_gf_decomp}) are asymptotically normal, and the fourth  and fifth terms are (asymptotically) linear combinations of $\chi^2$, according to (\ref{asymptotic normality}).  These four terms together describe the limiting distribution of $\la \bw, {\rm{P}}_{\mathsf{I}}\bw \ra-\frac{d_i^2-y}{d_i(d_i+y)}\la \bw, {\rm{Z}}_{\mathsf{I}}\bw \ra$, after appropriate scaling. 
We further emphasize here that the sizes of the first five terms may not be comparable and it is not uniformly determined which one is the  leading term in all cases. Under  different choices of $d_i$'s and $\mathbf{w}$, say, the leading term may change. However, in any case, the two error terms in (\ref{weak_gf_decomp}) are always smaller than the sum of the second to the fifth terms with high probability. This can be checked easily from the sizes of the entries in the covariance matrix $A_{\mathsf{I}}^{\bw} + \kappa_4  \frac{d_i^2-y}{d_i^2} B_{\mathsf{I}}^{\bw}$. 
Hence, from the expansion  (\ref{weak_gf_decomp}), one can get the limiting distribution of $\la \bw, {\rm{P}}_{\mathsf{I}}\bw \ra-\frac{d_i^2-y}{d_i(d_i+y)}\la \bw, {\rm{Z}}_{\mathsf{I}}\bw \ra$ in all cases. In the next  remark, we show how to get the limiting distribution for some specific examples. 
\end{rmk}

\begin{rmk}\label{rmk_inside}
 If $\bw \in \mathrm{Span}\{\bv_t\}_{t\in \mathsf{I}}$, then $\bw_{\mathsf{I}}=\bw$ and $ \mb{\varsigma}_{\mathsf{I}} =\mb{0}$ (c.f. (\ref{081510}), (\ref{def:v_I_varsigma_I})). Hence,  $\Lambda_{\mb{\varsigma}_{\mathsf{I}}}^{\bw}=0$ and $\Delta_{\bv_t}^{\bw}=0$ for all $t\in \mathsf{I}$. The conclusion of Theorem \ref{mainthm} is reduced to 
\begin{align}
\la \bw, {\rm{P}}_{\mathsf{I}}\bw \ra =&\frac{d_i^2-y}{d_i(d_i+y)}+
\frac{1}{\sqrt{N(d_i^2- y)}}\Theta_{\bw}^{\bw}
-\frac{1 }{N}\sum_{j\in \mathsf{I}^c}\frac{d_id_j}{(d_i-d_j)^2} (\Pi_{\bv_j}^{\bw})^2\notag \\
&+O_\prec\Big(\frac{N^{-\varepsilon}}{\sqrt{N(d_i^2-y)}}+\frac{N^{-\varepsilon}}{N}\sum_{j\in \mathsf{I}^c}\frac{d_id_j}{(d_i-d_j)^2} \Big), \label{082021}
\end{align}
for some small $\varepsilon>0$, and $(\Theta_{\bw}^{\bw}, \{\Pi_{\bv_j}^{\bw} \}_{j \in \mathsf{I}^c})$ is asymptotically Gaussian with mean $\mb{0}$ and covariance matrix with entries given by the RHS of the following equations
\begin{align*}
&\mathrm{var}(\Theta_{\bw}^{\bw}) \doteq 2y \mathtt h(d_i)^2(1+y \mathtt h(d_i)^2) + \kappa_4  \frac{d_i^2-y}{d_i^2}\mathtt f(d_i)^2 s_4(\mb{w}), \\
&  \mathrm{var}(\Pi_{\bv_j}^{\bw}) \doteq \mathtt l(d_i)^2 + \kappa_4  \frac{d_i^2-y}{d_i^2}\mathtt l(d_i)^2 s_{2,2} (\bv_j, \mb{w}), \\
& \mathrm{cov}(\Theta_{\bw}^{\bw},\Pi_{\bv_j}^{\bw}) \doteq \kappa_4  \frac{d_i^2-y}{d_i^2} \mathtt f(d_i) \mathtt l(d_i) s_{1,3}(\bv_j,\bw),\\
& \mathrm{cov}(\Pi_{\bv_j}^{\bw},\Pi_{\bv_{\bar{j}}}^{\bw})\doteq \kappa_4  \frac{d_i^2-y}{d_i^2} \mathtt l(d_i)^2 s_{1,1,2}(\bv_j, \bv_{\bar{j}}, \bw),
\end{align*}
for $j, \bar{j}\in \mathsf{I}^c$. In particular, if $\kappa_4=0$, the limiting distribution of $\la \bw, {\rm{P}}_{\mathsf{I}}\bw \ra$ does not depend on the specific choice of $\bw \in \mathrm{Span}\{\bv_t\}_{t\in \mathsf{I}}$. Here we recall the notation $A_N\doteq B_N$ for $A_N=B_N(1+o_N(1))$.

If $\bw \in \mathrm{Span}\{\mathbf v_j \}_{j\in \lb 1,M\rb \setminus \mathsf{I}}$, then $\bw =0$ and thus \eqref{weak_gf_decomp} becomes
\begin{align}
\la \bw, {\rm{P}}_{\mathsf{I}}\bw \ra = \frac{\| \mb{\varsigma}_{\mathsf{I}} \|^2}{Nd_i}  \sum_{t\in \mathsf{I}} (\Delta_{\bv_t}^{\bw})^2 + O_\prec \Big(\frac{\Vert \mb{\varsigma}_{\mathsf{I}}\Vert^2}{N^{1+\varepsilon}d_i}\Big). \label{0820100}
\end{align}
Hence,  the  eigenvectors $\mb{\xi_i}$'s are unbiased to any direction orthogonal to the spike directions $\bv_i$'s.  
\end{rmk}

 \begin{rmk}\label{rmk.extension} In this remark, we discuss several extensions. In (\ref{19071802}), we assumed $T=\Sigma^{\frac12}$. But our result holds under much more general assumption on $T$. We can indeed extend our result to the matrix $Q=TXX^*T^*$ with $(M+k)\times N$ random matrix $X$ and $M\times (M+k)$ matrix $T$. As long as $k\in \mathbb{N}$ is  fixed and $TT^*=\Sigma$ satisfying (\ref{19071901}), our result remains true. Such an extension  has been discussed in Section 8 of \cite{bloemendal2016principal} for the joint eigenvalue-eigenvector distribution in the subcritical regime and large deviation for eigenvector in the supercritical regime. It is easy to verify the validity of the same extension for the joint eigenvalue-eigenvector distribution in the supercritical regime.  The discussion in \cite{bloemendal2016principal} relies on the rewriting  $Q=TXX^*T^*=\Sigma^{\frac12}YY^*\Sigma^{\frac12}$, where $Y= (I_M\;0)OX$. Here $0$ is the $M\times k$ zero matrix and $O$ is some $(M+k)\times (M+k)$ orthogonal matrix.  It will be clear that all our arguments in the proof of Theorem \ref{mainthm} are valid, as long as the isotropic local law of the matrix $XX^*$ (c.f. Theorem \ref{isotropic}) holds. So here it would be sufficient to have the isotropic local law for the matrix $YY^*$, which has been demonstrated in Theorem 8.1 of  \cite{bloemendal2016principal}. Further,  we excluded the point $1$ from the possible values of $y\equiv y_N$ in case $d_i\equiv d_i(N)$ diverges. This is essentially due to the same restriction in one technical result we need, Theorem \ref{thm.VESD}, which is established in \cite{XYY}.   We emphasize here that Theorem \ref{thm.VESD} is only used in the case when $d_i=d_i(N)$ grows with $N$, but not needed in the case when $d_i$ is fixed. Therefore, when $d_i$ is fixed, the result of Theorem  \ref{mainthm} and also Theorem \ref{thm.simple case} below also apply to all $y\equiv y_N\in (\tau_1,\tau_2)$. Also, our result shall hold without the restriction $|y-1|\geq \tau_0$ even if $d_i$ diverges. 
 But this extension will be left as future work.
 \end{rmk}
 
   Our second result is  the joint eigenvalue-eigenvector distribution, i.e.,  joint distribution of the outlying eigenvalues and the generalized components of the associated eigenvectors. We state it   for the case when $d_i$ is simple, i.e. $\mathsf{I}=\{i\}$. For simplicity, we abbreviate the notation $\mathsf{A}_{\{i\}}$ to $\mathsf{A}_i$ for $\mathsf{A}=\bw, {\rm P}, \Phi$, etc. Further, we  use $\{i\}^c$ to represent $\llbracket 1, r\rrbracket\setminus\{i\}$. 
  \begin{thm} \label{thm.simple case}
  Under the same assumptions as Theorem \ref{mainthm}, with  $\mathsf{I}=\{i\}$ and  $\bw_i= \la \bw, \bv_i\ra \bv_i$,  the conclusion of Theorem \ref{mainthm} for the generalized component $\la \bw, {\rm{P}}_i\bw \ra = |\la \bw, \bv_i\ra|^2$ holds.
    Additionally, there exists a random variable $\Phi_i$ such that the outlying eigenvalue admits the  expansion
\begin{align}\label{expansion.eigv}
\mu_i=1+d_i+y+\frac{y}{d_i}+\frac{\sqrt{d_i^2-y}}{\sn}\Phi_i+O_\prec \Big(\frac{\sqrt{d_i^2-y}}{N^{\frac 12+\varepsilon}}\Big),
\end{align}
for some small constant $\varepsilon>0$,  and 
\begin{align*}
\Big(\Phi_i, \Theta_{\bw_i}^{\bw}, \Lambda_{\mb{\varsigma}_i}^{\bw}, \Delta_{\bv_i}^{\bw}, \{\Pi_{\bv_j}^{\bw}\}_{j\in \{i\}^c} \Big) \simeq \mathcal N \left(\mb{0}, C_i^\bw \right).
\end{align*}
Here  $\mathcal N(\mb{0}, C_i^\bw)$ represents a Gaussian vector with mean $\mb{0}$ and covariance matrix $C_i^\bw$ of size $r+3$. The lower right $(r+2)\times (r+2)$ corner of $C_i^\bw$ is given by $A_i^{\bw} + \kappa_4  \frac{d_i^2-y}{d_i^2} B_i^{\bw}$ as in \eqref{def:covA} and \eqref{def:covB}. The entries of the first row of $C_i^\bw$ is given by the RHS of the following equations
\begin{align*}
&\mathrm{var}(\Phi_i) \doteq (1+ d_i^{-1})^2\Big(2+ \kappa_4 \frac{d_i^2-y}{d_i^2}s_4(\bv_i)\Big) ,\nonumber\\
& \mathrm{cov}(\Phi_i, \Theta_{\bw_i}^{\bw})\doteq 2y \mathtt h(d_i)^2 (1+d_i^{-1})  \la \bw, \bv_i\ra^2+ 
\kappa_4 \frac{d_i^2-y}{d_i^2} (1+d_i^{-1}) \mathtt f(d_i) s_{2,2}(\bw_i, \bv_i),\\
&\mathrm{cov}(\Phi_i, \Lambda_{\mb{\varsigma}_i}^{\bw}) \doteq \kappa_4 \frac{d_i^2-y}{d_i^2}\mathtt g(d_i)(1+d_i^{-1}) s_{1,1,2}(\mb{\varsigma}_i^0, \bw_i, \bv_i), \nonumber\\
& \mathrm{cov}(\Phi_i, \Delta_{\bv_i}^{\bw})\doteq \kappa_4 \frac{d_i^2-y}{d_i^2}\sqrt{\mathtt h(d_i)}(1+d_i^{-1}) s_{1,3}(\mb{\varsigma}_i^0, \bv_i) ,\\
&\mathrm{cov}(\Phi_i, \Pi_{\bv_j}^{\bw})\doteq \kappa_4 \frac{d_i^2-y}{d_i^2} \mathtt l(d_i)(1+d_i^{-1}) s_{1,1,2}(\mb{v}_j, \bw_i, \bv_i), \quad\text{for } j \in \{i\}^c.
\end{align*}
Here we recall the notation $A_N\doteq B_N$ for $A_N=B_N(1+o_N(1))$.
 \end{thm}

 \begin{rmk} \label{rmk. multiple eigenvalue}Here we remark that in the supercritical regime, a generalized CLT for the eigenvalues  has been established in \cite{BY08} previously, for fixed $d_i$'s which are away from $y^{1/2}$ by a constant order distance. When there is a multiple $d_i$ in the supercritical regime with multiplicity $|\mathsf{I}|$, it is known from \cite{BY08} that the corresponding eigenvalues $\{\mu_t\}_{t\in \mathsf{I}}$ will converge jointly to the eigenvalues of a $|\mathsf{I}|\times |\mathsf{I}|$ Gaussian matrix  GOE. Since it is not convenient to express the distribution of the eigenvalues of this fixed-dimensional GOE and their dependence with the the generalized components of $\mb{\xi_i}$'s, we are not going to state the joint eigenvalue-eigenvector distribution in the multiple case here. Nevertheless,   we will state the joint distribution of the generalized components of  $\mb{\xi_i}$ and all the matrix entries of this limiting GOE in Appendix \ref{appendix.proof} which equivalently describes the joint eigenvalue-eigenvector distribution; see Proposition \ref{joint ev-evector multiple}.
 \end{rmk}

\subsection{Proof strategy and novelties} \label{s.pfst}

Finally, here we remark that a related problem has been previously  studied in \cite{BDW} for the so-called matrix denoising model, where the distribution of the leading singular vectors of this model was studied. Due to the additive structure of this model, the distribution of the singular vector may depend on the structure of the deformation and the entire distribution of entries of the noise matrix (rather than their first 4 moments only), and  may not be Gaussian or Chi-square or linear combinations of them. 
Such a phenomenon is called non-universality, which exists in  the additive models \cite{BDW, CDM18}. However, such a phenomenon does not show up for the spiked covariance matrix, as one can see from Theorem \ref{mainthm}, where the distribution has a {\it Gaussian nature} in the sense that it is a polynomial of Gaussian variables. This is essentially due to the multiplicative structure of the spiked covariance matrix where the structure of the spikes are smoothed out by the random matrix $X$. In addition, we emphasize here that in \cite{BDW} the assumption of the strengths of the deformation, counterpart of $d_i$'s,  is much more limited than the assumption here. In \cite{BDW}, the strengths are assumed to be bounded and thus cannot grow with $N$, and also the strengths are away from the critical threshold by a constant order distance. Further, the strengths in  \cite{BDW} are assumed to be simple, and thus no multiplicity is allowed, and also distinct strengths are away from each other by a constant order. In addition, only the projection of the random singular vector onto the directions of the deformations are discussed therein. Finally, no joint distribution of eigenvalue and eigenvectors is obtained in \cite{BDW}. 
The generality of the settings and the results bring various technical problems but meanwhile they are highly motivated by the applications. In the sequel, we highlight several novelties from both theoretical and application point of views. 

Similarly to \cite{BDW}, we start with  Green function representations of the eigenvector and eigenvalue statistics. It turns out that both the eigenvector and eigenvalue statistics can be expressed (approximately)
in terms of some quadratic forms of the Green function of the matrix $XX^*$. Then a recursive moment estimate for the linear combinations of these quadratic forms leads to the joint distribution of the eigenvector and eigenvalues statistics. A key technical input is the isotropic local laws derived in \cite{bloemendal2014isotropic, knowles2017anisotropic}. Equipped with this general strategy, we face several new technical obstacles, in contrast to the discussions in \cite{BDW}.  First, since the $d_i$'s could be either very close to the critical threshold or diverging, and meanwhile could be equal or close to each other, 
the control of the sizes of the terms (especially the error terms) becomes much more delicate. One needs to keep tracking the dependence of the size of terms on $d_i-\sqrt{y}$, $d_i-d_j$ and $d_i$ carefully to conduct a unified analysis in all cases of $d_i$'s. Especially, in case that $d_i$ is diverging, one needs to exploit a hidden cancellation between two quadratic forms of the Green function, which is absent in case that $d_i$ is fixed. In order to see such a cancellation, one needs to adopt the recently established nearly optimal convergence rate of the so-called {\it eigenvector empirical spectral distribution (VESD)} in \cite{XYY}. Second, in contrast to \cite{BDW}, where only the projection to the direction of the deformation is considered, here we consider the projection to arbitrary direction. Especially, when one considers the projection to the orthogonal complement 
of the direction of the deformation, the size of the whole projection will degenerate to a smaller order. In order 
to study the fluctuation of the eigenvector projection onto arbitrary directions, including the direction orthogonal to the one of the spike, one needs to express the eigenvector projection in terms of the Green function up to a higher order term, and involve the higher order term in the recursive moment estimate, since it could be significant.  Third, the joint distribution of the eigenvalue and eigenvector statistics is obtained for the first time in the supercritical regime for the whole range of $d_i$. Such a complete result was even not available for the eigenvalue statistics only. 

All the theoretical novelties are well motivated by our applications. First, in most of mathematical work on spiked models, $d_i$ is assumed to be bounded. However, in many  popular statistical models such as
 the factor model \cite{BN, BN2,BN3, ONA}, 
$d_i\equiv d_i(N)$ could be diverging. We provide a unified result in the whole supercritical regime, no matter $d_i$ is close to the threshold or diverging. Practically, that means our results can be applied no matter the spike is weak or strong.   
Second, in the application part, we consider two hypothesis testing problems. The first is to test whether an eigenspace formed by any part of the spikes is equal to some given subspace, while the second is to test whether it is orthogonal to certain given subspace. Both questions are significant in the statistics literature. We raised two testing statistics for these two problems respectively, and the limiting distribution of the first statistic relies on our joint eigenvalue-eigenvector distribution in case $\bw \in \mathrm{Span}\{\bv_t\}_{t\in \mathsf{I}}$, while  the limiting distribution of the second statistic relies on the joint eigenvalue-eigenvector distribution in case $\bw \in \mathrm{Span}\{\mathbf v_j \}_{j\in \lb 1,M\rb \setminus \mathsf{I}}$. This explains the necessity for us to derive the distribution of the projection onto general directions. Third,  in two applications, if we construct the testing statistics, using the result in 
Theorem \ref{mainthm} solely, the limiting distribution of the statistics will contain the parameter $d_i$, which is normally unknown in real application. Hence, in order to construct  {\it adaptive} statistics which do not depend on the unknown parameter    $d_i$, we use a plug-in estimator of $d_i$ which is given in terms of $\mu_i$. Then, in order to derive the distribution of these adaptive  statistics, we have to establish the joint eigenvalue-eigenvector distribution, as what we have in Theorem \ref{thm.simple case}, and its multiple extension in Proposition \ref{joint ev-evector multiple}.

{\section{Statistical inference for principal components} \label{s.application}
In this section, we apply our results and some of their variants to the statistical applications regarding the inference of principal components. We focus our discussion on the hypothesis testing regarding the eigenspaces of covariance matrices.  Eigenspaces of covariance matrices are important in many statistical methodologies and computational algorithms. A lot of efforts have been made to infer the eigenspace of the covariance matrices in the setting $M \ll N$, for instance, see \cite{LeCamtest, KLsplit, bootsfr, SF, SS}. 

In this section, we consider a generic index set $\mathcal{I}\subset\llbracket 1, r_0\rrbracket$, which may contain indices for both simple and multiple $d_t$'s.  Further, we set 
 \begin{align}
Z_{\mathcal{I}}= \sum_{t\in \mathcal{I}} \mb{v}_t\mb{v}_t^*.
 \end{align}
 We remark here that $ Z_{\mathcal{I}}$  shall be regarded as an extension of $ Z_{\mathsf{I}(i)}$  defined in (\ref{def of ZI}), in the sense that the former may be constituted of $\mb{v}_t$'s 
 associated with distinct $d_t$'s. 

Specifically,  in the literature, researchers are particularly interested in testing the following hypothesis: for $\mathcal{I} \subset \lb1, r_0\rb,$
\begin{equation} \label{eq_spectralprojection}
\mathbf{H}_0^{(1)}: \ Z_{\mathcal{I}}=Z_0 \ \text{vs} \ \mathbf{H}_a^{(1)}:\ Z_{\mathcal{I}} \neq Z_0
\end{equation}   
for a given projection $Z_0.$ For the general alternatives $\mathbf{H}_a^{(1)}$, we are particular interested in testing whether $Z_0$ is in the complement of $Z_{\mathcal{I}}.$ Specifically, the hypothesis testing problem can be formulated as 
\begin{equation}\label{eq_orthogobalteset}
\mathbf{H}_0^{(2)}: Z_{\mathcal{I}} \perp Z_0 \  \text{vs} \ \mathbf{H}_a^{(2)}:\ Z_{\mathcal{I}} \not\perp Z_0.
\end{equation} 
Note that (\ref{eq_orthogobalteset}) is the complement of the test considered in \cite{tyler1981, tyler1983} and hence it can be used to study the alternative in \cite{tyler1981,tyler1983}.  

 In Section \ref{sec.statanddist}, we propose accurate and powerful statistics for the aforementioned two hypothesis testing problems (\ref{eq_spectralprojection}) and (\ref{eq_orthogobalteset}) in the high dimensional regime (\ref{eq_highdimensionalregime}).  We construct test statistics using some plug-in estimators and then derive their distributions utilizing the joint distribution of the eigenvalues and eigenvectors established in Section \ref{sec.mainresult} and its extensions in Appendix \ref{appendix.proof}. For (\ref{eq_spectralprojection}) and (\ref{eq_orthogobalteset}), the plug-in estimators are nonlinear shrinkers of the sample eigenvalues. Consequently, the proposed statistics are adaptive to $d_i$'s.  We mention that this methodology can be potentially applied to perform statistical inference  and build up confidence intervals for other important statistics related to the principal components. For instance, the loadings of principal components \cite{JI}, the shrinkage of eigenvalues \cite{shrinkage}, the number of spikes \cite{ED, PYspike}, the estimation of eigenvectors \cite{Monasson_2015} and the invariant estimator for covariance matrices \cite[Section 6]{BBPA}. These applications will be studied  in the future.

\subsection{Test statistics and their asymptotic distributions}\label{sec.statanddist} In this section, we propose statistics to test (\ref{eq_spectralprojection}) and (\ref{eq_orthogobalteset}). We start with  (\ref{eq_spectralprojection}). 
In what follows, we construct a data-dependent statistic to address the high dimensional issue. Denote $Z_{0}= \sum_{i\in \mathcal{I}} \mb{u}_i\mb{u}_i^*$.  
We  first study 
\begin{align}\label{eq_teststat}
\mathcal{T}:=\sum_{i \in \mathcal{I}} \Big(\langle \bm{u}_i,  {\rm P}_{\mathcal{I}} \bm{u}_i\rangle - \vartheta(\widehat{d}_i)\Big),
\end{align}
where 
\begin{align}
{\rm P}_{\mathcal{I}}=\sum_{i \in \mathcal{I}} \mb{\xi}_i \mb{\xi}_i^* \label{def of PI general}, \quad   \vartheta(d)= \frac{d^2-y}{d(d+y)}
\end{align}
and $\widehat{d}_i=\gamma(\mu_i)$ is a nonlinear shrinkage of the sample eigenvalues  denoted as
\begin{equation}\label{eq_diconsistent}
\widehat{d}_i=\gamma(\mu_i),\ \gamma(x)=\frac{1}{2}(-y+x-1)+\frac{1}{2} \sqrt{(-y+x-1)^2-4y}.
\end{equation}
 We remark here that $ {\rm{P}}_{\mathcal{I}}$  shall be regarded as an extension of $ {\rm{P}}_{\mathsf{I}(i)}$  defined in (\ref{def of PI}), in the sense that the former may be constituted of $\mb{\xi}_t$'s 
 associated with distinct $d_t$'s.  Further, we remark here, according to the definition  in (\ref{eq_teststat}), the statistic $\mathcal{T}$ does not depend on the specific choice of the basis $\{\mb{u}_i\}$ of $Z_0$. Hence, we have a freedom to choose any basis $\{\mb{u}_i\}$ of $Z_0$ in the sequel.

Since we are studying general $\mathcal{I}$ in this section, the indices in $\mathcal{I}$ may not belong to the same multiple $d_t$'s. To facilitate our discussion in the sequel, we do a decomposition of  $\mathcal{I}$ into subsets with each consisting of the indices for one multiple (or simple) $d_t$. 
 For $\mathcal{I}= \{i_1, \cdots, i_{r_*}\}\subset \lb 1, r_0 \rb,$ we assume that $\mathcal{I}=\bigcup_{k=1}^\ell \mathcal{I}_k$  for some fixed integer $\ell$ such that $\mathcal{I}_{k}\cap \mathcal{I}_{j}=\varnothing$ for $k \neq j \in \lb 1, \ell\rb.$ We assume that (\ref{asd2}) holds for all the $d_i, i\in \mathcal{I}$. For each $i\in \mathcal{I}$, there is a $k_i\in \lb 1, \ell\rb$, such that $i\in \mathcal{I}_{k_i}$. 
 Moreover, we suppose that for $1 \leq k \leq \ell$, $d_t, t\in  \mathcal{I}_k$ are all the same,  and $d_i\neq d_j$ for $i \in \mathcal{I}_{k_i}, j \notin \mathcal{I}_{k_i}$. Note that by definition $\mathcal{I}_{k_i}\equiv \mathsf{I}(i)$ (c.f.  Assumption \ref{supercritical}).  Further note that $\ell=r_*$ corresponds to the case that all the spikes in $\mathcal{I}$ are simple , $\ell=1$ corresponds to the case that all the spikes are equal and $1< \ell <r_*$ corresponds to a mixture case.

 For the brevity of the discussion in the first test problem (\ref{eq_spectralprojection}), we further restrict ourselves to the case satisfying the following assumption. 
 \begin{asm} \label{asm.082021} Let the index set $\mathcal{I}\subset\llbracket 1, r_0\rrbracket$ be defined above. We assume that for any $i\in \mathcal{I}$, the following inequality holds
 \begin{align}
 \frac{1}{N} \sum_{j\in (\mathsf{I}(i))^c}\frac{d_id_j}{(d_i-d_j)^2}\leq  N^{-\varepsilon}\frac{1}{\sqrt{N}(d_i^2-y)}.  \label{082020}
 \end{align} 
 for some small but fixed $\varepsilon>0$. Here $i\in \mathcal{I}_{k_i}\equiv \mathsf{I}(i)$ for some $k_i\in \lb 1, \ell\rb$.
 \end{asm}
 \begin{rmk} 
 We remark here that the inequality (\ref{082020}) insures that the $\chi^2$ terms in (\ref{082021}) are suppressed by the Gaussian term. We impose such a condition in order to simplify the discussion in the application part, thanks to the simplicity of the Gaussianity. But our result can be applied without this additional assumption. In the general case, we need to work with a linear combination of Gaussian and $\chi^2$ variables. For brevity, we omit such a general discussion and leave it to the future work.  
 \end{rmk}

   We record the results regarding the asymptotic distribution of (\ref{eq_teststat}) in the following theorem and postpone its proof to Appendix \ref{appendix.proof}.

\begin{thm}\label{thm_firststatdist} Suppose that Assumptions \ref{assumption} ,  \ref{supercritical},  and the setting (\ref{19071802}) hold.  In case $d_i\equiv d_i(N)\to \infty$ as $N\to\infty$ for some $i\in \mathcal{I}$, we additionally assume that $|y-1|\geq \tau_0$ for some small but fixed $\tau_0>0$. Suppose that $\mathbf{H}_0^{(1)}$ of (\ref{eq_spectralprojection}) and Assumption \ref{asm.082021} hold. 
For the statistic  (\ref{eq_teststat}), we have that 
\begin{equation}\label{eq_statone}
\frac{\sqrt{N} \mathcal{T}}{\sqrt{\operatorname{\mathtt{V_{1}}}(\mathbf{d}_{\mathcal{I}})}} \simeq  \mathcal{N}(0,1),
\end{equation}
where $\mathtt{V}_{1}(\bm{d}_\mathcal{I}), \bm{d}_\mathcal{I}=(d_{i_1}, \cdots, d_{i_{r_*}})$ is defined as 
\begin{equation*}
\mathtt{V}_{1}(\bm{d}_\mathcal{I}):= \bm{\alpha}^* C_{\mathcal{I}} \bm{\alpha}.  
\end{equation*}
Here $\bm{\alpha}=(\alpha_1,\cdots, \alpha_{2 r_*})^* \in \mathbb{R}^{2r_*}$ is defined as
\begin{equation*}
\alpha_k= \left\{
\begin{array}{cc}
-{y(d_{i_k}^2+2d_{i_k}+y)}{(d_{i_k}+y)^{-2}(d_{i_k}^2-y)^{-\frac12}},& 1 \leq k \leq r_*; \\  (d_{i_{k-r_*}}^2-y)^{-\frac12}, & r_*+1 \leq k \leq 2 r_*,
\end{array}
\right.
\end{equation*}
and $C_{\mathcal{I}}$ is a positive definite matrix of dimension $2r_*$ and explicitly defined in Proposition \ref{prop.simples}. Particularly,  when all the spikes $d_t, t \in \mathcal{I}$ are equal to $d_e$,  i.e., $\mathcal{I}=\mathsf{I}(i)$ for some $i$, we have that

\begin{align}\label{def.V1(d_I)}
& \mathtt{V}_{1}( \bm{d}_\mathcal{I})=
2(d_e^2-y)^{-1}\Big( y\mathtt{h}(d_e)^2 |\mathcal{I}| - \frac{y(d_e^2+2d_e+y)(1+d_e)}{(d_e+y)^2d_e}\Big)^2 \notag\\
&\qquad\quad  + 2(d_e^2-y)^{-1} \Big(y\mathtt{h}(d_e)^2|\mathcal{I}|  + y^2\mathtt{h}(d_e)^4(|\mathcal{I}|- |\mathcal{I}| ^2)  \Big) \nonumber\\
&\qquad \quad + \kappa_4 \Big( \frac{\mathtt{f}(d_e)}{d_e} -  \frac{y(d_e^2+2d_e+y)(1+d_e)}{(d_e+y)^2d_e^2}\Big)^2 \sum_{k, t\in \mathcal{I}}s_{2,2}(\bv_k,\bv_t),
\end{align}
where $\mathtt{h}(\cdot)$ and $\mathtt{f}(\cdot)$ are defined in (\ref{eq_keyeqone}) and (\ref{eq_keyeqtwo}).
\end{thm}

By Theorems \ref{thm_firststatdist} and \ref{thm.simple case}, we can construct a pivotal statistic. Denote 
\begin{equation}\label{eq_finalstatistict1s}
\mathbb{T}_{1}=\frac{\sqrt{N} \mathcal{T}}{\sqrt{\texttt{V}_{1}(\widehat{\bm{d}}_{\mathcal{I}})}}, \ \widehat{\bm{d}}_{\mathcal{I}}=(\gamma(\mu_{i_1}), \cdots, \gamma(\mu_{i_{r_*}})). 
\end{equation}
We mention that  (\ref{eq_finalstatistict1s}) is  adaptive to the $d_i$'s by utilizing their estimators (\ref{eq_diconsistent}). We summarize the distribution of $\mathbb{T}_1$ in the corollary below, whose proof will also be postponed to Appendix \ref{appendix.proof}.

\begin{cor}\label{cor.082001}
Under the assumptions of Theorem \ref{thm_firststatdist}, we have that 
\begin{equation*}
\mathbb{T}_{1} \simeq \mathcal{N}(0,1). 
\end{equation*}
\end{cor}

Since $\mathbb{T}_{1}$ is asymptotically pivotal, we will use (\ref{eq_finalstatistict1s}) as our statistic for the testing of (\ref{eq_spectralprojection}). For an illustration, we record the behavior of our statistic for a single  spike model (i.e., $r_0=r_*=1$) in Figure \ref{fig_toy1size}. The more general and extensive simulations will be conducted in Section \ref{sec_simu}. We find that under the null hypothesis of (\ref{eq_spectralprojection}), our proposed statistic is close to $\mathcal{N}(0,1)$ for different values of $d$ and hence it is suitable for the hypothesis testing problem (\ref{eq_spectralprojection}).  We mention that even though we have not justified the case $d_i$ diverges under the assumption $y=1$ theoretically, our statistic is still accurate and powerful according to empirical illustrations for this case.  Hence, in the sequel, we also present the simulation results for the case $y=1$ for more extensive simulation study.

\begin{figure}[!ht]
 \hspace*{-1.1cm}
\includegraphics[width=13cm,height=6.5cm]{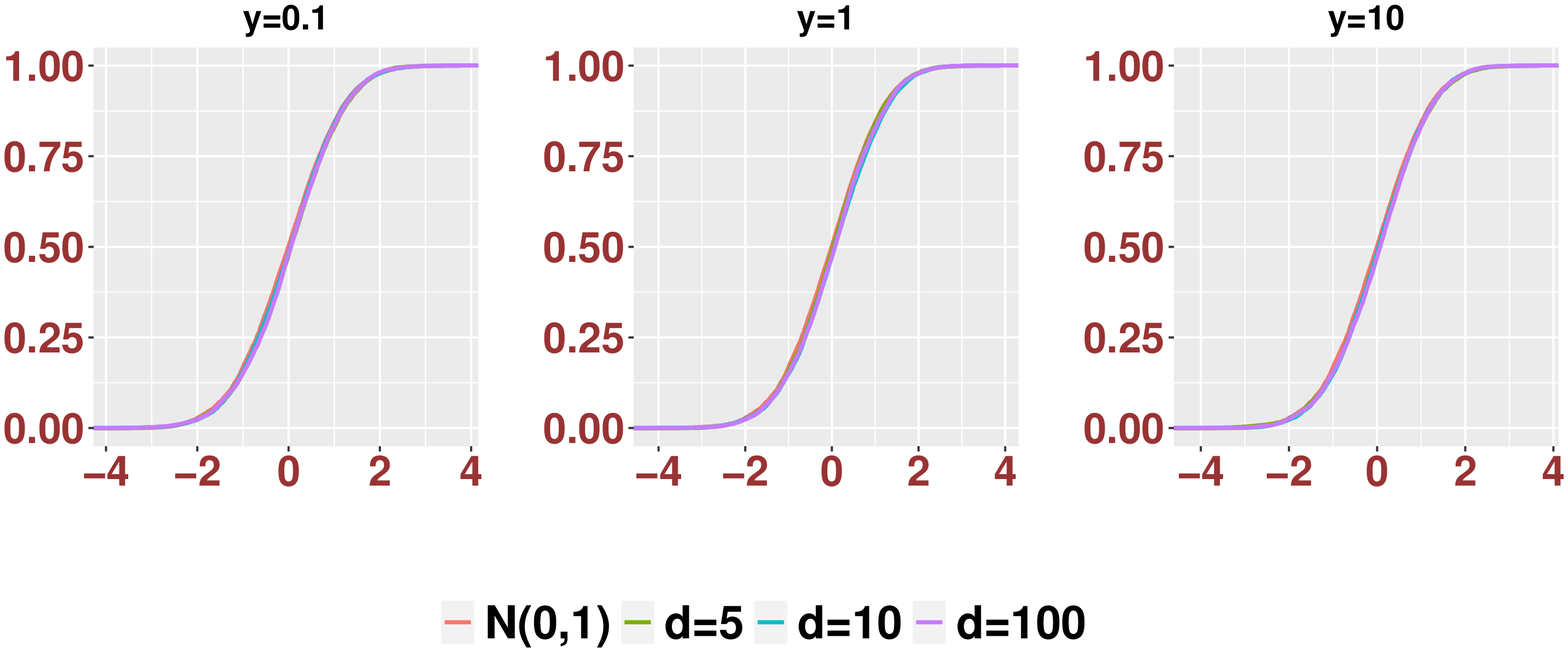}
\caption{ \footnotesize Simulated empirical cumulative distribution function (ECDF) for the proposed statistic (\ref{eq_finalstatistict1s}) under null of (\ref{eq_spectralprojection}) with $r_0=r_*=1.$ Here, the spiked covariance matrix is denoted as $\Sigma=\operatorname{diag}(d+1,1,\cdots,1)$ and we use the statistic $\sqrt{N}(|\langle \bm{\xi}_1, \bm{e}_1 \rangle|^2-\vartheta(\widehat{d}))/\sqrt{\mathtt{V}_1(\widehat{d}\,)},$ where
 $\mathtt{V}_1(d)=\frac{1}{2}\mathtt{V}_1(d,d)$;  see Example \ref{ex_test1_simple} for the definition of $\mathtt{V}_1(\cdot,\cdot)$.
 Here $N=500$ and we report our results based on 8,000 simulations with Gaussian random variables. }
\label{fig_toy1size}
\end{figure}

In what follows, we provide a few  examples with explicit formulas of  $\texttt{V}_{1}(\cdot)$.  These will be used for the simulations in Section \ref{sec_simu}.  
\begin{ex}\label{ex_test1_simple}
We consider that $\mathcal{I}=\{1,2\},$ both $d_1$ and $d_2$ are simple and $\bm{v}_i=\bm{e}_i, i=1,2.$ In this case, we have that 
\begin{align*}
 \mathtt{V}_{1}(d_1, d_2)= \sum_{i=1}^2 \mb{\alpha}^*
 \left[
 \begin{array}{cc}
 A_{11}(d_i) & A_{12}(d_i) \\
  A_{21}(d_i) & A_{22}(d_i) 
 \end{array} 
 \right]
 \mb{\alpha},
 \end{align*}
 where
 \begin{align*}
 &\mb{\alpha}  = \Big( { -} \frac{y(d_i^2+2d_i+y)}{(d_i+y)^2(d_i^2-y)^{\frac12}},  (d_i^2-y)^{-\frac12}\Big) ^* \notag\\
& A_{11}(d_i) = 2(1+d_i^{-1} )^2 +  \kappa_4 (1-yd_i^{-2})(1+d_i^{-1})^2,
\nonumber\\
  & A_{12}(d_i) = 2y \mathtt{h}(d_i)^2 (1+d_i^{-1} ) + \kappa_4 (1-yd_i^{-2}) \mathtt{f}(d_i)(1+d_i^{-1}),\nonumber\\
  & A_{22}(d_i) = 2y \mathtt{h}(d_i)^2 (1+ y \mathtt{h}(d_i) ^2 ) + \kappa_4 (1-yd_i^{-2}) \mathtt{f}(d_i)^2,
   \end{align*}
and  the functions $\mathtt{h}(\cdot)$ and $\mathtt{f}(\cdot)$ are defined in (\ref{eq_keyeqone}) and (\ref{eq_keyeqtwo}).
\end{ex}

\begin{ex}\label{ex_test1_equal}
We consider the case that $\mathcal{I}=\{1,2\}$ with $d_1=d_2=d$ and $\bm{v}_i=\bm{e}_i, i=1,2.$ In this case, we have that 
\begin{align*}
\mathtt{V}_{1}(d, d)=&2(d^2-y)^{-1}\Big( 2y\mathtt{h}(d)^2 - \frac{y(d^2+2d+y)(1+d)}{(d+y)^2d}\Big)^2 \notag\\
&\qquad\quad  + 2(d^2-y)^{-1} \Big(2 y\mathtt{h}(d)^2  -2 y^2\mathtt{h}(d)^4  \Big) \nonumber\\
&\qquad \quad + \kappa_4 \Big( \frac{\mathtt{f}(d)}{d} -  \frac{y(d^2+2d+y)(1+d)}{(d+y)^2d^2}\Big)^2 \sum_{k, t\in \mathcal{I}}s_{2,2}(\bv_k,\bv_t).
\end{align*}
\end{ex}

Next, we consider the hypo{thesis testing (\ref{eq_orthogobalteset}). In this case, we further assume that the true model or the population matrix $\Sigma$ only contains supercritical spikes, i.e., 
\begin{align}
r_0=r.  \label{082501}
\end{align}
It will be seen that if there exist  subcritical spikes, one will  need to provide a plug-in estimator of the subcritical $d_i$'s in order to  raise a testing statistic which is adaptive to all the spiked eigenvalues.
However, it is well-known now that an effective detection of subcritical $d_i$'s  based on $\mu_i$'s is impossible in general, unless one knows additional information such as the  structure of $\mb{v}_i$'s.  
 And also, indeed,  in many applications, $d_i$'s are very large and even divergent, and thus are certainly supercritical. Hence, in the sequel, we will focus on the case when (\ref{082501}) is satisfied. 
  Suppose that in this case $Z_{0}= \sum_{j\in \mathcal{J}} \mb{u}_j\mb{u}_j^*$ for some fixed index set $\mathcal{J}$ and $\{\bu_j\}_{j\in \mathcal{J}}$ is a family of orthonormal vectors.

We employ that 
\begin{equation}\label{eq_finalstat}
\mathbb{T}_{2}=\sum_{i \in \mathcal{I}, j\in\mathcal{J}} \langle  \mb{\xi}_i, \bm{u}_j \rangle^2.
\end{equation}

The asymptotic distribution of $\mathbb{T}_2$ is recorded in the following theorem. It turns out that its asymptotic distribution coincides with linear combinations of $\chi^2$ variables. 
 For convenience, we  first define $\mb{d}:= (d_1, \dots, d_r)$ and 
\begin{align*}
\mathtt{q}(\mb{d}):=  \max_{i\in\mathcal{I}, j \in \mathcal{J}} \sum_{k\in \lb 1,M\rb\setminus\mathcal{I}} \mathtt{h}(d_i)\frac{d_i(d_k+1)}{(d_i-d_k)^2} \langle \bu_j,\bv_k\rangle^2
\end{align*}
  which depends on the subspace $Z_0$ and all the  $d_i$'s for $i\in\lb 1, r\rb$. Here $d_{r+1}=\cdots d_M=0$. We emphasize that  all  $d_i$'s for $i\in\lb 1, r\rb$ satisfy  \rf{asd} and \rf{asd2} in this part.

\begin{thm}\label{thm_secondtestatistics} Suppose that Assumptions \ref{assumption} ,  \ref{supercritical},  and the settings (\ref{19071802}) and (\ref{082501}) hold.  In case $d_i\equiv d_i(N)\to \infty$ as $N\to\infty$ for some $i\in \mathcal{I}$, we additionally assume that $|y-1|\geq \tau_0$ for some small but fixed $\tau_0>0$.
Suppose that $\bm{H}_0^{(2)}$ of (\ref{eq_orthogobalteset}) holds true.
For the statistic $\mathbb{T}_2$ defined in \rf{eq_finalstat}, we have 
\begin{align}\label{asymp.test.ortho}
\frac{N \mathbb{T}_{2}}{\mathtt{q}(\mb{d})} \simeq \frac{\mb{g}^*\mb{U}\mb{g}}{\mathtt{q}(\mb{d})},
\end{align}
where $\mb{g}\in \mathbb{R}^{|\mathcal{I}||\mathcal{J}|}$, $\mb{g} \sim \mathcal{N}(0,{I}_{|\mathcal{I}||\mathcal{J}|})$, and $\mb{U} \equiv \mb{U}(\mb{d}) $ is a symmetric matrix of dimension $|\mathcal{I}||\mathcal{J}|$ defined explicitly in Proposition \ref{prop.simples}.
\end{thm}

\begin{rmk}
We remark here that in \rf{asymp.test.ortho}, we add the factor $1/\mathtt{q}(\mb{d})$ on both sides to scale them to order one random variables.   Since we use the notation `` $\simeq$ " which is defined in Definition \ref{defn_asymptotic} , we have to require the sequence of variables on both sides to be tight.
\end{rmk}

The results of Theorem \ref{thm_secondtestatistics}, especially $\mathtt{q}(\mb{d})$ and $\mathbf{U},$ still contain the values of $d_i, i \in \mathcal{I}$ and also the other nonzero spikes $d_j, j\in \mathcal{I}^c$ which are all supercritical (c.f. (\ref{082501})) so that  we can  use   (\ref{eq_diconsistent}) to estimate them all.
To construct a data-dependent statistic, we can use the plug-in estimator (\ref{eq_diconsistent})  to generate critical values of the hypothesis testing (\ref{eq_orthogobalteset}) using the samples. Recall $\widehat{\mb{d}}= (\widehat{d}_1, \dots,\widehat{d}_r)$.

\begin{cor}\label{cor_simulateddistribution} Under the assumptions of Theorem \ref{thm_secondtestatistics}, we have that
\begin{equation*}
\frac{N\mathbb{T}_{2}}{\mathtt{q}(\widehat{\mb{d}}\,)} \simeq \frac {\mb{g}^*\widehat{\mb{U}}\mb{g}}{\mathtt{q}(\widehat{\mb{d}}\,)},
\end{equation*}
where $\widehat{\mb{U}} = {\mb{U}} (\widehat{\mb{d}}\,)$. 
\end{cor}

}
We can use our statistic (\ref{eq_finalstat})  with the critical values generated from Corollary \ref{cor_simulateddistribution} to study the hypothesis testing problem (\ref{eq_orthogobalteset}). 
 Here we shall point out that although our statistic  $N\mathbb{T}_{2}/\mathtt{q}(\widehat{\mb{d}}\,) $ is adaptive to the $d_i$'s, it is nevertheless dependent of $\{\mb{v}_i\}_{i\in \llbracket 1, r\rrbracket\setminus \mathcal{I}}$ and also  a $\kappa_4$ term which involves some $\{\mb{v}_i\}_{i\in \mathcal{I}}$-dependent parameters of the form $s_{1,1,1,1}(\cdot, \cdot, \cdot,\cdot)$; see the definition of $\mb{U}$ in Proposition \ref{prop.simples}. Hence, first of all, we shall only apply our statistic $N\mathbb{T}_{2}/\mathtt{q}(\widehat{\mb{d}}\,) $ in case either $\{\mb{v}_i\}_{i\in \llbracket 1, r\rrbracket\setminus \mathcal{I}}$ is known a priori, or $\mathcal{I}=\llbracket 1, r\rrbracket$ such that the set $\{\mb{v}_i\}_{i\in \llbracket 1, r\rrbracket\setminus \mathcal{I}}$ is empty. In practice, this restriction is mild and  fits the following real scenario: if some of the $\mb{v}_i$'s are already known, we only need to do inference for those unknown $\mb{v}_i$'s, while if none of the $\mb{v}_i$ is known a priori, we consider the inference for all $\{\mb{v}_i\}_{i\in \llbracket 1, r\rrbracket}$ together. Certainly, it is also natural to consider a part of $\mb{v}_i$'s even if none of them is known, as  what we did in the test (\ref{eq_spectralprojection}). Nevertheless, due to the restriction of the theoretical result, we focus on the aforementioned scenario, which is more restricted but still very natural.  Second, the unknown $\kappa_4$ term will be absent in case we consider the Gaussian matrix $X$ which is often the case in reality. Hence, in Gaussian case, we can apply our statistic directly if we restrict ourself to the aforementioned scenario of $\mb{v}_i$'s. Nevertheless, for the reader's reference,  we also present our simulation study in the Appendix for the two-point case, as if the additional parameter, the $\kappa_4$ term, was known a priori.  Here, we remark that the restriction of scenario to apply our theoretical result for the test (\ref{eq_orthogobalteset}), which is not necessary for (\ref{eq_spectralprojection}), is actually quite reasonable. Note that, in (\ref{eq_spectralprojection}), the necessary parameters from $\{\mb{v}_i\}_{i\in\mathcal{I}}$ are completely fixed by the  null hypothesis, and the distribution of our statistic (under the null hypothesis) can be expressed in terms of these given parameters. However, in the test (\ref{eq_orthogobalteset}), our null hypothesis is $Z_{\mathcal{I}} \perp Z_0$. In case that the rank of the projection of the given $Z_0$ is small, as a low-rank subspace living in the complement of $Z_0$, $Z_{\mathcal{I}}$ can have many choices and thus there is a big uncertainty on the unknown parameters of $Z_{\mathcal{I}}$  which cannot be fixed by $Z_0$.

For an illustration, we record the behavior and power of our statistic for a single spiked model (i.e., $r_0=r=1$) in Figure \ref{fig_toy2size}.
We find that our proposed statistic is close to Chi-square distribution with one degree of freedom for various values of $d$ and hence it can be applied for testing (\ref{eq_orthogobalteset})  for this single spike model. For more general case, the asymptotic distribution of (\ref{eq_finalstat}) is a linear combination of Chi-square distributions. We will conduct extensive simulations in Section \ref{sec_simu}.

\begin{figure}[!ht]
\hspace*{-1.0cm}
\includegraphics[width=13cm,height=6.5cm]{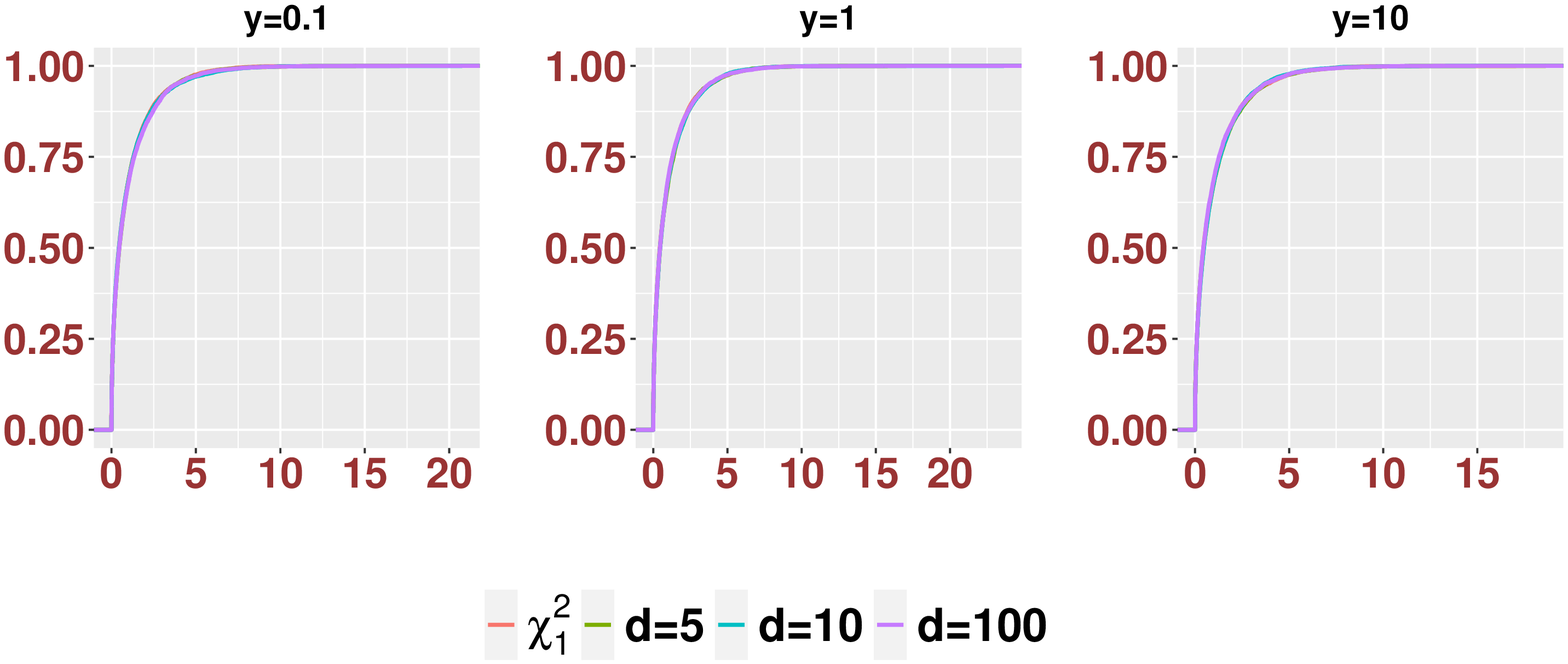}
\caption{\footnotesize Simulated empirical cumulative distribution function (ECDF) for the proposed statistic under null of (\ref{eq_orthogobalteset}) with $r_0=1.$ Here, the spiked covaraince matrix is denoted as $\Sigma=\operatorname{diag}\{d+1,1,\cdots,1\}$ and $Z_0=\mathbf{e}_3$ in (\ref{eq_orthogobalteset}).  We use the statistic $N\widehat{d}(\widehat{d}+y)|\langle \bm{\xi}_1, \bm{e}_3 \rangle|^2/(\widehat{d}+1).$ Here $\widehat{d}=\gamma(\mu_1), N=500$ and we report our results based on 8,000 simulations with Gaussian random variables. }
\label{fig_toy2size}
\end{figure}

We next consider a few examples to specify the asymptotic distribution stated in Theorem \ref{thm_secondtestatistics} and the results will be used in Section \ref{sec_simu}. 
\begin{ex}\label{ex_orthogonalsimple}
We consider that $r_0=3$,  $\mathcal{I}=\{1,2\},$ $d_i, i=1,2,3$ are simple and satisfy \rf{asd}, \rf{asd2} and $\bm{v}_i=\bm{e}_i, i=1,2,3$ and $Z_0=\bm{e}_3 \bm{e}_3^*+\bm{e}_4 \bm{e}_4^*.$ In this case, since  $\bm{\varsigma}_{\{1\}}^{\bm{e}_3}=\frac{d_1\sqrt{d_3+1}}{d_1-d_3}\bm{e}_3$, $\bm{\varsigma}_{\{2\}}^{\bm{e}_3}=\frac{d_2\sqrt{d_3+1}}{d_2-d_3}\bm{e}_3$  and $\bm{\varsigma}_{\{1\}}^{\bm{e}_4}=\bm{\varsigma}_{\{2\}}^{\bm{e}_4}=\bm{e}_4$ and $s_{1,1,1,1}(\bm{e}_{i_1}, \bm{e}_{i_2},  \bm{e}_{j_1}, \mb{e}_{j_2})=0$ for $i_{1,2}= 1,2, j_{1,2}= 3,4 $. Further, $\mathtt{q}(\mb{d})= \mathtt{q}(d_1,d_2, d_3)= \max\{\frac{(d_3+1) d_1\mathtt{h}(d_1)}{(d_1-d_3)^2}, \frac{\mathtt{h}(d_1)}{d_1}, \frac{(d_3+1) d_2\mathtt{h}(d_2)}{(d_2-d_3)^2}, \frac{\mathtt{h}(d_2)}{d_2}\}$. We have that the statistic $N \mathbb{T}_{2}/\mathtt{q}(\mb{d})$ will be asymptotically distributed as 
{ 
\begin{equation*}
\frac{1}{\mathtt{q}(\mb{d})}\, \bm{g}^* \operatorname{diag}\left( \frac{(d_3+1) d_1\mathtt{h}(d_1)}{(d_1-d_3)^2}, \frac{\mathtt{h}(d_1)}{d_1}, \frac{(d_3+1) d_2\mathtt{h}(d_2)}{(d_2-d_3)^2}, \frac{\mathtt{h}(d_2)}{d_2}  \right) \bm{g}, 
\end{equation*}
}
where $\bm{g} \in \mathbb{R}^4$ is a Gaussian random vector such that $\bm{g} \sim \mathcal{N}(\bm{0}, \rm{I}_4).$  
\end{ex}

\begin{ex} We consider that $r_0=3$, $\mathcal{I}=\{1,2\}, d_1=d_2=d, d_3$ is distinct from $d$ by a distance of order 1, and $\mathbf{v}_i=\bm{e}_i, i=1,2,3.$ In this case, $\mathtt{q}(\mb{d})= \max\{ \frac{(d_3+1) d\mathtt{h}(d)}{(d-d_3)^2}, \frac{\mathtt{h}(d)}{d}\}$. We have that the statistic $N\mathbb{T}_{2}/\mathtt{q}(\mb{d})$ will be asymptotically distributed as 
\begin{equation*}
\frac{1}{\mathtt{q}(\mb{d})}\, \bm{g}^* \operatorname{diag}\left( \frac{(d_3+1) d\, \mathtt{h}(d)}{(d-d_3)^2}, \frac{\mathtt{h}(d)}{d}, \frac{(d_3+1) d\, \mathtt{h}(d)}{(d-d_3)^2}, \frac{\mathtt{h}(d)}{d}  \right) \bm{g},
\end{equation*}
where $\bm{g} \in \mathbb{R}^4$ is a Gaussian random vector such that $\bm{g} \sim \mathcal{N}(\bm{0}, \rm{I}_4).$
\end{ex}

\subsection{Simulation studies}\label{sec_simu}
 In this section, we  perform extensive Monte Carlo simulations to study the finite-sample
accuracy and power of our proposed statistics. For (\ref{eq_spectralprojection}), we not only report our results but also compare them with the existing statistics in the literature. We will call our statistics as $\texttt{Fr-Adaptive}.$

We start with (\ref{eq_spectralprojection}). In what follows regarding (\ref{eq_spectralprojection}), we use $\texttt{Fr-boostrap}$ to stand for the bootstrapping method using the Frobenius norm proposed in \cite{bootsfr}, $\texttt{Fr-Bayes}$ to stand for the frequentist Bayes using the Frobenius norm proposed in \cite{SS},  $\texttt{Fr-DataDriven}$ to stand for the sample splitting method using the Frobenius norm proposed in \cite{KLsplit}, $\texttt{HPV-LeCam}$ to stand for the Le Cam optimal test proposed in \cite{LeCamtest}, $\texttt{En-bootstrap}$ and $\texttt{En-Bayes}$ to stand for the bootstrapping method and the frequentist Bayes method respectively using the power-enhanced norm introduced in \cite{SF} with $s_1=s_2=1$ (see Definition 3.1 of \cite{SF}), and  $\texttt{Sp-bootstrap}$ and $\texttt{Sp-Bayes}$  the bootstrapping method  and the frequentist Bayes method respectively using the spectral norm. In the following discussion, we compare the performance of our $\texttt{Fr-Adaptive}$ with all the aforementioned statistics.

In all the following simulations, we conduct 2,000 Monte-Carlo 
repetitions for the bootstrapping and frequentist Bayes procedure. For the accuracy of the tests, we focus on reporting the results with the type I error rate $0.1$ under different values of $y=0.1, 1, 10$ and various choices of the spikes. Moreover, we consider different scenarios to illustrate the usefulness and generality of our results.
\begin{itemize}
\item  Scenario I: We consider the case $r_0=3$ with $d_1=d+7,$ $d_2=7$ and $d_3=5,$ where $d$ takes a variety of values. We consider the hypothesis testing for the eigenspace of $\Sigma=I+ \sum_{i=1}^3d_i \bm{e}_i \bm{e}_i^*$ with $\mathcal{I}=\{1,2\},$ where the null is 
\begin{equation}\label{eq_trueeigenspace}
Z_0=\bm{e}_1 \bm{e}_1^*+\bm{e}_2 \bm{e}_2^*.
\end{equation}
 In this scenario, the spiked eigenvalues are simple. We will consider the standard Gaussian distribution and the two-point distribution $\frac{1}{3} \delta_{\sqrt{2}}+\frac{2}{3} \delta_{-\frac{1}{\sqrt{2}}}$ as the distribution of entries of $X$. 
\item Scenario II:  We consider the case $r_0=3$ with $d_1=d_2=d+5$ and $d_3=5,$ where $d$ takes a variety of values. We consider the hypothesis testing under the null (\ref{eq_trueeigenspace}).
In this scenario, we have multiple/identical $d_i$'s. We  consider both the standard Gaussian random variables and the two-point distribution $\frac{1}{3} \delta_{\sqrt{2}}+\frac{2}{3} \delta_{-\frac{1}{\sqrt{2}}}$ as the distribution of entries of $X$.     
\end{itemize}
We mention that the asymptotic distribution of our statistic (\ref{eq_finalstatistict1s}) under the null hypothesis of (\ref{eq_spectralprojection}) has been established in Examples \ref{ex_test1_simple} and \ref{ex_test1_equal} for Scenarios I and II, respectively.

For both of the above two scenarios, we consider the alternative
\begin{equation}\label{eq_alternative}
Z_a=\bm{v}_1(\varphi) \bm{v}_1(\varphi)^*+ \bm{v}_2(\varphi) \bm{v}_2(\varphi)^*,
\end{equation}
where for $\varphi \in [0,\frac{\pi}{2}]$
\begin{equation*}
\bm{v}_1(\varphi)=\cos \varphi \bm{e}_1+\sin \varphi \bm{e}_4, \ \bm{v}_2(\varphi)=\cos \varphi \bm{e}_2+\sin \varphi \bm{e}_5. 
\end{equation*}
Note that $\varphi=0$ corresponds to the null case of (\ref{eq_trueeigenspace}). It is easy to see that $\kappa_4=-1.5$ for the two-point distribution and $\kappa_4=0$ for standard Gaussian random variable.

For Scenario I, from Tables \ref{table_singu}--\ref{table_singu2}, we find that our proposed statistic (\ref{eq_finalstatistict1s}) is very accurate even for smaller values of $N$ and $d$. Moreover, our statistic reaches accuracy regardless of the values of $y.$ In contrast, for the other methods in the literature, we find that all of them lose their accuracy when $y$ increases (i.e. $M$ increases). Moreover, we find that most of the methods are conservative except for the Le Cam test.  Finally, we find that some of the methods, especially the frequentiest Bayes method with Frobenius norm, spectral norm or the power-enhanced norm in \cite{SF} are also reasonably accurate for larger values of $d$ when $y$ is small. Since our results allow us to deal with identical $d_i$'s, we report the simulation results  in Tables \ref{table_singudegenerate}--\ref{table_singu2denegenrate} for Scenario II. We find that our statistic (\ref{eq_finalstatistict1s}) is also very accurate and outperforms the other methods especially when either the $d_i$'s are small or $y$ is large.  In Appendix \ref{sec_additionalsimulation}, we report the simulation results for random variables with two-point distribution in Tables \ref{table_singutp}--\ref{table_singu2tp} for Scenario I and Tables \ref{table_singutpdenegate}--\ref{table_singu2tpdegenerate} for Scenario II. We can get analogous conclusions.

In summary, our proposed statistic (\ref{eq_finalstatistict1s}) is quite accurate for different values of $d$ satisfying Assumption \ref{supercritical}, even for smaller and multiple/identical ones. This accuracy is robust against different values of $y.$  As summarized in \cite[Section 7.4]{SF}, all the  previous methods in the literature request that $M \ll N.$ Therefore, when $y$ increases (i.e., $M$ diverges faster), we find that our methods perform better than all the other methods. Indeed, all our current results can be extended to the regime $\log M \asymp \log N$ following the discussion of \cite{bloemendal2016principal,bloemendal2014isotropic}. We will pursue this direction in the future work. Moreover, since the computational complexity of the aforementioned methods depends on the a polynomial order of the dimensionality $M$ (see \cite[Section 7.4]{SF}), they can be computationally intensive as $M$ diverges faster. In contrast, our method works faster since it only depends the sample eigenvalues and eigenvectors.

\begin{table}[!ht]

\setlength\arrayrulewidth{1.5pt}
\renewcommand{\arraystretch}{1.3}
 \captionsetup{width=1\linewidth}
 \addtolength{\tabcolsep}{-0.5pt} 
\begin{center}
\scriptsize
\hspace*{-0.9cm}
{
\begin{tabular}{clclclclclclc|c|c|c|c}
\toprule
 & \multicolumn{5}{c}{$N=200$}                                                                                                                       & \multicolumn{5}{c}{$N=500$}                                                                                                                       \\ \hline
Method   & $d=2$& $d=5$  & $d=10$  & $d=50$  & $ d=100$  & $d=2$ & $d=5$ & $d=10$ & $d=50$ & $d=100$ \\ \hline
$\texttt{Fr-bootstrap}$ & 0.041  &  0.044  &  0.044  & 0.054  & 0.062 & 0.047 &  0.051 &  0.049  & 0.063 & 0.067  \\
$\texttt{Fr-Bayes}$ & 0.055  & 0.049 &  {\bf 0.079} & {0.089}  & {\bf 0.095} & 0.045 & 0.053  &  0.082 & 0.093  & {0.094} \\
$\texttt{En-bootstrap}$& 0.047  & {0.053}  & 0.068  & {0.067} & 0.077 & 0.052 & 0.049 & 0.063  & 0.069  & 0.075  \\
$\texttt{En-Bayes}$ &  0.053 &  0.058 & 0.077  &  {\bf 0.093} & {\bf 0.096}  & 0.051  & {0.064}  & {\bf 0.088} & {\bf 0.098} & 0.093  \\
$\texttt{Fr-Datadriven}$& 0.046  &  0.049  &  0.051  & 0.063  &  0.067 & 0.041  &  0.043 &  0.057  & 0.059  & 0.065   \\
$\texttt{HPV-LeCam}$& 0.381 & 0.374 &  0.383 &  0.391 & 0.373   & 0.376 & 0.368  & 0.365 & 0.342  & 0.373  \\
$\texttt{Sp-bootstrap}$	& 0.047   & 0.054 & 0.062 & 0.073 & 0.079  &  0.042  & 0.053 & 0.063  & 0.069 &  0.076\\
$\texttt{Sp-Bayes}$&  {\bf 0.066}  & {\bf 0.069}   & 0.072  & 0.083  & 0.094  & {\bf 0.072}  &  {\bf 0.078} & 0.075 &  0.088 & {\bf 0.095} \\
$\texttt{Fr-Adaptive}$  & \bf 0.091  & \bf 0.108 & \bf 0.103 & \bf 0.107 & \bf 0.096  & \bf 0.11  & \bf 0.107 & \bf 0.095 & \bf 0.104  &  {\bf 0.102} \\ 
 \bottomrule
\end{tabular}
}
\end{center}
\caption{ Scenario I: simulated type I error rates under the nominal level $0.1$ for $y=0.1$.  We report our results based on 2,000 Monte-Carlo simulations with Gaussian random variables. We highlighted the two most accurate methods for each value of $d.$ 
} \label{table_singu}
\end{table}

\begin{table}[!ht]
\setlength\arrayrulewidth{1.5pt}
\renewcommand{\arraystretch}{1.3}
 \captionsetup{width=1\linewidth}
 \addtolength{\tabcolsep}{-0.5pt} 
\begin{center}
\scriptsize
\hspace*{-0.7cm}
{
\begin{tabular}{clclclclclclc|c|c|c|c}
\toprule
 & \multicolumn{5}{c}{$N=200$}                                                                                                                       & \multicolumn{5}{c}{$N=500$}                                                                                                                       \\ \hline
Method   & $d=2$& $d=5$  & $d=10$  & $d=50$  & $ d=100$  & $d=2$ & $d=5$ & $d=10$ & $d=50$ & $d=100$ \\ \hline
$\texttt{Fr-bootstrap}$ & 0.053   &  0.045  &  0.051  & 0.049  & 0.047 & 0.052 & 0.049 &  0.047  & 0.053  & 0.061  \\
$\texttt{Fr-Bayes}$ & 0.045   & 0.039  &   0.047& 0.063 & 0.071 & 0.039  &  0.041  & 0.052 &  0.061 & 0.068 \\
$\texttt{En-bootstrap}$& {\bf 0.057}  & 0.051   & 0.042   & 0.043 & 0.049 & 0.041 & 0.039  & 0.048 &  0.052  & 0.059   \\
$\texttt{En-Bayes}$ & {\bf 0.057}  & {0.053}  & {0.061}   &  0.067  & 0.075  & {0.048}  & {\bf 0.059}  & {\bf 0.064} &  {\bf 0.073} & {\bf 0.078}   \\
$\texttt{Fr-Datadriven}$& 0.026  & 0.023   & 0.034 & 0.037 & 0.039 & 0.041  & 0.04 & 0.048  & 0.042   & 0.047    \\
$\texttt{HPV-LeCam}$& 0.87 & 0.79 & 0.82  & 0.81 & 0.85   & 0.882 & 0.835 & 0.823 &  0.872 & 0.823 \\
$\texttt{Sp-bootstrap}$	&  0.049 & 0.057  & 0.056 & 0.062  & 0.059   &  0.043 & 0.041 & 0.045  & 0.053 &  0.062 \\
$\texttt{Sp-Bayes}$&  {\bf 0.057 }& {\bf 0.058}    &{\bf 0.062}  & {\bf 0.074}  & {\bf 0.081}   & {\bf 0.053} & 0.057 & 0.059 & 0.059 & 0.069 \\
$\texttt{Fr-Adaptive}$  & {\bf 0.103}  & {\bf 0.092} & {\bf 0.105}  & {\bf 0.107}  & {\bf 0.099}   &  {\bf 0.11} & {\bf 0.107} & {\bf 0.104} &  {\bf 0.097} & {\bf 0.103}  \\ 
 \bottomrule
\end{tabular}
}
\end{center}
\caption{ Scenario I: simulated type I error rates under the nominal level $0.1$ for $y=1$.  We report our results based on 2,000 Monte-Carlo simulations with Gaussian random variables. We highlighted the two most accurate methods for each value of $d.$   
} \label{table_singu1}
\end{table}

\begin{table}[!ht]
\setlength\arrayrulewidth{1.5pt}
\renewcommand{\arraystretch}{1.3}
 \captionsetup{width=1\linewidth}
 \addtolength{\tabcolsep}{-0.5pt} 
\begin{center}
\scriptsize 
\hspace*{-0.7cm}
{
\begin{tabular}{clclclclclclc|c|c|c|c}

\toprule
 & \multicolumn{5}{c}{$N=200$}                                                                                                                       & \multicolumn{5}{c}{$N=500$}                                                                                                                       \\ \hline
Method   & $d=2$& $d=5$  & $d=10$  & $d=50$  & $ d=100$  & $d=2$ & $d=5$ & $d=10$ & $d=50$ & $d=100$ \\ \hline
$\texttt{Fr-bootstrap}$ &  0.028 &  0.034  &  0.037  & 0.041 & 0.043 & 0.039 & 0.045  & 0.052   & 0.038  &  0.047 \\
$\texttt{Fr-Bayes}$ & 0.038  & {\bf 0.051}  & 0.049   &  {\bf 0.062} & {\bf 0.071} & {\bf 0.051} & {\bf 0.049} & 0.046 & 0.053  & {\bf 0.069} \\
$\texttt{En-bootstrap}$&  0.032 & 0.041   & 0.045   & 0.046  & 0.059  &  0.037 & 0.041 &  0.054 & 0.049  & 0.054   \\
$\texttt{En-Bayes}$ & {\bf 0.046} & 0.048  & {\bf 0.057}   & 0.061   & {0.068} & 0.039  & 0.042  & {0.049} & 0.052  & {0.064}   \\
$\texttt{Fr-Datadriven}$& 0.027 &  0.033 &  0.038  & 0.041 & 0.043  &  0.038  & 0.034  & 0.029 &  0.045  & 0.052   \\
$\texttt{HPV-LeCam}$& 0.897 & 0.939 & 0.964  & 0.971 & 0.972   & 0.891 & 0.911 & 0.943 &  0.932 & 0.953 \\
$\texttt{Sp-bootstrap}$	& {\bf 0.046} & 0.048  &0.045 &  0.054 & 0.052   & {0.039}   & {0.047}  &  0.049  & 0.053 &  0.058 \\
$\texttt{Sp-Bayes}$&  0.043  & {0.049}   & {0.052}  & {0.057}  & {0.068}  & {\bf 0.051} & 0.048 & {\bf 0.059} & {\bf 0.063} & 0.068 \\
$\texttt{Fr-Adaptive}$ &  {\bf 0.104} & {\bf 0.102} & {\bf 0.095}  & {\bf 0.098}  & {\bf 0.103}   & {\bf 0.091}   & {\bf 0.097} & {\bf 0.104} & {\bf 0.097}  & {\bf 0.103}  \\ 
 \bottomrule
\end{tabular}
}
\end{center}
\caption{ Scenario I: simulated type I error rates under the nominal level $0.1$ for $y=10$.  We report our results based on 2,000 Monte-Carlo simulations with Gaussian random variables. We highlighted the two most accurate methods for each value of $d.$   
} \label{table_singu2}
\end{table}

\begin{table}[!ht]
\setlength\arrayrulewidth{1.5pt}
\renewcommand{\arraystretch}{1.3}
 \captionsetup{width=1\linewidth}
 \addtolength{\tabcolsep}{-0.5pt} 
\begin{center}
\scriptsize 
\hspace*{-0.7cm}
{
\begin{tabular}{clclclclclclc|c|c|c|c}
\toprule
 & \multicolumn{5}{c}{$N=200$}                                                                                                                       & \multicolumn{5}{c}{$N=500$}                                                                                                                       \\ \hline
Method   & $d=2$& $d=5$  & $d=10$  & $d=50$  & $ d=100$  & $d=2$ & $d=5$ & $d=10$ & $d=50$ & $d=100$ \\ \hline
$\texttt{Fr-bootstrap}$ & 0.036  &  0.041   &  0.045   & 0.049  &  0.062 & 0.039 & 0.047  & 0.039   & 0.053 & 0.067 \\
$\texttt{Fr-Bayes}$ &  0.042 & 0.047 & 0.063   &  {\bf 0.087} &  {\bf 0.096} & 0.053 & 0.062  & {\bf 0.072} &  0.085  &  {\bf 0.096}\\
$\texttt{En-bootstrap}$ & 0.045  & 0.047  & 0.043  & 0.046 & 0.052  & {\bf 0.058} & 0.049 &  0.059& 0.062  & 0.066   \\
$\texttt{En-Bayes}$ & {\bf 0.062}  &  {\bf 0.069} &  {\bf 0.075} &  0.074 & 0.094  & 0.054  & {\bf 0.064}  & 0.063 & 0.085 & {0.108}  \\
$\texttt{Fr-Datadriven}$& 0.043  & 0.046  & 0.043 & 0.052 & 0.059  & 0.052  & 0.049 & 0.061   &  0.058  & 0.063   \\
$\texttt{HPV-LeCam}$& 0.394 & 0.379 & 0.433  & 0.436 & 0.398   & 0.431 & 0.441 & 0.423 &  0.412 & 0.393 \\
$\texttt{Sp-bootstrap}$	& {0.058}  & 0.056  & 0.047 & 0.059 & 0.068  & 0.051  & 0.042  & 0.049  &  0.047 & 0.062 \\
$\texttt{Sp-Bayes}$&   {0.058}  & {0.068}    & { 0.072}  & 0.081  & {0.095}  & 0.049 & {0.062} & {0.069} & {\bf 0.087} & {\bf 0.104}\\
$\texttt{Fr-Adaptive}$  & \bf 0.105  & \bf 0.103 & \bf 0.097 & \bf 0.102 & \bf 0.103  & \bf 0.104  & \bf 0.096 & \bf 0.095 & \bf 0.093 &  {0.105} \\ 
 \bottomrule
\end{tabular}
}
\end{center}
\caption{ Scenario II: simulated type I error rates under the nominal level $0.1$ for $y=0.1$.  We report our results based on 2,000 Monte-Carlo simulations with Gaussian random variables. We highlighted the two most accurate methods for each value of $d.$  
} \label{table_singudegenerate}
\end{table}

\begin{table}[!ht]
\setlength\arrayrulewidth{1.5pt}
\renewcommand{\arraystretch}{1.3}
 \captionsetup{width=1\linewidth}
 \addtolength{\tabcolsep}{-0.5pt} 
\begin{center}
\scriptsize 
\hspace*{-0.7cm}
{
\begin{tabular}{clclclclclclc|c|c|c|c}
\toprule
 & \multicolumn{5}{c}{$N=200$}                                                                                                                       & \multicolumn{5}{c}{$N=500$}                                                                                                                       \\ \hline
Method   & $d=2$& $d=5$  & $d=10$  & $d=50$  & $ d=100$  & $d=2$ & $d=5$ & $d=10$ & $d=50$ & $d=100$ \\ \hline
$\texttt{Fr-bootstrap}$ & 0.039  & 0.043  &  0.038  & 0.051 & 0.058& { 0.042} & 0.045 & 0.049  & 0.048  & 0.052 \\
$\texttt{Fr-Bayes}$ & 0.052   & {\bf 0.049} & 0.048   & {\bf 0.067} & 0.072 & 0.048 & 0.042  & 0.051  &  0.057  & 0.068  \\
$\texttt{En-bootstrap}$& {\bf 0.054} & {\bf 0.049}  & 0.053 & 0.047 & 0.061 & {\bf 0.052} & {\bf 0.047}   & 0.041  & 0.039   & 0.051   \\
$\texttt{En-Bayes}$ & 0.041   & {0.046}  & 0.052   &  {0.062}  & {\bf 0.078}  & 0.043  & {0.046}& {\bf 0.059} & {\bf 0.063} & {0.068}  \\
$\texttt{Fr-Datadriven}$& 0.325  & 0.029  & 0.038 & 0.041 & 0.047 & 0.039  & 0.042  &  0.047 &  0.046  & 0.052  \\
$\texttt{HPV-LeCam}$& 0.844 & 0.829 & 0.814  & 0.87 & 0.88   & 0.865 & 0.825 & 0.853 &  0.792 & 0.83 \\
$\texttt{Sp-bootstrap}$	& 0.038  & 0.047  & 0.042 & 0.057 & 0.049  & 0.037  & 0.042& {0.058} & 0.052 & 0.048  \\
$\texttt{Sp-Bayes}$&  0.038  &  0.046    & {\bf 0.057}  & 0.065  &  {0.069}  & 0.042 & 0.039  & 0.054 & 0.061 & {\bf 0.074} \\
$\texttt{Fr-Adaptive}$  & {\bf 0.11}  & {\bf 0.093} & {\bf 0.097}  & {\bf 0.096}  & {\bf 0.102}   &  {\bf 0.091} & {\bf 0.09} & {\bf 0.098} &  {\bf 0.096} & {\bf 0.102}  \\ 
 \bottomrule
\end{tabular}
}
\end{center}
\caption{ Scenario II: simulated type I error rates under the nominal level $0.1$ for $y=1$.  We report our results based on 2,000 Monte-Carlo simulations with Gaussian random variables. We highlighted the two most accurate methods for each value of $d.$  
} \label{table_singu1degenate}
\end{table}

\begin{table}[!ht]
\setlength\arrayrulewidth{1.5pt}
\renewcommand{\arraystretch}{1.3}
 \captionsetup{width=1\linewidth}
 \addtolength{\tabcolsep}{-0.5pt} 
\begin{center}
\scriptsize 
\hspace*{-0.7cm}
{
\begin{tabular}{clclclclclclc|c|c|c|c}
\toprule
 & \multicolumn{5}{c}{$N=200$}                                                                                                                       & \multicolumn{5}{c}{$N=500$}                                                                                                                       \\ \hline
Method   & $d=2$& $d=5$  & $d=10$  & $d=50$  & $ d=100$  & $d=2$ & $d=5$ & $d=10$ & $d=50$ & $d=100$ \\ \hline
$\texttt{Fr-bootstrap}$ & { 0.035}  & 0.041  &  0.035  & 0.045 & 0.058& {0.032} & 0.04 & 0.043  & 0.039  & 0.052 \\
$\texttt{Fr-Bayes}$ & 0.039   &{ 0.042} & 0.046   & {\bf 0.059} & {0.062} & 0.034 & {\bf 0.052}  & 0.049  &  {\bf 0.056}  & {\bf 0.065}  \\
$\texttt{En-bootstrap}$&  0.039 & 0.042  & 0.038  &  0.041 & 0.045 & 0.032 & 0.029  &0.042  &0.047  & 0.048 \\
$\texttt{En-Bayes}$ & 0.041  & 0.047  &0.051   & {\bf 0.059}  & {\bf 0.066}  & {\bf 0.045} & {0.036} & 0.047 & 0.053 & {\bf 0.065}  \\
$\texttt{Fr-Datadriven}$& 0.025 & 0.031  & 0.029 & 0.038 &  0.041 & 0.031  & 0.033 &  0.041 & 0.039  &0.048    \\
$\texttt{HPV-LeCam}$& 0.914 & 0.979 & 0.974  & 0.959 & 0.963   & 0.932 & 0.941 & 0.893 &  0.962 & 0.912 \\
$\texttt{Sp-bootstrap}$	&  {\bf 0.042} & {\bf 0.053}  & {\bf 0.052} & 0.052 & 0.057  & 0.039  & 0.042 & 0.049  & 0.052 & 0.049 \\
$\texttt{Sp-Bayes}$&   0.039 & {0.045}    & 0.046  & 0.054  & 0.062 & 0.042 & 0.048&  {\bf 0.051}&  {0.053} &  {0.064}\\
$\texttt{Fr-Adaptive}$ &  {\bf 0.107} & {\bf 0.103}  & {\bf 0.103}  & {\bf 0.094}  & {\bf 0.098}  &  {\bf 0.102} & {\bf 0.092}  & {\bf 0.097}  & {\bf 0.104}  & {\bf 0.098}  \\ 
 \bottomrule
\end{tabular}
}
\end{center}
\caption{ Scenario II: simulated type I error rates under the nominal level $0.1$ for $y=10$.  We report our results based on 2,000 Monte-Carlo simulations with Gaussian random variables. We highlighted the two most accurate methods for each value of $d.$  
} \label{table_singu2denegenrate}
\end{table}

Then we compare the power of the above statistics under the alternative (\ref{eq_alternative}) for different values of $\varphi$ regarding Scenarios I and II and $d=5, 50,$ respectively,  in Figures \ref{fig_power1}--\ref{fig_power250}.  
We find that our method is very powerful even when $y$ becomes larger and $d$  becomes smaller, and it outperforms the other methods especially when either $y$ is large or $d$ is small. When $y$ is small and $d$ is large, even though 
for larger values of $\varphi,$ many methods can obtain high power, we find that our proposed statistic (\ref{eq_finalstatistict1s}) is quite powerful even under a relatively weak alternative, i.e., smaller values of $\varphi$. We mention that the high power of $\texttt{HPV-LeCam}$ is not trustful since we have seen from the simulated type I error rates that it is not accurate. Similar results can be found for two-point random variables in Figures \ref{fig_powertp1}-\ref{fig_powertp250} of Appendix \ref{sec_additionalsimulation}.

\begin{figure}[!ht]
\hspace*{-2.0cm}
\begin{subfigure}{0.3\textwidth}
\includegraphics[width=5.8cm,height=5cm]{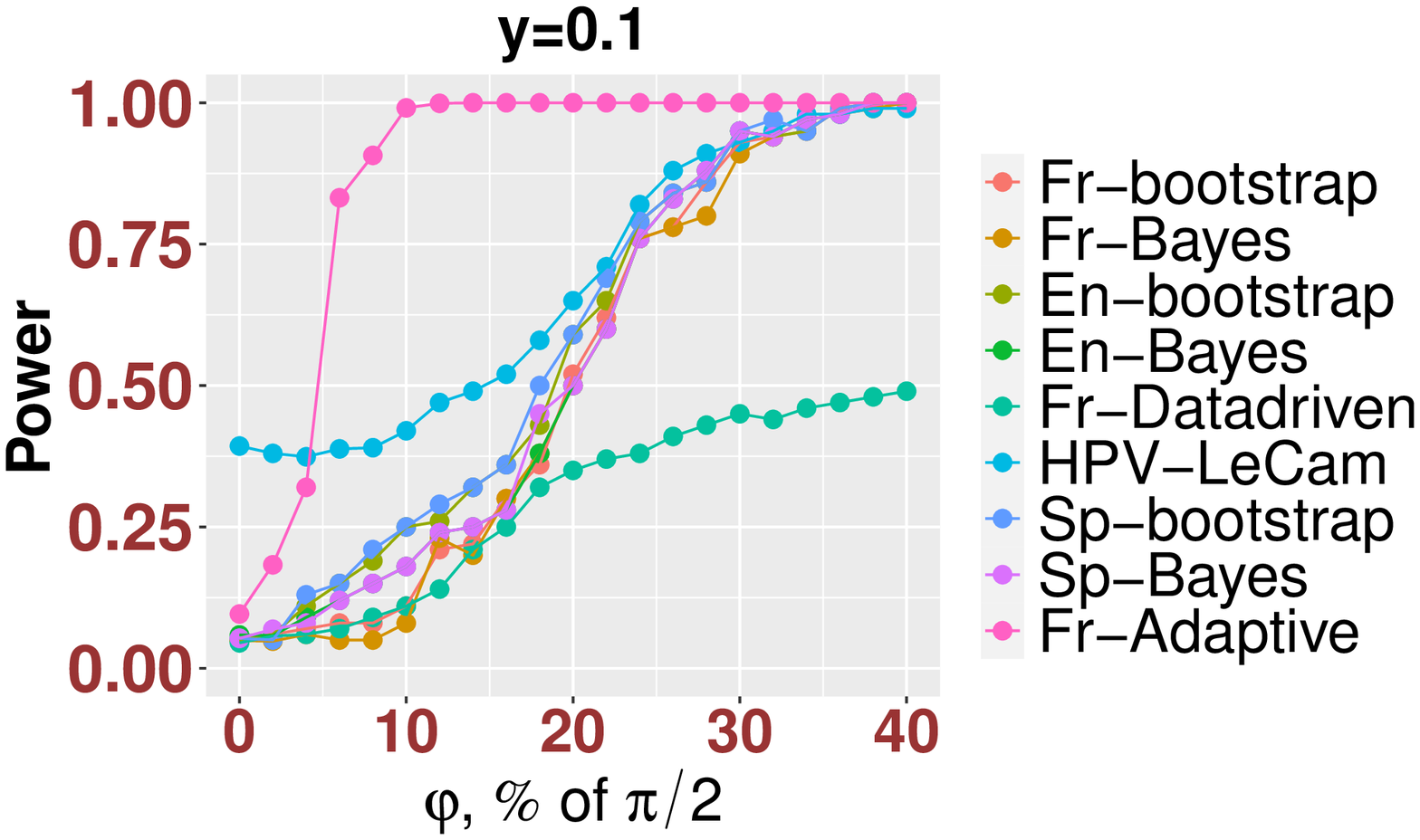}
\end{subfigure}
\begin{subfigure}{0.3\textwidth}
\includegraphics[width=5.8cm,height=5cm]{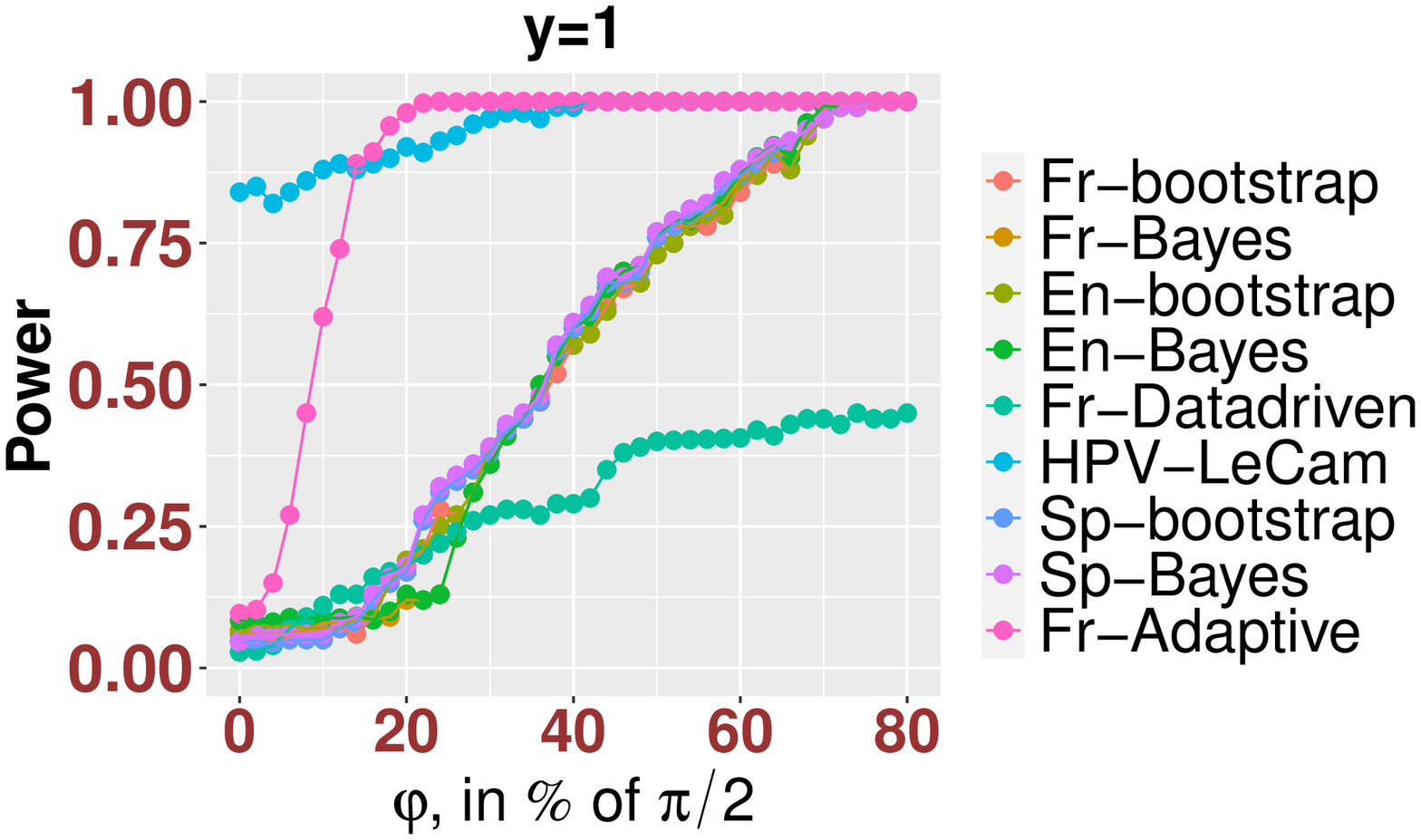}
\end{subfigure}
\begin{subfigure}{0.3\textwidth}
\includegraphics[width=5.8cm,height=5cm]{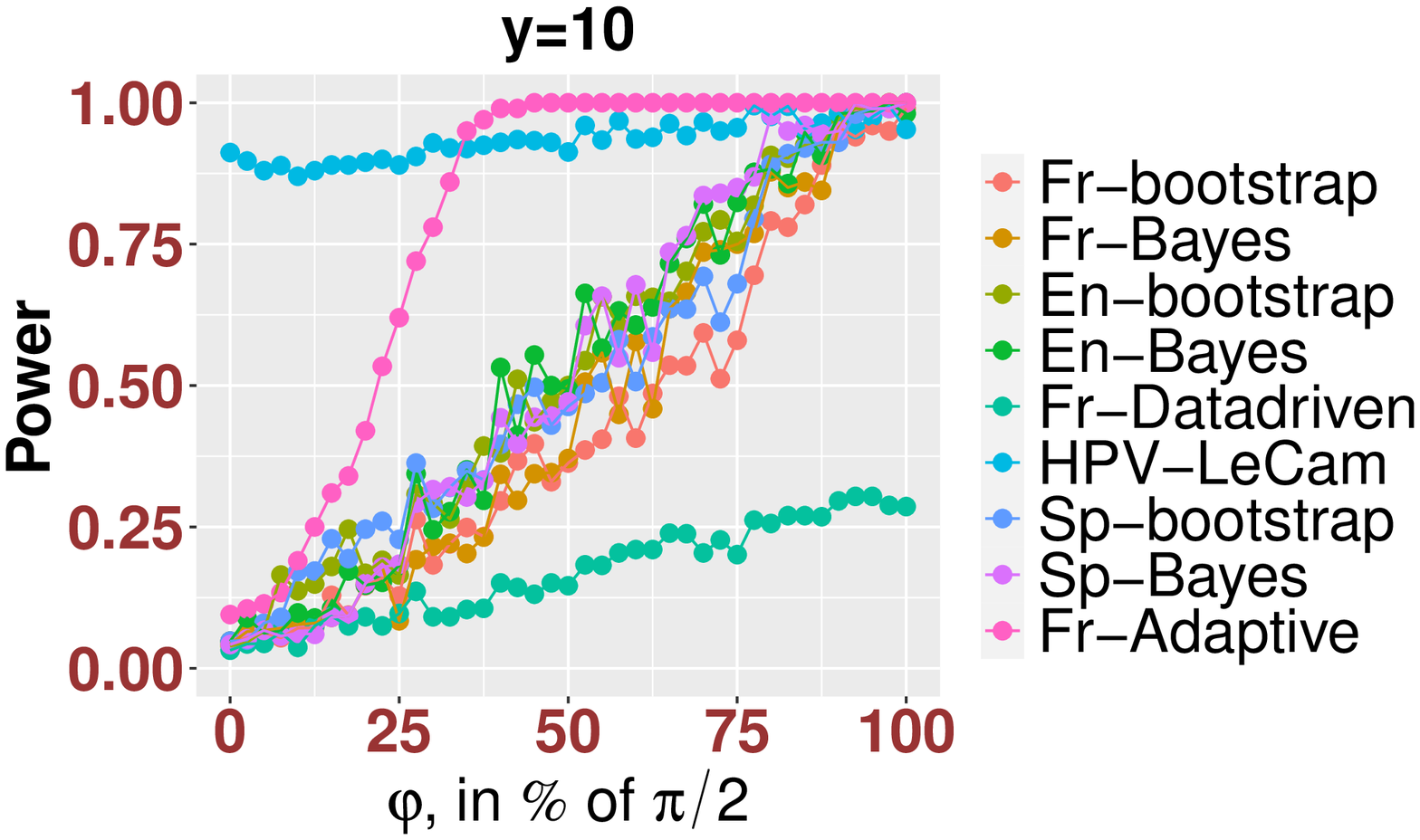}
\end{subfigure}
\caption{ {\footnotesize Comparison of power for Scenario I. We choose $d=5$ and use Gaussian random variables. We report our results under the nominal level 0.1 based on $2,000$ simulations. Here $N=500.$ } }
\label{fig_power1}
\end{figure}

\begin{figure}[!ht]
\hspace*{-2.0cm}
\begin{subfigure}{0.3\textwidth}
\includegraphics[width=5.8cm,height=5cm]{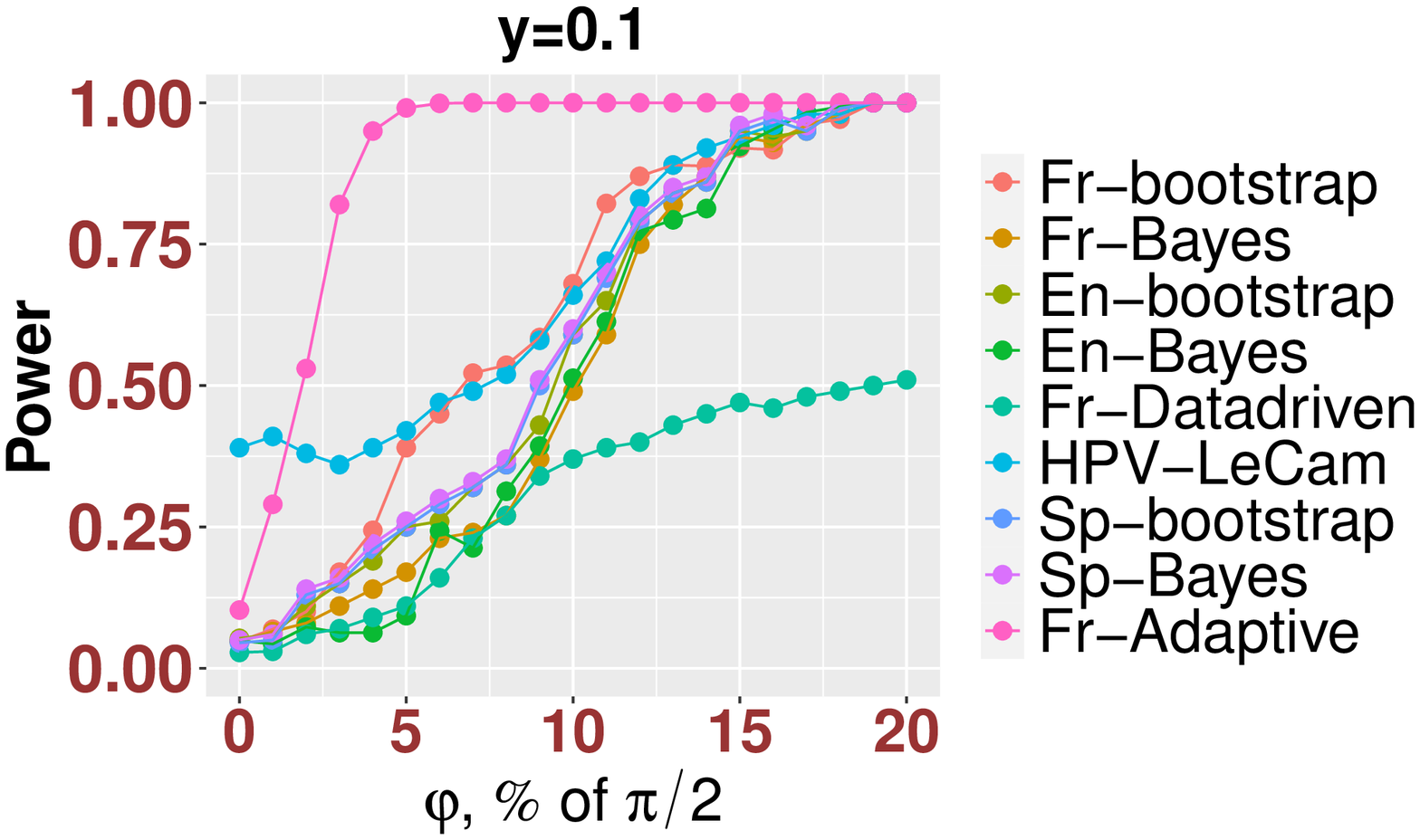}
\end{subfigure}
\begin{subfigure}{0.3\textwidth}
\includegraphics[width=5.8cm,height=5cm]{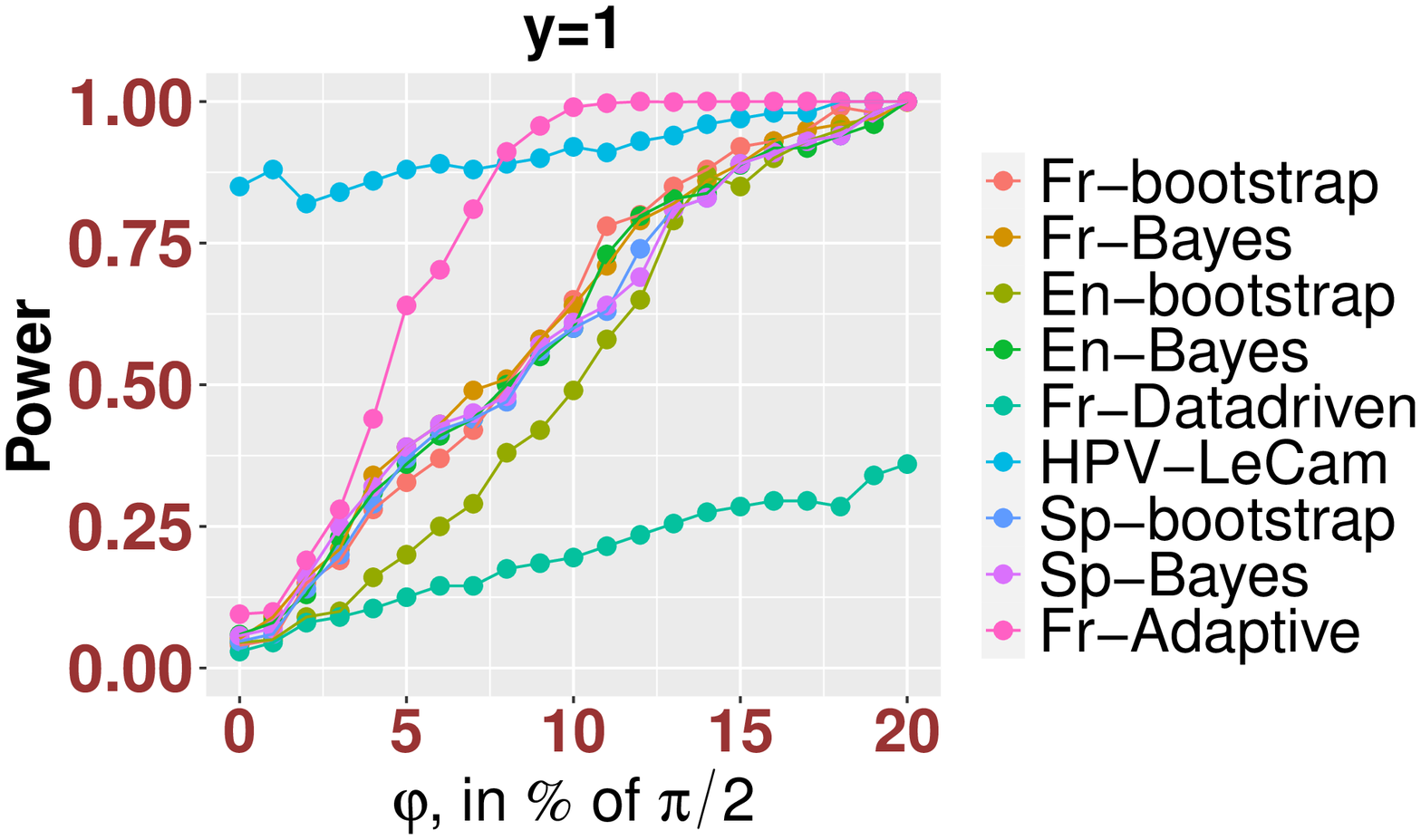}
\end{subfigure}
\begin{subfigure}{0.3\textwidth}
\includegraphics[width=5.8cm,height=5cm]{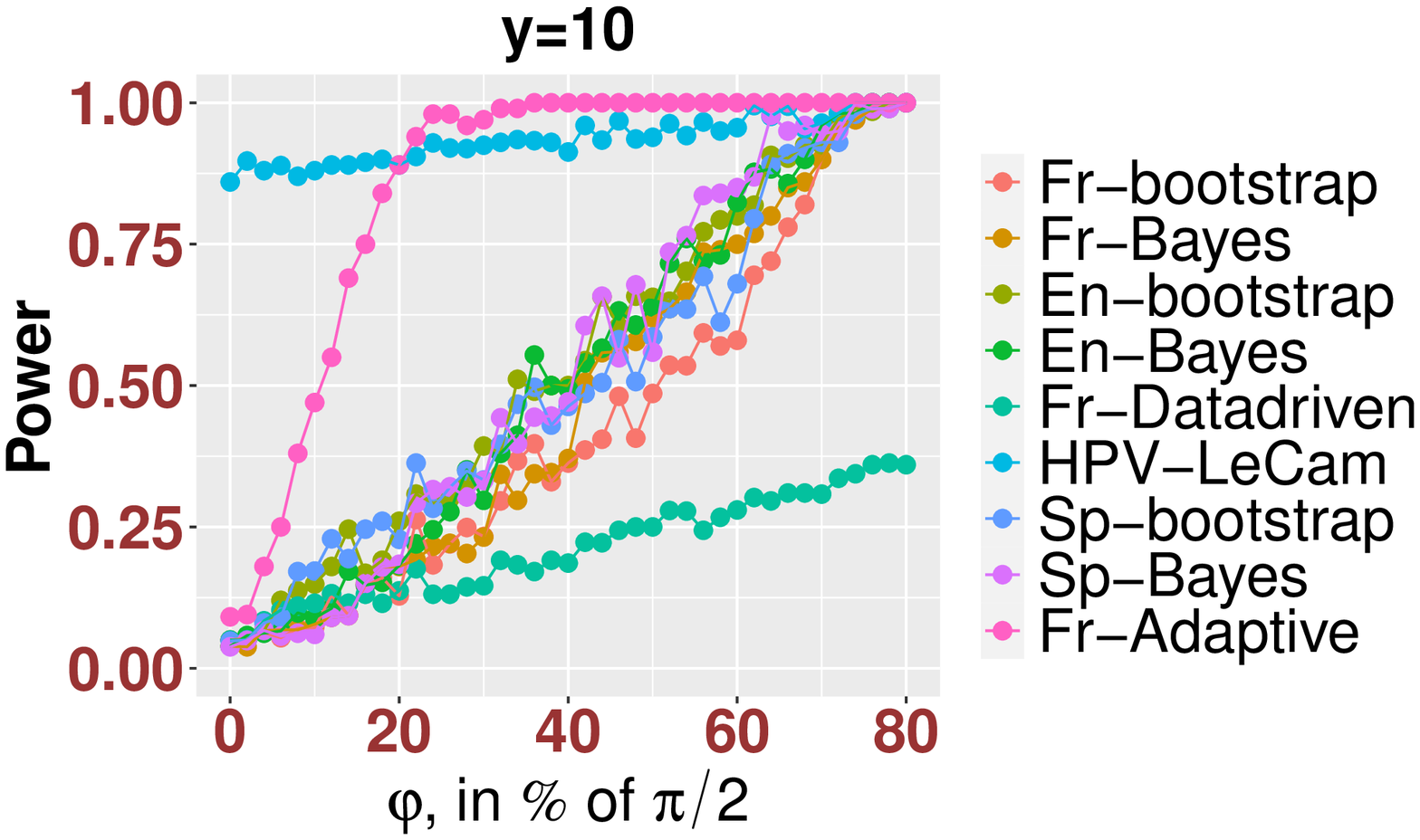}
\end{subfigure}
\caption{{  \footnotesize Comparison of power for Scenario I. We choose $d=50$ and use Gaussian random variables. We report our results under the nominal level 0.1 based on $2,000$ simulations. Here $N=500.$ }}
\label{fig_power150}
\end{figure}

\begin{figure}[!ht]
\hspace*{-2.0cm}
\begin{subfigure}{0.3\textwidth}
\includegraphics[width=5.8cm,height=5cm]{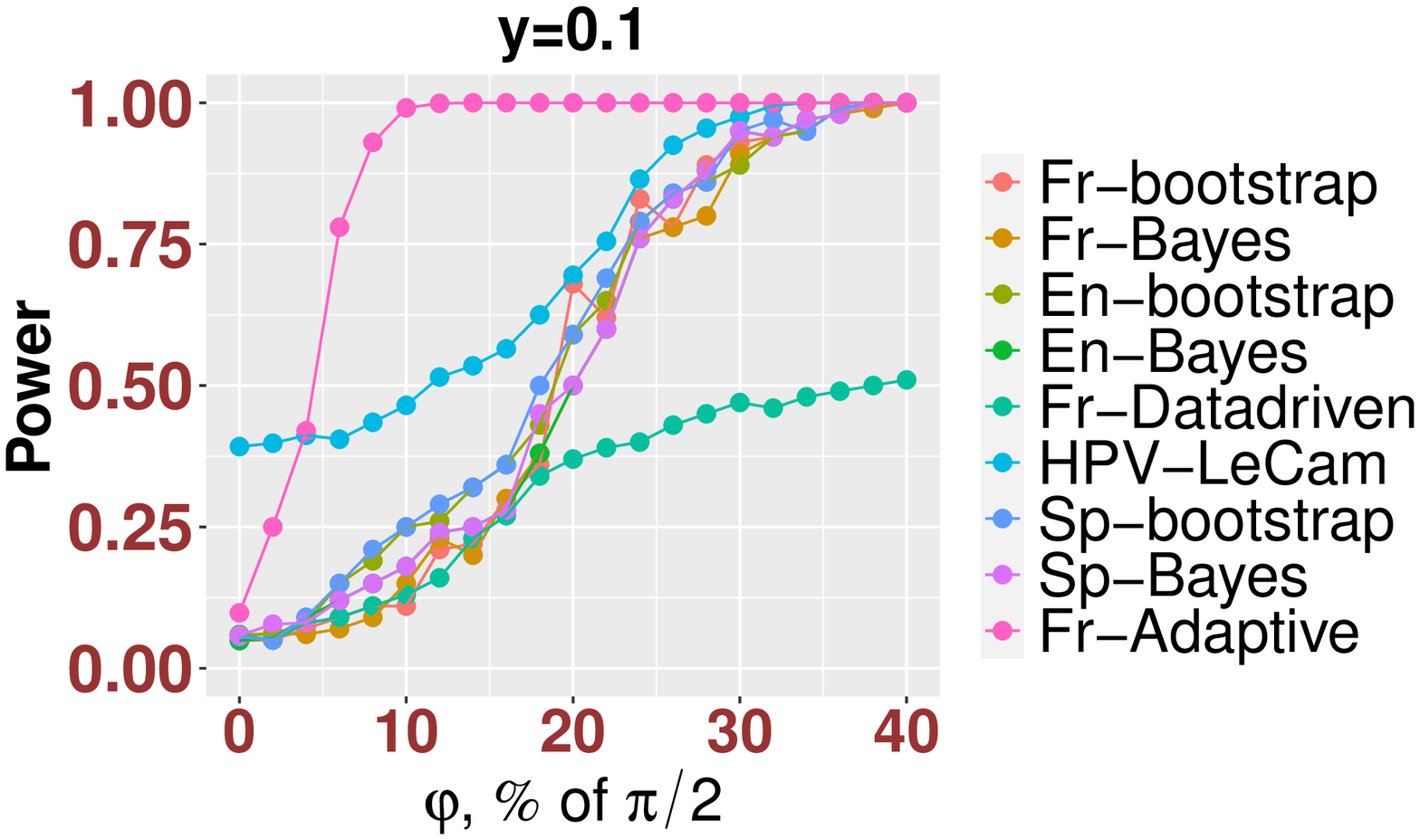}
\end{subfigure}
\begin{subfigure}{0.3\textwidth}
\includegraphics[width=5.8cm,height=5cm]{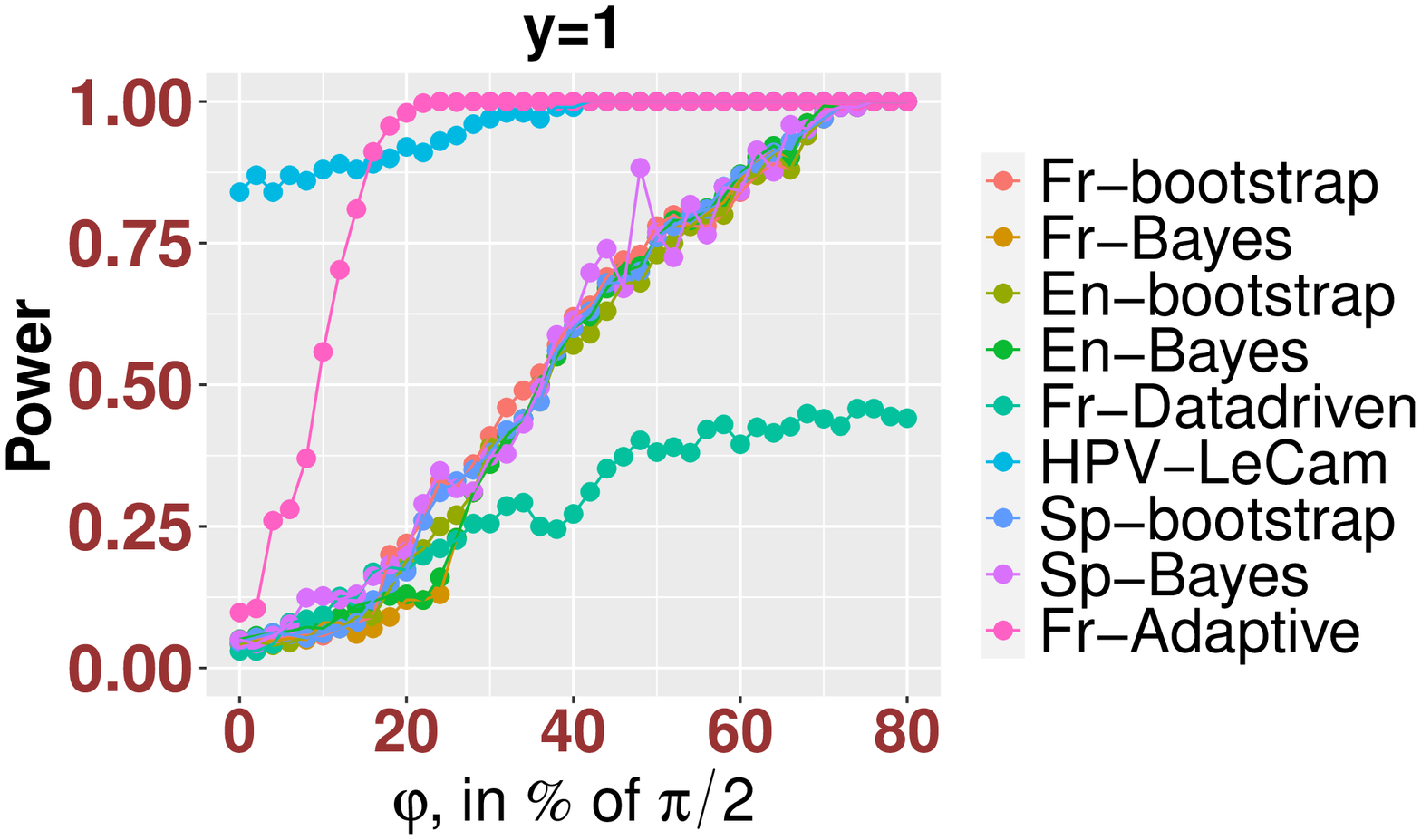}
\end{subfigure}
\begin{subfigure}{0.3\textwidth}
\includegraphics[width=5.8cm,height=5cm]{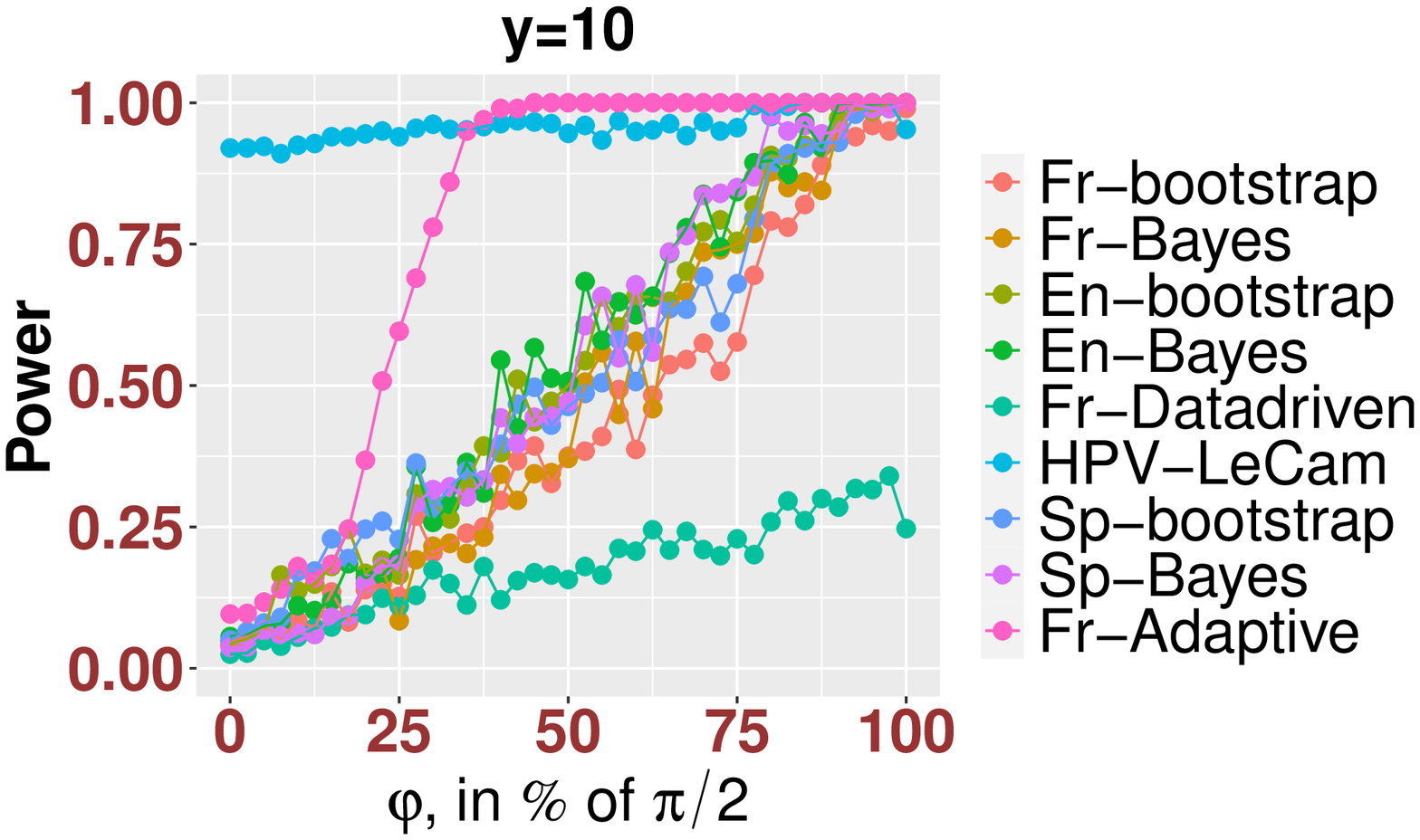}
\end{subfigure}
\caption{{ \footnotesize Comparison of power for Scenario II. We choose $d=5$ and use Gaussian random variables. We report our results under the nominal level 0.1 based on $2,000$ simulations.  Here $N=500.$} }
\label{fig_power2}
\end{figure}

\begin{figure}[!ht]
\hspace*{-2.0cm}
\begin{subfigure}{0.3\textwidth}
\includegraphics[width=5.8cm,height=5cm]{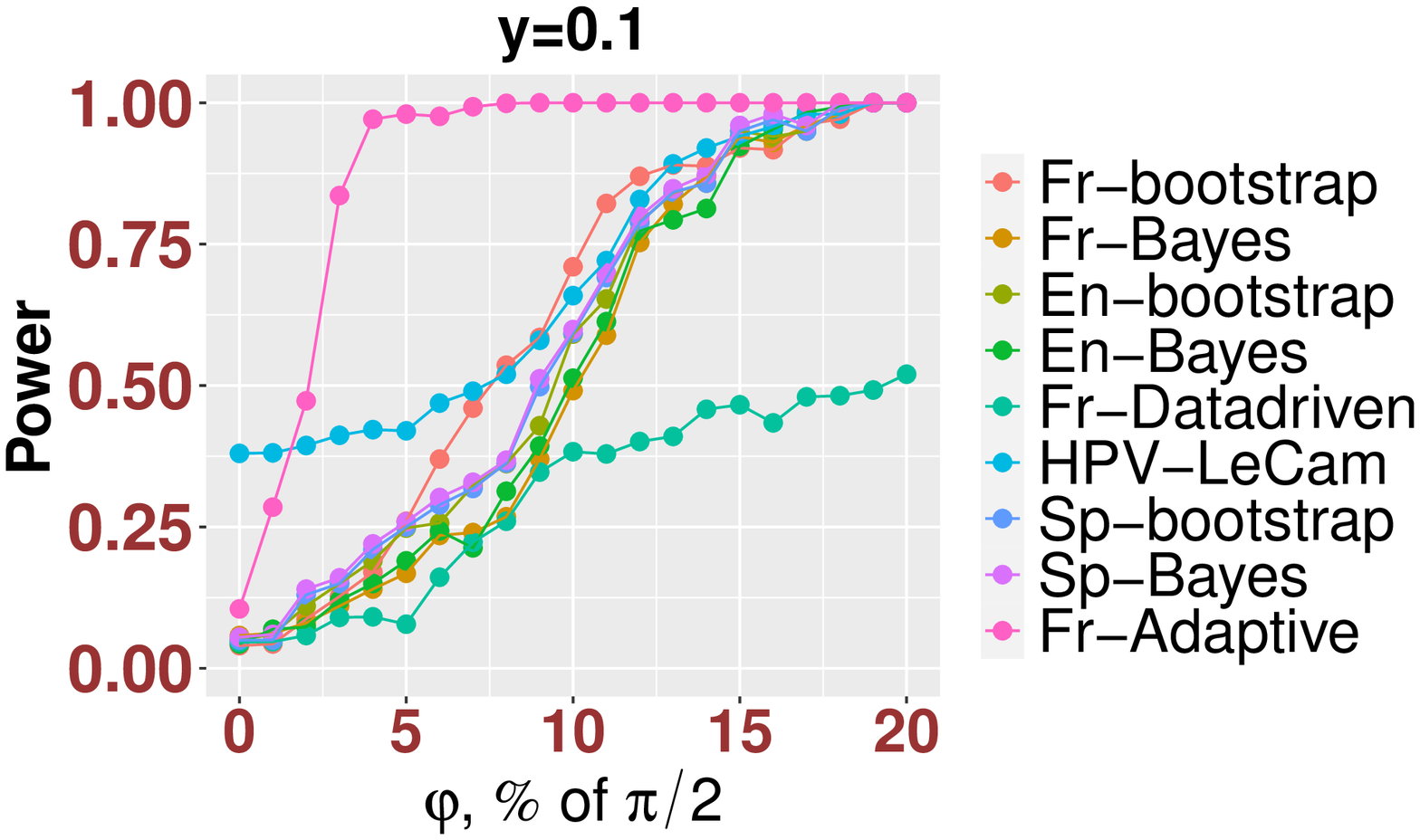}
\end{subfigure}
\begin{subfigure}{0.3\textwidth}
\includegraphics[width=5.8cm,height=5cm]{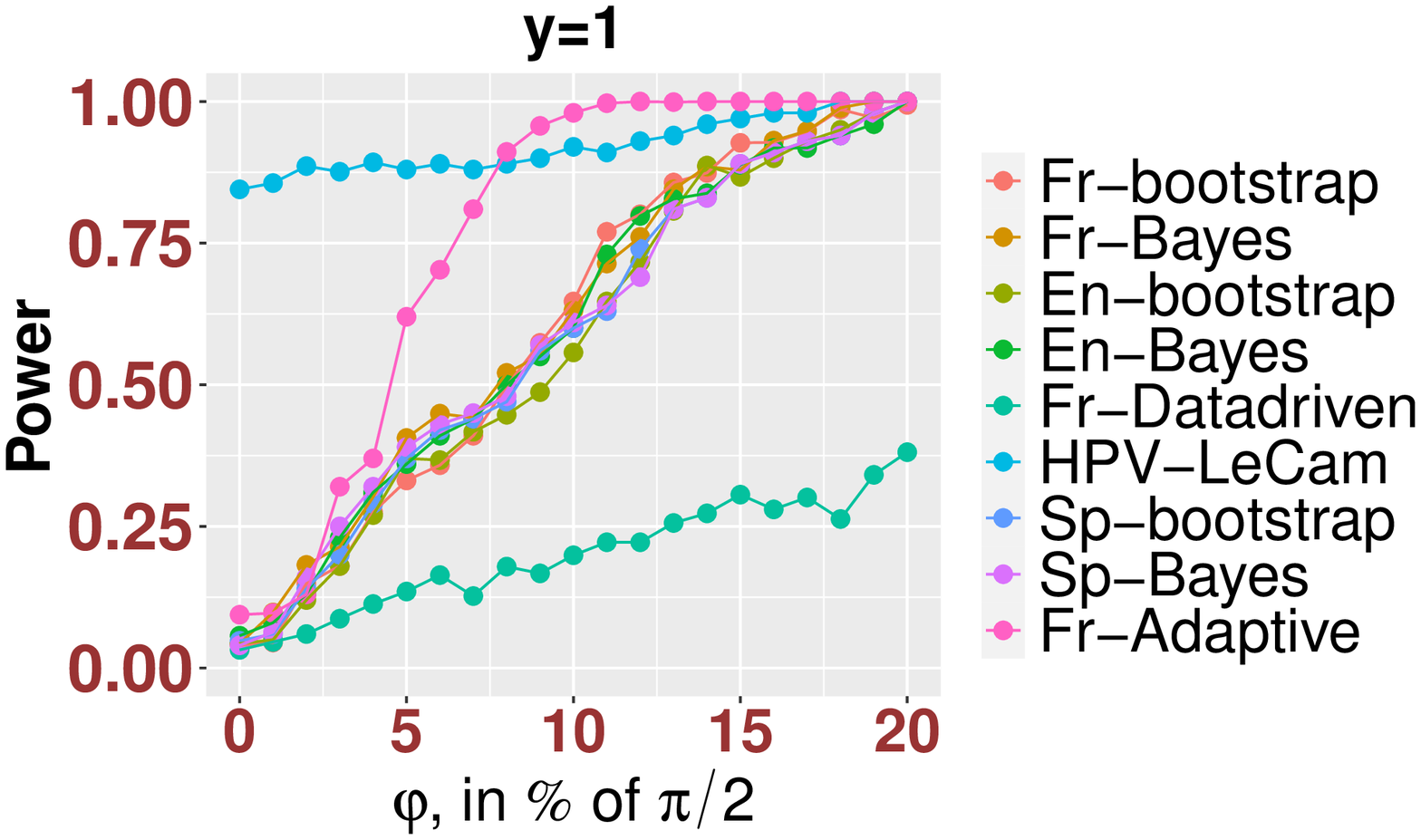}
\end{subfigure}
\begin{subfigure}{0.3\textwidth}
\includegraphics[width=5.8cm,height=5cm]{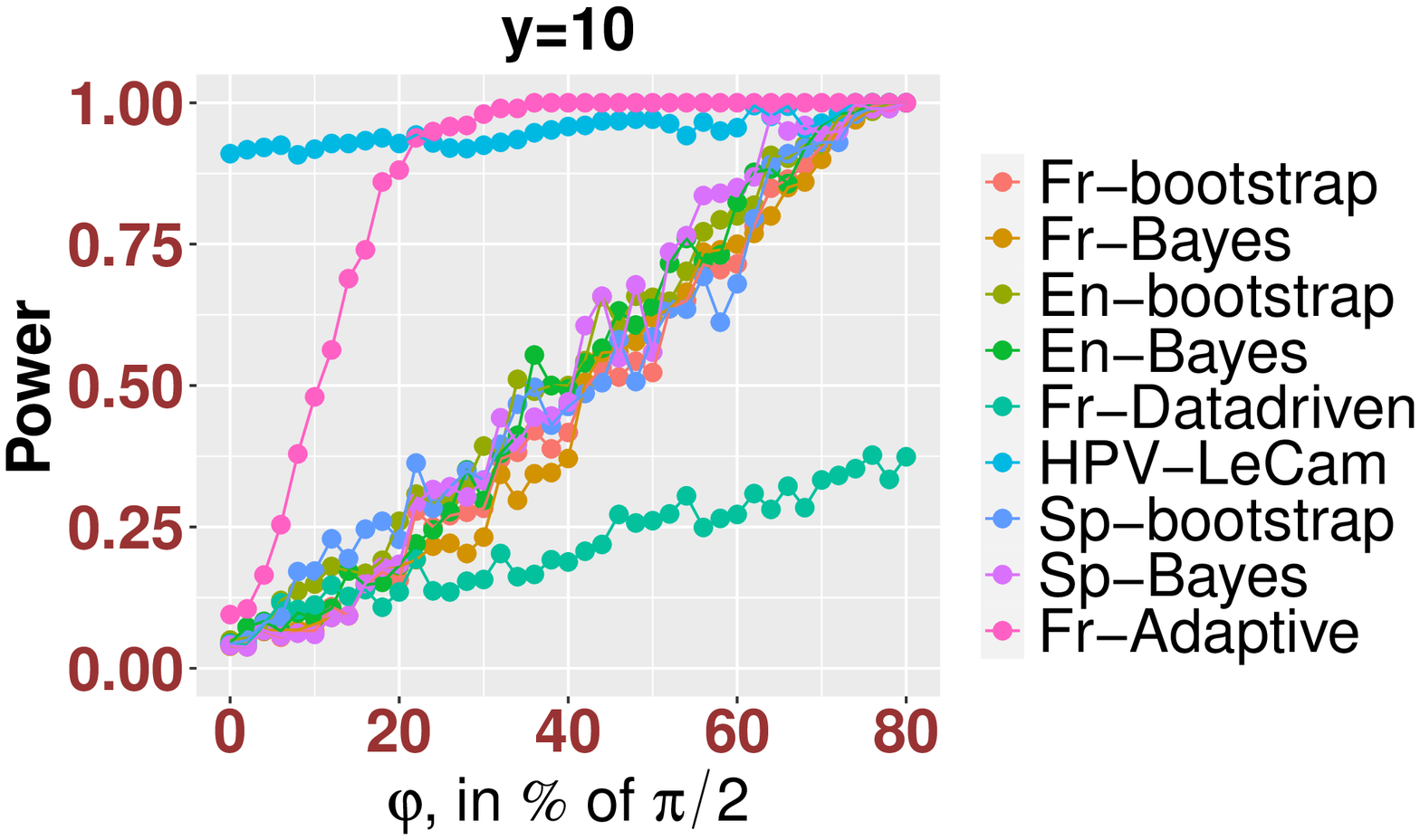}
\end{subfigure}
\caption{{ \footnotesize Comparison of power for Scenario II. We choose $d=50$ and use Gaussian random variables. We report our results under the nominal level 0.1 based on $2,000$ simulations.  Here $N=500.$  }  }
\label{fig_power250}
\end{figure}

Next, we investigate the size and power of the orthogonal test (\ref{eq_orthogobalteset}) using our proposed statistic (\ref{eq_finalstat}) with critical values generated adaptively according to Corollary \ref{cor_simulateddistribution}. We conduct numerical simulations using the following two scenarios. 
{
\begin{itemize}
\item  Scenario A: We consider the case $r_0=3$ with $d_1=d+7,$ $d_2=7$ and $d_3=5,$ where $d$ takes a variety of values. We consider the hypothesis testing of the eigenspace of $\Sigma=I + \sum_{i=1}^3d_i \bm{e}_i \bm{e}_i^*$ with $\mathcal{I}=\{1,2\},$ where the null is 
\begin{equation}\label{eq_trueeigenspace1}
Z_0=\bm{e}_3 \bm{e}_3^*+\bm{e}_4 \bm{e}_4^*.
\end{equation}
In this scenario, the spiked eigenvalues are simple. We will consider the standard Gaussian random variables and the two-point distribution $\frac{1}{3} \delta_{\sqrt{2}}+\frac{2}{3} \delta_{-\frac{1}{\sqrt{2}}}.$
\item Scenario B:  We consider the case $r_0=3$ with $d_1=d_2=d+5$ and $d_3=5,$ where $d$ takes a variety  of values. We consider the hypothesis testing under the null (\ref{eq_trueeigenspace1}).
In this scenario, we have multiple/identical $d_i$. We  consider both the standard Gaussian random variables and the two-point distribution $\frac{1}{3} \delta_{\sqrt{2}}+\frac{2}{3} \delta_{-\frac{1}{\sqrt{2}}}.$     
\end{itemize}

For both scenarios, we consider the alternative
\begin{equation}\label{eq_alternative1}
Z_a=\bm{v}_1(\varphi) \bm{v}_1(\varphi)^*+ \bm{v}_2(\varphi) \bm{v}_2(\varphi)^*,
\end{equation}
where for $\varphi \in [0,\frac{\pi}{2}]$
\begin{equation*}
\bm{v}_1(\varphi)=\cos \varphi \bm{e}_3+\sin \varphi \bm{e}_1, \ \bm{v}_2(\varphi)=\cos \varphi \bm{e}_4+\sin \varphi \bm{e}_2. 
\end{equation*}
Note $\varphi=0$ corresponds to the null case (\ref{eq_trueeigenspace1}).} For Gaussian random variables, we report the simulated type-I error rates in Figure \ref{fig_figorttypei} for Scenario A and in Figure \ref{fig_figorttypeii} for Scenario B for a variety of choices of $y$ and $d$. The results for the two-point random variables can be found in Figures \ref{fig_figorttypeitp} and \ref{fig_figorttypeiitp} of Appendix \ref{sec_additionalsimulation}. We can see that our proposed statistic (\ref{eq_finalstat})  with critical values generated adaptively according to Corollary \ref{cor_simulateddistribution} are reasonably accurate even for smaller values of $d$ and $N$ and we  have better accuracy when $N$ increases.  Finally, we examine the finite sample power for Gaussian  variables in Figures \ref{fig_figorttypeipower} and \ref{fig_figorttypeiipower} and two-point variables in Figures \ref{fig_figorttypeipowertp} and \ref{fig_figorttypeiipowertp}, we find that our statistic is powerful even for smaller values of $\varphi$ (i.e., weaker alternative) and $d.$

\begin{figure}[!ht]
\hspace*{-1.0cm}
\begin{subfigure}{0.45\textwidth}
\includegraphics[width=6.5cm,height=4.8cm]{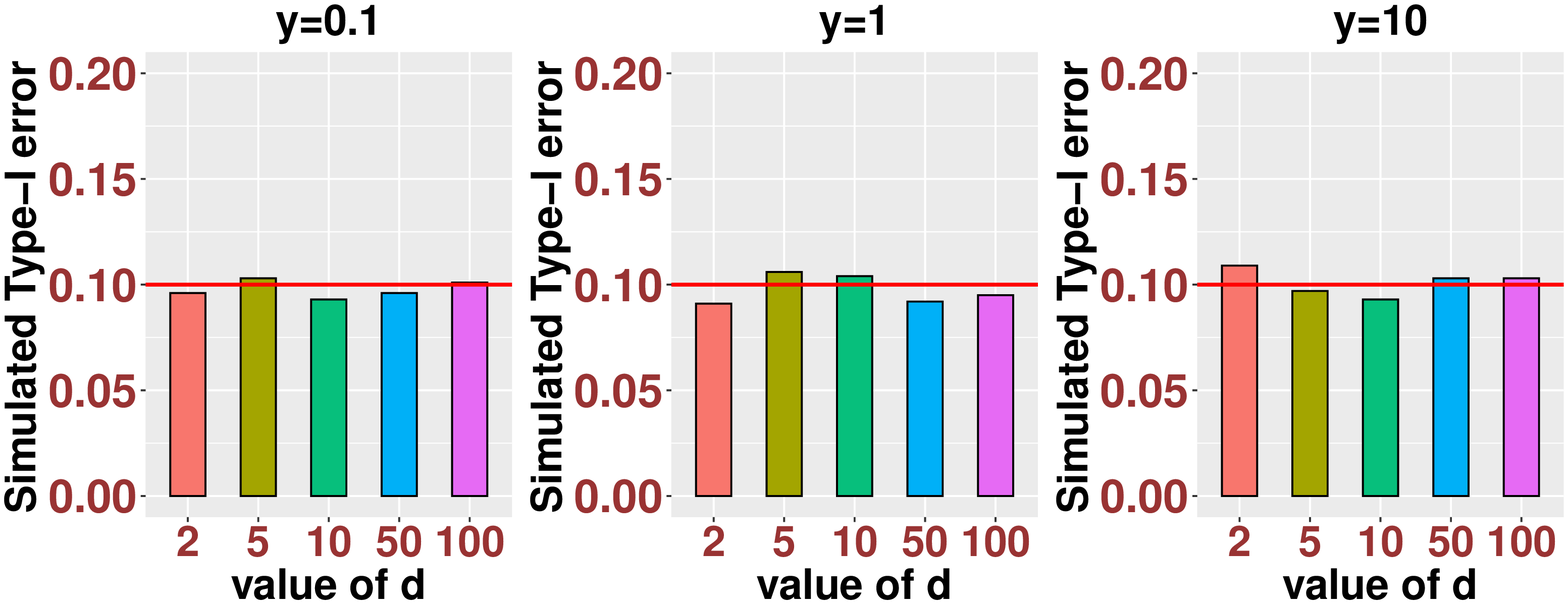}
\caption{$N=200.$}\label{subfig_nullorthogonaltypei200}
\end{subfigure}
\hspace{1cm}
\begin{subfigure}{0.45\textwidth}
\includegraphics[width=6.5cm,height=4.8cm]{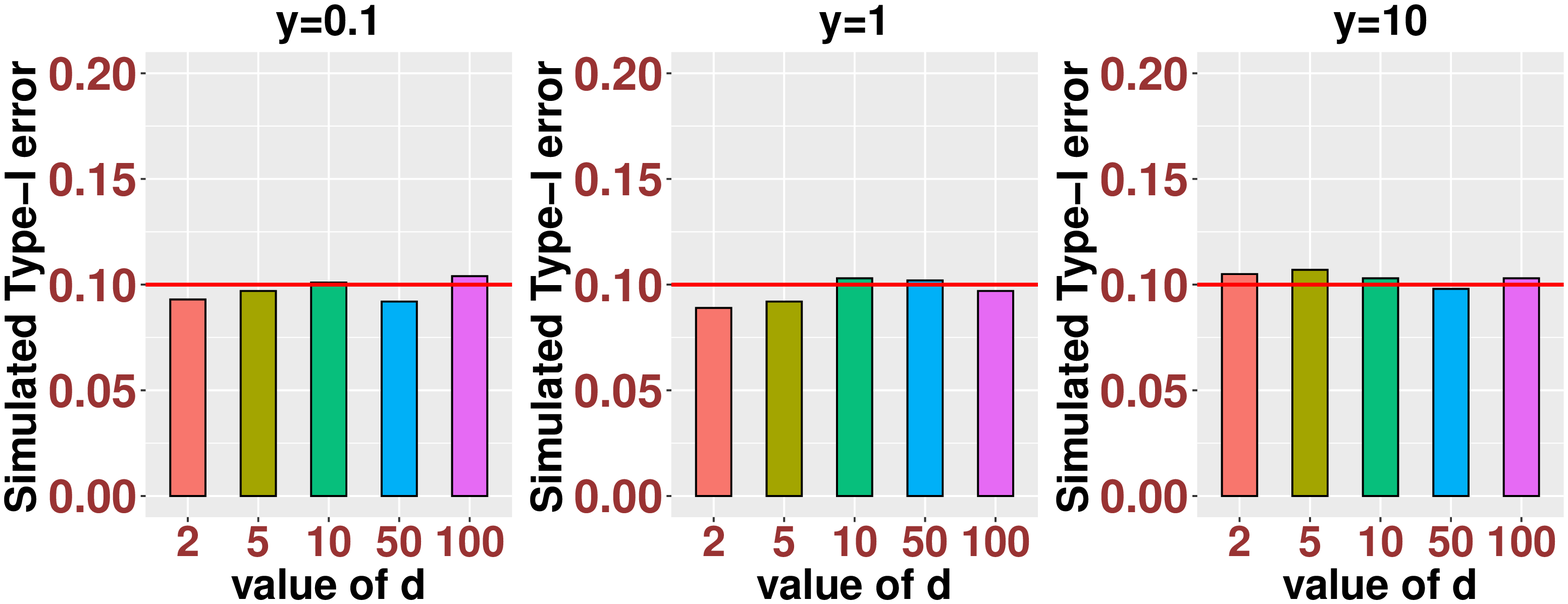}
\caption{$N=500.$}\label{subfig_nullorthogonaltypei500}
\end{subfigure}
\caption{{ \footnotesize Scenario A: simulated type I error rates for (\ref{eq_orthogobalteset}) using (\ref{eq_finalstat}). We report our results based on 2,000 Monte-Carlo simulations with Gaussian random variables. The critical values are generated using Corollary \ref{cor_simulateddistribution}. } }
\label{fig_figorttypei}
\end{figure}

\begin{figure}[!ht]
\hspace*{-1.0cm}
\begin{subfigure}{0.45\textwidth}
\includegraphics[width=6.5cm,height=4.8cm]{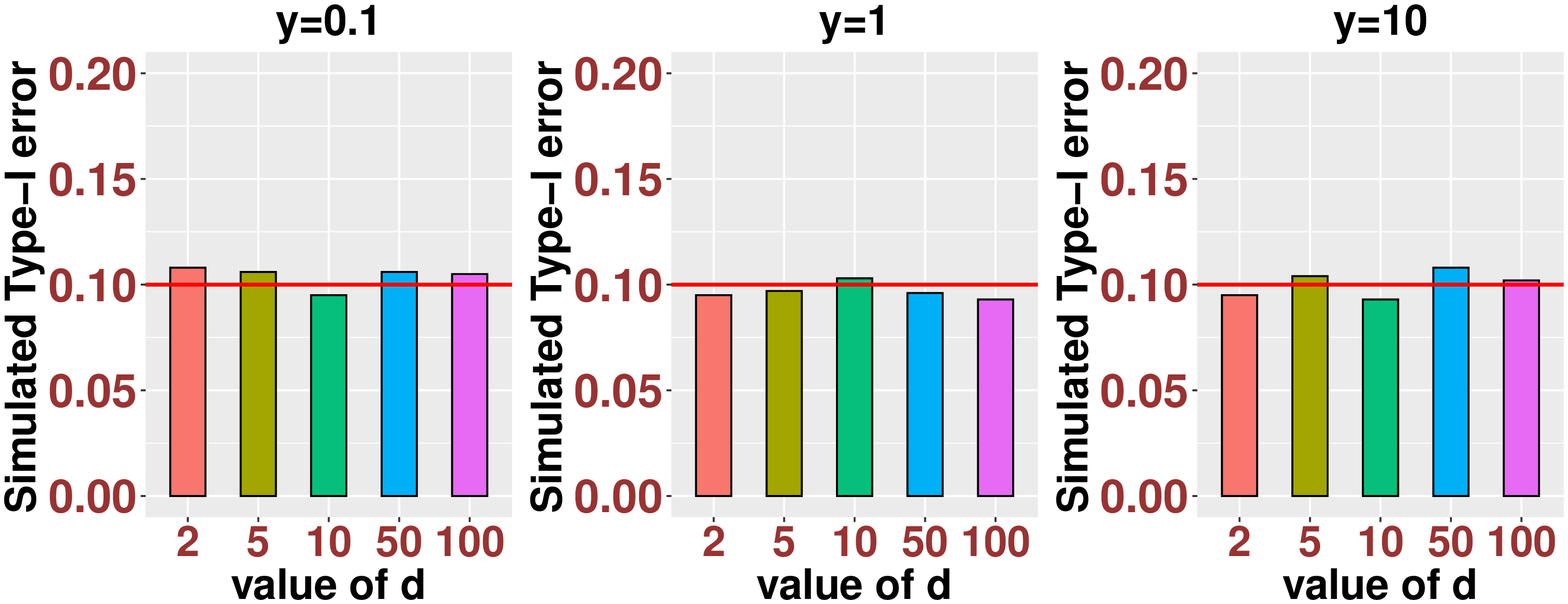}
\caption{$N=200.$}\label{subfig_nullorthogonaltypeii500}
\end{subfigure}
\hspace{1cm}
\begin{subfigure}{0.45\textwidth}
\includegraphics[width=6.5cm,height=4.8cm]{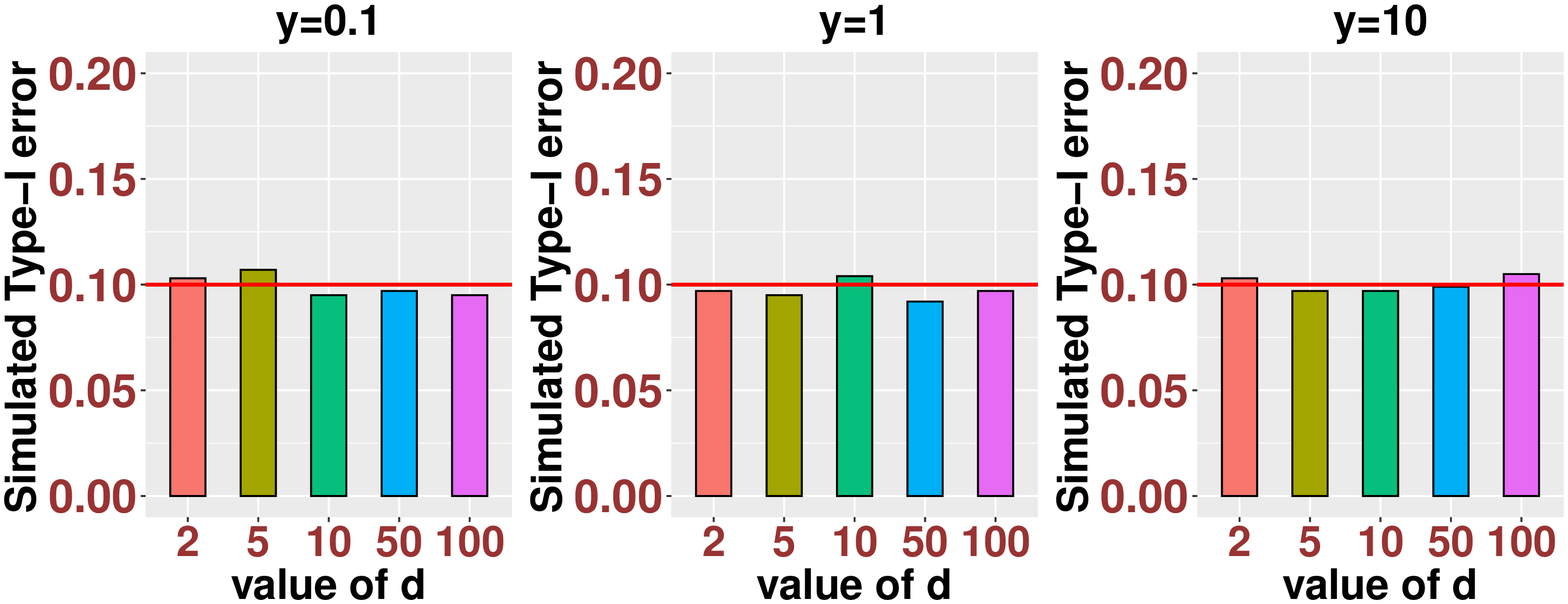}
\caption{$N=500.$}\label{subfig_nullorthogonaltypeii500}
\end{subfigure}
\caption{{ \footnotesize  Scenario B: simulated type I error rates for (\ref{eq_orthogobalteset}) using (\ref{eq_finalstat}). We report our results based on 2,000 Monte-Carlo simulations with Gaussian random variables. The critical values are generated from Corollary \ref{cor_simulateddistribution}. } }
\label{fig_figorttypeii}
\end{figure}

\begin{figure}[!ht]
\hspace*{-1.0cm}
\begin{subfigure}{0.45\textwidth}
\includegraphics[width=6.8cm,height=5cm]{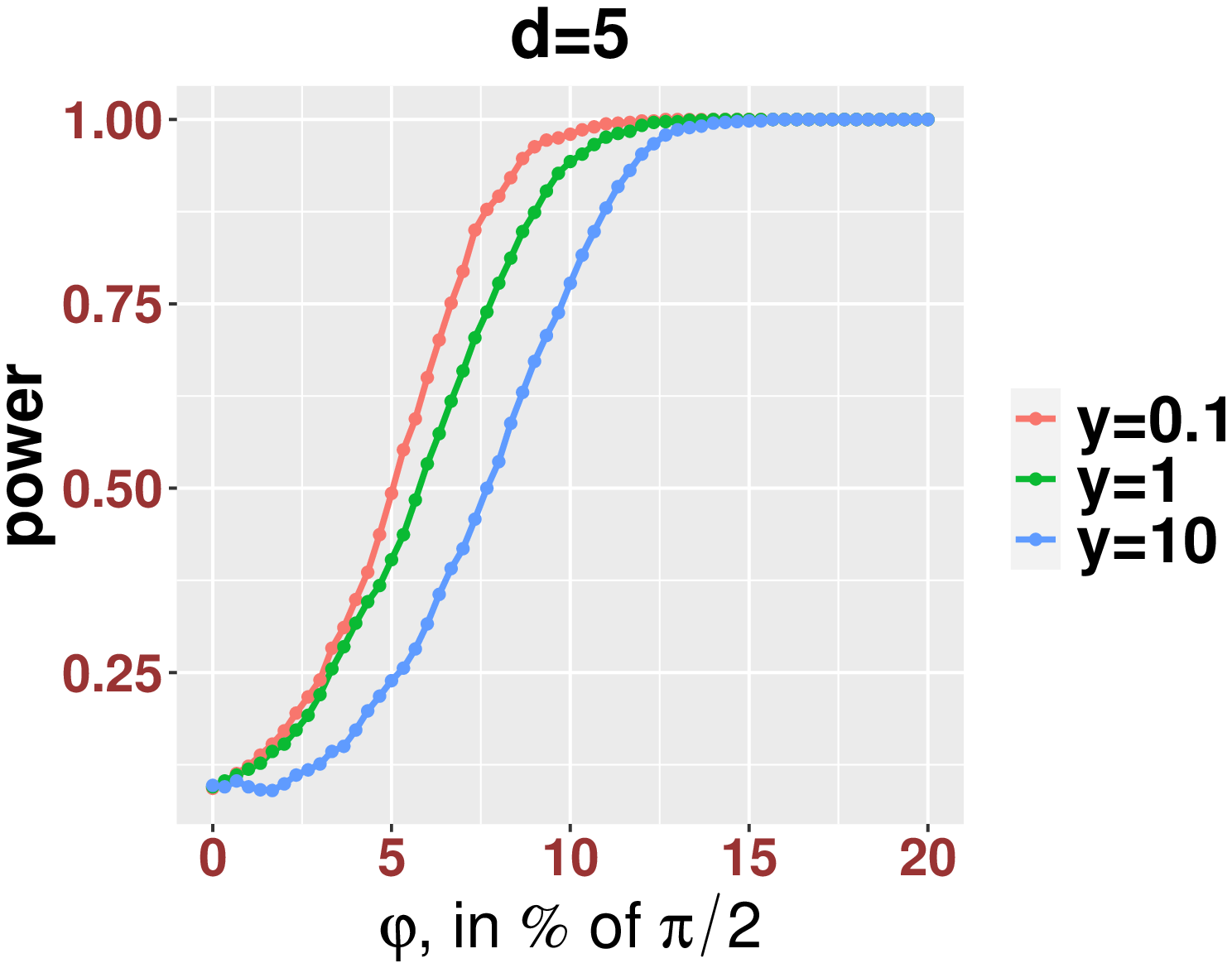}
\caption{$d=5.$}
\end{subfigure}
\hspace{1cm}
\begin{subfigure}{0.45\textwidth}
\includegraphics[width=6.8cm,height=5cm]{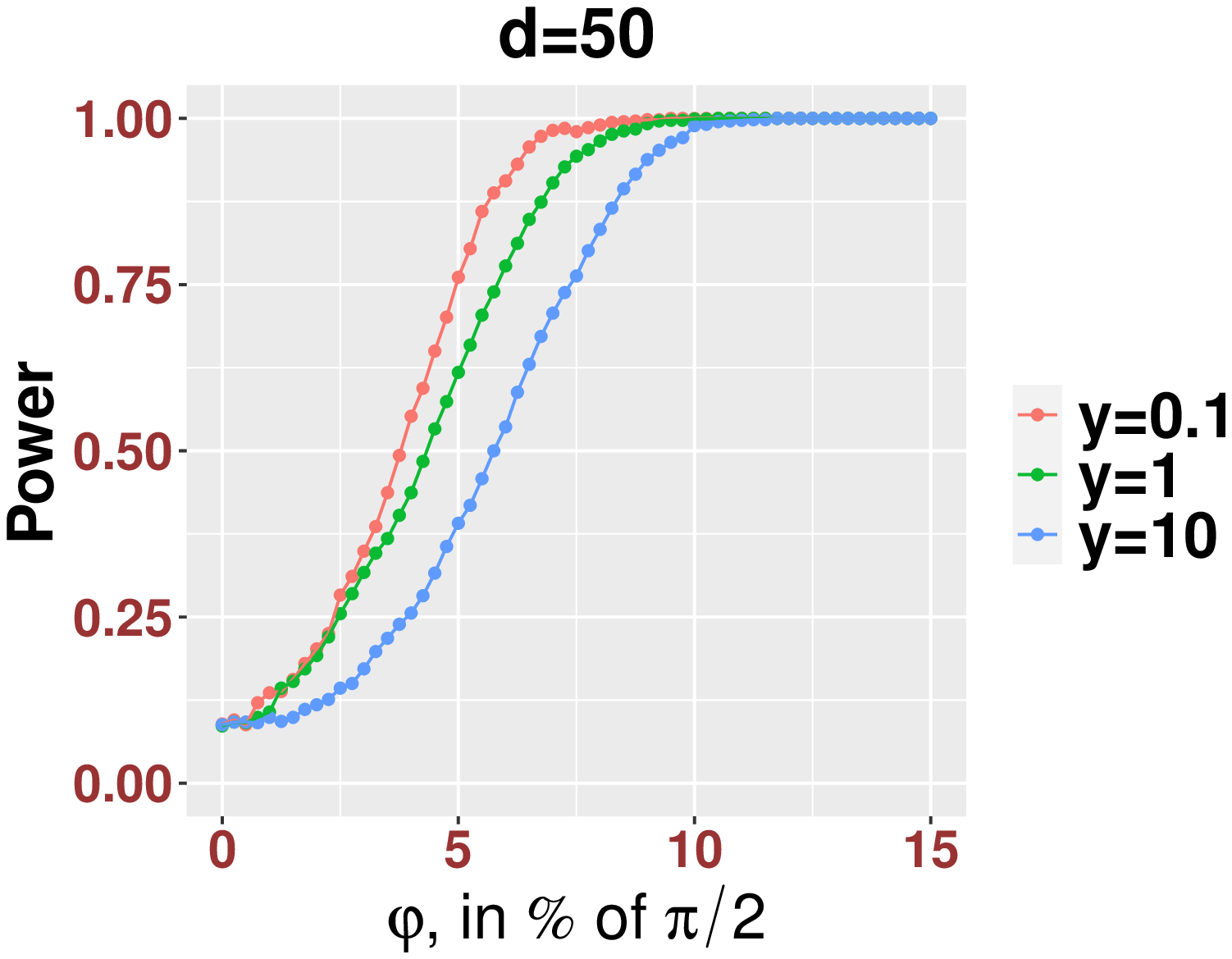}
\caption{$d=50.$}
\end{subfigure}
\caption{{ \footnotesize Power of Scenario A for Gaussian random variables using (\ref{eq_finalstat}).  We report our results under the nominal level 0.1 based on $2,000$ simulations.  Here $N=500$ and the critical values are generated using Corollary \ref{cor_simulateddistribution}. } }
\label{fig_figorttypeipower}
\end{figure}

\begin{figure}[!ht]
\hspace*{-1.0cm}
\begin{subfigure}{0.45\textwidth}
\includegraphics[width=6.8cm,height=5cm]{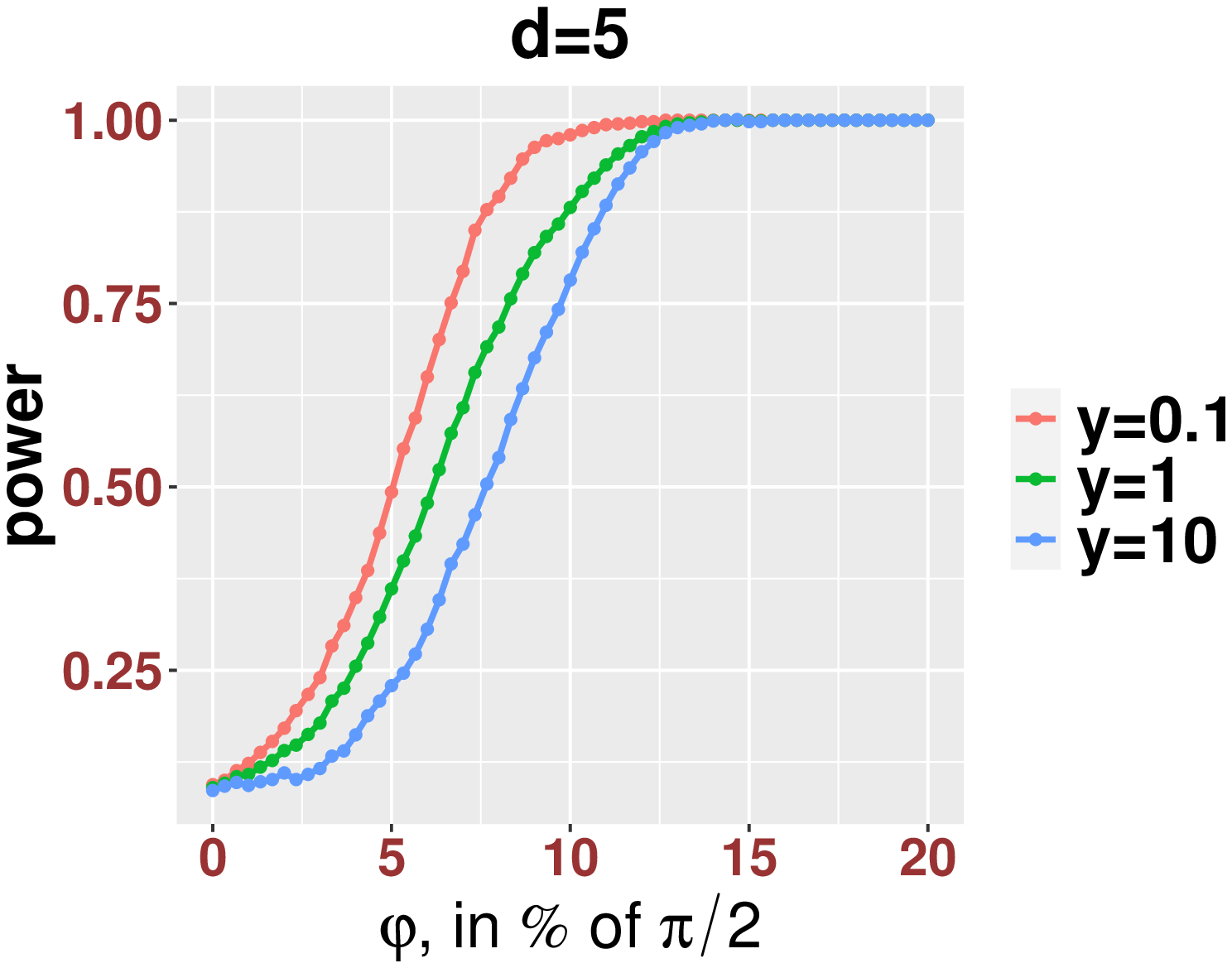}
\caption{$d=5.$}
\end{subfigure}
\hspace{1cm}
\begin{subfigure}{0.45\textwidth}
\includegraphics[width=6.8cm,height=5cm]{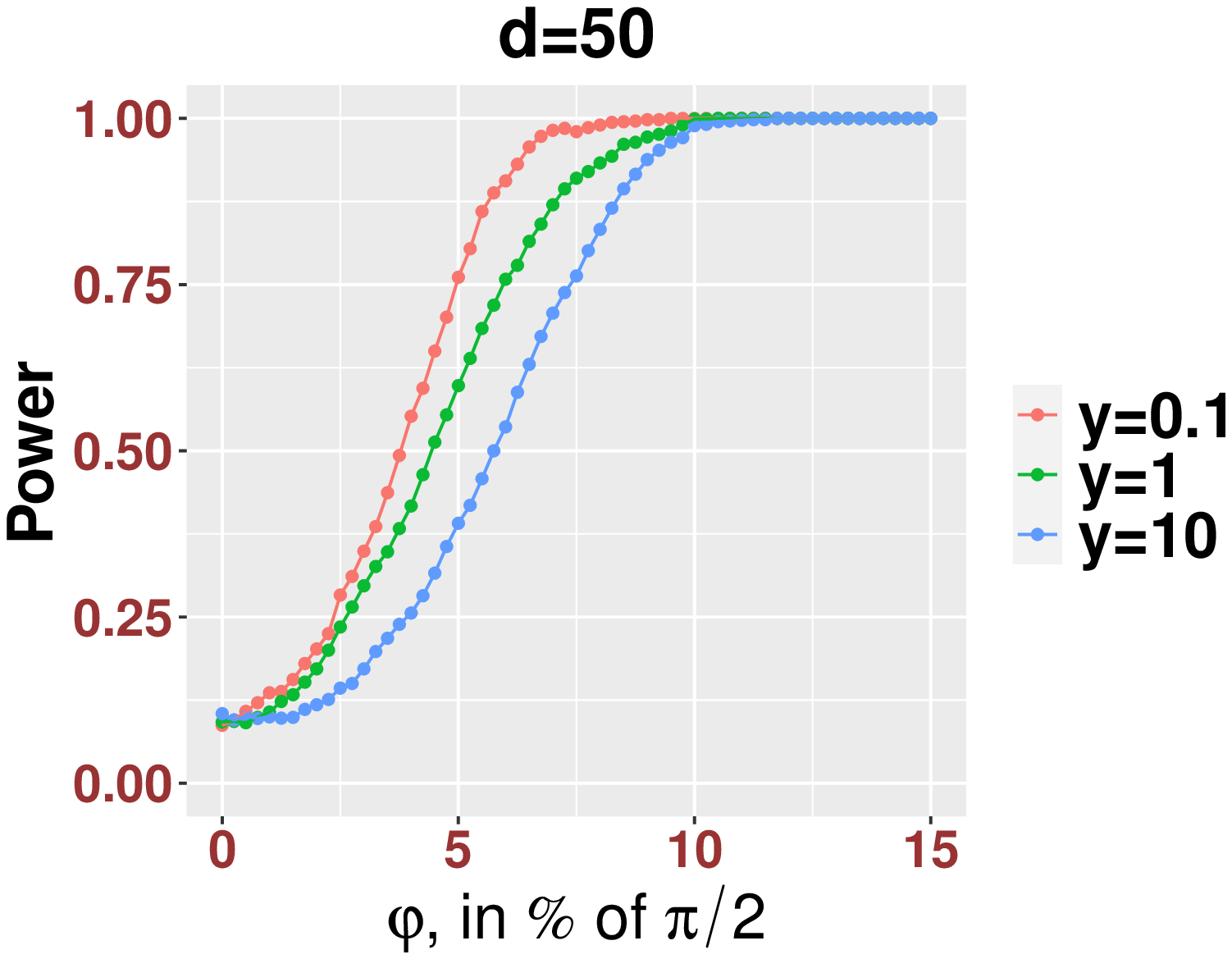}
\caption{$d=50.$}
\end{subfigure}
\caption{{  \footnotesize Power of Scenario B for Gaussian random variables using (\ref{eq_finalstat}).  We report our results under the nominal level 0.1 based on $2,000$ simulations.  Here $N=500$ and the critical values are generated using Corollary \ref{cor_simulateddistribution}. } }
\label{fig_figorttypeiipower}
\end{figure}

\section{Preliminaries} \label{s. pre}
In this section, we collect some basic notions and preliminary results which will be used in the proof of our main theorem. A key technical input is the isotropic local law from \cite{bloemendal2014isotropic, knowles2017anisotropic}.

\subsection{Basic notions}\label{basicnotations}

In the sequel,  we denote the Green function of $Q$ by
\begin{align}
G(z):=(Q-z)^{-1}, \quad z\in \mathbb{C}^+. \nonumber
\end{align}
The matrix $Q$ can be regarded as a finite-rank perturbation of the matrix $H:=XX^*$. In the sequel, we also need to consider $\mathcal{H}:=X^*X$ which shares the same non-zero eigenvalues with $H$. We further denote the Green functions of $H$ and $\mathcal{H}$ respectively  by
\begin{align}\label{def:green}
\mG_1(z):=(XX^*-z)^{-1}&, \qquad \mG_2(z):=(X^*X-z)^{-1}, \quad z\in \mathbb{C}^+,  
\end{align}
and their  normalized traces  by 
\begin{align*}
&m_{1N}(z):=\frac{1}{M}\text{Tr}\mG_1(z)=\int(x-z)^{-1}\,dF_{1N}(x),\nonumber\\
&m_{2N}(z):=\frac{1}{N}\text{Tr}\mG_2(z)= \int(x-z)^{-1}\,dF_{2N}(x),
\end{align*}
where  $F_{1N}(x)$, $F_{2N}(x)$ are the empirical spectral distributions of $H$ and $\mathcal{H}$ respectively, i.e.,
\begin{align}
F_{1N}(x):=\frac{1}{M}\sum_{i=1}^M \mathds{1}(\lambda_i(H)\leq x), \quad F_{2N}(x):=\frac{1}{N}\sum_{i=1}^N \mathds{1}(\lambda_i(\mathcal{H})\leq x). \nonumber
\end{align}
Here we used $\lambda_i(H)$ and $\lambda_i(\mathcal{H})$  to denote the $i$-th largest eigenvalue of $H$ and $\mathcal{H}$, respectively. 

It is well-known since \cite{MP67} that $F_{1N}(x)$ and $F_{2N}(x)$ converge weakly (a.s.) to the {\it Marchenko-Pastur} laws 
$\nu_{\text{MP},1}$ and $\nu_{\text{MP},2}$ (respectively) given below
\begin{align}
&\nu_{\text{MP},1}({\rm d}x):=\frac{1}{2\pi xy}\sqrt{\big((\lambda_+-x)(x-\lambda_-)\big)_+}{\rm d}x+(1-\frac{1}{y})_+\delta({\rm d}x),\nonumber\\
&\nu_{\text{MP},2}({\rm d}x):=\frac{1}{2\pi x}\sqrt{\big((\lambda_+-x)(x-\lambda_-)\big)_+}{\rm d}x+(1-y)_+\delta({\rm d}x), \label{19071801}
\end{align}
where $\lambda_{\pm}:=(1\pm \sqrt{y})^2$.  Note that here the parameter $y$ may be $N$-dependent. Hence, the weak convergence (a.s.) shall be understood as $\int g(x) {\rm d} F_{aN}(x)-\int g(x) \nu_{\text{MP},a}({\rm d}x) \stackrel{a.s.} \longrightarrow 0 $ for any given bounded continuous function $g:\mathbb{R}\to \mathbb{R}$, for $a=1,2$.  We further denote by $F_\alpha$ the cumulative distribution function of $\nu_{\text{MP},\alpha}$ for $\alpha=1,2$. 
Note that $m_{1N}$ and $m_{2N}$ can be regarded as the Stieltjes transforms of $F_{1N}$ and $F_{2N}$, respectively.  We further define their deterministic counterparts, i.e.,  Stieltjes transforms of $\nu_{\text{MP},1},\nu_{\text{MP},2}$,  by $m_1(z),m_2(z)$,  respectively, i.e.,
 \begin{align}
 m_1(z):=\int (x-z)^{-1}\nu_{\text{MP},1}({\rm d}x), \quad m_2(z):=\int (x-z)^{-1}\nu_{\text{MP},2}({\rm d}x). \nonumber
 \end{align}
 From the definition (\ref{19071801}), it is elementary to compute 
 \begin{align}
&m_1(z)=\frac{1-y-z+\ii\sqrt{(\lambda_+-z)(z-\lambda_-)}}{2zy},  \nonumber\\
&m_2(z)=\frac{y-1-z+\ii\sqrt{(\lambda_+-z)(z-\lambda_-)}}{2z}, \label{m1m2}
 \end{align}
where the square root is taken with a branch cut on the negative real axis. Equivalently, we can also characterize $m_1(z),m_2(z)$ as the unique solutions from $\mathbb{C}^+$ to $\mathbb{C}^+$ to the equations
\begin{align}
zym_1^2+[z-(1-y)]m_1+1=0, \qquad  zm_2^2+[z+(1-y)]m_2+1=0. \label{selfconeqt}
\end{align} 
Using \eqref{m1m2} and \eqref{selfconeqt}, one can easily derive the following identities  
\begin{align}
m_1=-\frac{1}{z(1+m_2)}, \quad 1+zm_1=\frac{1+zm_2}{y},\quad m_1\big((zm_2)'+1\big)=\frac{m_1'}{m_1}, \label{identitym1m2}
\end{align}
which will be used in the later discussions.

\subsection{Isotropic local law}
In this section, we state the isotropic local law from \cite{bloemendal2014isotropic, knowles2017anisotropic} together with some  consequences which will serve as the main technical inputs in the proofs. But before the statements of these estimates, we first introduce a basic lemma about  the notion ``stochastic domination" introduced in Definition \ref{def.sd}. 
\begin{lem} \label{prop_prec} Let
	\begin{equation*}
   {X}_i=({X}_{N,i}(u):  N \in \mathbb{N}, \ u \in {U}_{N}), \   {Y}_i=({Y}_{N,i}(u):  N \in \mathbb{N}, \ u \in {U}_{N}),\quad i=1,2
	\end{equation*}
	be families of  random variables, where ${Y}_i, i=1,2,$ are nonnegative, and ${U}_{N}$ is a possibly $N$-dependent parameter set.	Let 
	\begin{align*}
	\Phi=(\Phi_{N}(u): N \in \mathbb{N}, \ u \in {U}_{N})
	\end{align*}
	be a family of deterministic nonnegative quantities. We have the following results:
	
(i)	If ${X}_1 \prec {Y}_1$ and ${X}_2 \prec {Y}_2$ then ${X}_1+{X}_2 \prec {Y}_1+{Y}_2$ and  ${X}_1{X}_2 \prec {Y}_1 {Y}_2$.

 (ii) Suppose ${X}_1 \prec \Phi$, and there exists a constant $C>0$ such that  $|{X}_{N,1}(u)| \leq N^C$ a.s.\ and $\Phi_{N}(u)\geq N^{-C}$ uniformly in $u$ for all sufficiently large $N$. Then $\mathbb{E} {X}_1 \prec \Phi$.
\end{lem}

\begin{proof}
	Part (i) is obvious from Definition \ref{def.sd}. For any fixed $\varrho>0$, we have
\begin{align*}
	|\mathbb{E} X_1| &\leq \mathbb{E} |X_1\mathds{1}(|X_1|\leq  N^{\varrho}\Phi)|+\mathbb{E} |X_1\mathds{1}(|X_1|\geq  N^{\varrho}\Phi)|\nonumber\\
	&\leq N^{\varrho}\Phi+ N^C \mathbb{P}(|X_1|\geq N^{\varrho}\Phi)=O(N^{\varrho}\Phi)
\end{align*}
	for for sufficiently large $N\geq N_0(\varrho)$. This proves part (ii). 
\end{proof}

In the sequel, we state the isotropic local law and related estimates.  We first introduce the following domain. For a small (but fixed) $\tau>0$, we set
\begin{align}
{\mathscr{D}}\equiv {\mathscr{D}}(\tau):=\big\{z=E+\ii\eta\in\mathbb{C}:E\geq \lambda_{+}+N^{-\frac 23+\tau}, 0<\eta\leq \tau^{-1} \big\}. \label{19071810}
\end{align}
Conventionally,  for $a=1,2$,  we denote by $\mathcal{G}_a^l$ and $\mathcal{G}_a^{(l)}$  the $l$-th power of $\mathcal{G}_a$ and the $l$-th derivative of $\mathcal{G}_a$ w.r.t. $z$, respectively. 
With the above notation, we have the following theorem.
\begin{thm} \label{isotropic}
Let  $\tau>0$ in (\ref{19071810}) be a small but fixed constant. Let  $\bu,\bv$ be complex deterministic unit vectors of proper dimensions. Suppose $X$ satisfies Assumption \ref{assumption}. Then,  for any given $l\in\mathbb{N}$ and $\alpha=1,2$, we have  
\begin{align}
&|\langle \bu,\mG_\alpha^{(l)}(z)\bv\rangle-m_\alpha^{(l)}(z)\langle \bu,\bv\rangle|\prec \frac{\min\{ (\kappa+\eta)^{-2}, (\kappa+\eta)^{-\frac 14}\}}{(\kappa+\eta)^{l}\sqrt N},  \label{est.DG}
\\
&|\langle \bu,X^*\mG_1^{l+1}(z)\bv\rangle|\prec  \sqrt{\frac{\Im m_1(z)}{N\eta}} \frac{1}{(\kappa+\eta)^{l}}, \quad  |\langle  \bu,X\mG_2^{l+1}(z) \bv\rangle|\prec  \sqrt{\frac{\Im m_2(z)}{N\eta}} \frac{1}{(\kappa+\eta)^{l}}, \label{081501}
\\
&|m_{\alpha N}^{(l)}(z)-m^{(l)}_\alpha(z)|=O_\prec\Big(\frac{1}{N(\kappa+\eta)^{l+1}}\Big)
\label{est_m12N}, 
\end{align}
uniformly in $z\in{\mathscr{D}} $, where $\kappa=E-\lambda_+$. Further, when $\kappa>K$ for some sufficiently large constant $K>0$ and $|y-1|\geq \tau_0$ for any small but fixed $\tau_0>0$,  (\ref{est_m12N}) can be improved to 
\begin{align}
|m_{\alpha N}^{(l)}(z)-m^{(l)}_\alpha(z)|=O_\prec\Big(\frac{1}{N(\kappa+\eta)^{l+2}}\Big).
\label{est_m12N for large z}
\end{align}
\end{thm}
\begin{rmk}
The case of $l=0$ of (\ref{est.DG}) and  (\ref{081501}) is  from the isotropic law for extended spectral domain in  Proposition 3.8 of \cite{bloemendal2016principal} which is derived from Theorem 3.12 of \cite{bloemendal2014isotropic} and the anisotropic laws in Theorem 3.7 of   \cite{knowles2017anisotropic}.  We emphasize that the extension to all $z$ for (\ref{081501}) is not directly included in Proposition 3.8 of \cite{bloemendal2016principal}. Nevertheless, the approach used there can be adapted to prove  (\ref{081501}) for all $z$, starting from the result in \cite{knowles2017anisotropic}  which is stated for fixed $z$ only. Specifically, writing $2\langle \bu,X^*\mG_1(z)\bv\rangle=(X\bu+\bv)^*\mG_1(z)(X\bu+\bv)-(X\bu)^*\mG_1(z)(X\bu)-\bv^*\mG_1(z)\bv$, one can use the approach for the extension in \cite{bloemendal2016principal} to each symmetric quadratic forms and conclude the extension of the estimate for $\langle \bu,X^*\mG_1^{l}(z)\bv\rangle$. 
For other $l\geq 1$, we can derive the estimate easily from the case $l=0$ by using Cauchy integral with  the radius of the contour taking value $|z-\lambda_+|/2 \asymp \kappa+\eta$ . We also remark here that the original isotropic local laws in \cite{bloemendal2014isotropic, knowles2017anisotropic} were stated in much larger domains which also include the bulk and  edge regimes of the MP law. But here we only need the result for the domain far away from the support of the MP law. The (\ref{est_m12N}) can be obtained by the rigidity estimates of eigenvalues in \cite[Theorem 3.3]{PY14}  and the definition of the Stieltjes transform;  see e.g. \cite[Remark 4.5]{BDW} for more details.  For (\ref{est_m12N for large z}), denoting by $F_\alpha$ the cumulative distribution function of the measure $\nu_{\text{MP},\alpha}$ and using integration by parts, we have
\begin{align*}
m_{\alpha N}^{(l)}(z)-m^{(l)}_\alpha(z) &\asymp \int \frac{1}{(\lambda-z)^{\ell+1}} {\rm d} (F_{\alpha N}(\lambda)-F_\alpha(\lambda))\nonumber\\
&\asymp \int \frac{1}{(\lambda-z)^{\ell+2}}  (F_{\alpha N}(\lambda)-F_\alpha(\lambda)) {\rm d} \lambda\notag\\
&=O_\prec \Big(\frac{1}{N}\int_{0}^{\lambda_++N^{-\frac 23+\frac{\tau}{2}}} \frac{1}{|\lambda-z|^{\ell+2}} {\rm d}\lambda\Big ),
\end{align*}
where the last step follows from the rigidity results in Lemma \ref{rigidity of H} below and its consequence on the convergence rate,  $\sup_x|F_{\alpha N}(x)-F_\alpha(x)|\prec \frac{1}{N}$; see for instance \cite{PY14}. 
Then  (\ref{est_m12N for large z}) follows by elementary calculation. 
\end{rmk}

Further,  in the following lemma, we collect some basic estimates of $m_1$ and $m_2$ which can be verified by elementary computations.

\begin{lem} \label{lem.19072501*}
 Recall the definition of $m_1$ and $m_2$ in \rf{m1m2}. For $a=1,2$, we have
\begin{align}
|zm_{a}(z)|\asymp 1,\quad  |1+zm_a(z)|\asymp \min\{1,(\kappa+\eta)^{-1}\},
 \quad 
{\Im} m_{a}(z) \asymp \frac{\eta}{|z|^{\frac 32}\sqrt{\kappa+\eta}},  \label{estm1m2*}
\end{align}
and 
\begin{align}  \label{estm1m2**}
 |z^{\frac 32}m_{a}^{(l)}(z)|\asymp (\kappa+\eta)^{-l+\frac 12},
\quad 
|(zm_a)^{(l)}| \asymp \frac{(\kappa+\eta)^{-l+\frac 12}}{|z|^{\frac 32}},  \text{ for } l\geq 1
 \end{align}
uniformly in $z\in{\mathscr{D}}$ for some positive constant $C$, where $\kappa=E-\lambda_+$.
\end{lem}

\begin{rmk} \label{boundrmk*}
By Theorem \ref{isotropic}, Lemma \ref{lem.19072501*} and Lemma \ref{lemrelation} , we can easily bound  $(\kappa+\eta)^l z\bu^*\mG_1^{(l)}\bv $, $(\kappa+\eta)^l z\bu^*X^*\mG_1^{(l)}\bv$ and $(\kappa+\eta)^l\bu^*X^* \mG_1 ^{(l)}X \bv$ for any deterministic unit vectors $\mathbf{u}$ and $\mathbf{v}$ of appropriate dimensions. For instance, for $(\kappa+\eta)^l\bu^*X^*\mG_1 ^{(l)}X\bv$, we have the following estimates on ${\mathscr{D}}$,
\begin{align*}
&\bu^*X^*\mG_1 X\bv =(1+zm_2) \bu^*\bv+ O_\prec \big({\min\{ (\kappa+\eta)^{-2}, (\kappa+\eta)^{-\frac 14}\}}N^{-\frac 12} \big)\nonumber\\
&\qquad\qquad\quad= O_\prec(\min\{1,(\kappa+\eta)^{-1}\}),\\
&(\kappa+\eta)^l\bu^*X^*\mG_1 ^{(l)}X\bv =(\kappa+\eta)^{l}\bu^* (\mG_1^{l}+z\mG_1^{l+1})\bv\notag\\
&\qquad\qquad\quad= O_\prec\Big(|(\kappa+\eta)^{l} (zm_2)^{(l)} \bu^*\bv |\Big)=O_\prec \big(|z|^{-\frac 32} \sqrt {\kappa+\eta}\, \big).
\end{align*}
\end{rmk}

Using the isotropic local law, one can also get the following result, which gives the location of the outlier and the  extremal non-outlier.

\begin{lem}\label{locationeig}
{\rm(}{\it Theorem 2.3} of \cite{bloemendal2016principal}{\rm)} Under Assumption \ref{assumption} and \rf{asd}, we have for $i\in \lb 1, r_0\rb$
\begin{align*}
|\mu_i-\theta(d_i)|\prec  (d_i-\sqrt y)^{\frac 12}d_i^{\frac 12} N^{-\frac12},\qquad |\mu_{r_0+1}-\lambda_+|\prec N^{-2/3},
\end{align*}
where
\begin{align}
\theta(z):=1+z+y+yz^{-1}, \qquad \text{for   } z\in \mathbb{C},\quad \Re z>\sqrt{y}. \label{def. of theta}
\end{align}
\end{lem}

Further, for the eigenvalues of $H$,  we have the following rigidity estimate. 
\begin{lem}[Theorem 3.1 of \cite{PY14}] \label{rigidity of H} Suppose that Assumption \ref{assumption} holds. Denote by $\gamma_i$ the $i$-th largest $N$-quantile of $F_1$, i.e, $1-F_1(\gamma_i)=\frac{i-1/2}{N}$. We have  
\begin{align}
|\lambda_i(H)-\gamma_i|\prec  N^{-\frac23}\big(\min \{\min\{N,M\}+1-i,i\}\big)^{-\frac13},  \label{190726100}
\end{align}
for  $1\leq i\leq (1-\tau_0)\min\{N,M\}$ or  for all $1\leq i\leq \min\{N,M\}$ in case $|y-1|\geq \tau_0$,  for some small but fixed $\tau_0>0$. 
Especially, (\ref{190726100}) implies
\begin{align*}
|\lambda_1(H)-\lambda_+|\prec N^{-\frac23}.
\end{align*}
Also, we note that $\lambda_i(H)=0$ if $i\geq \min\{N,M\}+1$.  
\end{lem}

\begin{rmk} \label{Rmk:UpsilonDelta}
By setting $z:=\theta(d_i)+\ii N^{-K}$ for a sufficiently large integer $K$, we have the estimates  
\begin{align}\label{est_z_kap}
|z| \asymp \theta(d_i) \asymp d_i, \quad \kappa+\eta\asymp \kappa\asymp \Upsilon.
\end{align}
Here we set 
$$\Upsilon\equiv \Upsilon(d_i):=(d_i-\sqrt y)^2/d_i .$$ 
Applying the above estimates, we can further simplify the error bounds in Theorem \ref{isotropic} to get 
\begin{align}
&|\langle \bu, \Upsilon^l\mG_\alpha^{(l)}(z)\bv\rangle- \Upsilon^lm_\alpha^{(l)}(z)\langle \bu,\bv\rangle|=O_\prec\big(\Delta(d_i)N^{-1/2}\big), \label{weak_est_DG} \\
&\max \big\{ |\langle  \bu, \Upsilon^{l-1}X^*\mG_1^{l}(z)\bv\rangle|, |\langle  \bu, \Upsilon^{l-1}X\mG_2^{l}(z)\bv\rangle| \big\}\notag\\
&\qquad\qquad\qquad=O_\prec\big( (d_i-\sqrt y)^{-\frac 12}d_i^{-\frac 12} N^{-1/2}\big), \label{weak_est_XG}\\
&|\Upsilon^{l}m_{\alpha N}^{(l)}(z)-\Upsilon^{l}m^{(l)}_\alpha(z)|=O_\prec\big((d_i-\sqrt y)^{-2} N^{-1}\big)  \label{weak_est_m12N},
\end{align}
for any deterministic unit vectors $\bu, \bv$ of appropriate dimensions, where for convenience, we set 
\begin{align} \label{def:Delta(d)}
\Delta(d):=\min\{(d-\sqrt y)^{-4}d^2, (d-\sqrt y)^{- 1/2}d^{1/4}\}.
\end{align}
We remark here that due to the Lipschitz continuity of $\mG_\alpha^{(l)}(z), m_\alpha^{(l)}(z), \mG_\alpha^{l}(z) $,  the estimates in \eqref{weak_est_DG}-\eqref{weak_est_m12N} also hold for $z=\theta(d_i)$ with the error bounds unchanged.
\end{rmk}

\begin{rmk}\label{Rmk:Delta(d)}
According to the definition of $\Delta(d)$, by simple calculations,  we see that for $i\in \lb 1, r_0\rb$, i.e. $d_i>\sqrt y$, 
\begin{align}
\Delta(d_i)&=\left \{
\begin{array}{cc}
(d_i-\sqrt y)^{- 1/2}d_i^{1/4} & \text{if } d_i<\Big(\frac 12+\sqrt{\sqrt y+\frac 14}\Big)^2\\
(d_i-\sqrt y)^{-4}d_i^2 &\text{if } d_i\geq \Big(\frac 12+\sqrt{\sqrt y+\frac 14}\Big)^2
\end{array}
\right. \nonumber\\
&\asymp
\left \{
\begin{array}{cc}
(d_i-\sqrt y)^{- 1/2}& \text{if } d_i<\Big(\frac 12+\sqrt{\sqrt y+\frac 14}\Big)^2\\
d_i^{-2} &\text{if } d_i\geq \Big(\frac 12+\sqrt{\sqrt y+\frac 14}\Big)^2.
\end{array}
\right.  \label{Delta:asymp}
\end{align}
\end{rmk}

Finally, for the large $d_i$ regime, i.e., $d_i>K$ for some sufficiently large $K$, we will also need a convergence rate of the so-called {\it eigenvector empirical spectral distribution (VESD)}, which is recently obtained in 
\cite{XYY}. For simplicity, we  state the result with necessary modification to adapt to our assumption. We also refer to \cite{XYY} for the original statement under more general assumption. 

 Recall the notation $\lambda_i(H)$ for the $i$-th largest eigenvalue of $H=XX^*$. We further denote by $\mb{\phi}_i$ the unit eigenvector of $H$ associated with $\lambda_i(H)$.  
 { Then, for a fixed unit vector $\bv$, we define the so-called {\it eigenvector empirical spectral distribution (VESD)} with respect to $\bv$ by 
 \begin{align}\label{def:VESD0}
F_{1N}^{\bv} (x) = \sum_{i=1}^M \big| \la \mb{\phi}_i, \bv\ra \big|^2   \mathds{1} (\lambda_i(H)\leq x).
\end{align} 
With the definition, we now introduce the following theorem  which gives the convergence rate of {\it VESD}. 
 \begin{thm}\label{thm.VESD}
 Under the Assumption \ref{assumption} and the additional restriction $|y-1|\geq \tau_0$ with any small but fixed $\tau_0>0$, for a given deterministic  vector $\bv\in S^{M-1}_\mathbb{R}$, we have 
 \begin{align}
 \sup_x | F_{1N}^{\bv} (x) - F_1(x)| \prec N^{-\frac 12},
 \end{align}
 where $F_1$ is the cumulative distribution function of $\nu_{\text{MP,1}}$ defined in \rf{19071801}.
\end{thm}
}

\subsection{Auxiliary lemmas} 
The following \emph{cumulant expansion formula} plays a central role in our computation, whose proof can be found in \cite[Proposition 3.1]{LP09} or \cite[Section II]{KKP96}, for instance.

\begin{lem}\label{cumulantexpansion}
(Cumulant expansion formula) For a fixed $\ell\in \mathbb{N}$,  let $f\in C^{\ell+1}(\mathbb{R})$. Supposed $\xi$ is a centered random variable with finite moments  to order $\ell+2$.  Recall the notation $\kappa_k(\xi)$ for  the k-th cumulant of $\xi$.  Then we have  
\begin{align}
\mbE(\xi f(\xi))=\sum_{k=1}^{\ell}\frac{\kappa_{k+1}(\xi)}{k!}\mbE(f^{(k)}(\xi))+\mbE(r_{\ell}(\xi f(\xi))),
\end{align}
where the error term $r_{\ell}(\xi f(\xi))$ satisfies
\begin{align}
|\mbE(r_{\ell}(\xi f(\xi)))|\leq &C_\ell\mbE(|\xi|^{\ell+2}){\rm{sup}}_{|t|\leq s}|f^{\ell+1}(t)|\notag\\
&+C_\ell\mbE(|\xi|^{\ell+2}\mathds{1}(|\xi|>s)){\rm{sup}}_{t\in \mathbb{R}}|f^{\ell+1}(t)|  \label{remaining}
\end{align}
for any  $s>0$ and $C_\ell$ satisfied $C_\ell\leq (C\ell)^{\ell}/\ell!$ for some constant $C>0$.
\end{lem}

Next we collect some basic identities for the Green functions in (\ref{def:green}) without proof.

\begin{lem} \label{lemrelation}
For any integer $l\geq1$, we have
\begin{align}
&\mG_1^{l}=\frac{1}{(l-1)!}\frac{\partial^{l-1}\mG_1}{\partial z^{l-1}} = \frac{1}{(l-1)!} \mG_1^{(l-1)},\label{eq:basicG}\\   
&\mG_1^lXX^*=\mG_1^{l-1}+z\mG_1^{l}, \quad X^*\mG_1^lX=\mG_2^lX^*X=\mG_2^{l-1}+z\mG_2^{l}. \label{relationXG}
\end{align}
\end{lem}

 Further, for $a \in \lb 1, M\rb$ and $b\in \lb 1, N\rb$, we denote by  $E_{ab}$ the $M\times N$ matrix with entires $(E_{ab})_{cd}=\delta_{ac}\delta_{bd}$. Let
\begin{align}
\mathscr{P}_{0}^{ab}=E_{ab}(E_{ab})^*,\quad \mathscr{P}_{1}^{ab}=E_{ab}X^*,\quad \mathscr{P}_{2}^{ab}=X(E_{ab})^*. \label{P012}
\end{align}
For any integer $l\geq1$, it is also elementary to compute that 
\begin{align}
&\frac{\partial \mG_1^{l}}{\partial x_{ab}}=-\sum_{\alpha=1}^2\sum_{\begin{subarray}{c}l_1,l_2\geq 1\\ l_1+l_2=l+1 \end{subarray}}\mG_1^{l_1}\mathscr{P}_{\alpha}^{ab}\mG_1^{l_2}.\label{derivative}
\end{align}
Repeatedly applying the  identity \rf{derivative}, we can  get the formulas for higher order derivatives of $\mG_1^l$ w.r.t.  $x_{ab}$. Moreover, by (\ref{derivative}) and  the product rule, we can easily deduce the derivatives of $X^*\mG_1^{l}$ w.r.t. $x_{ab}$.   For the convenience of the reader, we collect more basic formulas of the  derivatives of Green functions in Appendix \ref{s.derivative of G}.

\section{Green function representation}\label{secgfr}
In this section, we  express  $\la \bw, {\rm{P}}_{\mathsf{I}}\bw \ra$ in terms of  the Green function $\mG_1(z)$ in \eqref{def:green}.  And for singleton $\mathsf{I}=\{i\}$, we also express the eigenvalue $\mu_i$ in terms of  the Green function $\mG_1(z)$ in \eqref{def:green}. Both  representations are obtained via doing expansions of certain functionals of the Green function. The expansion for the eigenvalue in the multiple case can be done similarly, and the details will be stated in Appendix \ref{appendix.proof}.  This representation will allow us to work with the Green function instead of the eigenvalue and eigenvector statistics.  We also remark here that similar derivation of the Green function representation has appeared in previous work such as \cite{KY13, bloemendal2016principal, KY14}. But here for eigenvectors, we need to do it up to a higher order precision, in order to capture all contributing terms for the fluctuation.  For instance, when $\bw \in \mathrm{Span}\{\mathbf v_j \}_{j\in \lb 1,M\rb \setminus \mathsf{I}}$, the fluctuation of $\la \bw, {\rm{P}}_{\mathsf{I}}\bw \ra$ is one order smaller than that of the case $\bw \in \mathrm{Span}\{\mathbf v_t \}_{t \in \mathsf{I}}$. In order to cover the situation like the former case, we will need to investigate a higher order term in the expansion. 

We start with a few more notations. The vector $\bw$ is decomposed as
\begin{align}
\mb{w}=\sum_{j=1}^r \langle\mb{w}, \bv_j\rangle \bv_j+\bu, \quad \text{where $ \bu\in \mathrm{Span}\{\bv_1\cdots,\bv_r\}^\perp$}. \label{19071820}
\end{align}
Hereafter,  we take (\ref{19071820}) as the definition of $\mathbf{u}$. 

We define the centered Green function by
\begin{align}
\Xi(z) := \mG_1(z) - m_1(z) I, \label{19071905}
\end{align} 
and introduce its quadratic forms
\begin{align}
\chi_{ij} (z) = \bv_i^* \Xi(z) \bv_j,\qquad \chi_{\bu j}(z) = \bu^* \Xi(z) \bv_j,\qquad i,j\in \lb 1, r\rb. \label{19071920}
\end{align}
For brevity, we further set
 \begin{align}\label{def:wtw}
\widetilde{\mb{w}}:= \Sigma^{-\frac{1}{2}}\mb{w}=\sum_{j=1}^r \wt w_j\bv_j+\bu \quad\text{with}\quad \wt w_j:= \frac{\langle \mb{w},\bv_j \rangle}{\sqrt{1+d_j}}.
 \end{align}
Also, for $d>0$, we define the following functions
\begin{align}
&f(d):=\frac{1}{d}(d+1)(d^2-y), \qquad g(d):=\frac{1}{d} (d+1)(d+y)(d^2-y) , \label{19071921}
\end{align}
which are of order $O(d(d-\sqrt y))$ and $O(d^2(d-\sqrt y))$  respectively.
And for $i\in \lb 1,r\rb$, we set for $d\neq d_i$, 
\begin{align}
&\nu_i (d) := \frac{d_i (d+1)}{d_i -d}, \label{19071945}
\end{align}
and further introduce the following shorthand notations
\begin{align}\label{def:d_i-d_j}
\delta_{ij}:=|d_i-d_j|,\qquad \delta_{i0}:=|d_i-\sqrt y|
\end{align}
for $i, j \in \lb 1,r\rb$.

With the above notations and $\bw_{\mathsf{I}}, \mb{\varsigma}_\mathsf{I}$ from  (\ref{081510}) and \eqref{def:v_I_varsigma_I}, we have the following lemma. 
\begin{lem}\label{weak_lem_decomp}Suppose that the assumptions of Theorem \ref{mainthm} hold. 
For $i \in \lb 1, r_0\rb$, we have 
\begin{align}\label{eq:weakgreen}
\la \bw, {\rm{P}}_{\mathsf{I}}\bw \ra= &\frac{d_i^2-y}{d_i(d_i+y)}\Vert \bw_{\mathsf{I}}\Vert^2-2d_i(d_i+1)\bw_{\mathsf{I}}^*\Xi(\theta(d_i))\bw_{\mathsf{I}}-2\frac{f(d_i)}{\sqrt{1+d_i}}\mb{\varsigma}_{\mathsf{I}}^*\Xi(\theta(d_i))\bw_{\mathsf{I}}  \nonumber \\
& -\frac{f(d_i)^2}{1+d_i} \bw_{\mathsf{I}}^*\Xi'(\theta(d_i))\bw_{\mathsf{I}}+ g(d_i)\sum_{t \in \mathsf{I}} \Big(\bv_t^*\Xi(\theta(d_i)) \mb{\varsigma}_{\mathsf{I}}\Big)^2 \nonumber\\
&-d_i(1+d_i)g(d_i)\sum_{j \in \mathsf{I}^c}\frac{d_j}{(d_i-d_j)^2}\Big(\bv_j^* \Xi(\theta(d_i))\bw_{\mathsf{I}}\Big)^2\notag\\
&+O_\prec\Big( N^{-\frac 12-\varepsilon}\big(\delta_{i0}\sqrt{d_i}\, \Vert \mb{\varsigma}_{\mathsf{I}}\Vert + d_i \Vert \bw_{\mathsf{I}} \Vert \big)\,\Vert\bw_{\mathsf{I}}\Vert \, \Delta(d_i) \Big) \nonumber\\
&+O_\prec\bigg(  N^{-1-\varepsilon}\Big( d_i\Vert \mb{\varsigma}_{\mathsf{I}}\Vert^2 + d_i^2\,\Vert \mb{w}_\mathsf{I}\Vert^2\Big( \sum_{j\in \mathsf{I}^c}\frac {d_j}{\delta_{ij}}\Big) \Big(  \sum_{j\in \mathsf{I}^c}\frac {d_i}{\delta_{ij}}\Big) \Big) \delta_{i0}\, d_i\, \Delta(d_i)^2 \bigg)
\end{align}
for some small fixed constant $\varepsilon>0$. 
\end{lem}

\begin{rmk}\label{rmk:representation} It will be seen that the bounds in \rf{weak_est_XG} actually give the  true typical size of 
the quadratic forms. In light of this, 
 Lemma \ref{weak_lem_decomp} suggests that  the  distribution of $ \la \bw, {\rm{P}}_i\bw \ra$ is ultimately governed by the joint distribution of the quadratic forms $\bw_{\mathsf{I}}^*\Xi \bw_{\mathsf{I}}, \mb{\varsigma}_{\mathsf{I}}^*\Xi \bw_{\mathsf{I}}, \bw_{\mathsf{I}}^*\Xi' \bw_{\mathsf{I}}, \{ \bv_t^*\Xi \mb{\varsigma}_{\mathsf{I}}\}_{t\in \mathsf{I}}, \{\bv_j^* \Xi \bw_{\mathsf{I}}\}_{j\in \mathsf{I}^c}$, since the last two error terms in (\ref{eq:weakgreen}) are smaller than the sum of the second to the fifth fluctuating terms in any case, with high porbability.  Here we drop the parameter $\theta(d_i)$ from $\Xi$ for simplicity.
\end{rmk}

 \begin{proof}[Proof of Lemma \ref{weak_lem_decomp}]
 Recall (\ref{19071901}) together with (\ref{19071902}).  We rewrite   $S$ as
$$S=V\diag(d_1,\cdots,d_r)V^*,$$ 
by setting $V=(\bv_1,\cdots,\bv_r)$. Therefore, we have 
 $$\Sigma^{-1}=I-VDV^*.$$
 with
 $$D=\diag\left(\frac{d_1}{1+d_1},\cdots,\frac{d_r}{1+d_r} \right).$$
  From Lemma \ref{locationeig}, we observe that all $\mu_i, i\in \mathsf{I}$ tend to the identical limit $\theta(d_i)$ with error bound $O_\prec (N^{-1/2}(d_i-\sqrt y)^{1/2})$.
 
  Let $\Gamma_i$ be the boundary of a disc centered at $d_i$ with  radius $\rho_i$,
  \begin{align} \label{def:radius}
  \rho_i:=\frac 12 \Big( \min_{j\in \mathsf{I}^c}|d_i-d_j| \wedge (d_i-\sqrt{y})\Big),
    \end{align}
    such that the disc is away from the critical value $\sqrt y$ and other distinct $d_j$'s.
  Therefore under the definitions of $\rho_i$ and $\theta(z)$ in \rf{def. of theta},  with  Assumption \ref{supercritical}, the contour $\theta(\Gamma_i)$, which is the image of $\Gamma_i$ under the map $\theta(\cdot)$, encloses exactly $|\mathsf{I}|$ eigenvalues of $Q$, i.e. $\mu_i,i\in \mathsf{I}$. This follows from the fact 
  \begin{align*}
  |\theta(d_i+\rho_i)-\theta(d_i)| \asymp (d_i-\sqrt y)\rho_i > (d_i-\sqrt y)^{\frac 12 } N^{-\frac 12+ \varepsilon }, \text { for some $\varepsilon >0$},
  \end{align*}
where the last inequality is guaranteed  by   Assumption \ref{supercritical}. 
 According to Lemma  \ref{locationeig}, together with the Cauchy integral, we have the following equality with high probability 
\begin{align}
\la \bw, {\rm{P}}_{\mathsf{I}}\bw \ra=-\frac{1}{2\pi \ii}\oint_{\theta(\Gamma_i)}\mb{w}^*G(z)\mb{w}\,{\rm d}z.\label{greenfunctionrepre}
\end{align} 
 With the notations for $V, S, D$ and $\Sigma^{-1}$, using the setting (\ref{19071802}), we can write 
\begin{align*}
G(z)&=\left( \Sigma^{\frac12} XX^*  \Sigma^{\frac12} - z I \right)^{-1}= \Sigma^{-\frac12} \left( \mG_1^{-1}(z) +z \Sigma^{-\frac12} S \Sigma^{-\frac12} \right)^{-1} \Sigma^{-\frac12}\\
& = \Sigma^{-\frac12} \left( \mG_1^{-1}(z) + z VDV^* \right)^{-1} \Sigma^{-\frac12}.
\end{align*}
Then, it follows from the matrix inversion lemma that
\begin{align*}
G(z)&= \Sigma^{-\frac12}\mG_1(z) \Sigma^{-\frac12}-z \Sigma^{-\frac12}\mG_1(z) V\big(D^{-1}+zV^*\mG_1(z) V\big)^{-1} V^*\mG_1(z) \Sigma^{-\frac12}. 
\end{align*}
With the notation introduced in \eqref{def:wtw}, we can further write
 \begin{align}
 \mb{w}^*G(z)\mb{w}=\widetilde{\mb{w}}^*\mG_1(z)\widetilde{\mb{w}}-z\widetilde{\mb{w}}^*\mG_1(z){V}\big(D^{-1}+zV^*\mG_1(z)V\big)^{-1}V^*\mG_1(z) \widetilde{\mb{w}}. \label{19071903}
 \end{align}
Plugging (\ref{19071903}) into \eqref{greenfunctionrepre}, and noticing that the contour integral of $\widetilde{\mb{w}}^*\mG_1(z)\widetilde{\mb{w}}$ on $\theta(\Gamma_i)$ is zero with high probability by Assumption \ref{supercritical} and the rigidity of eigenvalues of $H$ (\cf (\ref{190726100})), one has
  \begin{align}\label{eq:intrep1}
\la \bw, {\rm{P}}_{\mathsf{I}}\bw \ra =\frac{1}{2\pi \ii}\oint_{\theta(\Gamma_i)} z\widetilde{\mb{w}}^*\mG_1(z){V}\big(D^{-1}+zV^*\mG_1(z)V\big)^{-1}V^*\mG_1(z) \widetilde{\mb{w}}\,{\rm d}z
  \end{align}
with high probability. 

For the integrand in (\ref{eq:intrep1}), we first recall the notation in (\ref{19071905}) and then 
we apply resolvent expansion 
\begin{align}
\big(D^{-1}+zV^*\mG_1(z)V\big)^{-1} 
=&L(z) - z L(z) V^*\Xi(z)V L(z) \nonumber\\
&+ \big( z L(z)V^*\Xi(z)V \big)^2 \big(D^{-1}+zV^*\mG_1(z)V\big)^{-1}, \label{19071907}
\end{align}
where 
$$L(z) := \left( D^{-1} + z m_1(z)  \right)^{-1}.$$
With (\ref{19071907}), we can further rewrite  \eqref{eq:intrep1}  as
\begin{align*}
&\la \bw, {\rm{P}}_{\mathsf{I}}\bw \ra=\frac{1}{2\pi \ii} \oint_{\theta(\Gamma_i)} z \big( m_1(z) \wt{\bw}^*V +\wt{\bw}^* \Xi(z) V \big)\Big( L(z) - z L(z) V^*\Xi(z)V L(z)\\
& \qquad+ \big( z L(z)V^*\Xi(z)V \big)^2 \big(D^{-1}+zV^*\mG_1(z)V\big)^{-1} \Big) \big( m_1(z) V^* \wt{\bw} + V^* \Xi(z) \wt{\bw} \big) \,{\rm d}z.
\end{align*}
Hence,  we  can  write
\begin{align}\label{eq:decomposew}
\la \bw, {\rm{P}}_{\mathsf{I}}\bw \ra = S_1 + S_2 + S_3 +R,
\end{align}
by defining
 \begingroup
\allowdisplaybreaks
\begin{align}
&S_1:= \frac{1}{2\pi \ii}\oint_{\theta(\Gamma_i)} z m_1^2(z) \wt{\bw}^* V L(z) V^* \wt\bw \,{\rm d}z, \notag\\
&S_2:= \frac{1}{2\pi \ii}\oint_{\theta(\Gamma_i)} \Big( 2 z m_1(z)\wt{\bw}^* V L(z) V^* \Xi(z) \wt\bw - z^2 m_1^2(z) \wt{\bw}^* V L(z) V^* \Xi(z) V L(z) V^* \wt\bw   \Big) \,{\rm d}z, \notag\\
&S_3:=  \frac{1}{2\pi \ii} \oint_{\theta(\Gamma_i)} \Big( z  \wt\bw^* \Xi(z) V L(z) V^* \Xi(z) \wt\bw - 2 z^2 m_1(z) \wt\bw^* (VL(z)V^* \Xi(z))^2 \wt\bw \notag\\
&\qquad \qquad \qquad \qquad+ z^3 m_1^2(z) \wt\bw^* (VL(z)V^* \Xi(z))^2V L(z) V^* \wt\bw  \Big) \,{\rm d}z,
\label{def of S_i}
\end{align}
\endgroup
and we take \rf{eq:decomposew} as the definition of the remainder term $R$. It remains to estimate $S_1, S_2, S_3$ and $R$.
From the definitions in (\ref{m1m2}) and (\ref{def. of theta}),  
it is easy to check the identity
\begin{align}
1+z^{-1} + \theta(z) m_1 (\theta(z))=0.  \label{19071910}
\end{align} 
With the above identity, we see that
\begin{align*}
V L(\theta(z)) V^* = \sum_{j=1}^r  \frac{\bv_j \bv_j^*}{1+d_j^{-1} + \theta(z) m_1 (\theta(z))} = \sum_{j=1}^r  \frac{ z d_j \bv_j \bv_j^*}{z- d_j}.
\end{align*}
Therefore, by the residue theorem,
 \begingroup
\allowdisplaybreaks
\begin{align}
S_1 &= \frac{1}{2\pi \ii}\oint_{\Gamma_i}  \theta(z) \theta'(z) m_1^2(\theta(z)) \wt{\bw}^* V L(\theta(z)) V^* \wt\bw \,{\rm d}z \nonumber\\
&= \theta(d_i) \theta'(d_i) m_1^2 (\theta(d_i)) d_i^2 \sum_{t\in \mathsf{I}}(\wt\bw^* \bv_t)^2 = \frac{d_i^2 - y }{d_i (d_i+ y)} \sum_{t\in \mathsf{I}}(\bw^* \bv_t)^2.
\end{align}
\endgroup
Similarly, using \rf{19071910}, we can get
\begin{align*}
S_2 &= \frac{1}{2\pi \ii}\oint_{\Gamma_i}  \Big( 2\theta(z) \theta'(z) m_1(\theta(z))  \sum_{j=1}^r \frac{z d_j}{z- d_j} (\wt\bw^* \bv_j) \wt\bw^* \Xi(\theta(z)) \bv_j \\
&\quad- \theta^2(z) \theta'(z) m_1^2(\theta(z)) \sum_{j,k=1}^r \frac{z^2 d_j d_k}{(z-d_j)(z-d_k)} (\wt\bw^* \bv_j) (\wt\bw^* \bv_k) \bv_j^*\Xi(\theta(z)) \bv_k \Big) \,{\rm d}z.
\end{align*} 
Further by the residue theorem together with the definition of $\theta$ and $m_1$ in (\ref{def. of theta}) and (\ref{m1m2}), we can get
\begin{align}
S_2&=-2d_i (d_i+1)^2 \sum_{j,k\in \mathsf{I}}{\wt w_j}{\wt w_k}  \chi_{jk}(\theta(d_i)) - 2 f(d_i) \sum_{j\in \mathsf{I}^{c}} \nu_i(d_j) \wt w_j \sum_{t\in \mathsf{I}}\wt w_t \chi_{tj}(\theta(d_i)) \nonumber\\
&\quad- 2 f(d_i) \sum_{t\in \mathsf{I}}\wt w_t \chi_{\bu t}(\theta(d_i)) - f(d_i)^2 \sum_{j,k\in \mathsf{I}}{\wt w_j}{\wt w_k}  \chi_{jk}'(\theta(d_i))\nonumber\\
&=-2d_i(d_i+1) \bw_{\mathsf{I}}^*\Xi(\theta(d_i))\bw_{\mathsf{I}}-2 \frac{f(d_i)}{\sqrt{1+d_i}}\bw_{\mathsf{I}}^*\Xi(\theta(d_i))\mb{\varsigma}_{\mathsf{I}} - \frac{f(d_i)^2}{1+d_i}\bw_{\mathsf{I}}^*\Xi'(\theta(d_i))\bw_{\mathsf{I}},  \label{weak_19071941}
\end{align}
 where we recall the notations in \rf{19071920}, \rf{19071921} and \rf{def:v_I_varsigma_I}. 
 Next, we fix a sufficiently large $K>0$. For the case $d_i\leq K$, by isotropic local law  \rf{weak_est_DG}, we can bound both the first and third term on the RHS of \rf{weak_19071941}  by $O_\prec ( \sum_{t\in \mathsf{I}} |\wt w_t|^2  \Delta(d_i) N^{-1/2})$. In the case $d_i>K$,  this  crude bound shall be refined to $O_\prec ( \sum_{t\in \mathsf{I}} |\wt w_t|^2  d_i^3\Delta(d_i) N^{-1/2})$. Here in the latter case we carry the additional factor $d_i^3$ since it can be very large; for instance, when $d_i\sim N^C$. Nevertheless,  a more precise bound can be obtained for the combination of the first and third terms (\ref{weak_19071941}) due to a hidden cancellation in case $d_i>K$.  In order to see this cancellation, we need to apply in Theorem \ref{thm.VESD}. To this end,  we first set the following notation for the normalized vector  
\begin{align}
\bw_{\mathsf{I}}^0:=            
\begin{cases}
{\bw_{\mathsf{I}}}/{\| \bw_{\mathsf{I}}\|}, & \text{if } \bw_{\mathsf{I}} \neq \mb{0};\\
\mb{0}, & \text{otherwise}.
\end{cases} \label{normalized bw}
\end{align}
Note that due to the triviality of the case $\bw_{\mathsf{I}}=\mb{0}$, we only discuss the case $\bw_{\mathsf{I}}\neq \mb{0}$ in the sequel. 

Recall the definition of  the VESD from \rf{def:VESD0} with $\mathbf{v}=\bw_{\mathsf{I}}^{0}$ 
\begin{align}\label{def:VESD}
F_{1N}^{\mb{w}_\mathsf{I}^0} (x) = \sum_{i=1}^M \big| \la \mb{\phi}_i, \bw_{\mathsf{I}}^0\ra \big|^2  \mathscr{1} \mathds{1} (\lambda_i(H)\leq x),
\end{align} 
and we denote by $F_1(x)$ the distribution function of $\nu_{\text{MP},1}({\rm d}x)$ in \rf{19071801}.
By Theorem \ref{thm.VESD}, we obtain,
\begin{align}\label{bdd:Koldis_VESD}
\sup_x| F_{1N}^{\mb{w}_\mathsf{I}^0}(x) - F_1(x)  | = O_\prec (N^{-\frac 12}).
\end{align}
In light of the fact that $\bw_{\mathsf{I}}^0$ is the normalized $\bw_{\mathsf{I}}$, we can write the combination of the first and third terms on the RHS of (\ref{weak_19071941}) as
\begin{align}\label{cancel:1}
&\quad \Big|-2d_i(d_i+1) \bw_{\mathsf{I}}^*\Xi(\theta(d_i))\bw_{\mathsf{I}} - \frac{f(d_i)^2}{1+d_i}\bw_{\mathsf{I}}^*\Xi'(\theta(d_i))\bw_{\mathsf{I}} \Big| \nonumber\\
& = \Vert \bw_{\mathsf{I}}\Vert^2 \Big|\int \Big(\frac{-2d_i(d_i+1)}{x-\theta(d_i)}  - \frac{(1+d_i)(d_i^2-y)^2/d_i^2}{\big(x-\theta(d_i)\big)^2} \Big) {\rm{d}}  \big(F_{1N}^{\mb{w}_\mathsf{I}^0} (x)- F_1 (x)\big)\Big| \nonumber\\
&\asymp \Vert \bw_{\mathsf{I}}\Vert^2 \bigg|\int^{\lambda_+ + N^{-\frac 23 + \frac \tau 2} }_{\lambda_- - N^{-\frac 23 + \frac \tau 2}} \Big(\frac{2d_i(d_i+1)}{\big(x-\theta(d_i)\big)^2}  + \frac{2(1+d_i)(d_i^2-y)^2/d_i^2}{\big(x-\theta(d_i)\big)^3} \Big)   \big(F_{1N}^{\mb{w}_\mathsf{I}^0} (x)- F_1 (x)\big) {\rm{d}}x \bigg|,
\end{align}
where in the last step, we used integration by parts, together with the rigidity results in Lemma \ref{rigidity of H}. 
In the case $d_i>K$ for sufficiently large $K$, by the definition of $\theta(d_i)$, we see that $\theta(d_i) \sim d_i$ and thus $\theta(d_i)$ is also sufficiently large. Hence,  applying Taylor expansion for both $1/(x-\theta(d_i))^2, 1/(x-\theta(d_i))^3$ at $x=0$, we can easily get the bound 
\begin{align}\label{cancel:Taylor}
\frac{2d_i(d_i+1)}{\big(x-\theta(d_i)\big)^2}  + \frac{2(1+d_i)(d_i^2-y)^2/d_i^2}{\big(x-\theta(d_i)\big)^3} = O(1/d_i)
\end{align} 
uniformly for $x\in [\lambda_- - N^{-\frac 23 + \frac \tau 2}, \lambda_+ + N^{-\frac 23 + \frac \tau 2}]$. Then using \rf{cancel:Taylor} and \rf{bdd:Koldis_VESD}, we finally get the estimate 
\begin{align} \label{2020072106}
\Big|-2d_i(d_i+1) \bw_{\mathsf{I}}^*\Xi(\theta(d_i))\bw_{\mathsf{I}} - \frac{f(d_i)^2}{1+d_i}\bw_{\mathsf{I}}^*\Xi'(\theta(d_i))\bw_{\mathsf{I}} \Big| =  O_\prec \Big(\sum_{t\in \mathsf{I}} |\wt w_t|^2 N^{-\frac 12}\Big).
\end{align}
We therefore get a unified estimate $ O_\prec \Big(\sum_{t\in \mathsf{I}} |\wt w_t|^2 d_i^2 \Delta(d_i) N^{-\frac 12}\Big)$ for this combination, no matter $d_i\leq K$ or $d_i>K$.

Further, from the isotropic local law \rf{weak_est_DG},  Assumption \ref{supercritical}, and the definition of $\bw_{\mathsf{I}}, \mb{\varsigma}_{\mathsf{I}}$ in \rf{def:v_I_varsigma_I}, we get the following crude bound  
\begin{align} \label{2020011401}
S_2&=O_\prec \bigg( \sum_{t\in \mathsf{I}}|\wt w_t| \Delta(d_i) N^{-1/2}\Big( d_i^2 \sum_{t\in \mathsf{I}}|\wt w_t| + \sum_{j\in \mathsf{I}^c} |\wt w_j| \frac {d_i^2d_j\delta_{i0}}{\delta_{ij}}  + d_i\delta_{i0}\Vert \bu\Vert \Big)\bigg)\nonumber\\
&=O_\prec (\Lambda_2),
\end{align}
where we recall the notations $\delta_{ij}, \delta_{i0}$ in \rf{def:d_i-d_j} and  define 
\begin{align} \label{def:Lambda_2}
\Lambda_2&:=\Big( \sum_{t\in \mathsf{I}}\frac{|\wt w_t|d_i}{\delta_{i0}}+ \sum_{j\in \mathsf{I}^c}\frac{|\wt w_j|d_id_j}{\delta_{ij}} +  {\Vert \bu \Vert}\Big) \, \Big(\sum_{t\in \mathsf{I}}\frac{|\wt w_t|d_i}{\delta_{i0}}\Big) \, \delta_{i0}^2 \Delta(d_i) N^{-\frac 12}.
\end{align}
Since $\Lambda_2$ can degenerate depending on the size of $\sum_{t\in \mathsf{I}}|\wt w_t|$, it is necessary to also consider the fluctuation of the higher order term $S_3$. We shall obtain more precise estimates for the summands in $S_2$ in later context, after first obtaining an estimate of $S_3$.

Next, we turn to estimate $S_3$ (c.f. (\ref{def of S_i})). We estimate the integrals of three terms in the integrand separately. First, using the residue theorem together with the notations in (\ref{19071920}) and (\ref{19071921}), we have  
 \begin{align} \label{weak_d_is31}
&  \frac{1}{2\pi \ii} \oint_{\theta(\Gamma_i)}  z  \wt\bw^* \Xi(z) V L(z) V^* \Xi(z) \wt\bw\,{\rm d}z=\theta(d_i)\theta'(d_i)d_i^2 \sum_{t\in \mathsf{I}}\Big(\wt \bw^*\Xi(\theta(d_i))\bv_t\Big)^2 \nonumber\\
&  = g(d_i)\sum_{t\in \mathsf{I}}\Big( { \sum_{j,k\in \mathsf{I}^c} }  \wt w_j \wt w_k \chi_{tj}(\theta(d_i)) \chi_{tk}(\theta(d_i)) +2 { \sum_{j\in \mathsf{I}^c}} \wt w_j \chi_{tj}(\theta(d_i)) \chi_{\bu t}(\theta(d_i)) + \chi_{\bu t}^2(\theta(d_i))\Big)\nonumber\\
&\quad +O_\prec\Big({d_i^2(d_i-\sqrt y)\Delta(d_i)^2\sum_{t\in \mathsf{I}}|\wt w_t|\Vert  \wt\bw\Vert }{N^{-1}} \Big),
\end{align}
where the error bound follows from the isotropic local law \rf{weak_est_DG} and the  expression for $g(d)$.

For the second part of the integral, we have
 \begingroup
\allowdisplaybreaks
\begin{align}\label{weak_d_is32}
 &\frac{1}{2\pi \ii} \oint_{\theta(\Gamma_i)} - 2 z^2 m_1(z) \wt\bw^* (VL(z)V^* \Xi(z))^2 \wt \bw \,{\rm d}z \nonumber \\
 &=-2 \theta^2(d_i)\theta'(d_i)m_1(\theta(d_i)) \sum_{j\in \mathsf{I}^c, \; t\in \mathsf{I}} \frac{d_i^3 d_j}{d_i - d_j} (\wt\bw^*\bv_j) \wt\bw^*\Xi(\theta(d_i))\bv_t\chi_{tj}(\theta(d_i)) \nonumber \\
&\quad -2 \theta^2(d_i)\theta'(d_i)m_1(\theta(d_i)) \sum_{j\in \mathsf{I}^c,\; t\in \mathsf{I}} \frac{d_i^3 d_j}{d_i - d_j} (\wt\bw^*\bv_t)\wt\bw^*\Xi(\theta(d_i))\bv_j\chi_{tj}(\theta(d_i)) \nonumber\\
&\quad-2d_i^2 \sum_{t, k\in \mathsf{I}}(\wt\bw^*\bv_t) \Big(\theta^2(z)\theta'(z)m_1(\theta(z))z^2 \wt\bw^*\Xi(\theta(z))\bv_k \chi_{tk}(\theta(z)) \Big)'\Big|_{z=d_i} \nonumber\\
 &=2g(d_i)\sum_{j\in \mathsf{I}^c}\frac{(1+d_i)d_j}{d_i-d_j}\wt w_j \sum_{t\in \mathsf{I}} \chi_{tj}(\theta(d_i))\Big(\sum_{k\in \mathsf{I}^c} \wt w_k \chi_{tk}(\theta(d_i)) +\chi_{\bu t}(\theta(d_i)) \Big) \nonumber \\
 &\quad+O_\prec\Big(\sum_{t\in \mathsf{I}, j\in \mathsf{I}^c} \Big(\frac{d_i^3d_j\delta_{i0}}{\delta_{ij}} + d_i^4\Big)|\wt w_t|\Vert \wt \bw\Vert \Delta(d_i)^2 {N^{-1}}\Big).
\end{align}
\endgroup
Similarly, we used isotropic local law \rf{weak_est_DG} for $l=0,1$ and the definition for $g(d)$ in (\ref{19071921})  for the estimates in the last step.

Analogously, for  the last term, we obtain
 \begingroup
\allowdisplaybreaks
\begin{align} \label{weak_d_is33}
&\frac{1}{2\pi \ii} \oint_{\theta(\Gamma_i)} z^3 m_1^2(z) \wt\bw^* (VL(z)V^* \Xi(z))^2V L(z) V^* \wt\bw   \,{\rm d}z \nonumber \\
&=g(d_i)\sum_{j,k\in \mathsf{I}^c}\frac{(1+d_i)^2d_jd_k}{(d_i-d_j)(d_i-d_k)} \wt w_j \wt w_k \sum_{t\in \mathsf{I}} \chi_{tj}(\theta(d_i)) \chi_{tk}(\theta(d_i))\nonumber\\
& +2g(d_i)\sum_{j,k\in \mathsf{I}^c}\frac{(1+d_i)^2d_jd_k}{(d_i-d_j)(d_i-d_k)} \wt w_k\chi_{jk}(\theta(d_i)) \sum_{t\in \mathsf{I}} \wt w_t\chi_{tj}(\theta(d_i)) \nonumber\\
&+2\sum_{t,k\in \mathsf{I}, j\in \mathsf{I}^c}\wt w_t \wt w_j d_i^2d_j\Big(\theta(z)^3\theta'(z)m_1(\theta(z))^2\frac{z^3}{z-d_j}\chi_{tk}(\theta(z))\chi_{kj}(\theta(z)) \Big)'\Big|_{z=d_i} \nonumber\\
&+\sum_{t,k\in \mathsf{I}, j\in \mathsf{I}^c}\wt w_t \wt w_k d_i^2d_j \Big(\theta(z)^3\theta'(z)m_1(\theta(z))^2\frac{z^3}{z-d_j}\chi_{tj}(\theta(z))\chi_{kj}(\theta(z)) \Big)'\Big|_{z=d_i} \nonumber\\
&+\sum_{t, j, k\in \mathsf{I}}\wt w_t \wt w_kd_i^3 \Big(\theta(z)^3\theta'(z)m_1(\theta(z))^2z^3\chi_{tj}(\theta(z))\chi_{kj}(\theta(z)) \Big)''\Big|_{z=d_i}.
\end{align}
\endgroup
In the sequel, we will keep the first term on the RHS of (\ref{weak_d_is33}) as it is and estimate the other terms. 
First, notice that the second term  on the RHS of (\ref{weak_d_is33}) can be estimated by 
\begin{align}
&2g(d_i)\sum_{j,k\in \mathsf{I}^c}\frac{(1+d_i)^2d_jd_k}{(d_i-d_j)(d_i-d_k)} \wt w_k\chi_{jk}(\theta(d_i)) \sum_{t\in \mathsf{I}} \wt w_t\chi_{tj}(\theta(d_i)) \nonumber\\
&=O_\prec \Big(\sum_{j,k\in \mathsf{I}^c, t\in \mathsf{I}}  |\wt w_t| |\wt w_k| \frac {d_i^4d_jd_k\delta_{i0}}{\delta_{ij}\delta_{ik}} \Delta(d_i)^2N^{-1}\Big), \label{S_33:omitted_2}
\end{align}
Similarly, simple calculation shows that the third term on the RHS of (\ref{weak_d_is33}) can be bounded crudely by   
\begin{align}
&\sum_{t,k\in \mathsf{I}, j\in \mathsf{I}^c} \wt w_t\wt w_j d_i^2d_j\Big[\Big(\frac{\theta''(d_i)d_i^4 +\theta'(d_i) ^2 d_i^3+ \theta'(d_i) d_i^3}{d_i-d_j} + \frac{\theta'(d_i) d_i^4  }{(d_i-d_j)^2 } \Big)\chi_{tk}(\theta(d_i))\chi_{kj}(\theta(d_i)) \nonumber\\
&+ \frac{\theta'(d_i) ^2 d_i^4}{d_i-d_j}\Big(\chi_{tk}'(\theta(d_i))\chi_{kj}(\theta(d_i)) + \chi_{tk}(\theta(d_i))\chi_{kj}'(\theta(d_i))\Big)\Big] \nonumber\\
&=O_\prec\bigg(\sum_{t\in \mathsf{I}, j \in \mathsf{I}^c}|\wt w_t| |\wt w_j| \Big(\frac{d_i^5d_j}{\delta_{ij}} + \frac{d_i^5d_j\delta_{i0}}{\delta_{ij}^2}\Big) \Delta(d_i)^2N^{-1}\bigg). \label{S_33:omitted_3}
\end{align}

Next,   we turn to  the last term on the RHS of (\ref{weak_d_is33}). We split the discussion into two cases: $d_i\leq K$ or $d_i>K$, for some sufficiently large (but fixed) $K>0$. For $d_i\leq K$ , it is easy to compute the derivative by chain rule and use isotropic local law \rf{weak_est_DG} to get the bound $O_\prec \big( \sum_{t,k\in \mathsf{I}} |\wt w_t| |\wt w_k| \delta_{i0}^{-1}\Delta(d_i)^2{N^{-1}}\big)$ directly for the last term on the RHS of (\ref{weak_d_is33}), and thus we omit the detail. In case $d_i>K$,  we first introduce a new form of eigenvector empirical spectral distribution(VESD) of $H$ with respect to two fixed unit vectors $\bu, \bv$, which is defined by 
\begin{align} \label{def:VESDnew}
F_{1N}^{\bu,\bv}(x)  = \sum_{i=1}^M  \la \mb{\phi}_i, \bu\ra \la \mb{\phi}_i, \bv\ra    \mathscr{1} \mathds{1} (\lambda_i(H)\leq x).
\end{align}
Correspondingly, we set $F_1^{\bu, \bv}(x) = \la \bu, \bv \ra   F_1(x)$. Then by triangle inequality and Theorem \ref{thm.VESD}, we have  the estimate
\begin{align} \label{est:Kol uv}
\sup_{x}| F_{1N}^{\bu,\bv}(x)- F_1^{\bu, \bv}(x) | & \leq    \frac 12 \Vert\bu+\bv\Vert^2 \sup_x| F_{1N}^{(\bu+\bv)^0}(x) - F_1(x)| \nonumber\\
&+\frac 12 \sup_x |F_{1N}^{\bu}(x) - F_1(x)| + \frac 12\sup_x|F_{1N}^{\bv} - F_1|\notag\\
& = O_\prec(N^{-\frac 12}).
\end{align}
Here we use $(\bu+\bv)^0$ to denote the normalized $\bu+\bv$. 
Similarly to \rf{def:VESD}- \rf{2020072106}, after computing the derivative of  the last term on the RHS of (\ref{weak_d_is33}), we will see the hidden cancellations 
 among the resulting terms using the following estimates,
 \begin{align*}
 &2\chi_{tj}(\theta(z)) + \theta(d_i)\chi_{tj}'(\theta(z))  = O_\prec(d_i^{-3} N^{-1/2}),  \notag\\
 & \chi_{tj}(\theta(z)) \chi_{kj}'(\theta(z)) \theta''(d_i) = O_\prec(d_i^{-4} \Delta(d_i)^2N^{-1/2}), \nonumber\\
 &6\chi_{tj}(\theta(z)) + 6\theta(d_i)\chi_{tj}'(\theta(z))  + \theta(d_i)^2\chi_{tj}''(\theta(z)) = O_\prec(d_i^{-3} N^{-1/2})
 \end{align*}
 which can be checked readily  by using (\ref{est:Kol uv}).
  Then after elementary computations,  
we can get for the case $d_i>K$, the  last term on the RHS of (\ref{weak_d_is33}) admits the bound $O_\prec \Big( \sum_{t,k\in \mathsf{I}} |\wt w_t| |\wt w_k| {N^{-1}}\Big)$. Combining the estimates for the two cases, $d_i\leq K$ or $d_i>K$, we can uniformly bound  the last term on the RHS of (\ref{weak_d_is33}) by 
\begin{align}
O_\prec \Big( \sum_{t,k\in \mathsf{I}} |\wt w_t| |\wt w_k| d_i^5\delta_{i0}^{-1}  \Delta(d_i)^2 {N^{-1}}\Big).\label{S_33:omitted_5}
\end{align}

Next, we turn to estimate the fourth term on the RHS of \rf{weak_d_is33}, 
which by product rule can be written up as 
\begin{align} \label{2020072201}
&\sum_{t,k\in \mathsf{I}, j\in \mathsf{I}^c}\wt w_t \wt w_k d_i^2d_j \Big(\theta(z)^3\theta'(z)m_1(\theta(z))^2\frac{z^3}{z-d_j}\chi_{tj}(\theta(z))\chi_{kj}(\theta(z)) \Big)'\Big|_{z=d_i}\nonumber\\
&=-d_i(1+d_i)^2g(d_i)\sum_{t, k\in \mathsf{I}, j\in \mathsf{I}^c}\wt w_t \wt w_k \frac{d_j}{(d_i-d_j)^2}\chi_{tj}(\theta(d_i))\chi_{kj}(\theta(d_i))  \nonumber\\
&\quad +  \sum_{t,k\in \mathsf{I}, j\in \mathsf{I}^c}\wt w_t \wt w_k \frac{d_i^2d_j }{d_i-d_j}\Big(\theta(z)^3\theta'(z)m_1(\theta(z))^2{z^3}\chi_{tj}(\theta(z))\chi_{kj}(\theta(z)) \Big)'\Big|_{z=d_i}.
\end{align}
The first term on the RHS of \rf{2020072201} can be further simplified to 
\begin{align}\label{weak_d_1s33term}
-d_i(1+d_i)g(d_i)\sum_{ j\in \mathsf{I}^c}\frac{d_j}{(d_i-d_j)^2}\Big(\bw_{\mathsf{I}}^*\Xi(\theta(d_i))\bv_j\Big)^2.
\end{align}
The above may exceed the stochastic bound $\Lambda_2$ defined in (\ref{def:Lambda_2}). While a posteriori, it will be clear that $\Lambda_2$ gives the typical size of $S_2$. Hence,  we shall keep the  term (\ref{weak_d_1s33term}) explicit instead of grouping it into the errors. As for the second term of \rf{2020072201},
analogously to the last term on the RHS of \rf{weak_d_is33}, by splitting the discussion into two cases based on the  size of $d_i$, we can  get the estimate 
\begin{align}
O_\prec\bigg(\sum_{t,k\in \mathsf{I}, j \in \mathsf{I}^c}|\wt w_t| |\wt w_k| \frac{d_i^4d_j}{\delta_{ij}}  \Delta(d_i)^2N^{-1}\bigg).  \label{S33:omitted_4}
\end{align}

Combining \eqref{weak_d_is31}-\eqref{S33:omitted_4}, after necessary simplification, we arrive at 
\begin{align} \label{2020011402}
S_3&=g(d_i)\sum_{t\in \mathsf{I}}\Big(\bv_t^*\Xi(\theta(d_i))\mb{\varsigma}_{\mathsf{I}}\Big)^2- d_i(1+d_i)g(d_i)\sum_{ j\in \mathsf{I}^c}\frac{d_j}{(d_i-d_j)^2}\Big(\bv_j^*\Xi(\theta(d_i))\bw_{\mathsf{I}}\Big)^2\notag\\
&\quad+O_\prec (\mathcal{R}_3)\nonumber\\
&=:\wt S_3+ O_\prec (\mathcal{R}_3),
\end{align}
where  
 \begingroup
\allowdisplaybreaks
\begin{align} \label{def:R_3}
\mathcal{R}_3&:= \Big(\sum_{t\in \mathsf{I}} \frac{|\wt w_t|d_i}{\delta_{i0}} \Big) \Vert \wt\bw \Vert \Big( \sum_{j\in \mathsf{I}^c}\frac{d_j}{ \delta_{ij}}+\frac{d_i}{\delta_{i0}}\Big) d_i^2\delta_{i0}^2 \Delta(d_i)^2N^{-1} \nonumber\\
&+\Big(\sum_{j\in \mathsf{I}^c}  \frac{|\wt w_j|d_id_j}{\delta_{ij}}\Big) \Big( \sum_{t\in \mathsf{I}}  \frac{|\wt w_t|d_i}{\delta_{i0}}\Big)\Big( \sum_{k\in \mathsf{I}^c} \frac{d_k}{\delta_{ik}} + \frac{d_i}{\delta_{i0}} + \sum_{k\in \mathsf{I}^c} \frac{d_i}{\delta_{ik}} \Big)  d_i^2\delta_{i0} ^2\Delta(d_i)^2N^{-1}\nonumber\\
&+ \Big(\sum_{t\in \mathsf{I}} \frac{|\wt w_t|d_i}{\delta_{i0}} \Big)^2 \Big(\frac{d_i}{\delta_{i0}}+ \sum_{j\in \mathsf{I}^c}\frac {d_j}{\delta_{ij}} \Big)  {d_i^2}{\delta_{i0}^2} \Delta(d_i)^2N^{-1}.
\end{align}
\endgroup
Now we claim that $\mathcal{R}_3\prec N^{-\varepsilon} \Lambda_2$ for some small (but fixed) $\varepsilon>0$, where $\Lambda_2$ is defined in \rf{def:Lambda_2}. This can be achieved by applying assumption  \rf{asd} and \rf{asd2} and  using the following three estimates 
\begin{align} \label{2020011405}
&\frac{d_i^3}{\delta_{i0}}\Delta(d_i) N^{-\frac 12} \prec N^{-\varepsilon},\quad \quad \sum_{k\in \mathsf{I}^c} \frac{d_i^3}{\delta_{ik}}\Delta(d_i) N^{-\frac 12}\prec N^{-\varepsilon} ,\notag\\
& \sum_{k\in \mathsf{I}^c} \frac{d_i^2 d_k}{\delta_{ik}}\Delta(d_i) N^{-\frac 12} \prec N^{-\varepsilon}.
\end{align}
The first two estimates in \rf{2020011405} can be easily checked by combining the assumption for $\delta_{i0}=d_i-\sqrt y$ in \rf{asd}, the non-overlapping condition  \rf{asd2}  with  the definition of $\Delta(d_i)$ in \rf{def:Delta(d)} and its estimates in  Remark \ref{Rmk:Delta(d)}. The last estimate in \eqref{2020011405} can be derived as follows
\begin{align}
&\sum_{k} \frac{d_i^2 d_k}{\delta_{ik}}\Delta(d_i) N^{-\frac 12}\leq \sum_{k}\Big( \frac{d_i^3}{\delta_{ik}}+ d_i^2\Big)\Delta(d_i) N^{-\frac 12} \notag\\
&\leq   \sum_{k}\Big( \frac{d_i^3}{\delta_{ik}}+ \frac{d_i^3}{\delta_{i0}}\Big)\Delta(d_i) N^{-\frac 12}\prec N^{-\varepsilon}.
\end{align}
Here for the first inequality, we use the triangle inequality $d_k\leq d_i +\delta_{ik}$. The last inequality above holds from the definition of $\Delta(d_i)$ and the non-overlapping condition \rf{asd2} for $\delta_i=\min_{j\in \mathsf{I}^c}\delta_{ij}$ as well as the first estimate in \rf{2020011405}.

Thus to see $\mathcal{R}_3\prec N^{-\varepsilon}\Lambda_2$, we first observe that the second and third terms on the RHS of \rf{def:R_3} are much smaller than $\Lambda_2$ by applying the estimates in \rf{2020011405} and the trivial fact $d_i>\delta_{i0}$. For the first term on the RHS of \rf{def:R_3}, notice that 
\begin{align} \label{20200417}
\Vert \wt \bw \Vert \leq \Big( \sum_{t\in \mathsf{I}}\frac{|\wt w_t|d_i}{\delta_{i0}}+ \sum_{j\in \mathsf{I}^c}\frac{|\wt w_j|d_id_j}{\delta_{ij}} +  {\Vert \bu \Vert}\Big),
\end{align}
which is a consequence of the fact  
\begin{align*}
 \frac{d_id_j}{|d_i-d_j|}>C
\end{align*}
for some constant $C>0$. Using \rf{20200417} and the first and third estimates in  \rf{2020011405}, we get the first term on the RHS of \rf{def:R_3} is also negligible compared to $\Lambda_2$. 
To conclude, we have $\mathcal{R}_3\prec N^{-\varepsilon} \Lambda_2$.

 Therefore, we can rewrite 
\begin{align} \label{2020011406}
S_3=\wt S_3+O_\prec \Big(\frac {\Lambda_2}{N^{\varepsilon}}\Big),
\end{align}
for some small fixed positive $\varepsilon$. 
Further, using isotropic local law \rf{weak_est_DG} together with the definitions of $\mb{w}_\mathsf{I}, \mb{\varsigma}_{\mathsf{I}}$ in   \rf{def:v_I_varsigma_I},
 we can crudely estimate the size  of $\wt S_3$ in \rf{2020011402} by $O_\prec (\Lambda_3)$ where 
\begin{align} \label{est.S_3}
\Lambda_3:=&\Big( \sum_{j\in \mathsf{I}^c} \frac{|\wt w_j|d_j}{\delta_{ij}}+\frac{\Vert \bu \Vert}{d_i}\Big)^2 d_i^4 \,\delta_{i0} \, \Delta(d_i)^2N^{-1} \nonumber\\
& +  \Big( \sum_{t\in \mathsf{I}, j\in \mathsf{I}^c} \frac{|\wt w_t |d_j}{\delta_{ij}}\Big)\, \Big( \sum_{t\in \mathsf{I}, j\in \mathsf{I}^c} \frac{|\wt w_t |d_i}{\delta_{ij}}\Big) d_i^4 \delta_{i0} \, \Delta(d_i)^2N^{-1}.
\end{align}

Now it remains to estimate the remainder term $R$ in (\ref{eq:decomposew}). We shall prove that 
$$|R|= O_\prec\Big( \frac{\Lambda_2+\Lambda_3}{N^\varepsilon}\Big).$$
Following the derivation of \eqref{eq:decomposew}, one has 
  \begin{align}
R&=\la \bw, {\rm P}_{\mathsf{I}} \bw\ra-S_1-S_2-S_3 \nonumber\\
&=-\frac{1}{2\pi \ii}\oint_{\theta(\Gamma_i)}zm_1(z)\wt {\mb{w}}^*V ( zL(z)V^* \Xi(z)V )^3 (D^{-1}+zV^*\mG_1(z)V)^{-1}m_1(z) V^*\wt {\mb{w}} \,{\rm d}z\nonumber\\
&\quad +\frac{2}{2\pi \ii}\oint_{\theta(\Gamma_i)}zm_1(z)\wt {\mb{w}}^*V( zL(z)V^* \Xi(z)V )^2 (D^{-1}+zV^*\mG_1(z)V)^{-1}  V^*\Xi(z) \wt {\mb{w}} \,{\rm d}z\nonumber\\
&\quad -\frac{1}{2\pi \ii}\oint_{\theta(\Gamma_i)}z\wt {\mb{w}}^* \Xi(z)V  ( zL(z)V^* \Xi(z)V )(D^{-1}+zV^*\mG_1(z)V)^{-1}  V^*\Xi(z) \wt {\mb{w}} \,{\rm d}z\nonumber\\
&=:T_1+T_2+T_3.
\end{align} 
The estimations of $T_1, T_2, T_3$ cannot be computed by only applying the residue theorem. Instead, we shall follow the method used in Lemma $5.6$ of \cite{bloemendal2016principal}. More specifically,
 for $T_1$, we first use the resolvent expansion formula 
$$(D^{-1}+zV^*\mG_1(z)V)^{-1}=L(z)-(D^{-1}+zV^*\mG_1(z)V)^{-1}zV^*\Xi(z)V L(z)$$ 
to split it into two parts
\begin{align} \label{Remainder:T_1}
T_1&=-\frac{1}{2\pi \ii}\oint_{\theta(\Gamma_i)}zm_1(z)\wt {\mb{w}}^*V ( zL(z)V^* \Xi(z)V )^3 L(z)m_1(z) V^*\wt {\mb{w}} \,{\rm d}z\nonumber\\
&\quad +\frac{1}{2\pi \ii}\oint_{\theta(\Gamma_i)}zm_1(z)\wt {\mb{w}}^*V ( zL(z)V^* \Xi(z)V )^3 
(D^{-1}+zV^*\mG_1(z)V)^{-1}V^*\Xi(z)V  L(z) m_1(z) V^*\wt {\mb{w}} \,{\rm d}z\nonumber\\
&=:T_{11}+T_{12}.
\end{align}
For $T_{11}$, we estimate it using the residue theorem. Similar to the previous calculations as in \eqref{weak_19071941}-\eqref{2020011401}, one can derive that the magnitude of $T_{11}$ is stochastically bounded by 
\begin{align}\label{20200623}
&\Big(\sum_{j\in \mathsf{I}^c} \frac {|\wt w_j| d_j}{\delta_{ij}} \Big) \Big(\sum_{j\in \mathsf{I}^c} \frac {|\wt w_j| d_j}{\delta_{ij}} + \sum_{t\in \mathsf{I}, j\in \mathsf{I}^c}\frac{|\wt w_t| d_j}{\delta_{ij}}\Big) \Big( \sum_{j\in \mathsf{I}^c} \frac { d_j}{\delta_{ij}}\Big)d_i^6\delta_{i0} \Delta(d_i)^3N^{-\frac 32} \nonumber\\
&+\Big(\sum_{j\in \mathsf{I}^c} \frac {|\wt w_j| d_j}{\delta_{ij}} + \sum_{t\in \mathsf{I}, j\in \mathsf{I}^c}\frac{|\wt w_t| d_j}{\delta_{ij}}\Big)  \Big(\sum_{j\in \mathsf{I}^c} \frac {|\wt w_j| d_j}{\delta_{ij}} + \sum_{t\in \mathsf{I}, j\in \mathsf{I}^c}\frac{|\wt w_t| d_j}{\delta_{ij}}+\sum_{t\in \mathsf{I}}\frac{|\wt w_t| d_i}{\delta_{i0}}\Big) d_i^6 \Delta(d_i)^3N^{-\frac 32} \nonumber\\
&+ \Big(\sum_{j\in \mathsf{I}^c} \frac {|\wt w_j| d_j}{\delta_{ij}^2} + \sum_{t\in \mathsf{I}, j\in \mathsf{I}^c}\frac{|\wt w_t| d_j}{\delta_{ij}^2}\Big) \Big(\sum_{j\in \mathsf{I}^c} \frac {|\wt w_j| d_j}{\delta_{ij}} + \sum_{t\in \mathsf{I}, j\in \mathsf{I}^c}\frac{|\wt w_t| d_j}{\delta_{ij}} + \sum_{t\in \mathsf{I}}\frac{|\wt w_t|}{\delta_{i0}}\Big) d_i^7 \delta_{i0}\Delta(d_i)^3N^{-\frac 32}
\nonumber\\
&+\Big(\sum_{j\in \mathsf{I}^c} \frac {|\wt w_j| d_j}{\delta_{ij}^2} + \sum_{t\in \mathsf{I}, j\in \mathsf{I}^c}\frac{|\wt w_t| d_j}{\delta_{ij}^2}\Big) \,\Big(\sum_{t\in \mathsf{I},j\in \mathsf{I}^c} \frac { |\wt w_t|d_i}{\delta_{ij}}\Big) d_i^7 \delta_{i0} \Delta(d_i)^3N^{-\frac 32} \nonumber\\
&
+ \Big(\sum_{t\in \mathsf{I}}\frac{|\wt w_t| d_i}{\delta_{i0}}\Big)^2 d_i^6 \Delta(d_i)^3N^{-\frac 32}. 
\end{align}
It is easy to check that \eqref{20200623} is  negligible compared to $\Lambda_2, \Lambda_3$ in \rf{def:Lambda_2} and \rf{est.S_3}. Hence, we can write 
\begin{align}\label{2020011501}
|T_{11}|=O_\prec\Big( \frac{\Lambda_2+\Lambda_3}{N^\varepsilon}\Big).
\end{align}

For the term $T_{12}$, instead of using the residue theorem as above, we directly bound the integral as follows
\begin{align} \label{est_T_{12}}
|T_{12}|\leq &C\int_{\Gamma_i}\big|z^2(1+z)^2\theta(z)^2\theta'(z)\big| \, \Big(\sum_{j}\frac{|\wt w_j|d_j}{|z-d_j|}\Big)^2\Big(\sum_{k}\frac{d_k}{|z-d_k|}\Big)^2  |\Delta(\theta(z))|^4 N^{-2}\nonumber\\
&\qquad\quad \times \Vert (D^{-1}+\theta(z)V^*\mG_1(\theta(z))V)^{-1}\Vert_{\text{op}} \,|{\rm d}z|.
\end{align}
Notice that for $z\in \Gamma_i$,
\begin{align} \label{2020011303}
|z|\asymp d_i, \quad |z-\sqrt{y}|\asymp d_i-\sqrt{y}  \quad\text{ and }\quad  |z-d_j|\asymp \left\{
\begin{array}{cc}
|d_i-d_j| & \text{ for } j\in \mathsf{I}^c\\
\rho_i  & \text{ for } j\in \mathsf{I}
\end{array}
\right. .
\end{align}
Besides,
\begin{align}\label{2020011301}
&\quad\Big\Vert (D^{-1}+\theta(z)V^*\mG_1(\theta(z))V)^{-1}\Big\Vert_{\text{op}} \notag\\
&=\Big\Vert \Big( \big(D^{-1}+\theta(z)m_1(\theta(z))\big)+ \theta(z)V^*\Xi(\theta(z))V \Big)^{-1} \Big\Vert_{\text{op}}  \nonumber\\
&\prec \frac{1}{\min_{j}|z-d_j|/ |zd_j| -\Vert \theta(z)V^*\Xi(\theta(z))V\Vert_{\text{op}} },
\end{align}
where for the last bound we used the fact that $\theta(z)m_1(\theta(z))=-(1+1/z)$. By elementary computations and the definitions of $\rho_i, \delta_{ij}$ (see \rf{def:radius} and  \rf{def:d_i-d_j}), one has
\begin{align*}
&\min_{j\in \mathsf{I}^c}|z-d_j|/ |z d_j|\asymp \min_{j\in \mathsf{I}^c}|d_i-d_j|/d_id_j\geq \min_{j\in \mathsf{I}^c}\frac {\delta_{ij}}{d_i(d_i+\delta_{ij})} \geq \frac{\rho_i}{d_i(d_i+\rho_i)}\asymp \frac{\rho_i}{d_i^2}, \\
&\min_{j\in \mathsf{I}} |z-d_j|/|z d_j| \asymp\frac{\rho_i}{d_i^2},
\end{align*}
and  
\begin{align*}
\Vert \theta(z)V^*\Xi(\theta(z))V\Vert_{\text{op}}  &\prec d_i\Vert V^*\Xi(\theta(z))V\Vert_{\text{op}}  \notag\\
&=O_\prec (d_i\min\{ |\theta(z)-\lambda_+|^{-2}, |\theta(z)-\lambda_+|^{-1/4}\} N^{-1/2})  \nonumber\\
&=O_\prec (d_i\Delta(d_i)N^{-1/2}).
\end{align*}
Note that the following bounds are implied by the non-overlapping condition \rf{asd2}
\begin{align*}
\frac{\delta_i}{d_i^3\Delta(d_i)N^{-1/2}}\geq CN^{\varepsilon},\quad \frac{d_i-\sqrt y}{d_i^3\Delta(d_i)N^{-1/2}}\geq CN^{\varepsilon},
\end{align*}
where $\delta_i=\min_{j\in \mathsf{I}^c}|d_i-d_j|$. 
 Recall $\rho_i$ from \rf{def:radius}. Thus
\begin{align*}
 \Vert \theta(z)V^*\Xi(\theta(z))V\Vert_{\text{op}}\prec N^{-\varepsilon} \frac{\rho_i}{d_i^2} .
\end{align*}
Consequently, from \rf{2020011301}, we get 
\begin{align} \label{2020011304}
\Vert (D^{-1}+\theta(z)V^*\mG_1(\theta(z))V)^{-1}\Vert_{\text{op}}  \prec \frac{d_i^2}{\rho_i}.
\end{align} 

Inserting the estimates \rf{2020011303} and \rf{2020011304} into \eqref{est_T_{12}}, we arrive at 
\begin{align}\label{eq:anotherT12}
|T_{12}| \prec & \Big(\sum_{j\in \mathsf{I}^c}\frac{|\wt w_j|d_j}{\delta_{ij}}+
\sum_{t\in \mathsf{I}} \frac{|\wt w_t|d_i}{\delta_{i0}}+ \sum_{t\in \mathsf{I},j\in \mathsf{I}^c}\frac{|\wt w_t|d_i}{\delta_{ij}}\Big)^2\notag\\
& \times\Big( \sum_{j\in \mathsf{I}^c}\frac{d_j}{\delta_{ij}}+\sum_{j\in \mathsf{I}^c}\frac {d_i}{\delta_{ij}} +
\frac{d_i}{\delta_{i0}} \Big)^2 d_i^7\delta_{i0}\, \Delta(d_i)^4N^{-2}.
\end{align}
Further, applying the estimates \rf{2020011405} to the term $\sum_{j\in \mathsf{I}^c}\frac{d_j}{\delta_{ij}}+\sum_{j\in \mathsf{I}^c}\frac {d_i}{\delta_{ij}} +
\frac{d_i}{\delta_{i0}}$ above, we get 
\begin{align}
|T_{12}| \prec N^{-\varepsilon} \Big(\sum_{j\in \mathsf{I}^c}\frac{|\wt w_j|d_j}{\delta_{ij}}+
\sum_{t\in \mathsf{I}} \frac{|\wt w_t|d_i}{\delta_{i0}}+ \sum_{t\in \mathsf{I},j\in \mathsf{I}^c}\frac{|\wt w_t|d_i}{\delta_{ij}}\Big)^2 d_i^3\delta_{i0}\, \Delta(d_i)^2 N^{-1}. \label{eq:2anotherT12}
\end{align}

Next, we expand the squared term on the RHS of \eqref{eq:2anotherT12} and analyse them term by term.
Recalling the expression of  $\Lambda_3$ in \rf{est.S_3}, we directly have 
\begin{align} \label{eq:2020040501}
\Big(\sum_{j\in \mathsf{I}^c}\frac{|\wt w_j|d_j}{\delta_{ij}}\Big)^2   d_i^3\delta_{i0}\, \Delta(d_i)^2N^{-1}\leq  \Lambda_3. 
\end{align}
Further recalling  the estimates for $S_2$, i.e. $\Lambda_2$ in \rf{def:Lambda_2} and using the first estimate in \rf{2020011405}, we obtain  
\begin{align} \label{eq:2020040503}
\Big(\sum_{t\in \mathsf{I}} \frac{|\wt w_t|d_i}{\delta_{i0}}\Big)^2  d_i^3\delta_{i0}\, \Delta(d_i)^2 N^{-1}\prec N^{-\varepsilon} 
\Big(\sum_{t\in \mathsf{I}} \frac{|\wt w_t|d_i}{\delta_{i0}}\Big)^2 \delta_{i0}^2\, \Delta(d_i) N^{-\frac 12} \leq N^{-\varepsilon}\Lambda_2.
\end{align}
It remains to bound the last squared term $( \sum_{t\in \mathsf{I},j\in \mathsf{I}^c}{|\wt w_t|d_i}/{\delta_{ij}} )^2$. First notice that 
it follows from the triangle inequality and $d_i/\delta_{i0}\geq 1$ that 
\begin{align} \label{eq:triangle}
 \sum_{t\in \mathsf{I},j\in \mathsf{I}^c}\frac{|\wt w_t|d_i}{\delta_{ij}}\leq \sum_{t\in \mathsf{I},j\in \mathsf{I}^c}\frac{|\wt w_t|d_j}{\delta_{ij}}+ \sum_{t\in \mathsf{I},j\in \mathsf{I}^c}\frac{|\wt w_t|d_i}{\delta_{i0}}.
 \end{align}
By \rf{eq:triangle}, we have  
\begin{align} \label{eq:2020040506}
&\Big(\sum_{t\in \mathsf{I},j\in \mathsf{I}^c}\frac{|\wt w_t|d_i}{\delta_{ij}}\Big)^2 d_i^3\delta_{i0}\, \Delta(d_i)^2N^{-1} \nonumber\\
\leq 
&\Big(\sum_{t\in \mathsf{I},j\in \mathsf{I}^c}\frac{|\wt w_t|d_i}{\delta_{ij}}\Big) \Big(\sum_{t\in \mathsf{I},j\in \mathsf{I}^c}\frac{|\wt w_t|d_j}{\delta_{ij}}\Big)  d_i^3\delta_{i0}\, \Delta(d_i)^2N^{-1} \notag\\
&\qquad+
 \Big(\sum_{t\in \mathsf{I},j\in \mathsf{I}^c}\frac{|\wt w_t|d_i}{\delta_{ij}}\Big) \Big( \sum_{t\in \mathsf{I}}\frac{|\wt w_t| d_i}{\delta_{i0}} \Big)  d_i^3\delta_{i0}\, \Delta(d_i)^2N^{-1}\nonumber \\
\leq
&\Lambda_3 + \Lambda_2.
\end{align}
To obtain the last inequality above, we compared the first term with the latter part in the  expression of $\Lambda_3$ in \rf{est.S_3}. And for the second term, we used the non-overlapping condition \rf{asd} for $\delta_{ij}$ and further compared it with $\Lambda_2$ in \rf{def:Lambda_2}.

Combining \rf{eq:anotherT12}-\rf{eq:2020040506}, we finally have 
\begin{align}
|T_{12}|=O_\prec \Big(\frac{\Lambda_2+\Lambda_3}{N^\varepsilon} \Big)
\end{align}
for some small fixed positive $\varepsilon$.

 Next, we turn to estimate $T_2$ and $T_3$. The method is analogous to that of $T_{12}$; we omit the details. By definition, we can bound $T_2$ and $ T_3$ as follows
 \begingroup
\allowdisplaybreaks
 \begin{align} \label{eq:est_T2}
|T_2|&=\Big| \frac{1}{2\pi \ii}\oint_{\theta(\Gamma_i)}zm_1(z)\wt {\mb{w}}^*V( zL(z)V^* \Xi(z)V )^2 (D^{-1}+zV^*\mG_1(z)V)^{-1}  V^*\Xi(z) \wt {\mb{w}} \,{\rm d}z \Big|\nonumber\\
&\leq C\oint_{\Gamma_i}\Big| z(1+z)\theta(z)^2\theta'(z)\Big|  \Big(\sum_{j}\frac{|\wt w_j|d_j}{|z-d_j|}\Big)\Big(\sum_{k}\frac{d_k}{|z-d_k|}\Big) \Vert \wt{\mb{w}}\Vert  \frac{|\Delta(z)|^3}{ N^{\frac 32}}\nonumber\\
&\qquad\qquad \quad \times\Vert (D^{-1}+\theta(z)V^*\mG_1(\theta(z))V)^{-1}\Vert_{\text{op}}  \,|{\rm d}z| \nonumber\\
&\prec  N^{-\frac 32}\Big(\sum_{j\in \mathsf{I}^c}\frac{|\wt w_j|d_j}{\delta_{ij}}+
\sum_{t\in \mathsf{I}}\frac{|\wt w_t|d_i}{\delta_{i0}}+ \sum_{t\in \mathsf{I},j\in \mathsf{I}^c}\frac{|\wt w_t|d_i}{\delta_{ij}}\Big) \Vert  \wt {\mb{w}}\Vert \notag\\
&\qquad \qquad \quad \times \Big( \sum_{j\in \mathsf{I}^c}\frac{d_j}{\delta_{ij}}+\sum_{j\in \mathsf{I}^c}\frac{d_i}{\delta_{ij}}+
\frac{d_i}{\delta_{i0}} \Big) d_i^5 \delta_{i0}\, \Delta(\theta(d_i))^3\nonumber\\
&\prec N^{-1-\varepsilon}\Big(\sum_{j\in \mathsf{I}^c}\frac{|\wt w_j|d_j}{\delta_{ij}}+
\sum_{t\in \mathsf{I}}\frac{|\wt w_t|d_i}{\delta_{i0}}+ \sum_{t\in \mathsf{I},j\in \mathsf{I}^c}\frac{|\wt w_t|d_i}{\delta_{ij}}\Big)^2 d_i^3\delta_{i0}\, \Delta(\theta(d_i))^2 \nonumber\\
&\prec N^{-\varepsilon} (\Lambda_2+\Lambda_3)
\end{align}
\endgroup
and  
\begin{align} \label{eq:est_T3}
|T_3|&=\Big| -\frac{1}{2\pi \ii}\oint_{\theta(\Gamma_i)}z\wt {\mb{w}}^* \Xi(z)V  ( zL(z)V^* \Xi(z)V )(D^{-1}+zV^*\mG_1(z)V)^{-1}  V^*\Xi(z) \wt {\mb{w}} \,{\rm d}z \Big|\nonumber\\
&\leq C\oint_{\Gamma_i} \Vert \wt {\mb{w}}\Vert^2 \Big| z\theta(z)^2\theta'(z) \Big| \Big(\sum_{k}\frac{d_k}{|z-d_k|}\Big)  \frac{|\Delta(z)|^3}{ N^{\frac 32}}\Vert (D^{-1}+\theta(z)V^*\mG_1(\theta(z))V)^{-1}\Vert_{\text{op}}  \,|{\rm d}z| \nonumber\\
&\prec \Vert  \wt {\mb{w}}\Vert^2  \Big( \sum_{j\in \mathsf{I}^c}\frac{d_j}{\delta_{ij}}+\sum_{j\in \mathsf{I}^c}\frac{d_i}{\delta_{ij}}+
\frac{d_i}{\delta_{i0}} \Big) d_i^4 \delta_{i0}\, \Delta(\theta(d_i))^3N^{-\frac 32}\nonumber\\
& \prec N^{-\varepsilon} (\Lambda_2+\Lambda_3).
\end{align}
To obtain the above bounds, we used the estimates in \rf{2020011405} and the  inequality \rf{20200417} for $\Vert \wt \bw \Vert$.
Hence, using the estimates of $T_1,T_2$ and $T_3$, we get that 
\begin{align}
\la \bw, {\rm P}_{\mathsf{I}} \bw\ra=S_1+S_2+\wt S_3+O_\prec \Big(\frac{\Lambda_2+\Lambda_3}{N^\varepsilon}\Big),
\end{align} 
for some  small  constant $\varepsilon>0$. Finally, using the definitions of $\Lambda_2, \Lambda_3$ in \rf{def:Lambda_2} and \rf{est.S_3} and the expression of $\bw_{\mathsf{I}}, \mb{\varsigma}_{\mathsf{I}}$ in \rf{def:v_I_varsigma_I}, one can rewrite the error bound 
\begin{align*}
\frac{\Lambda_2+\Lambda_3}{N^\varepsilon}&=O_\prec\Big( \big(\delta_{i0}\sqrt{d_i}\, \Vert \mb{\varsigma}_{\mathsf{I}}\Vert + d_i\Vert \bw_{\mathsf{I}} \Vert \big)\,\Vert\bw_{\mathsf{I}}\Vert \, \Delta(d_i)N^{-\frac 12-\varepsilon} \Big) \\
&+O_\prec\bigg(  \Big( d_i\Vert \mb{\varsigma}_{\mathsf{I}}\Vert^2 + d_i^2\,\Vert \mb{w}_\mathsf{I}\Vert^2\Big( \sum_{j\in \mathsf{I}^c}\frac {d_j}{\delta_{ij}}\Big) \Big(  \sum_{j\in \mathsf{I}^c}\frac {d_i}{\delta_{ij}}\Big) \Big) \delta_{i0}\, d_i\, \Delta(d_i)^2N^{-1-\varepsilon} \bigg)
\end{align*}
as in \rf{eq:weakgreen}. This completes the proof. 
 \end{proof}
{
Besides the Green representation for the eigenvector, we also have the following lemma  for the representation of the eigenvalue $\mu_i$ in the simple case $\mathsf{I}=\{i\}$. The  extension to multiple case can be found in Appendix \ref{appendix.proof}; see Proposition \ref{repre.multiple.eigv}. 
\begin{lem}\label{lem.representation.eigv}
Suppose that the assumptions in Theorem \ref{mainthm} hold. In the simple case $\mathsf{I}=\{i\}$,  we have 
\begin{align} \label{decomeigenv}
\mu_i=\theta(d_i) - (d_i^2-y) \theta(d_i) \chi_{ii} \big(\theta(d_i)\big) + O_\prec \Big( d_i^2\delta_{i0}\Delta(d_i)N^{-\frac 12-\varepsilon }\Big)
\end{align}
for some small fixed constant $\varepsilon>0$.
\end{lem}

\begin{proof}[Proof of Lemma \ref{lem.representation.eigv}]
Recall the notations at the beginning of proof of Lemma \ref{weak_lem_decomp}, $$V= (\bv_1,\cdots, \bv_r), \quad D={\rm {diag}} \Big(d_1/(1+d_1), \cdots, d_r/(1+d_r)\Big).$$
 Then, by elementary calculation, we have 
 \begin{align*}
 Q-z=\Sigma^{1/2}\mG_1^{-1}(z)\big(I_M+z\mG_1(z)VDV^*\big)\Sigma^{1/2}.
 \end{align*}
 Notice that $\mu_i$ is the $i$th largest real value such that $\det{(Q-\mu_i)}=0$. Further by the fact that $\mu_i$ stays away from the spectrum of $H$ with high probability (\cf Lemma  \ref{locationeig},  Assumption \ref{supercritical}), together with the identity $\det \big(I_M+z\mG_1(z)VDV^*\big)=\det(D)\det  \big(D^{-1}+zV^*\mG_1(z)V\big)$, we have that $\mu_i$ is the $i$th largest real solution to the equation $\det  \big(D^{-1}+z V^*\mG_1(z)V\big)=0$ with high probability. 
 
For $x\in [\lambda_++ N^{-2/3+ \delta}, \infty) \cap \big(\theta(d_i)- d_i^{1/2}\delta_{i0}^{1/2}N^{-1/2 + \delta}, \theta(d_i)+ d_i^{1/2}\delta_{i0}^{1/2}N^{-1/2 + \delta} \big)$ with sufficiently small constant $\delta>0$,  we define the matrices $\mathcal{A}(x)=(\mathcal{A}_{ij}(x))$ and $\wt{\mathcal{A}}(x)=(\wt{\mathcal{A}}_{ij}(x))$ by setting
 \begin{align*}
& \mathcal{A}_{ij}(x)=(1+d_i^{-1})\delta_{ij}+x\mb{v}_i^*\mG_1(x)\mb{v}_j-xm_1(x)\delta_{ij}, \notag\\ &\wt{\mathcal{A}}_{ij}(x)=\delta_{ij}((1+d_i^{-1})+x\mathbf{v}_i^*\mG_1(x)\mb{v}_i-xm_1(x)),
 \end{align*}
 Further, we denote the eigenvalues of $\mathcal{A}(x)$ and $\wt{\mathcal{A}}(x)$ by $a_1(x)\leq \ldots\leq a_r(x)$ and $\wt{a}_1(x)\leq \ldots\leq \wt{a}_r(x)$ respectively. Apparently, one has
 \begin{align}
 \wt{a}_i(x)= (1+d_i^{-1})+x\mathbf{v}_i^*\mG_1(x)\mb{v}_i-xm_1(x), \label{19072601}
 \end{align}
 with high probability by the isotropic local law (\ref{est.DG}) and the Assumption \ref{supercritical}. We then claim that, in order to prove  (\ref{decomeigenv}), it suffices to show  the following two estimates
 \begin{align}
 &\mu_im_1(\mu_i)=-a_i(\theta(d_i))+ O_\prec\Big( d_i^{-\frac 12} \delta_{i0}^{-\frac 12} N^{-\frac 12- \varepsilon}\Big),\label{19072602}\\
 & a_i(\theta(d_i))=\wt{a}_i(\theta(d_i))+ O_\prec\Big( d_i^{-\frac 12} \delta_{i0}^{-\frac 12} N^{-\frac 12- \varepsilon}\Big).  \label{19072603}
 \end{align}
 for some small constant $\varepsilon>0$.
 Combining (\ref{19072601}), (\ref{19072602}) and (\ref{19072603}), we have 
 \begin{align*}
&\mu_im_1(\mu_i)- \theta(d_i)m_1(\theta(d_i))\notag\\
&\quad= -(1+d_i^{-1})-\theta(d_i)\mathbf{v}_i^*\mG_1(\theta(d_i))\mb{v}_i + O_\prec\Big( d_i^{-\frac 12} \delta_{i0}^{-\frac 12} N^{-\frac 12- \varepsilon}\Big). 
 \end{align*}
 Expanding $\mu_im_1(\mu_i)$ around $\theta(d_i)m_1(\theta(d_i))$ with the aid of Lemma \ref{locationeig}  will then lead to  (\ref{decomeigenv}). Therefore, what remains is to prove (\ref{19072602}) and (\ref{19072603}).  
 
 We start with (\ref{19072602}).  First, by the fact that $\mu_i$ is a solution to $\det  \big(D^{-1}+z V^*\mG_1(z)V\big)=0$, it is easy to see that 
 $\mu_im_1(\mu_i)= -a_k(\mu_i)$ for some $k$.  But by isotropic law (\ref{est.DG}),  we see that  $\mathcal{A}=\wt{\mathcal{A}}+O_\prec(d_i^{-\frac 12} \delta_{i0}^{-\frac 12} N^{-\frac12})$, where $O_\prec(d_i^{-\frac 12} \delta_{i0}^{-\frac 12} N^{-\frac12})$ represents a matrix bounded in  operator norm  by $O_\prec(d_i^{-\frac 12} \delta_{i0}^{-\frac 12} N^{-\frac12})$. This leads to the  estimate $a_k(\mu_i)=\wt{a}_k(\mu_i)+O_\prec(d_i^{-\frac 12} \delta_{i0}^{-\frac 12} N^{-\frac12})$. Further, by  (\ref{est.DG}) and Lemma \ref{locationeig} one can easily show that $\mu_im_1(\mu_i)=-(1+d_i^{-1})+O_\prec(d_i^{-\frac 12} \delta_{i0}^{-\frac 12} N^{-\frac12})$ and $\wt{a}_k(\mu_i)=(1+d_k^{-1})+O_\prec(d_i^{-\frac 12} \delta_{i0}^{-\frac 12} N^{-\frac12})$.  Therefore, due to the fact that $d_i$'s are well separated, more specifically, non-overlapping condition \rf{asd2}, we have $\mu_im_1(\mu_i)= -a_i(\mu_i)$  with high probability. Next, by the isotropic law (\ref{est.DG}), one can also check that 
 \begin{align*}
 \|\partial_x \mathcal{A}(x)\|_{\text{op}}\prec d_i^{\frac 12} \delta_{i0}^{-\frac 52} N^{-\frac 12} 
 \end{align*}
 This together with Lemma \ref{locationeig} leads to
 \begin{align*}
 |a_i(\mu_i)-a_i(\theta(d_i))|\prec  d_i \delta_{i0}^{-2} N^{-1}\prec d_i^{-\frac 12} \delta_{i0}^{-\frac 12}N^{-\frac12-\varepsilon},
 \end{align*}
 where the last step is due to the fact that  $d_i\delta_{i0}^{-\frac 32} N^{-\frac 12} = O(N^{-\varepsilon})$ following from \rf{asd}.
 
 Combining the above with the fact $\mu_im_1(\mu_i)= -a_i(\mu_i)$  with high probability, we arrive at (\ref{19072602}). 
 
  Next, we prove (\ref{19072603}).  Observe that the diagonal entries of $\mathcal{A}-\wt{\mathcal{A}}$ are $0$, and $\wt{\mathcal{A}}$ is a diagonal matrix. So expanding the eigenvalues of $\mathcal{A}$ around the eigenvalues of $\wt{\mathcal{A}}$ using the perturbation theory, we see that the first order  term vanishes. Hence, it suffices to estimate the second order term. More specifically, we have 
 \begin{align*}
 |a_i(\theta(d_i))-\wt{a}_i(\theta(d_i))|\prec \frac{\|\mathcal{A}-\wt{\mathcal{A}}\|_{\text{op}}^2}{\min_{j\neq i} |\wt{a}_j(\theta(d_i))-\wt{a}_i(\theta(d_i))|}\prec d_i^{-\frac 12} \delta_{i0}^{-\frac 12} N^{-\frac 12- \varepsilon},
 \end{align*} 
for some small fixed $\varepsilon>0$, where the last step follows from the fact that $\wt{a}_k(\theta(d_i))=1+d_k^{-1}+O_\prec(d_i^{-\frac12} \delta_{i0}^{-\frac12}N^{-\frac12})$ and the fact that the $d_i$'s are well separated and satisfy \rf{asd2}. This concludes the proof of (\ref{19072603}).
\end{proof}
}

\section{Proof of Theorems \ref{mainthm} and \ref{thm.simple case} } \label{s.proof of thm}

In this section, we prove the  main result Theorem \ref{mainthm}, based on Proposition \ref{recursivemP}. The proof of Theorem \ref{thm.simple case} is nearly the same, and thus it will be  only briefly stated in the end of the section. The proof of Proposition \ref{recursivemP} will be deferred to Section \ref{sec:prop}. The starting point  is Lemma \ref{weak_lem_decomp} and  Remark \ref{rmk:representation}, which state that  the study of $\la \bw, {\rm P}_{\mathsf{I}} \bw\ra$ can be  reduced to the study of the random vector 
\begin{align*}
\begin{pmatrix}
 \bw_{\mathsf{I}}^*\Xi' (z)\bw_{\mathsf{I}}, &\bw_{\mathsf{I}}^*\Xi (z)\bw_{\mathsf{I}},& \mb{\varsigma}_{\mathsf{I}}^*\Xi(z) \bw_{\mathsf{I}}, &\{ \bv_t^*\Xi (z)\mb{\varsigma}_{\mathsf{I}}\}_{t\in \mathsf{I}}, &\{\bv_j^* \Xi (z)\bw_{\mathsf{I}}\}_{j\in \mathsf{I}^c}
\end{pmatrix}^*
\end{align*}
at  $z=\theta(d_i)$.

In view of the isotropic local laws in Theorem \ref{isotropic} and Lemma \ref{weak_lem_decomp}, we shall study the limiting distribution of the rescaled random vectors 
 \begin{align} \label{def:chi_I}
 &\mb{\chi}_{\mathsf{I}}(z):= \frac{\sn}{ \Delta(d_i)}\Big( c_1 (\bw_{\mathsf{I}}^0)^*\Xi(z)\bw_{\mathsf{I}}^0+ c_2(\bw_{\mathsf{I}}^0)^*\Xi'(z)\bw_{\mathsf{I}}^0,\; (\mb{\varsigma}_{\mathsf{I}}^0)^*\Xi(z)\bw_{\mathsf{I}}^0,\; \{\bv_t^*\Xi  (z)\mb{\varsigma}_{\mathsf{I}}^0\}_{t\in \mathsf{I}},\; \{ \bv_j^*\Xi (z)\bw_{\mathsf{I}}^0\}_{j\in \mathsf{I}^c} \Big)
\end{align}
at $z=\theta(d_i)$, where  $\bw_{\mathsf{I}}^0$ and $\mb{\varsigma}_{\mathsf{I}}^0$ are defined  in (\ref{normalized bw}) and (\ref{def.nor.Varsigma}) and we further introduced the shorthand notations
 \begin{align} \label{def:c12}
c_1\equiv c_1(d_i):= 2d_i \quad \text{and}\quad c_2\equiv c_2(d_i):= \frac{(d_i^2-y)^2}{d_i^2}. 
\end{align}
We remark here, in the case when $d_i>K$ for sufficiently large $K$ or even grows with $N$, that either summand in the first component of $\mb{\chi}_{\mathsf{I}}(z)$ may have larger typical order than the other components of the vector $\mb{\chi}_{\mathsf{I}}(z)$, by a direct application of the isotropic local laws \rf{weak_est_DG}. However, for such large $d_i$, referring to  the estimate of  \rf{weak_19071941}  in the proof of Lemma \ref{weak_lem_decomp}, more precisely, \rf{cancel:1}-\rf{2020072106}, we shall also exploit a cancellation between two summands in the first component of (\ref{def:chi_I}) that is simply proportional to the combination in \rf{weak_19071941}.
Consequently, all the components of $\mb{\chi}_{\mathsf{I}}(z)$ are of comparable sizes. The details of such arguments will be provided in later proofs. Thus our goal is to prove that $\mb{\chi}_{\mathsf{I}}(z)$ is asymptotically Gaussian (c.f. Lemma \ref{lem:normal}) and this leads to the proof of Theorem \ref{mainthm}.

Before we give the precise statement, let us introduce some necessary notations. For brevity, in the sequel, we will very often omit the $z$-dependence from the notations. For instance, we will write $m_{1,2}(z)$ as $m_{1,2}$. 

In the sequel, we fix an $i\in \lb 1, r_0\rb$ and its corresponding index set ${\mathsf{I}}\equiv {\mathsf{I}}(i)$. For simplicity, we also denote
\begin{align} \label{def:wtc12}
 \wt c_1= \frac{d_i^2+2d_iy+y}{d_i+y},\quad  \wt c_2= c_2= \frac{(d_i^2-y)^2}{d_i^2}.
\end{align}
Define a symmetric matrix  $\mathcal{M}_{\mathsf{I}}\equiv \mathcal{M}_{\mathsf{I}}(z)\in{\mathbb{C}^{(r+2)\times (r+2)}}$ with diagonal entries 
\begin{align}\label{def:Mdiag}
\mathcal M_\mathsf{I} (j,j) = 
\begin{cases}
\frac{2}{\Delta(d_i)^2}  \; \sum\limits_{a,b=1}^2  c_a \wt c_bm_1 \sum\limits_{
\begin{subarray}{c}
a_1,a_2\geq 1\\ a_1+a_2=a+1
\end{subarray}}
\frac{m_1^{(a_1-1)} (zm_1)^{(b+a_2-1)}}{(a_1-1)!(b+a_2-1)!},
&\text{ if } j=1;\\
\\
\frac{1}{\Delta(d_i)^2}\;m_1^2(zm_1)', &\text{ otherwise } ;\\
\end{cases}
\end{align}
 and its non-zero off-diagonal  entries as
 \begin{align} \label{def:Moff}
 & \mathcal{M}_\mathsf{I}(2, \lb 3,r+2 \rb)=
\frac{1}{\Delta(d_i)^2}\Big(m_1^2(zm_1)'\Big) \bigg( \Big\{ \la \bw_{\mathsf{I}}^0, \bv_t\ra\Big\}_{t\in \mathsf{I}} , \;  \Big\{ \la \mb{\varsigma}_{\mathsf{I}}^0, \bv_j\ra \Big\}_{j\in \mathsf{I}^c}\bigg),\nonumber\\
  & \mathcal{M}_\mathsf{I}(\lb 3,|\mathsf{I}|+ 2 \rb, \lb |\mathsf{I}|+3,r+2 \rb)=\frac{1}{\Delta(d_i)^2} \Big(m_1^2(zm_1)'\Big) \Big(  \la \bw_{\mathsf{I}}^0, \bv_t\ra  \la \mb{\varsigma}_{\mathsf{I}}^0, \bv_j\ra \Big)_{t\in \mathsf{I},j\in \mathsf{I}^c}.
 \end{align}
Here  for any matrix $A=(a_{ij})_{n\times n}$ and  index sets $\mathcal{S}_1,\mathcal{S}_2\subset \lb1, n\rb$,   we denote by $A(\mathcal{S}_1,\mathcal{S}_2)$ the submatrix of $A$ obtained via taking rows from $\mathcal{S}_1$ and columns from $\mathcal{S}_2$. More specifically, 
\begin{align*}
A(\mathcal{S}_1,\mathcal{S}_2):=(a_{ij})_{i\in \mathcal{S}_1, j\in\mathcal{S}_2}.
\end{align*}
When $\mathcal{S}_1=\{a\}$ is a singleton, we use the abbreviation $A(a, \mathcal{S}_2)$ instead of $A(\{a\}, \mathcal{S}_2)$. 
We then further define the  symmetric matrix  $\mathcal{K}_\mathsf{I}\equiv \mathcal{K}_\mathsf{I}(z)\in \mathbb{C}^{(r+2)\times (r+2)}$ whose 
$r \times r $ lower right corner has the block structure 
\begin{align*}
\left(
\begin{array}{cc}
\mathcal{K}_\mathsf{I}(\mathsf{I},\mathsf{I}) & \mathcal{K}_\mathsf{I}(\mathsf{I}, \mathsf{I}^c)\\
\mathcal{K}_\mathsf{I}(\mathsf{I}^c, \mathsf{I}) & \mathcal{K}_\mathsf{I}(\mathsf{I}^c, \mathsf{I}^c)
\end{array}
\right),
\end{align*}
given by 
\begingroup
\allowdisplaybreaks
\begin{align} \label{def_Klowdiag}
&\mathcal{K}_\mathsf{I}(\mathsf{I}, \mathsf{I})= \frac{1}{\Delta(d_i)^2}(zm_2m_1^2)^2 \Big(\mathbf{s}_{1,1,2}(\bv_{t_1},\bv_{t_2},\mb{\varsigma}_{\mathsf{I}}^0)\Big)_{t_1,t_2\in \mathsf{I}}, \nonumber\\
&\mathcal{K}_\mathsf{I}(\mathsf{I}^c, \mathsf{I}^c)= \frac{1}{\Delta(d_i)^2} (zm_2m_1^2)^2 \Big(\mathbf{s}_{1,1,2}(\bv_{j},\bv_{k},\bw_{\mathsf{I}}^0)\Big)_{j,k \in \mathsf{I}^c},\nonumber\\
&\mathcal{K}_\mathsf{I}(\mathsf{I}, \mathsf{I}^c)= \frac{1}{\Delta(d_i)^2} (zm_2m_1^2)^2 \Big(\mathbf{s}_{1,1,1,1}(\bv_{t},\bv_{j}, \mb{\varsigma}_{\mathsf{I}}^0,\bw_{\mathsf{I}}^0)\Big)_{t\in \mathsf{I}, j\in \mathsf{I}^c}.
\end{align}
\endgroup
The remaining  entries  are defined by 
 \begingroup
\allowdisplaybreaks
\begin{align} \label{def_Kothers}
&\mathcal{K}_\mathsf{I}(1,1)=   \frac{\mathbf{s}_{4}(\bw_{\mathsf{I}}^0)}{\Delta(d_i)^2} \Big(\sum_{s=1}^2 \wt c_s\frac {m_1(zm_2m_1)^{(s-1)}}{(s-1)!}\Big)\Big( \sum_{s=1}^2  c_s\frac {(zm_2m_1^2)^{(s-1)}}{(s-1)!} \Big),
\nonumber\\
&\mathcal{K}_\mathsf{I}(1, \lb 2, r+2\rb )= \frac{zm_2m_1^2}{2\Delta(d_i)^2}\Big(\sum_{s=1}^2 \wt c_s\frac {m_1(zm_2m_1)^{(s-1)}}{(s-1)!} +  \sum_{s=1}^2  c_s\frac {(zm_2m_1^2)^{(s-1)}}{(s-1)!} \Big) \nonumber\\
&\qquad \qquad \qquad\qquad \times
 \,\Big(\mathbf{s}_{1,3}(\mb{\varsigma}_{\mathsf{I}}^0, \bw_{\mathsf{I}}^0),\;\{\mathbf{s}_{1,1,2}(\bv_t , \mb{\varsigma}_{\mathsf{I}}^0 ,\bw_{\mathsf{I}}^0 )\}_{t\in \mathsf{I}},\; \{\mathbf{s}_{1,3}(\bv_j, \bw_{\mathsf{I}}^0)\}_{j\in \mathsf{I}^c} \Big), \nonumber \\
&\mathcal{K}_\mathsf{I}(2,\lb 2, r+2\rb)=\frac{(zm_2m_1^2)^2}{\Delta(d_i)^2} \notag\\
&\qquad\qquad\qquad\qquad \times\Big(\mathbf{s}_{2,2}(\mb{\varsigma}_{\mathsf{I}}^0,\bw_{\mathsf{I}}^0),\;\{\mathbf{s}_{1,1,2}(\bv_t,\bw_{\mathsf{I}}^0 , \mb{\varsigma}_{\mathsf{I}}^0)\}_{t\in \mathsf{I}},\; \{\mathbf{s}_{1,1,2}(\bv_j, \mb{\varsigma}_{\mathsf{I}}^0 , \bw_{\mathsf{I}}^0)\}_{j\in \mathsf{I}^c} \Big) .
\end{align}
\endgroup
Finally, we  define 
\begin{align}\label{weak_def:V}
\mathcal{V}_\mathsf{I}(z)=\mathcal{M}_\mathsf{I}(z)+\kappa_4 \mathcal{K}_\mathsf{I}(z).
\end{align}

Recall $\mathbf{\chi}_\mathsf{I}$ defined in (\ref{def:chi_I}).  We have the following lemma. 
\begin{lem}\label{lem:normal} Under the assumptions of Theorem \ref{mainthm}, we have
$$\mathbf{\chi}_\mathsf{I}(\theta(d_i)) \simeq \mathcal N\left(\mathbf{0}, \mathcal V_\mathsf{I}(\theta(d_i)) \right).$$
\end{lem}

With the above lemma, we now can finish the proof of Theorem \ref{mainthm}.
\begin{proof}[Proof of Theorem \ref{mainthm}]
Recall the shorthand notation $\delta_{i0}=d_i-\sqrt y$. Set
 \begingroup
\allowdisplaybreaks
\begin{align}\label{weak_def:variables}
&\Theta_\mathsf{I}^{\bw}:=- (d_i^2- y)^{\frac 12}(1+d_i)\Delta(d_i)\bigg(  \frac{\sn}{ \Delta(d_i)}\Big( c_1 (\bw_{\mathsf{I}}^0)^*\Xi(z)\bw_{\mathsf{I}}^0+ c_2(\bw_{\mathsf{I}}^0)^*\Xi'(z)\bw_{\mathsf{I}}^0\Big)\bigg),
\nonumber \\
&\Lambda_\mathsf{I}^{\bw}:=-2(d_i^2- y)\sqrt{\frac {d_i+1}{\delta_{i0}}} \;\Delta(d_i) \Big(\frac{\sn}{\Delta(d_i)}(\mb{\varsigma}_{\mathsf{I}}^{0})^*\Xi(\theta(d_i))\bw_{\mathsf{I}}^0\Big),\nonumber \\
&\Delta_{\mathsf{I}t}^{\bw}:=-\sqrt{d_ig(d_i)}\, \Delta(d_i) \; \Big(\frac{\sn}{\Delta(d_i)}\bv_t^*\Xi(\theta(d_i))\mb{\varsigma}_{\mathsf{I}}^0\Big) \quad \text{ for } t\in \mathsf{I},\nonumber \\
&\Pi_{\mathsf{I}j}^{\bw}:=-\sqrt{(1+d_i)g(d_i)}\,\Delta(d_i)  \Big(\frac{\sn}{\Delta(d_i)}\bv_j^*\Xi(\theta(d_i))\bw_{\mathsf{I}}^0\Big) \quad \text{ for } j\in \mathsf{I}^c.
\end{align}
\endgroup
Notice that $\Theta_\mathsf{I}^{\bw}$,  $\Lambda_\mathsf{I}^{\bw}$, $\Delta_{\mathsf{I}t}^{\bw}$ and $\Pi_{\mathsf{I}j}^{\bw}$ for $t\in \mathsf{I}, j\in \mathsf{I}^c$ are all linear combinations of the components of $\mathbf{\chi}_\mathsf{I}(\theta(d_i))$. Then by Lemma \ref{lem:normal}, they are also asymptotically jointly Gaussian with mean $\mb{0}$. The entries of the covariance matrix of $(\Theta_\mathsf{I}^{\bw},  \Lambda_\mathsf{I}^{\bw}, \{\Delta_{\mathsf{I}t}^{\bw}\}_{t\in \mathsf{I}}, \{\Pi_{\mathsf{I}j}^{\bw}\}_{j\in \mathsf{I}^c})$ can be obtained from $\mathcal V_\mathsf{I}(\theta(d_i))$. Further simplifications can be achieved using the following identities at $z=\theta(d_i)$
\begin{align*}
&m_1^2(zm_1)' = \frac{1}{(d_i+y)^2(d_i^2-y)},\qquad  m_1^2(zm_1)''= -\frac{2d_i^3}{(d_i+y)^2(d_i^2-y)^{3}}, \notag\\
& m_1^2(zm_1)'''= \frac{6d_i^4(d_i^2+y)}{(d_i+y)^2(d_i^2-y)^{5}}, \qquad m_1m_1'(zm_1)'= - \frac{d_i^2}{(d_i+y)^3(d_i^2-y)^2} ,\\
&m_1m_1'(zm_1)'' = \frac{2d_i^5}{(d_i+y)^3(d_i^2-y)^4}, \qquad zm_2m_1^2=-\frac{1}{d_i(d_i+y)},
\\
&(zm_2m_1^2)'=\frac{2d_i+y}{(d_i+y)^2(d_i^2-y)}, \qquad m_1(zm_2m_1)'= \frac{1}{(d_i+y)(d_i^2-y)},
\end{align*}
as well as the explicit expressions for $c_1, c_2, \wt c_{1}, \wt c_{2}$ in \rf{def:c12} and \rf{def:wtc12}. 

We further set 
 $$\Theta_{\bw_{\mathsf{I}}}^{\bw}=\Theta_\mathsf{I}^{\bw} \|\bw_{\mathsf{I}}\|^2, \quad
\Lambda_{\mb{\varsigma}_{\mathsf{I}}}^{\bw}=\Lambda_\mathsf{I}^{\bw} \|\bw_{\mathsf{I}}\|, \quad\Delta_{\bv_t}^{\bw}=\Delta_{\mathsf{I}t}^{\bw}, \quad\Pi_{\bv_j}^{\bw}= \Pi_{\mathsf{I}j}^{\bw} \|\bw_{\mathsf{I}}\|. $$ 
 It follows immediately from Lemma \ref{weak_lem_decomp} that $\la \bw, {\rm{P}}_{\mathsf{I}}\bw \ra$ can be written as the RHS of \rf{weak_gf_decomp} in  Theorem \ref{mainthm} and 
 $$\Big( \Theta_{\bw_{\mathsf{I}}}^{\bw}, \Lambda_{\mb{\varsigma}_{\mathsf{I}}}^{\bw}, \{\Delta_{\bv_t}^{\bw}\}_{t\in \mathsf{I}}, \{\Pi_{\bv_j}^{\bw}\}_{j\in \mathsf{I}^c} \Big) \simeq \mathcal N \left(\mb{0}, A_{\mathsf{I}}^{\bw} + \kappa_4  \frac{d_i^2-y}{d_i^2} B_{\mathsf{I}}^{\bw}\right),$$ 
 with the matrices $A_{\mathsf{I}}^{\bw}$ and $B_{\mathsf{I}}^{\bw}$ defined in \eqref{def:covA} and $\eqref{def:covB}$. This completes the proof of Theorem \ref{mainthm}.
\end{proof}

In the rest of this section, we prove Lemma \ref{lem:normal}, based on our key technical result, Proposition \ref{recursivemP}. In order to show the asymptotic Gaussianity of $\mathbf{\chi}_\mathsf{I}(\theta(d_i))$, it suffices to show that all  linear combinations of the components of $\mathbf{\chi}_\mathsf{I}(\theta(d_i))$ are asymptotic Gaussian.  Our proof will be based on a moment estimate. This requires a deterministic bound for the Green function,  in order to control the contribution of the bad event in the isotropic local laws. To this end, we introduce a tiny imaginary part to the parameter $z$, such that the Green functions can be bounded by $1/\Im z$ deterministically. Specifically, in the sequel, we set 
\begin{align}
z=\theta(d_i)+\ii N^{-K}, \label{19072401}
\end{align}
for some sufficiently large constant $K>0$. 

For a fixed deterministic  column vector $\mb{c}=(c_{01}, \; c_{02},\; \{c_{1t}\}_{t\in \mathsf{I}}, \; \{c_{2j}\}_{j\in \mathsf{I}^c} )^*\in \mathbb{R}^{r+2}$, we define
\begin{align}
\mP:=\mb{c}^*\mathbf{\chi}_\mathsf{I}(z), \label{19071950}
\end{align}
where $z$ is given in (\ref{19072401}).  Notice that $|\mathcal{P}|\prec 1$ by the isotropic local law. 
 Here we omit the dependence of $\mathcal{P}$ on $\mathbf{c}$ and the index set ${\mathsf{I}}$ for simplicity. Hereafter we always assume that $\mathbf{c}$ and ${\mathsf{I}}$ are fixed. 
The following proposition is our main technical task. 

\begin{prop}[Recursive moment estimate] \label{recursivemP} Let $\mathcal{P}$ be defined in (\ref{19071950}) with $z$ given in (\ref{19072401}). Under the assumption of  Theorem \ref{mainthm}, we have 
\begin{align}
&(i) \quad\mbE \mathcal{P}=O_\prec(N^{-\frac12}\delta_{i0}^{-\frac 32}d_i^{\frac 32}), \label{est_EmP}\\
&(ii)\quad \mbE \mathcal{P}^l =(l-1)\mathcal{V}_{\mathsf{I},\mathbf{c}}\mbE \mathcal{P}^{l-2}+O_\prec(N^{-\frac12}\delta_{i0}^{-\frac 32}d_i^{\frac 32}),\label{est_EmPk}
\end{align}
where {$\mathcal{V}_{\mathsf{I},\mathbf{c}}=\mb{c}^*{\mathcal{V}_\mathsf{I} (\theta(d_i))}\mb{c }$} with  $\mathcal{V}_\mathsf{I}(\theta(d_i))$ defined in \eqref{weak_def:V} and $\delta_{i0}= d_i-\sqrt y$ is defined in  \rf{def:d_i-d_j}.
\end{prop}

With the above proposition, we can now show the proof of Lemma \ref{lem:normal}. 

\begin{proof}[Proof of Lemma \ref{lem:normal}]
By Proposition \ref{recursivemP}, one observes that $\mP(z)$ is asymptotically Gaussian with mean $\mb{0}$ and variance $\mathcal{V}_{\mathsf{I},\mathbf{c}}$. By the definition of $z$ in (\ref{19072401}), and a simple continuity argument for the Green function, one can easily see that $\mathcal{P}(\theta(d_i))$ admits the same asymptotic distribution as  $\mP(z)$, when $K$ is chosen to be sufficiently large. 
Since  {$\mathcal{V}_{\mathsf{I},\mathbf{c}}=\mb{c}^*{\mathcal{V}_\mathsf{I}(\theta(d_i))}\mb{c }$} and $\mb{c}$ is arbitrary, we have 
$$\mathbf{\chi}_\mathsf{I}(\theta(d_i)) \simeq \mathcal N\left(\mathbf{0}, \mathcal V_\mathsf{I}(\theta(d_i)) \right).$$
This concludes the proof of Lemma \ref{lem:normal}. 
\end{proof}

In the end of this section, we briefly state the proof of Theorem \ref{thm.simple case}.
 \begin{proof} [Proof of Theorem \ref{thm.simple case}]
 In addition to the expansion of $\la \bw,  {\rm{P}}_i \bw\ra$ which can be easily obtained from Lemma \ref{weak_lem_decomp} in case $\mathsf{I}=\{i\}$, by setting 
 \begin{align}
 \Phi_i:= - \sn \sqrt{d_i^2-y} \theta(d_i)\chi_{ii} \big(\theta(d_i)\big),  \label{090301}
 \end{align} 
 it follows immediately from Lemma \ref{lem.representation.eigv} that 
 the  eigenvalue $\mu_i$ can be written in terms of $\Phi_i$. Note that $\Phi_i$ is also a quadratic form similar to the other quadratic forms in   the expansion of eigenvector in Lemma  \ref{weak_lem_decomp} in case  $\mathsf{I}=\{i\}$. Hence their joint distribution can de deduced analogously to the proof of Theorem \ref{mainthm} by adding one more quadratic form into the linear combinations of quadratic forms $\mathcal{P}$. Since the proof strategy of Theorem \ref{mainthm} can apply mutatis mutandis to that of Theorem  \ref{thm.simple case}, we omit the details and conclude the proof. 
 \end{proof}

{\bf \large Acknowledgment.} The second author would like to thank Igor Silin for sharing the Python codes of \cite{SF} and providing some insights on the statistical applications.

\begin{appendix}
\section{ Collection of derivatives} \label{s.derivative of G}

In this section, we summarize some derivatives that appear in the previous sections. And all these derivatives can be obtained by repeatedly applying the second identity in \rf{derivative} and chain rule. 
For convenience, we set 
\begin{align}
\mathcal{O}_l:=\{(o_1,\cdots,o_l):\exists i\in \{1,\cdots,l\}, o_i=0, \text{and }o_{j}=1,2  \text{ for all $j\neq i$}\} \label{def of O}
\end{align}

For simplicity, in the sequel,  we still use $\mG_a^{s}, \mG_a^{(s)}$ to denote $\Upsilon^{s-1}\mG_a^{s}, \Upsilon^{s}\mG_a^{(s)}$ respectively  for $a=1,2$. Below,  we first collect  the derivatives of $(X^*\mG_1^s\bv)_{k}$  for any deterministic unit vector $\bv$, 
which can be derived by using (\ref{derivative}) and the product rule.
The first derivative of $(X^*\mG_1^s\bv)_{k}$ is
\begin{align}
\label{deriXG^l}
 &\frac{\partial (X^*\mG_1^{s}\bv)_{k}}{\partial x_{qk}}=(\mG_1^s\bv)_{q}-\sum_{a=1}^2\sum_{\begin{subarray}{c}s_1,s_2\geq 1;\\ s_1+s_2=s+1 \end{subarray} }(X^*\mG_1^{s_1}\mathscr{P}_a^{qk}\mG_1^{s_2}\bv)_{k}.
\end{align}
The second derivative of  $(X^*\mG_1^s\bv)_{k}$ is
\begin{align} \label{eq:2ndXG}
\frac{\partial^2 (X^*\mG_1^s\bv)_{k}}{\partial x_{qk}^2}=&2\bigg(-\sum_{a=1}^2\sum_{\begin{subarray}{c}s_1,s_2\geq 1;\\ s_1+s_2=s+1 \end{subarray}}\Big(\mG_1^{s_1}\mathscr{P}_a^{qk}\mG_1^{s_2}\bv\Big)_{q}\notag\\
&+\sum_{a_1,a_2=1}^2\sum_{\begin{subarray}{c}s_1, s_2, s_3\geq 1;\\ \sum_{i=1}^3s_i=s+2 \end{subarray}}\Big(X^*\mG_1^{s_1}\mathscr{P}_{a_1}^{qk}\mG_1^{s_2}\mathscr{P}_{a_2}^{qk}\mG_1^{s_3}\bv\Big)_{k}\nonumber\\
&
-\sum_{\begin{subarray}{c}s_1,s_2\geq 1;\\ s_1+s_2=s+1 \end{subarray}}\Big(X^*\mG_1^{s_1}\mathscr{P}_{0}^{qk}\mG_1^{s_2}\bv\Big)_{k}\bigg).
\end{align}
The third derivative of  $(X^*\mG_1^s\bv)_{k}$ is
\begin{align} 
\label{rankrorder3derivative}
\frac{\partial^3 (X^*\mG_1^s\bv)_{k}}{\partial x_{qk}^3}=& 6\bigg(\sum_{a_1,a_2=1}^2\sum_{\begin{subarray}{c}s_1,s_2,s_3\geq 1;\\ \sum_{i=1}^3s_i=s+2 \end{subarray}}\Big(\mG_1^{s_1}\mathscr{P}_{a_1}^{qk}\mG_1^{s_2}\mathscr{P}_{a_2}^{qk}\mG_1^{s_3}\bv\Big)_{q} \notag\\
&-\sum_{\begin{subarray}{c}s_1,s_2\geq 1;\\ s_1+s_2=s+1 \end{subarray}}\Big(\mG_1^{s_1}\mathscr{P}_{0}^{qk}\mG_1^{s_2}\bv\Big)_{q} \nonumber\\
&-\sum_{a_1,a_2,a_3=1}^2\sum_{\begin{subarray}{c}s_1,\cdots,s_4\geq 1;\\ \sum_{i=1}^4s_i=s+3 \end{subarray}}\Big(X^*\Big(\prod_{i=1}^3(\mG_1^{s_i}\mathscr{P}_{a_i}^{qk})\Big)\mG_1^{s_4}\bv\Big)_{k}\nonumber\\
&+\sum_{(a_1,a_2)\in \mathcal{O}_2}\sum_{\begin{subarray}{c}s_1,\cdots,s_3\geq 1;\\ \sum_{i=1}^3s_i=s+2 \end{subarray}}\Big(X^*\Big(\prod_{i=1}^2(\mG_1^{s_i}\mathscr{P}_{a_i}^{qk})\Big)\mG_1^{s_3}\bv\Big)_{k}\bigg).
\end{align}
The fourth derivative of  $(X^*\mG_1^s\bv)_{k}$ is
\begin{align} 
\frac{\partial^4 (X^*\mG_1^s\bv)_{k}}{\partial x_{ik}^4}=&4!\bigg(-\sum_{a_1,a_2,a_3=1}^2\sum_{\begin{subarray}{c}s_1,\cdots,s_4\geq 1;\\ \sum_{i=1}^4s_i=s+3 \end{subarray}}\Big(\Big(\prod_{i=1}^3(\mG_1^{s_i}\mathscr{P}_{a_i}^{qk})\Big)\mG_1^{s_4}\bv\Big)_{q} \nonumber\\
&+\sum_{(a_1,a_2)\in \mathcal{O}_2}\sum_{\begin{subarray}{c}s_1,s_2,s_3\geq 1;\\ \sum_{i=1}^3s_i=s+2 \end{subarray}}\Big(\Big(\prod_{i=1}^2(\mG_1^{s_i}\mathscr{P}_{a_i}^{qk})\Big)\mG_1^{s_3}\bv\Big)_{q}\nonumber\\
&+\sum_{a_1,\cdots,a_4=1}^2\sum_{\begin{subarray}{c}s_1,\cdots,s_5\geq 1;\\ \sum_{i=1}^5s_i=s+4\end{subarray}}\Big(X^*
\Big(\prod_{i=1}^4(\mG_1^{s_i}\mathscr{P}_{a_i}^{qk})\Big)\mG_1^{s_5}\bv\Big)_{k} \nonumber\\
&-\sum_{(a_1,a_2,a_3)\in \mathcal{O}_3}\sum_{\begin{subarray}{c}s_1,\cdots,s_4\geq 1;\\ \sum_{i=1}^4s_i=s+3 \end{subarray}}\Big(X^*\Big(\prod_{i=1}^3(\mG_1^{s_i}\mathscr{P}_{a_i}^{qk})\Big)\mG_1^{s_4}\bv\Big)_{k}\nonumber\\
&+\sum_{\begin{subarray}{c}s_1,s_2,s_3\geq 1;\\ \sum_{i=1}^3s_i=s+2 \end{subarray}}\Big(X^*\Big(\prod_{i=1}^2(\mG_1^{s_i}\mathscr{P}_{0}^{qk})\Big)\mG_1^{s_3}\bv\Big)_{k}\bigg).
\end{align}

Next, for $\mathcal P$ defined in \eqref{rankrrepresentmP}, we collect its derivatives  $$ \frac{\partial^s \mP}{\partial x_{qk}^s}=\frac{\sn}{\Delta(d_i)} \mb{y}_0^* \Big(\sum_{a=1}^2 c_a
\frac{\partial^s \mG_1^{a}}{\partial x_{qk}^s}\Big) \mb{\vartheta}_0 + \frac{\sn}{\Delta(d_i)} \sum_{t=1}^2   \mb{y}_t^* \frac{\partial^s \mG_1}{\partial x_{qk}^s}\mb{\vartheta}_t.
$$ For the first derivative, we have
\begin{align}\label{eq:Pd1}
\frac{\partial \mP}{\partial x_{qk}}
&=-\frac{\sn}{\Delta(d_i)} \sum_{a=1}^2 c_a \sum_{\begin{subarray}{c}a_1,a_2\geq 1;\\ a_1+a_2=a+1 \end{subarray}}\Big((\mb{y}_0^*\mG_1^{a_1})_q(X^*\mG_1^{a_2}\mb{\vartheta}_0)_k+(X^*\mG_1^{a_1}\mb{y}_0)_k(\mG_1 ^{a_2}\mb{\vartheta}_0)_q\Big)\nonumber\\
&\quad - \frac{\sn}{\Delta(d_i)} \sum_{t=1}^2 \Big((\mb{y}_t^*\mG_1)_q(X^*\mG_1\mb{\vartheta}_{t})_k+(X^*\mG_1\mb{y}_t)_k(\mG_1\mb{\vartheta}_{t})_q\Big).
\end{align}
By \rf{weak_est_XG} and Remark \ref{boundrmk*}, it is easy to see that the second term on the RHS of \rf {eq:Pd1} can be estimated by 
\begin{align*}
\Big| \frac{\sn}{\Delta(d_i)} \sum_{t=1}^2 \Big((\mb{y}_t^*\mG_1)_q(X^*\mG_1\mb{\vartheta}_{t})_k+(X^*\mG_1\mb{y}_t)_k(\mG_1\mb{\vartheta}_{t})_q\Big) \Big| &\prec \Delta(d_i)^{-1}d_i^{-\frac 32}(d_i-\sqrt y)^{-\frac 12 } \nonumber\\
&\asymp 1, 
\end{align*}
where the last step follows from the estimate of $\Delta(d_i)$ in \rf{Delta:asymp}. For the first term on the RHS of  \rf {eq:Pd1}, first using  {\rf{weak_est_DG}}, We see
that 
\begin{align*}
\frac{\partial \mP}{\partial x_{qk}} &= - \frac{\sn}{\Delta(d_i)} \bigg[ 
\Big(\frac {c_1}{2} m_1+ c_2 m_1'\Big)(X^*\mG_1\mb{\vartheta}_{0})_k y_{0q}
\notag\\
&+ \Big(\frac {c_1}{2} (X^*\mG_1\mb{\vartheta}_{0})_k + c_2  (X^*\mG_1^2\mb{\vartheta}_{0})_k \Big)(\mG_1^*\mb{y}_0)_q+ \Big(\frac {c_1}{2} m_1+ c_2 m_1'\Big) (X^*\mG_1\mb{y}_{0})_k \vartheta_{0q} \nonumber\\
&
+ \Big(\frac {c_1}{2} (X^*\mG_1\mb{y}_{0})_k + c_2  (X^*\mG_1^2\mb{y}_{0})_k \Big) (\mG_1^*\mb{\vartheta}_0)_q
\bigg] +O_\prec(d_i^{\frac 12} \delta_{i0}^{-\frac 12}N^{-\frac 12}).
\end{align*}
Note that by plugging the values of $c_1,c_2$ and $m_1, m_1'$ at $z=\theta(d_i)$, we have
\begin{align*}
\frac {c_1}{2} m_1+ c_2 m_1'=-\frac{y(d_i+1)}{(d_i+y)^2} =O(d_i^{-1}).
\end{align*}
If $d_i$ is large, i.e. $d_i>K$ for sufficiently large constant $K>0$, performing the expansions for $\mG_1, \mG_1^2$ around  $-1/\theta(d_i)$ and $1/(\theta(d_i))^2$ respectively, we can further see the cancellation
\begin{align*}
&\quad\frac {c_1}{2} (X^*\mG_1\mb{\vartheta}_{0})_k + c_2  (X^*\mG_1^2\mb{\vartheta}_{0})_k\notag\\
& =O(d_i) \times \Big( (X^*\mG_1\mb{\vartheta}_{0})_{k}  +\theta(d_i) (X^*\mG_1^2\mb{\vartheta}_{0})_{k}  \Big)+O_\prec(N^{-\frac12}) \nonumber\\
 &= O(d_i) \times O_\prec\Big( \frac{(X^*H\mb{\vartheta}_{0})_{k}}{\theta(d_i)^2}\Big) = O_\prec(d_i^{-1}N^{-\frac12}).
\end{align*}
which leads to ${\partial \mP}/{\partial x_{qk}}=O_\prec(1)$. Here we used  the estimate 
$
(X^*H\mb{\vartheta}_{0})_{k} =O_\prec(N^{-\frac12}),
$ which can be easily verified by the moment method.
 For the case $d_i=O(1)$, simply using isotropic local laws \rf{weak_est_XG} and  Remark \ref{boundrmk*}, we  obtain ${\partial \mP}/{\partial x_{qk}}=O_\prec(1)$. To conclude, we always have
\begin{align}\label{19072410}
\frac{\partial \mP}{\partial x_{qk}}=O_\prec(1).
\end{align}

The second derivative of $\mP$ is
\begin{align}\label{eq:Pd2}
 \frac{\partial^2 \mP}{\partial x_{qk}^2}&=\frac{2\sn}{\Delta(d_i)} \sum_{b=1}^2 c_b \Big(\sum_{a_1,a_2=1}^2\sum_{\begin{subarray}{c}b_1,b_2,b_3\geq 1;\\ \sum_{i=1}^3b_i=b+2 \end{subarray}}\mb{y}_0^*\Big(\prod_{j=1}^2(\mG_1^{b_j}\mathscr{P}_{a_j}^{qk})\Big)\mG_1^{b_3}\mb{\vartheta}_0 \nonumber\\
 &\qquad \qquad\qquad -\sum_{\begin{subarray}{c}b_1,b_2\geq 1;\\ b_1+b_2=b+1 \end{subarray}}\mb{y}_0^*\mG_1^{b_1}\mathscr{P}_0^{qk}\mG_1^{b_2}\mb{\vartheta}_{0}\Big)\nonumber\\
 &+\frac{2\sn}{\Delta(d_i)} \sum_{t=1}^2  \Big(\sum_{a_1,a_2=1}^2 \mb{y}_t^*\Big(\prod_{j=1}^2(\mG_1\mathscr{P}_{a_j}^{qk})\Big)\mG_1\mb{\vartheta}_t - 
 \mb{y}_t^*\mG_1\mathscr{P}_0^{qk}\mG_1\mb{\vartheta}_{t}\Big).
\end{align}
Recalling the definition of $\mathscr{P}^{qk}_i$ for $i=0,1,2$ in \rf{P012}, one observes that all the terms in the parenthesis admit  one of the following forms 
\begin{align*}
&(X^*\mG_1^{b_1}\mb{\eta}_1)_k(X^*\mG_1^{b_2})_{kq}(\mG_1^{b_3}\mb{\eta}_2)_q, \qquad (X^*\mG_1^{b_1}\mb{\eta}_1)_k(\mG_1^{b_2})_{qq}(X^*\mG_1^{b_3}\mb{\eta}_2)_k, \notag\\
&
(\mG_1^{b_1}\mb{\eta}_1)_q(X^*\mG_1^{b_2}X)_{kk}^a(\mG_1^{b_3}\mb{\eta}_2)_q
\end{align*}
for $\mb{\eta}_1,\mb{\eta}_2=\mb{y}_0, \mb{y}_1,\mb{y}_2, \mb{\vartheta}_{0}, \mb{\vartheta}_{1}, \mb{\vartheta}_{2}$, $a=0,1$ and $b_1,b_2,b_3=1,2$ satisfying 
$b_1+b_3-1+a(b_2-1)=b \in \{1,2\}$, which are bounded by $O_\prec (N^{-1}(d_i-\sqrt y)^{-1}d_i^{-2})$, $O_\prec (N^{-1} (d_i-\sqrt y)^{-1}d_i^{-2})$, $O_\prec( d_i^{-2})$ respectively for $b=1$ and by $O_\prec (N^{-1}(d_i-\sqrt y)^{-3}d_i^{-1})$, $O_\prec (N^{-1} (d_i-\sqrt y)^{-3}d_i^{-1})$, $O_\prec( (d_i-\sqrt y)^{-1}d_i^{-2})$ respectively for $b=2$, in light of \rf{weak_est_XG} and Remark \ref{boundrmk*}. Therefore, combining with the prefactor $\sn$ and coefficients $c_{1,2}$ in (\ref{eq:Pd2}), we get 
\begin{align}
 \frac{\partial^2 \mP}{\partial x_{qk}^2}&= \frac{2\sn}{\Delta(d_i)} \bigg[
 c_1 \Big((\mG_1\mb{y}_0)_q(X^*\mG_1X)_{kk}(\mG_1\mb{\vartheta}_0)_q
 - (\mG_1\mb{y}_0)_q(\mG_1\mb{\vartheta}_0)_q\Big) \nonumber\\
 &\quad + c_2 \Big( (\mG_1^2\mb{y}_0)_q(X^*\mG_1X)_{kk}(\mG_1 \mb{\vartheta}_0)_q + 
 (\mG_1\mb{y}_0)_q(X^*\mG_1^2X)_{kk}(\mG_1\mb{\vartheta}_0)_q\nonumber\\
 &\quad
 +(\mG_1\mb{y}_0)_q(X^*\mG_1X)_{kk}(\mG_1^2\mb{\vartheta}_0)_q - (\mG_1^2\mb{y}_0)_q(\mG_1\mb{\vartheta}_0)_q- (\mG_1\mb{y}_0)_q(\mG_1^2\mb{\vartheta}_0)_q\Big)
 \bigg]  \nonumber\\
 &\quad
 + O_\prec (\sn\Delta(d_i)^{-1} d_i^{-2})\nonumber\\
  &= \frac{2\sn}{\Delta(d_i)} \bigg[
 c_1 \Big((\mG_1\mb{y}_0)_q(z\mG_2)_{kk}(\mG_1\mb{\vartheta}_0)_q
 \Big) + c_2 \Big( (\mG_1^2\mb{y}_0)_q(z\mG_2)_{kk}(\mG_1 \mb{\vartheta}_0)_q\notag\\
  &\quad + 
 (\mG_1\mb{y}_0)_q(\mG_2+z\mG_2^2)_{kk}(\mG_1\mb{\vartheta}_0)_q   +(\mG_1\mb{y}_0)_q(z\mG_2)_{kk}(\mG_1^2\mb{\vartheta}_0)_q\Big) \bigg]  \notag\\
 &\quad
 + O_\prec (\sn(d_i-\sqrt y)^{\frac 12} d_i^{-\frac 12})\nonumber\\
&= O_\prec \Big( \frac{2\sn}{\Delta(d_i)} \big(c_1m_1^2(zm_2)  + c_2(2m_1'm_1(zm_2)+ m_1^2(zm_2)')\big) \Big) \nonumber\\
&\quad + O_\prec (\sn(d_i-\sqrt y)^{\frac 12} d_i^{-\frac 12}) \nonumber\\
& = O_\prec (\sn(d_i-\sqrt y)^{\frac 12} d_i^{-\frac 12}) .  \label{19072411}
\end{align}
In the last step above, we used
$$c_1m_1^2(zm_2)  + c_2(2m_1'm_1(zm_2)+ m_1^2(zm_2)')=-\frac{y(d_i^2+2d_i+y)}{d_i^2(d_i+y)^2}= O(d_i^{-2}). $$

The third derivative of $\mP$ is
\begin{align}\label{eq:Pd3} 
\frac{\partial^3 \mP}{\partial x_{qk}^3}
&=-\frac{6\sn}{\Delta(d_i)} \sum_{b=1}^2 c_b  \bigg(\sum_{a_1,a_2,a_3=1}^2\sum_{\begin{subarray}{c}b_1,\cdots,b_4\geq 1;\\\sum_{i=1}^4 b_i=b+3 \end{subarray}}\Big(\mb{y}_0^*\Big(\prod_{j=1}^3(\mG_1^{b_j}\mathscr{P}_{a_j}^{qk})\Big)\mG_1^{b_4} \mb{\vartheta}_{0}\Big)\nonumber\\
&\qquad-\sum_{(a_1,a_2)\in \mathcal{O}_2}\sum_{\begin{subarray}{c}b_1,b_2,b_3\geq 1;\\ \sum_{i=1}^3b_i=b+2 \end{subarray}}\Big(\mb{y}_0^*\Big(\prod_{j=1}^2(\mG_1^{b_j}\mathscr{P}_{a_j}^{qk})\Big)\mG_1^{b_3}\mb{\vartheta}_{0}\Big)\bigg) \nonumber\\
&\quad -\frac{6\sn}{\Delta(d_i)} \sum_{t=1}^2  \bigg(\sum_{a_1,a_2,a_3=1}^2\Big(\mb{y}_t^*\Big(\prod_{j=1}^3(\mG_1\mathscr{P}_{a_j}^{qk})\Big)\mG_1 \mb{\vartheta}_{t}\Big) \nonumber\\
&\qquad
-\sum_{(a_1,a_2)\in \mathcal{O}_2}\Big(\mb{y}_t^*\Big(\prod_{j=1}^2(\mG_1\mathscr{P}_{a_j}^{qk})\Big)\mG_1\mb{\vartheta}_{t}\Big)\bigg).
\end{align}
By plugging in $\mathscr{P}^{qk}_a$ defined in (\ref{P012}), one can see that all summands  above contain either more than two quadratic forms of  $(X^*\mG_1^a)$  and  one quadratic form of $(\mG_1^b)$, or one  quadratic form of  $(X^*\mG_1^a)$  and  two quadratic forms of $(\mG_1^b)$ for some $a, b=1, 2,$  which  by \rf{weak_est_XG} and \rf{weak_est_DG} will contribute a $O_\prec (N^{-1/2} d_i^{-1/2+(a-1)}(d_i-\sqrt y)^{-1/2-2(a-1)})$ factor and a $O_\prec(d_i^{-1}(d_i-\sqrt y)^{-(a-1)})$ factor respectively. Using this fact together with contribution of $c_b/\Delta(d_i)$, one can easily obtain the crude bound 
\begin{align}
\Big|\frac{\partial^3 \mP}{\partial x_{qk}^3}\Big|\prec 1.  \label{eq:bdPd3}
\end{align}

The fourth derivative of $\mP$ is
\begin{align}\label{eq:Pd4}
 \frac{\partial^4 \mP}{\partial x_{qk}^4}
 &=\frac{4!\sn}{\Delta(d_i)} \sum_{b=1}^2 c_b \bigg[\sum_{a_1,\cdots,a_4=1}^2\sum_{\begin{subarray}{c}b_1,\cdots,b_5\geq 1;\\\sum_{i=1}^5 b_i=b+4 \end{subarray}}\Big(\mb{y}_b^*
\Big(\prod_{j=1}^4(\mG_1^{b_j}\mathscr{P}_{a_j}^{qk})\Big)\mG_1^{b_5}\mb{\vartheta}_{0}\Big)\nonumber\\
&\qquad-\sum_{(a_1,a_2,a_3)\in \mathcal{O}_3}\sum_{\begin{subarray}{c}b_1,\cdots,b_4\geq 1;\\\sum_{i=1}^4 b_i=b+3 \end{subarray}}\Big(\mb{y}_0^*\Big(\prod_{j=1}^3(\mG_1^{b_j}\mathscr{P}_{a_j}^{qk})\Big)\mG_1^{b_4}\mb{\vartheta}_{0}\Big)\notag\\
&\qquad +\sum_{\begin{subarray}{c}b_1,b_2,b_3\geq 1;\\ \sum_{i=1}^3b_i=b+2 \end{subarray}}\Big(\mb{y}_0^*\Big(\prod_{j=1}^2(\mG_1^{b_j}\mathscr{P}_{0}^{qk})\Big)\mG_1^{b_3}\mb{\vartheta}_{0}\Big)\bigg] \nonumber\\
&\quad +\frac{4!\sn}{\Delta(d_i)} \sum_{t=1}^2\bigg[ \sum_{a_1,\cdots,a_4=1}^2\Big(\mb{y}_t^*
\Big(\prod_{j=1}^4(\mG_1\mathscr{P}_{a_j}^{qk})\Big)\mG_1\mb{\vartheta}_{t}\Big)\notag\\
&\qquad
-\sum_{(a_1,a_2,a_3)\in \mathcal{O}_3}\Big(\mb{y}_t^*\Big(\prod_{j=1}^3(\mG_1\mathscr{P}_{a_j}^{qk})\Big)\mG_1\mb{\vartheta}_{t}\Big)+\Big(\mb{y}_t^*\Big(\prod_{j=1}^2(\mG_1\mathscr{P}_{0}^{qk})\Big)\mG_1\mb{\vartheta}_{t}\Big) \bigg]. 
\end{align}

\section{Proof of Proposition \ref{recursivemP}} \label{sec:prop}

This section is devoted to the proof of Proposition \ref{recursivemP}, the recursive moment estimates of $\mathcal{P}$ defined in (\ref{19071950}). The basic strategy is to use the cumulant expansion formula in Lemma \ref{cumulantexpansion} to the functionals of Green functions. In the context of Random Matrix Theory, such an idea dates back to \cite{KKP96}. We also refer to \cite{HLS, LS18, HK, HK2} for some recent applications of this strategy for other problems in Random Matrix Theory. 

First,  we provide estimates for some random terms that appear frequently in the proof. Indeed, they have  one of the following forms: $\mb{\eta}_1^*X^*\mG_1^s\mb{\eta}_2$, $\mb{\eta}_1^*X^*\mG_1^sX\mb{\eta}_2$, $\mb{\eta}_1^*\mG_1^s\mb{\eta}_2$, for any fixed $s\in \mathbb{N}$ and deterministic vectors $\mb{\eta}_1, \mb{\eta}_2$ whose $\ell^2$-norm are bounded by some constant $C>0$. Notice that under the choice of $z$ in (\ref{19072401}),
we have the deterministic bound
 \begin{align}\label{trivialbd_G1}
 |\mb{\eta}_1^*\mG_1^s\mb{\eta}_2|\leq C(\Im z)^{-s}.
 \end{align}
Similarly, by Cauchy-Schwarz inequality, we have 
 \begin{align}\label{trivialbd_XG1}
 |\mb{\eta}_1^*X^*\mG_1^s\mb{\eta}_2|&\leq  \|\mb{\eta}_1\|\|X^*\mG_1^s\mb{\eta}_2\|\leq C  \Big(\mb{\eta}_2^*\mG_1^s(\bar{z})XX^*\mG_1^s(z)\mb{\eta}_2 \Big)^{\frac12}\nonumber\\
 &= \Big(\mb{\eta}_2^*\mG_1^s(\bar{z})(I+z\mG_1)\mG_1^{s-1}(z)\mb{\eta}_2 \Big)^{\frac12}\leq C(\Im z)^{-s} 
  \end{align}
 and 
  \begin{align}\label{trivialbd_XG1X}
 |\mb{\eta}_1^*X^*\mG_1^sX\mb{\eta}_2|= |\mb{\eta}_1^*X^*X\mG_2^s\mb{\eta}_2|= |\mb{\eta}_1^*\mG_2^{s-1}\mb{\eta}_2+z\mb{\eta}_1^*\mG_2^s\mb{\eta}_2|\leq C(\Im z)^{-s}.
 \end{align}
According to Lemma \ref{prop_prec} (ii), the above deterministic bounds allow us to use the high probability bounds  for the aforementioned quantities (following from the isotropic local law) directly in the calculation of the expectations in (\ref{est_EmP}) and (\ref{est_EmPk}).

Next, we derive an initial bound on $\mP:=\mb{c}^*\mathbf{\chi}_\mathsf{I}(z)$ in \rf{19071950}, where $\mathbf{\chi}_\mathsf{I}$ is the random vector defined in \rf{def:chi_I} and $\mb{c}=(c_{01}, \; c_{02}, \; \{c_{1t}\}_{t\in \mathsf{I}}, \; \{c_{2j}\}_{j\in \mathsf{I}^c} )^*\in \mathbb{R}^{r+2}$ is an arbitrary fixed deterministic column vector of size $r+2$. For the case $d_i\leq K$ for sufficiently large constant $K>0$, one can easily get $|\mP|=O_\prec(1)$ by a direct application of the isotropic local law \rf{weak_est_DG}. As for the case $d_i>K$ for sufficiently large constant $K>0$, it is easy to see that all the components of $\mathbf{\chi}_\mathsf{I}(z)$ other than the first one are of  size $O_\prec(1)$ by applying the isotropic local law directly. Actually, as we mentioned earlier, there is a cancellation between two summands in the first component of $\mathbf{\chi}_\mathsf{I}(z)$. After exploiting this cancellation, similarly to the discussion in \rf{cancel:1}-\rf{2020072106}, we can also get an $O_\prec(1)$ bound for the first component of $\mathbf{\chi}_\mathsf{I}(z)$.
Hence, according to the above discussion, we  have 
\begin{align} \label{est:mP}
\mP=O_\prec(1).
\end{align}

With the above preliminary bounds at hand, we can now start formally the proof of Proposition \ref{recursivemP}. 
The main tool for the proof is the cumulant expansion formula in Lemma \ref{cumulantexpansion}. For convenience, we first introduce the following three column vectors
 \begin{align}
 \mb{y}_0:= c_{01}\bw_{\mathsf{I}}^0, \quad 
\mb{ y}_1:= \sum_{t\in \mathsf{I}} c_{1t} \bv_t , \quad
 \mb{y}_2:=c_{02}\mb{\varsigma}_{\mathsf{I}}^0+  \sum_{j\in \mathsf{I}^c}c_{2j}\bv_j
  \label{def:y012}
 \end{align}
 and define $$\mb{\vartheta}_{0}=\mb{\vartheta}_{2}=\bw_{\mathsf{I}}^0,\quad  \mb{\vartheta}_{1}=\mb{\varsigma}_{\mathsf{I}}^0.$$
Then we can rewrite
\begin{align}
\mP=\frac{\sn}{\Delta(d_i)}\mb{y}_0^*\big( \sum_{a=1,2} c_a(\mG_1^{(a-1)}-m_1^{(a-1)})\big)\mb{\vartheta}_{0}
+
\frac{\sn}{\Delta(d_i)}\sum_{t=1,2}\mb{y}_t^*\big(\mG_1-m_1\big)\mb{\vartheta}_{t}\label{rankrrepresentmP}
\end{align}
and 
\begin{align}
\mathbb{E}(\mP^l)
=\frac{\sn}{\Delta(d_i)} \mathbb{E} 
\Big( \mb{y}_0^*\big( \sum_{a=1,2} c_a(\mG_1^{(a-1)}-m_1^{(a-1)})\big)\mb{\vartheta}_{0}
+
\sum_{t=1,2}\mb{y}_t^*\big(\mG_1-m_1\big)\mb{\vartheta}_{t}
\Big)
 \mP^{l-1}. \label{19071960}
\end{align}
Using the identity
\begin{align*}
\mG_1^t=z^{-1}(H\mG_1^t-\mG_1^{t-1}), \qquad t=1,2, 
\end{align*}
we  rewrite  (\ref{19071960}) as
\begin{align*}
\mathbb{E}(\mP^l)=&\frac{\sn}{\Delta(d_i)} \mathbb{E} 
\bigg(
\mb{y}_0^* \Big( c_2\Big(\frac{1}{(1+m_2)z} H\mG_1^2 + \frac{m_2}{1+m_2} \mG_1^2 +\frac{(zm_2)'}{(1+m_2)z}\mG_1 -\frac{(zm_2)' + 1 }{(1+m_2)z}\mG_1 -m_1'\Big)\nonumber\\
&+ c_1 \Big( \frac{1}{(1+m_2)z} H\mG_1 + \frac{m_2}{1+m_2} \mG_1-  \frac{1}{(1+m_2)z} -m_1 \Big)
\Big) \mb{\vartheta}_{0}   \nonumber\\
&+ \sum_{t=1,2}\mb{y}_t^*   \Big( \frac{1}{(1+m_2)z} H\mG_1 + \frac{m_2}{1+m_2} \mG_1-  \frac{1}{(1+m_2)z} -m_1 \Big)  \mb{\vartheta}_{t}
\bigg)
\mP^{l-1}. 
\end{align*}
Repeatedly using the first and third identities in \rf{identitym1m2}, we have
\begin{align*}
\mathbb{E}(\mP^l)=&- \frac{\sn m_1}{\Delta(d_i)} \mathbb{E} 
\bigg(
\mb{y}_0^* \Big( c_2\Big( H\mG_1^2 + zm_2 \mG_1^2 +(zm_2)'\mG_1 \Big)\notag\\
&+ c_1 \Big(  H\mG_1 + zm_2 \mG_1 \Big) - c_2 \frac{m_1'}{m_1^2} (\mG_1-m_1)
\Big) \mb{\vartheta}_{0}   \nonumber\\
&+ \sum_{t=1,2}\mb{y}_t^*   \Big(  H\mG_1 +zm_2 \mG_1 \Big)  \mb{\vartheta}_{t}
\bigg)
\mP^{l-1},
\end{align*}
which can be further simplified to 
\begin{align} \label{2020050701}
\mathbb{E}(\mP^l)=&- \frac{\sn m_1}{\Delta(d_i)} \mathbb{E} 
\bigg(
\mb{y}_0^* \Big( c_2\Big( H\mG_1^2 + zm_2 \mG_1^2 +(zm_2)'\mG_1 \Big)\notag\\
&+\Big( c_1 + c_2 \frac{m_1'}{m_1} \Big)\Big(  H\mG_1 + zm_2 \mG_1 \Big) 
\Big) \mb{\vartheta}_{0}   \nonumber\\
&+ \sum_{t=1,2}\mb{y}_t^*   \Big(  H\mG_1 +zm_2 \mG_1 \Big)  \mb{\vartheta}_{t}
\bigg)
\mP^{l-1}. 
\end{align}
For the coefficient $c_1 + c_2 {m_1'}/{m_1} $ on the RHS of \eqref{2020050701}, by elementary computations, we see 
\begin{align*}
c_2+c_2\frac{m_1'(\theta(d_i))}{m_1(\theta(d_i))} = \wt c_1,
\end{align*}
where $\wt c_1$ is defined in \rf{def:wtc12} and $c_1,c_2$ are given in \eqref{def:c12}.  If we substitute  $z=\theta(d_i)+\ii N^{-K}$ into $m_1', m_1$ above, since the imaginary part can be taken arbitrarily small, we can get $c_1+c_2m_1'/m_1= \wt c_1 + O_\prec(N^{-\wt K})$ for some sufficient large $\wt K$. Thus in later estimates, we shall simply replace the coefficient $c_1+c_2m_1'/m_1$ in \eqref{2020050701} by $\wt c_1$ and the resulting error will be negligible by taking $K$ sufficiently large. Similarly, in the following, once we do the substitution to get the value of certain function of $z$, we simply take $z=\theta(d_i)$ for the computation, up to negligible error. 

To  ease the notation, we use $\wt c_2=c_2$ (see the definition in \rf{def:wtc12}) in \rf{2020050701}. For $t=0,1,2$, we introduce the following functions in terms of $\mG_1$  
 \begin{align}
 T_t(\mG_1)=
 \left\{
 \begin{array}{cl}
 \sum_{a=1}^2  \wt c_a \mG_1^a,  & \text{ if $t=0$;}\\
 \mG_1, &  \text{ if $t=1,2.$}
 \end{array}
 \right.   \label{def:Tt(mG)}
 \end{align}
With the above notations, we can simply write \rf{2020050701} as 
\begin{align}\label{rankrEmP^l}
\mathbb{E}(\mP^l)= & - \frac{\sn m_1}{\Delta(d_i)} \mbE \sum_{t=1}^3  \mb{y}_t^* H T_t(\mG_1)\mb{\vartheta}_{t}  \mP^{l-1} \nonumber\\
&-  \frac{\sn m_1}{\Delta(d_i)} \mbE \bigg( \mb{y}_0^* \Big(\wt c_2 (zm_2\mG_1^2+ (zm_2)'\mG_1 ) + \wt c_1 zm_2\mG_1 \Big) \mb{\vartheta}_{0}\notag\\
& \qquad+ \sum_{t=1,2}\mb{y}_t^*zm_2\mG_1 \mb{\vartheta}_{t}\bigg) \mP^{l-1}.
\end{align}

Now, we apply the cumulant expansion  to the terms $ - \frac{\sn m_1}{\Delta(d_i)} \mbE   \mb{y}_t^* H T_t(\mG_1)\mb{\vartheta}_{t}$ for $t=0,1,2$. For simplicity, we use the following shorthand notation for the summation
$$\sum_{q,k}=\sum_{q=1}^M \sum_{k=1}^N,$$
and similar shorthand notations are also used for single sum. 
By Lemma \ref{cumulantexpansion},  we have 
\begin{align}
&\frac{\sn m_1}{\Delta(d_i)} \mbE   \mb{y}_t^* H T_t(\mG_1)\mb{\vartheta}_{t} \mP^{l-1} = 
\frac{\sn m_1}{\Delta(d_i)}  \mbE \sum_{q,k} y_{tq}  \, x_{qk} (X^*T_t(\mG_1) \mb{\vartheta}_{t})_{k} \mP^{l-1}
\nonumber\\
&\qquad \qquad= \frac{\sn m_1}{\Delta(d_i)} \mathbb{E}\sum_{q,k}y_{tq}\sum_{\alpha=1}^3\frac{\kappa_{\alpha+1}(x_{qk})}{\alpha!}\frac{\partial^{\alpha}}{\partial x_{qk}^{\alpha}}\Big((X^*T_t(\mG_1) \mb{\vartheta}_{t} )_k \mP^{l-1}\Big)+\mathcal{R}_{t}, \label{19071971}
\end{align}
where $\mathcal{R}_{t}$ satisfies
\begin{align}
|\mathcal{R}_{t}|\leq &\sn \Big|\frac{m_1}{\Delta(d_i)} \Big|\sum_{q,k} \bigg(C\mathbb{E}(|x_{qk}|^{5}) \mathbb{E}\Big(\sup_{|x_{qk}|\leq c}\Big|y_{tq}  \frac{\partial^{4}}{\partial x_{qk}^{4}} \Big( (X^*T_t(\mG_1) \mb{\vartheta}_{t} )_k \mP^{l-1}\Big) \Big| \Big) \nonumber\\
&+C\mathbb{E}\big(|x_{qk}|^{5}\mathds{1}(|x_{qk}|>c)\big) \mathbb{E}\Big(\sup_{x_{qk}\in \mathbb{R}}\Big|y_{tq}\frac{\partial^{4}}{\partial x_{qk}^{4}}\Big((X^*T_t(\mG_1) \mb{\vartheta}_{t} )_k \mP^{l-1}\Big) \Big|\Big) \bigg) \nonumber
\end{align}
for any  constant $c>0$ and  some  constant $C>0$.

By the product rule, we have 
\begin{align}
\frac{\partial^{\alpha}}{\partial x_{qk}^{\alpha}}\Big((X^*T_t(\mG_1) \mb{\vartheta}_{t} )_k \mP^{l-1}\Big)= \sum_{\begin{subarray}{c}\alpha_1,\alpha_2\geq 0\\\alpha_1+\alpha_2=\alpha\end{subarray}} {\alpha \choose \alpha_1}\frac{\partial^{\alpha_1}(X^*T_t(\mG_1) \mb{\vartheta}_{t} )_k }{\partial x_{qk}^{\alpha_1}}\frac{\partial^{\alpha_2}\mP^{l-1}}{\partial x_{qk}^{\alpha_2}}. \label{19071970}
\end{align}
For $t=0,1,2$,  we set the notation
 \begin{align}
h_{t}(\alpha_1,\alpha_2):=\sn  \frac{m_1}{\Delta(d_i)} \sum_{q,k}y_{tq}\frac{\kappa_{\alpha_1+\alpha_2+1}(x_{qk})}{(\alpha_1+\alpha_2)!}\binom{\alpha_1+\alpha_2}{\alpha_1}\frac{\partial^{\alpha_1}(X^*T_t(\mG_1) \mb{\vartheta}_{t} )_k }{\partial x_{qk}^{\alpha_1}}\frac{\partial^{\alpha_2}\mP^{l-1}}{\partial x_{qk}^{\alpha_2}}. \label{rankrdefhts}
\end{align}
Note that $h_{t}(\alpha_1,\alpha_2)$  depends on   $l$ and $i$. However, we drop this dependence for brevity.

Using (\ref{19071971})-(\ref{rankrdefhts}), we can now write
\begin{align}
\frac{\sn m_1}{\Delta(d_i)} \mbE   \mb{y}_t^* H T_t(\mG_1)\mb{\vartheta}_{t} \mP^{l-1} 
&=\sum_{\begin{subarray}{c}\alpha_1,\alpha_2\geq 0\\1\leq \alpha_1+\alpha_2\leq 3\end{subarray}}\mathbb{E}h_{t}(\alpha_1,\alpha_2)+\mathcal{R}_{t} \label{rankrcumuexpansion^t}.
\end{align} 

In the sequel, we estimate $h_{t}(\alpha_1,\alpha_2)$  and the remainder terms $\mathcal R_{t}$ for $t=0, 1, 2$.   We collect the estimates  in the following lemma,  whose proof will be postponed to the end of this section. Before we state the lemma, let us first recall the shorthand notation $\delta_{i0}=d_i-\sqrt y$.
 \begin{lem}\label{rankrlemh_ts} Let $l$ in (\ref{rankrdefhts}) be any fixed positive integer. 
With the convention $m_a^{(-1)}/(-1)!=1$ for $a=1,2$, we have the following estimates on $h_{t}(\alpha_1,\alpha_2)$  and $\mathcal{R}_{t}$ where $t=0,1,2$. 
\begin{itemize}
\item [(1):] For $h_{t}(\alpha_1,\alpha_2)$, the non-negligible terms are
\begingroup
\allowdisplaybreaks
\begin{align}
&\hspace{-3ex}h_0(1,0)=- \frac{\sn m_1}{\Delta(d_i)}   \Big(\wt  c_1(\mb{y}_0^*\mG_1\mb{\vartheta}_{0}) (zm_2)+ \wt c_2(\mb{y}_0^*\mG_1\mb{\vartheta}_{0}) (zm_2)' \notag\\
&\qquad\qquad\qquad+ \wt c_2(\mb{y}_0^*\mG_1^2\mb{\vartheta}_{0}) (zm_2)\Big) \mP^{l-1}
+O_\prec\Big({d_i^{\frac 32} \delta_{i0}^{-\frac32}} N^{-\frac 12} \Big), \label{rankrest:h0(1,0)}
\\
&\hspace{-3ex}h_0(0,1)= - \frac{(l-1)m_1}{\Delta(d_i)^2} \bigg(\sum_{a,b =1}^2 c_a \wt c_b \bigg(\sum_{\begin{subarray} {c} a_1,a_2\geq 1,\\ a_1+a_2=a+1
\end{subarray}}  \frac{m_1^{(a_1-1)} (zm_1)^{(b+a_2-1)}}{(a_1-1)! (b+a_2-1)!}\bigg)\notag\\
&\qquad\qquad\qquad\times \Big((\mb{y}_0^*\mb{y}_0) (\mb{\vartheta}_0^*\mb{\vartheta}_0) +  (\mb{\vartheta}_0^*\mb{y}_0)^2\Big)  \nonumber\\
&\qquad \quad+ \sum_{s=1}^2 m_1 \Big( \sum_{b=1}^2\wt c_b \frac{(zm_1)^{(b)}}{b!}\Big)
 \Big( (\mb{y}_0^*\mb{y}_s) (\mb{\vartheta}_0^*\mb{\vartheta}_{s})+(\mb{y}_0^*\mb{\vartheta}_{s})(\mb{\vartheta}_0^*\mb{y}_s)\Big)
\bigg)\mP^{l-2} \notag\\
&\qquad \quad+O_\prec(d_i^{\frac 32} \delta_{i0}^{-\frac 32} N^{-\frac12}), \label{rankresth0(0,1)}
\\
&\hspace{-3ex}h_0(1,2)=- \frac{\kappa_4 (l-1) m_1}{\Delta(d_i)^2} \sum_{s=1}^2 \wt c_s\bigg(\frac{(zm_2m_1)^{(s-1)}}{(s-1)!}\bigg) 
\bigg(\sum_{b=1}^2 c_b \frac{(zm_2m_1^2)^{(b-1)}}{(b-1)!}
 \mathbf{s}_{2,2}(\mb{y}_{0},\mb{ \vartheta}_{0})\nonumber\\
&\qquad  \quad+\sum_{s=1}^2 (zm_2m_1^2)   \mathbf{s}_{1,1,1,1}(\mb{y}_{0},\mb{ \vartheta}_{0}, \mb{y}_{s},\mb{ \vartheta}_{s})
 \bigg)\mP^{l-2}  + O_\prec (N^{-\frac 12}),  \label{rankresth0(1,2)}
  \end{align}
 \endgroup 
 and for $t=1,2$, 
\begingroup
\allowdisplaybreaks 
 \begin{align}
h_t(1,0)&=- \frac{\sn m_1}{\Delta(d_i)} \mb{y}_t^* \mG_1 \mb{\vartheta}_t (zm_2) \mP^{l-1} + O_\prec\Big({d_i^{\frac 32} \delta_{i0}^{-\frac32}} N^{-\frac 12} \Big), \label{rankrest:ht(1,0)}
 \\
h_t(0,1) &= - \frac{(l-1)}{\Delta(d_i)^2} \bigg(\sum_{a=1}^2 c_a \frac{\big((zm_1)'m_1^2\big)^{(a-1)}}{a!}   \Big((\mb{y}_t^*\mb{y}_0)(\mb{\vartheta}_t^*\mb{\vartheta}_0) + (\mb{y}_t^*\mb{\vartheta}_0)(\mb{\vartheta}_t^*\mb{y}_0) \Big)\nonumber\\
&\qquad + \sum_{s=1}^2  m_1^2 (zm_1)'  \Big((\mb{y}_t^*\mb{y}_s) ( \mb{\vartheta}_t^*\mb{\vartheta}_{s})+
(\mb{y}_t^*\mb{\vartheta}_{s}) (\mb{\vartheta}_t^*\mb{y}_s)\Big)
\bigg)\mP^{l-2}\notag\\
&\qquad +O_\prec(d_i^{\frac 32} \delta_{i0}^{-\frac 32} N^{-\frac12}),
\label{rankrestht(0,1)}
\\
h_{t}(1,2)&=-\frac{\kappa_4(l-1)}{\Delta(d_i)^2} (zm_2m_1^2)
 \bigg(\sum_{b=1}^2 c_b \frac{(zm_2m_1^2)^{(b-1)}}{(b-1)!}
\mathbf{s}_{1,1,1,1}(\mb{y}_{t},\mb{ \vartheta}_{t}, \mb{y}_{0},\mb{ \vartheta}_{0}) \nonumber\\
 &\qquad +\sum_{s=1}^2 (zm_2m_1^2) \mathbf{s}_{1,1,1,1}(\mb{y}_{t},\mb{ \vartheta}_{t}, \mb{y}_{s},\mb{ \vartheta}_{s})
\bigg)\mP^{l-2}  + O_\prec (N^{-\frac 12}). \label{rankrestht(1,2)}
\end{align}
\endgroup

\vspace{1ex}
\item[(2):]
Except for the above terms,
 all the other  $h_{t}(\alpha_1, \alpha_2)$ terms with $\alpha_1,\alpha_2\geq 0$ and $\alpha_1+\alpha_2\le 3$ can be bounded by $O_\prec(N^{-\frac12})$.

\vspace{1ex}
\item [(3):] For the remainder terms, we have 
\begin{align}
\mathcal{R}_{t}=O_\prec(N^{-\frac12}). \label{19071980}
\end{align} 
\end{itemize}
\end{lem}

Now we show the proof of Proposition \ref{recursivemP},  based on Lemma \ref{rankrlemh_ts}.
\begin{proof}[Proof of Proposition \ref{recursivemP}]
First we show the proof of \eqref{est_EmP}. 
Using  Lemma \ref{rankrlemh_ts}  with $l=1$, we can rewrite \eqref{rankrcumuexpansion^t} as
\begin{align}
\sn\frac{m_1}{\Delta(d_i)} \mathbb{E}\mb{y}_t^*HT_t(\mG_1)\mb{\vartheta}_{t}=\mathbb{E}h_{t}(1,0)+O_\prec(N^{-\frac12}). \nonumber
\end{align}
Plugging  the above estimate into \eqref{rankrEmP^l} with $l=1$, we obtain 
\begin{align*}
&\mathbb{E}\mP=-\bigg(\mathbb{E}h_{0}(1,0)+\frac{\sn m_1}{\Delta(d_i)}\mathbb{E} \Big(\wt  c_1(\mb{y}_0^*\mG_1\mb{\vartheta}_{0}) zm_2 + \wt c_2(\mb{y}_0^*\mG_1\mb{\vartheta}_{0}) (zm_2)' + \wt c_2(\mb{y}_0^*\mG_1^2\mb{\vartheta}_{0}) zm_2\Big) \nonumber\\
&\qquad+\mathbb{E}h_{1}(1,0)+\sn \frac{zm_2m_1}{\Delta(d_i)}\mathbb{E}\mb{y}_1^*\mG_1\mb{\vartheta}_{1}+ \mathbb{E}h_{2}(1,0)+\sn  \frac{zm_2m_1}{\Delta(d_i)}\mathbb{E}\mb{y}_2^*\mG_1\mb{\vartheta}_{2}\bigg) \notag\\
&\qquad +O_\prec(N^{-\frac12} ).
\end{align*}
We substitute the estimates  \eqref{rankrest:h0(1,0)} and \eqref{rankrest:ht(1,0)}  with $l=1$ into the above estimate and immediately get
\begin{align}
\mathbb{E}\mP=O_\prec(N^{-\frac12} \delta_{i0}^{-\frac 32}d_i^{\frac 32}). \label{19071975}
\end{align}
This proves (\ref{est_EmP}).

Next we turn to prove \eqref{est_EmPk}. By (\ref{rankrcumuexpansion^t}) and  Lemma \ref{rankrlemh_ts}, 
 we observe that 
\begin{align}
\sn \frac{m_1}{\Delta(d_i)}\mathbb{E}\mb{y}_t^*HT_t(\mG_1) \mb{\vartheta}_{t} \mP^{l-1}
=\mathbb{E} h_{t}(1,0) + \mathbb{E} h_{t}(0,1) + \mathbb{E} h_{t}(1,2) +O_\prec(N^{-\frac12} ) \label{19072402}
\end{align}
for $t=0,1,2$. Plugging \rf{19072402} into \eqref{rankrEmP^l}, together with the estimates \eqref{rankrest:h0(1,0)} and \eqref{rankrest:ht(1,0)}, we get 
\begin{align}
\mathbb{E}(\mP^l)=&-\Big(\mathbb{E}h_{0}(0,1)+\mathbb{E}h_{0}(1,2)+\mathbb{E}h_{1}(0,1)+\mathbb{E}h_{1}(1,2)+ \mathbb{E} h_{2}(0,1) + \mathbb{E} h_{2}(1,2) \Big) \nonumber\\
& + O_\prec \big(N^{-\frac12} \delta_{i0}^{-\frac 32}d_i^{\frac 32} \big). \label{19071983}
\end{align}
It remains to compute the explicit expression for the RHS of the above equation. First, using \rf{rankresth0(0,1)} and \rf{rankrestht(0,1)}, we get
\begingroup
\allowdisplaybreaks
\begin{align*}
&-\Big(\mathbb{E}h_{0}(0,1)+\mathbb{E} h_{1}(0,1) + \mathbb{E} h_{2}(0,1) \Big)\nonumber\\
&=\frac{l-1}{\Delta(d_i)^2}\bigg( \Big(\Vert {\mb{y}}_0 \Vert^2 +({\mb{y}}_0^* \bw_{\mathsf{I}}^0)^2\Big) \sum_{a,b =1}^2 c_a \wt c_b  \sum_{\begin{subarray} {c} a_1,a_2\geq 1,\\ a_1+a_2=a+1
\end{subarray}}  \frac{m_1m_1^{(a_1-1)} (zm_1)^{(b+a_2-1)}}{(a_1-1)! (b+a_2-1)!}\notag\\
&\qquad+
\Big(\Vert {\mb{y}}_1 \Vert^2 +( {\mb{y}}_1^* \mb{\varsigma}_{\mathsf{I}}^0)^2\Big)  \Big(m_1^2(zm_1)'\Big)
\\
&\quad +\Big( {\mb{y}}_1^*{\mb{y}}_2 (\mb{\varsigma}_{\mathsf{I}}^0)^* \bw_{\mathsf{I}}^0+  {\mb{y}}_1^*\bw_{\mathsf{I}} ^0 {\mb{y}}_2^*\mb{\varsigma}_{\mathsf{I}}^0\Big) 
\Big(2m_1^2(zm_1)'\Big)+\Big(\Vert {\mb{y}}_2\Vert^2 +( {\mb{y}}_2^* \mb{w}_\mathsf{I}^0)^2\Big)  \Big(m_1^2(zm_1)'\Big)\\
&\quad +\Big( {\mb{y}}_0^*{\mb{y}}_1(\mb{w}_\mathsf{I}^0)^*\mb{\varsigma}_{\mathsf{I}}^0 +  {\mb{y}}_0^*\mb{\varsigma}_{\mathsf{I}}^0 \mb{y}_1^*{\mb{w}}_I^0 + {\mb{y}}_0^*{\mb{y}}_2\Vert \mb{w}_\mathsf{I}^0\Vert ^2 +  {\mb{y}}_0^*\mb{w}_\mathsf{I}^0 \mb{y}_2^*{\mb{w}}_I^0\Big)\nonumber\\
&\qquad\times\bigg(\sum_{a=1}^2 c_a \frac{\big((zm_1)'m_1^2\big)^{(a-1)}}{a!} + \sum_{a=1}^2 \wt c_a \frac{m_1^2(zm_1)^{(a)}}{a!}\bigg)
\ \bigg) \mathbb{E} \mathcal P^{l-2} + O_\prec(d_i^{\frac 32} \delta_{i0}^{-\frac 32} N^{-\frac12}).
\end{align*}
Recall the definitions for $\mb{y}_{0}, \mb{y}_1$ and $\mb{y}_2$ in \rf{def:y012} and the matrix $\mathcal M_\mathsf{I}$ in \eqref{def:Mdiag} and \eqref{def:Moff}.  By elementary calculation, we arrive at
\begin{align}\label{ViI}
&-\Big(\mathbb{E}h_{0}(0,1)+\mathbb{E} h_{1}(0,1)+ \mathbb{E} h_{2}(0,1))  = (l-1) \mathbf{c}^* \mathcal M_\mathsf{I} \mathbf{c} \mathbb{E} \mathcal P^{l-2}+ O_\prec(d_i^{\frac 12} \delta_{i0}^{-\frac 12} N^{-\frac12}).
\end{align}
Next, by \eqref{rankresth0(1,2)} and \eqref{rankrestht(1,2)} , we have
\begin{align*}
&-\big(\mathbb{E}h_{0}(1,2)+\mathbb{E}h_{1}(1,2)+ \mathbb{E}h_{2}(1,2)\big) \\&=\frac{(l-1) \kappa_4}{\Delta(d_i)^2} \bigg( \mathbf{s}_{2,2}({\mb{y}}_0, \bw_{\mathsf{I}}^0) \Big( \sum_{s=1}^2 \wt c_s \frac{m_1(zm_2m_1)^{(s-1)}}{(s-1)!}\Big) 
\Big(  \sum_{s=1}^2  c_s \frac{(zm_2m_1^2)^{(s-1)}}{(s-1)!}\Big)\\
&+\Big( \mathbf{s}_{1,1,1,1}({\mb{y}}_0, \mathbf{w}_I^0,{\mb{y}}_1, \mathbf{\varsigma}_I^0 ) +   \mathbf{s}_{1,1,2}({\mb{y}}_0,{\mb{y}}_2, \mathbf{w}_I^0 )\Big)
zm_2m_1^2\notag\\
&\qquad\times\Big(  \sum_{s=1}^2 \wt c_s \frac{m_1(zm_2m_1)^{(s-1)}}{(s-1)!}+  \sum_{s=1}^2  c_s \frac{(zm_2m_1^2)^{(s-1)}}{(s-1)!}\Big)\\
&  +  \Big(\mathbf{s}_{2,2}({\mb{y}}_1, \mathbf{\varsigma}_I^0)+ \mathbf{s}_{2,2}({\mb{y}}_2, \mathbf{w}_I^0) \Big)(zm_2m_1^2)^2 + \mathbf{s}_{1,1,1,1}({\mb{y}}_1, \mathbf{\varsigma}_I^0,{\mb{y}}_2, \mathbf{w}_I^0 ) \Big(2(zm_2m_1^2)^2\Big)
\bigg) \mathbb{E} \mathcal P^{l-2},
\end{align*}
\endgroup
which, by the definitions of $\mb{y}_{0}, \mb{y}_1$ and $\mb{y}_2$ in \rf{def:y012} and the matrix $\mathcal K_\mathsf{I}$ in \eqref{def_Klowdiag}-\eqref{def_Kothers}, can be simplified to
\begin{align}\label{ViJ}
&-\big(\mathbb{E}h_{0}(1,2)+\mathbb{E}h_{1}(1,2)+ \mathbb{E}h_{2}(1,2)\big) 
=(l-1) \kappa_4 \mathbf{c}^* \mathcal K_\mathsf{I}\mathbf{c} \mathbb{E} \mathcal P^{l-2}.
\end{align}
Combining \rf{ViI} and $\rf{ViJ}$, we complete the proof of  \eqref{est_EmPk} in Proposition \ref{recursivemP}. Hence, we conclude the proof of  Proposition \ref{recursivemP}. 
\end{proof}

The rest of this section is devoted to the proof of Lemma \ref{rankrlemh_ts}. It is convenient to first introduce the next lemma, which will be used  to control the negligible terms in the proof of Lemma \ref{rankrlemh_ts}.

\begin{lem}\label{Importantests}
For a fixed integer $n\geq 1$, let $\mb{\eta}_i=(\eta_{i1},\ldots, \eta_{iM})^*\in \mathbb{C}^M, i\in \lb 0, n\rb$ be any given deterministic vectors with $\max_i\|\mathbf{\eta}_i\|\leq C$ for some positive constant $C$.  For positive integers $s_0,s_1,\cdots,s_n\in \{1,2\}$ and $a=0,1$, $t=0,1,2$, we have the following estimates:
\begin{align}
&\sum_{q}\Big|\eta_{0q}\big(\mG_1^{s_0})_{qq}\big)^a\Big(\prod_{k=1}^n(\mG_1^{s_k}\mb{\eta}_k)_q\Big)\Big|=O_\prec\Big(d_i^{-n-a}(d_i-\sqrt y)^{-\sum_{k=1}^n (s_k-1)- a(s_0-1)}\Big) \label{Importestsum1},\\
&\Big|\sum_{k}\big((X^*\mG_1^{s_0}X)_{kk}\big)^a(X^*T_t(\mG_1)\mb{\eta}_1)_{k}\Big|=O_\prec \Big((d_i-\sqrt y)^{-\frac 12-a(s_0-1)}d_i^{-\frac 12-a} \Big). \label{Importestsum2}
\end{align}
\end{lem}
\begin{proof}[Proof of Lemma \ref{Importantests}]
The first estimate \rf{Importestsum1} can be proved by \eqref{eq:basicG} and the isotropic local law \eqref{weak_est_DG} as follows: 
\begingroup
\allowdisplaybreaks
\begin{align*}
&\sum_{q}\Big|\eta_{0q}\big((\mG_1^{s_0})_{qq}\big)^a\Big(\prod_{k=1}^n((\mG_1^{s_k}\mb{\eta}_k)_q\Big)\Big|\\
&=\sum_{q}\Big|\eta_{0q}\big((\mG_1^{s_0})_{qq}\big)^a\Big(\prod_{k=1}^n\Big\langle \mathbf{e}_q, \frac{\mG_1^{(s_k-1)}}{(s_k-1)!}\mb{\eta}_k \Big\rangle\Big)\Big|\\
&=\sum_q\bigg|\eta_{0q}\bigg(\frac{m_1^{(s_0-1)}}{(s_0-1)!}+O_\prec\Big(\frac{\Delta(d_i)}{\Upsilon^{s_0-1}\sqrt{N}}\Big)\bigg)^a \prod_{k=1}^n\bigg(\frac{m_1^{(s_k-1)}}{(s_k-1)!} \eta_{kq}+O_\prec\Big(\frac{\Delta(d_i)}{\Upsilon^{s_k -1}\sqrt{N}}\Big)\bigg)\bigg|\\
&=O_\prec\Big(d_i^{-n-a}(d_i-\sqrt y)^{-\sum_{k=1}^n (s_k-1)- a(s_0-1)}\Big),
\end{align*}
\endgroup
where in the last step we used the first and  last estimates in \rf{estm1m2*} and the fact $|z|\asymp |\theta(d_i)|\asymp d_i$.

f
Recall  the  definition of $T_t(\mG_1)$ from (\ref{def:Tt(mG)}). To prove \rf{Importestsum2}, we first claim that for $t=0,1,2$,  
\begin{align} \label{est:sumXTG}
\Big|\sum_{k}(X^*T_t(\mG_1)\mb{\eta}_1)_{k}\Big|= O_\prec\big((d_i-\sqrt y)^{-\frac 12}d_i^{-\frac 12}\big).
\end{align}
Notice that for $t=1,2$, by \rf{weak_est_XG},
\begin{align*}
\Big|\sum_{k}(X^*T_t(\mG_1)\mb{\eta}_1)_{k}\Big|=\Big|\sum_{k}(X^*\mG_1\mb{\eta}_1)_{k}\Big|=\sn \mb{1}_N^*X^*\mG_1\mb{\eta}_1=O_\prec\big((d_i-\sqrt y)^{-\frac 12}d_i^{-\frac 12}\big),
\end{align*}
where we recalled the notation $\sqrt{N}\mb{1}_N\in \mathbb{R}^N$ for the all-1 vector. For $t=0$, by definition, $T_0(\mG_1)= \wt c_1\mG_1 + \wt c_2 \mG_1^2$ where $\wt c_{1,2}$ are defined in \rf{def:wtc12}. If $d_i\leq K$ with sufficiently large constant $K>0$, then similarly to the above estimate, we have 
\begin{align*}
\Big|\sum_{k}(X^*T_0(\mG_1)\mb{\eta}_1)_{k}\Big| &\leq  \wt c_1 \sn \mb{1}_N^*X^*\mG_1\mb{\eta}_1 + \wt c_2 \sn \mb{1}_N^*X^*\mG_1^2\mb{\eta}_1\nonumber\\
&= O_\prec\big((d_i-\sqrt y)^{-\frac 12}d_i^{-\frac 12}\big).
\end{align*} 
However, in the case that $d_i>K$ with sufficiently large constant $K>0$, $\wt c_{1,2}$ are no longer bounded. Instead, we do the expansion for $\mG_1, \mG_1^2$ around $-1/\theta(d_i)$ and $1/(\theta(d_i))^2$, respectively,  to get a cancellation. More specifically, 
\begin{align*}
\Big|\sum_{k}(X^*T_0(\mG_1)\mb{\eta}_1)_{k}\Big| &=  \sn \mb{1}_N^*X^*(\theta(d_i)\mG_1+ \theta(d_i)^2 \mG_1^2)\mb{\eta}_1 + O_\prec\big(d_i^{-1}\big) \nonumber\\
&= \sn \frac { \mb{1}_N^*X^*H\mb{\eta}_1 }{ \theta(d_i)}+ O_\prec\big(d_i^{-1}\big)= O_\prec\big(d_i^{-1}\big).
\end{align*}
In the last two steps above, we used 
\begin{align}\label{est:from MoM}
\mb{1}_N^*X^*H^a\mb{\eta}_1 =  O_\prec(N^{-\frac 12}) \quad \text{for any $a\in \mathbb{Z}^+$,}
\end{align}  
which can be checked easily by the moment method.

Hence, we conclude the proof of \rf{est:sumXTG}.
Further,  by \rf{relationXG} and the isotropic local law Theorem \ref{isotropic}, we get
\begin{align*}
&\quad \Big|\sum_{k}\big((X^*\mG_1^{s_0}X)_{kk}\big)^a(X^*T_t(\mG_1)\mb{\eta}_1)_{k}\Big|=\Big|\sum_{k}\big((\mG_2^{s_0-1}+z \mG_2^{s_0} )_{kk}\big)^a(X^*T_t(\mG_1)\mb{\eta}_1)_{k}\Big|\\
&=\Big|\sum_{k}\Big(\frac{m_2^{(s_0-2)}}{(s_0-2)!}+\frac{zm_2^{(s_0-1)}}{(s_0-1)!}+O_\prec\big(N^{-\frac{1}{2}} d_i^{s_0}\delta_{i0}^{-2(s_0-1)}\Delta(d_i) \big)\Big)^a ((X^*T_t(\mG_1)\mb{\eta}_1)_{k}\Big|\\
&=\Big|\sum_{k}\Big(\frac{(1+zm_2)^{(s_0-1)}}{(s_0-1)!}+O_\prec\big(N^{-\frac{1}{2}} d_i^{s_0}\delta_{i0}^{-2(s_0-1)}\Delta(d_i) \big)\Big)^a ((X^*T_t(\mG_1)\mb{\eta}_1)_{k}\Big|\\
&=O_\prec \Big((d_i-\sqrt y)^{-\frac 12-a(s_0-1)}d_i^{-\frac 12-a} \Big),\quad a=0,1,
\end{align*}
by using   the convention $m_2^{(-1)}/(-1)!=1$, the product rule and also noticing that $1+zm_2= O(1/d_i)$ and $(zm_2)'= O(d_i^{-1}(d_i-\sqrt y)^{-1})$ 
  from \rf{estm1m2*} and \rf{estm1m2**}. This concludes the proof of Lemma \ref{Importantests}.  
\end{proof}

With Lemma \ref{Importantests}, we can now prove Lemma \ref{rankrlemh_ts}.
\begin{proof} [Proof of Lemma \ref{rankrlemh_ts}] In this proof, we fix $l$ and $i$. First, we compute $h_{t}(1,0)$ for $t=0,1,2$. Recall the definition in \rf{rankrdefhts} and the expression of $T_t(\mG_1)$ in \rf{def:Tt(mG)}.  We need to compute separately for $t=0$ and $t=1,2$. By \rf{deriXG^l}, we have

\begin{align*}
h_0(1,0) &= \frac{ m_1}{\Delta(d_i) \sn } \sum_{q,k} y_{0q} \frac{\partial \Big(X^* (\sum_{s=1}^2 \wt c_s \mG_1^s) \mb{\vartheta}_0 \Big)_{k}}{\partial x_{qk}}\,\mP^{l-1}\nonumber\\
& =  \frac{ m_1}{\Delta(d_i) \sn } \sum_{q,k} y_{0q} \bigg( \sum_{s=1}^2 \wt c_s \Big((\mG_1^{s}\mb{\vartheta}_0 )_{q} -  \sum_{a=1}^2 \sum_{\begin{subarray}{c}s_1,s_2\geq 1\\ s_1+s_2=s+1 \end{subarray}}\big( X^*\mG_1^{s_1}  \mathscr{P}_a^{qk}\mG_1^{s_2} \mb{\vartheta}_0\big)_k \Big) \bigg)\mP^{l-1},
\end{align*}
where $\mathscr{P}_a^{qk}, a=1,2$ are defined in \rf{P012}. Taking the sum over $q,k$, we get 
\begin{align}
h_{0}(1,0)
&=\frac{\sqrt N m_1}{\Delta(d_i)}  \sum_{s=1}^2 \wt c_s  \mb{y}_0^*\mG_1^s \mb{\vartheta}_{0}\mP^{l-1} \nonumber\\
&-  \frac{m_1}{\sn\Delta(d_i)}   \sum_{s=1}^2 \wt c_s    \sum_{\begin{subarray}{c}s_1,s_2\geq 1\\s_1+s_2=s+1\end{subarray}} \Big(\mb{y}_0^* \mG_1^{s_1} X X^* \mG_1^{s_2} \mb{\vartheta}_{0}+ \mb{y}_0^* \mG_1^{s_2} \mb{\vartheta}_{0} \; \text{Tr}(X^* \mG_1^{s_1} X)\Big)\mP^{l-1}\label{19072011}.
\end{align}
By the identity in \eqref{relationXG} and the explicit expression for $\wt c_{1,2}$ in \rf{def:wtc12}, we get 
\begin{align*}
& \frac{m_1}{\sn\Delta(d_i)}  \sum_{s=1}^2 \wt c_s   \sum_{\begin{subarray}{c}s_1,s_2\geq 1\\s_1+s_2=s+1\end{subarray}} \mb{y}_t^* \mG_1^{s_1} X X^* \mG_1^{s_2} \mb{\vartheta}_{0} \nonumber\\
 &= \frac{m_1}{\sn\Delta(d_i)}  \Big(\wt c_1(\mb{y}_0^*\mG_1XX^*\mG_1\mb{\vartheta}_{0})+ 2 \wt c_2 (\mb{y}_0^*\mG_1XX^*\mG_1^2\mb{\vartheta}_{0})\Big) \nonumber\\
 &= \frac{m_1}{\sn\Delta(d_i)} \Big( \frac{d_i^2+2d_iy+y}{d_i+y}(\mb{y}_0^* (\mG_1+ z\mG_1')\mb{\vartheta}_{0})+ \frac{(d_i^2-y)^2}{d_i^2}(\mb{y}_0^* (2\mG_1'+z\mG_1'')  \mb{\vartheta}_{0} )\Big)  
 ,
\end{align*}
which, by isotropic local law \rf{weak_est_DG}, can be estimated by
\begin{align} \label{19072010}
& \frac{m_1}{\sn\Delta(d_i)}  \sum_{s=1}^2 \wt c_s   \sum_{\begin{subarray}{c}s_1,s_2\geq 1\\s_1+s_2=s+1\end{subarray}} \mb{y}_t^* \mG_1^{s_1} X X^* \mG_1^{s_2} \mb{\vartheta}_{0}  \nonumber\\
&=  \frac{m_1}{\sn\Delta(d_i)} \bigg( \frac{d_i^2+2d_iy+y}{d_i+y} \bigg((zm_1)'  \mb{y}_0^* \mb{\vartheta}_{0}+ O_\prec \Big(\frac{\Delta(d_i)d_i^2}{ \delta_{i0}^2 \sn } \Big) \bigg)\notag\\
&\qquad+ \frac{(d_i^2-y)^2}{d_i^2} \bigg((zm_1)'' \mb{y}_0^* \mb{\vartheta}_{0} + O_\prec \Big(\frac{\Delta(d_i)d_i^3}{ \delta_{i0}^4 \sn }\Big)  \bigg)\bigg)
\nonumber\\
&=\frac{m_1}{\sn\Delta(d_i)} \bigg( \frac{d_i^2+2d_iy+y}{d_i+y}(zm_1)' + \frac{(d_i^2-y)^2}{d_i^2} (zm_1)'' \bigg) \mb{y}_0^* \mb{\vartheta}_{0} + O_\prec\Big( \frac{d_i^2}{\delta_{i0}^2 N}\Big) \nonumber\\
&=\frac{m_1}{\sn\Delta(d_i)}  \Big(-\frac{1}{d_i+y}\Big) \mb{y}_0^* \mb{\vartheta}_{0} + O_\prec\Big( \frac{d_i^2}{ \delta_{i0}^2 N}\Big) = O_\prec\Big( \frac{d_i^{-\frac 12} \delta_{i0}^{\frac 12}}{\sn}\Big)
    + O_\prec\Big( \frac{d_i^2}{ \delta_{i0}^2 N}\Big)\notag\\
    & = O_\prec\Big(  \sqrt{\frac{{d_i} }{ \delta_{i0}N }} \,\Big).
\end{align} 
Here we used the facts $(zm_1)'=1/(d_i^2-y)$ and $(zm_1)''=-2d_i^3/(d_i^2-y)^3$ together with the asymptotic expression for $\Delta(d)$ in \rf{Delta:asymp}. The last step is due to the assumption \rf{asd}.

Substituting (\ref{19072010}) into (\ref{19072011}) and using the fact $\mP=O_\prec(1)$, we get 
\begin{align}
h_{0}(1,0)&=\frac{\sqrt N m_1}{\Delta(d_i)}  \sum_{s=1}^2 \wt c_s  \mb{y}_0^*\mG_1^s \mb{\vartheta}_{0}\mP^{l-1} \notag\\
&\quad-  \frac{m_1}{\sn\Delta(d_i)}   \sum_{s=1}^2 \wt c_s    \sum_{\begin{subarray}{c}s_1,s_2\geq 1\\s_1+s_2=s+1\end{subarray}} \Big(\mb{y}_0^* \mG_1^{s_2} \mb{\vartheta}_{0} \; \text{Tr}(X^* \mG_1^{s_1} X)\Big)\mP^{l-1}\nonumber\\
&\quad+ O_\prec\Big(  \sqrt{\frac{{d_i} }{\delta_{i0} N}} \,\Big)
.  \label{19072015}
\end{align}
Using the second identity in (\ref{relationXG}), and Theorem \ref{isotropic}, more specifically, \rf{weak_est_m12N},  we have 
\begingroup
\allowdisplaybreaks
\begin{align} \label{2019121701}
& \frac{m_1}{\sn\Delta(d_i)}   \sum_{s=1}^2 \wt c_s    \sum_{\begin{subarray}{c}s_1,s_2\geq 1\\s_1+s_2=s+1\end{subarray}} \Big(\mb{y}_0^* \mG_1^{s_2} \mb{\vartheta}_{0} \; \text{Tr}(X^* \mG_1^{s_1} X)\Big)\mP^{l-1}\nonumber\\
 &=\frac{\sn m_1}{\Delta(d_i)}   \Big( \wt c_1(\mb{y}_0^*\mG_1\mb{\vartheta}_{0}) \frac{{\rm{Tr}} (X^*\mG_1X)}{N}+ \wt c_2(\mb{y}_0^*\mG_1\mb{\vartheta}_{0}) \frac{{\rm{Tr}} (X^*\mG_1^2X)}{N} + \wt c_2(\mb{y}_0^*\mG_1^2\mb{\vartheta}_{0}) \frac{{\rm{Tr}} (X^*\mG_1X)}{N}\Big) \mP^{l-1} \nonumber\\
 &=\frac{\sn m_1}{\Delta(d_i)}   \Big( \wt c_1(\mb{y}_0^*\mG_1\mb{\vartheta}_{0}) (1+zm_2)+ \wt c_2(\mb{y}_0^*\mG_1\mb{\vartheta}_{0}) (zm_2)' +\wt c_2(\mb{y}_0^*\mG_1^2\mb{\vartheta}_{0}) (1+zm_2)\Big) \mP^{l-1} \nonumber\\
&\quad + O_\prec\Big(\Big|\frac{\sn}{\Delta(d_i)} m_1^2\wt  c_1\Big(\frac{{\rm{Tr}} (z\mG_2)}{N}-zm_2\Big)\Big|\Big)+ O_\prec\Big( \Big|\frac{\sn}{\Delta(d_i)}  m_1^2 \wt c_2 \Big( \frac{{\rm{Tr}} (\mG_2+z\mG_2^2)}{N} - (zm_2)'\Big)\Big|\Big)\nonumber\\
 &\quad +  O_\prec\Big( \Big|\frac{\sn}{\Delta(d_i)} m_1m_1' \wt c_2 \Big( \frac{{\rm{Tr}} (z\mG_2)}{N} - zm_2\Big)\Big| \Big).
\end{align}
\endgroup
 Notice that $\wt c_1=O(d_i)$ and $\wt c_2=O_\prec(\delta_{i0}^2)$. 
 To estimate the error bounds in \eqref{2019121701}, again,  we need  to separate the discussion into two cases:  $d_i\leq K$ or $d_i>K$ for sufficiently large constant $K>0$. In the case $d_i\leq K$, we can use the isotropic law \rf{est_m12N} directly together with $\Delta(d_i)\asymp \delta_{i0}^{-\frac 12}$ and $m_1=O(1), m_1'= O(1/  \delta_{i0})$ to get an $O_\prec(\delta_{i0}^{-3/2} N^{-1/2})$ bound for all the error terms in \eqref{2019121701}. In case $d_i>K$, we use the improved isotropic local law \rf{est_m12N for large z} to get 
 \begin{align}
 \Big|\frac{{\rm{Tr}} (z\mG_2)}{N}-zm_2 \Big|=O_\prec(N^{-1} d_i^{-1}),\quad 
\Big|  \frac{{\rm{Tr}} (\mG_2+z\mG_2^2)}{N} - (zm_2)'\Big|= O_\prec(N^{-1} d_i^{-2}).
 \end{align}
 The above two estimates with the facts $m_1(z)= O(1/d_i) $ and $m_1'(z)= O(d_i^{-1} \delta_{i0}^{-1})$ yield the $O_\prec(N^{-1/2})$  bound for  all the error  terms in \eqref{2019121701} in the case $d_i\leq K$. To unify the two cases, we use the $O_\prec(d_i^{3/2}\delta_{i0}^{-3/2} N^{-1/2})$ bound to estimate all the error terms in \eqref{2019121701}.

Then, we obtain 
\begin{align*}
&\frac{m_1}{\sn\Delta(d_i)}   \sum_{s=1}^2 \wt c_s    \sum_{\begin{subarray}{c}s_1,s_2\geq 1\\s_1+s_2=s+1\end{subarray}} \Big(\mb{y}_0^* \mG_1^{s_2} \mb{\vartheta}_{0} \; \text{Tr}(X^* \mG_1^{s_1} X)\Big)\mP^{l-1} \nonumber\\
&=\frac{\sn m_1}{\Delta(d_i)}   \Big(\wt  c_1(\mb{y}_0^*\mG_1\mb{\vartheta}_{0}) (1+zm_2)+ \wt c_2(\mb{y}_0^*\mG_1\mb{\vartheta}_{0}) (zm_2)' + \wt c_2(\mb{y}_0^*\mG_1^2\mb{\vartheta}_{0}) (1+zm_2)\Big) \mP^{l-1}\notag\\
&\qquad+O_\prec\Big( d_i^{\frac 32} \delta_{i0}^{-\frac 32} N^{-\frac 12}\;\Big),
\end{align*}
which, together with \rf{19072015}, leads to 
\begin{align}
h_0(1,0)=&- \frac{\sn m_1}{\Delta(d_i)}   \Big(\wt  c_1(\mb{y}_0^*\mG_1\mb{\vartheta}_{0}) (zm_2)+ \wt c_2(\mb{y}_0^*\mG_1\mb{\vartheta}_{0}) (zm_2)' + \wt c_2(\mb{y}_0^*\mG_1^2\mb{\vartheta}_{0}) (zm_2)\Big) \mP^{l-1}\notag\\
&\quad+O_\prec\Big( d_i^{\frac 32} \delta_{i0}^{-\frac 32} N^{-\frac 12}\;\Big).
\end{align}

For the case $t=1,2$, we have
\begin{align*}
h_t(1,0) &= \frac{ m_1}{\Delta(d_i) \sn } \sum_{q,k} y_{tq} \frac{\partial \big(X^* \mG_1 \mb{\vartheta}_t \big)_{k}}{\partial x_{qk}}\,\mP^{l-1}\nonumber\\
& =  \frac{ m_1}{\Delta(d_i) \sn } \sum_{q,k} y_{tq} \Big((\mG_1\mb{\vartheta}_t )_{q} -  \sum_{a=1}^2 \big( X^*\mG_1  \mathscr{P}_a^{qk}\mG_1 \mb{\vartheta}_t\big)_k \Big) \mP^{l-1} \nonumber\\
&= \frac{\sn m_1}{\Delta(d_i)} \mb{y}_t^* \mG_1 \mb{\vartheta}_t - \frac{ m_1}{\Delta(d_i) \sn } (  \mb{y}_t^* \mG_1 XX^*\mG_1 \mb{\vartheta}_t ) -  \frac{\sn m_1}{\Delta(d_i)  } \mb{y}_t^* \mG_1 \mb{\vartheta}_t \frac{1}{N}{\rm{Tr}} (X^*\mG_1X) \nonumber\\
&= - \frac{\sn m_1}{\Delta(d_i)} \mb{y}_t^* \mG_1 \mb{\vartheta}_t (zm_2) + O_\prec\Big({d_i^{\frac 32} \delta_{i0}^{-\frac32}} N^{-\frac 12} \Big)
\end{align*}
where we used the estimates
\begin{align*}
 \frac{ m_1}{\Delta(d_i) \sn } (  \mb{y}_t^* \mG_1 XX^*\mG_1 \mb{\vartheta}_t ) &= 
 \frac{ m_1}{\Delta(d_i) \sn }(zm_1)' \mb{y}_t^* \mb{\vartheta}_t + O_\prec \Big(\frac{d_i}{N\delta_{i0}^2}\Big)  \notag\\
 &= O_\prec( d_i^{-\frac 12} \delta_{i0}^{-\frac 12}N^{-\frac 12})
\end{align*}
and 
\begin{align*}
 \frac{\sn m_1}{\Delta(d_i)  } \mb{y}_t^* \mG_1 \mb{\vartheta}_t \frac{1}{N}{\rm{Tr}} (X^*\mG_1X) =
  \frac{\sn m_1}{\Delta(d_i)  } \mb{y}_t^* \mG_1 \mb{\vartheta}_t (1+zm_2) + O_\prec\Big({d_i^{\frac 32} \delta_{i0}^{-\frac32}} N^{-\frac 12} \Big).
 \end{align*}
Both follow directly from the isotropic local laws \rf{weak_est_DG} and \rf{weak_est_m12N}.

Next, we turn to estimate $h_{t}(0,1)$, which by the definition in \rf{rankrdefhts} reads
\begin{align}
h_{t}(0,1)&=(l-1)\frac{m_1}{\sn\Delta(d_i)} \sum_{q,k}y_{tq} \big(X^*T_t(\mG_1) \mb{\vartheta}_{t} \big)_{k}\frac{\partial \mP}{\partial x_{qk}} \,\mP^{l-2}.
\label{rankrhts(1,0)}
\end{align}

Using the formula \rf{eq:Pd1} to \rf{rankrhts(1,0)}, we get
\begin{align} 
&h_{t}(0,1)\nonumber\\
&=- \frac{(l-1)m_1}{\Delta(d_i)^2} \bigg[\sum_{a=1}^2 c_a \sum_{\begin{subarray}{c}a_1,a_2\geq 1;\\ a_1+a_2=a+1 \end{subarray}}\Big((\mb{y}_t^*\mG_1^{a_1}\mb{y}_0) (\mb{\vartheta}_t^*T_t(\mG_1)XX^*\mG_1^{a_2}\mb{\vartheta}_0)\notag\\
&\qquad\qquad\qquad\qquad+(\mb{y}_t^*\mG_1 ^{a_2}\mb{\vartheta}_0) (\mb{\vartheta}_t^*T_t(\mG_1)XX^*\mG_1^{a_1}\mb{y}_0) \Big)\nonumber\\
&+ \sum_{s=1}^2 \Big((\mb{y}_t^*\mG_1\mb{y}_s)( \mb{\vartheta}_t^*T_t(\mG_1)XX^*\mG_1\mb{\vartheta}_{s})+
(\mb{y}_t^*\mG_1\mb{\vartheta}_{s})(\mb{\vartheta}_t^*T_t(\mG_1)XX^*\mG_1\mb{y}_s)\Big)
\bigg]\mP^{l-2}. 
\label{17092027}
\end{align}
By \eqref{relationXG}, we have
\begin{align*}
\mG_1^{s_1}XX^*\mG_1^{s_2} = \mG_1^{s_1+s_2-1} + z \mG_1^{s_1+s_2}, \qquad a=1,2.
\end{align*}
Hence, for $t=0$, by  \eqref{weak_est_DG} and the bound $|\mP|\prec 1$, we have 
\begin{align}
h_0(0,1)& = - \frac{(l-1)m_1}{\Delta(d_i)^2} \bigg[\sum_{a=1}^2 c_a \sum_{\begin{subarray}{c}a_1,a_2\geq 1;\\ a_1+a_2=a+1 \end{subarray}}\bigg( \frac{m_1^{(a_1-1)}}{(a_1-1)!} (\mb{y}_0^*\mb{y}_0) \Big( \sum_{b=1}^2\wt c_b \frac{(zm_1)^{(b+a_2-1)}}{(b+a_2-1)!}\Big) (\mb{\vartheta}_0^*\mb{\vartheta}_0)\nonumber\\
&+\frac{m_1^{(a_2-1)}}{(a_2-1)!}(\mb{y}_0^*\mb{\vartheta}_0) \Big(\sum_{b=1}^2\wt c_b\frac{(zm_1)^{(b+a_1-1)}}{(b+a_1-1)!}\Big) (\mb{\vartheta}_0^*\mb{y}_0) \bigg)\nonumber\\
&+ \sum_{s=1}^2 \Big( m_1(\mb{y}_0^*\mb{y}_s) \Big( \sum_{b=1}^2\wt c_b \frac{(zm_1)^{(b)}}{b!}\Big)( \mb{\vartheta}_0^*\mb{\vartheta}_{s})\notag\\
&\qquad\qquad+
m_1(\mb{y}_0^*\mb{\vartheta}_{s}) \Big( \sum_{b=1}^2\wt c_b \frac{(zm_1)^{(b)}}{b!}\Big)(\mb{\vartheta}_0^*\mb{y}_s)\Big)
\bigg]\mP^{l-2} + \mathcal{R},
\end{align}
where the error term $\mathcal{R}$ is dominated by the sum of the absolute values of the following two terms 
\begin{align*}
& \mathcal{R}_1:=- \frac{(l-1)m_1}{\Delta(d_i)^2}
 \bigg[\sum_{a=1}^2 c_a \sum_{\begin{subarray}{c}a_1,a_2\geq 1;\\ a_1+a_2=a+1 \end{subarray}}\Big((\mb{y}_0^*\Xi^{(a_1-1)}\mb{y}_0)\Big( \sum_{b=1}^2\wt c_b \frac{(zm_1)^{(b+a_2-1)}}{(b+a_2-1)!}\Big) (\mb{\vartheta}_0^*\mb{\vartheta}_0)\nonumber\\
 &\qquad  +(\mb{y}_0^*\Xi^{(a_2-1)}\mb{\vartheta}_0) \Big(\sum_{b=1}^2\wt c_b\frac{(zm_1)^{(b+a_1-1)}}{(b+a_1-1)!}\Big) (\mb{\vartheta}_0^*\mb{y}_0) \Big)\nonumber\\
&\qquad + \sum_{s=1}^2 \Big((\mb{y}_0^*\Xi\mb{y}_s) \Big( \sum_{b=1}^2\wt c_b \frac{(zm_1)^{(b)}}{b!}\Big)( \mb{\vartheta}_0^*\mb{\vartheta}_{s})+
(\mb{y}_0^*\Xi\mb{\vartheta}_{s})\Big( \sum_{b=1}^2\wt c_b \frac{(zm_1)^{(b)}}{b!}\Big)(\mb{\vartheta}_0^*\mb{y}_s)\Big)
\bigg], 
\end{align*}
and 
\begin{align} \label{def:mathcalR2}
& \mathcal{R}_2: = - \frac{(l-1)m_1}{\Delta(d_i)^2} \bigg[\sum_{a=1}^2 c_a \sum_{\begin{subarray}{c}a_1,a_2\geq 1;\\ a_1+a_2=a+1 \end{subarray}}\bigg( \frac{m_1^{(a_1-1)}}{(a_1-1)!} (\mb{y}_0^*\mb{y}_0) \Big( \sum_{b=1}^2\wt c_b  \mb{\vartheta}_0^*\Big[\frac{(z\Xi)^{(b+a_2-1)}}{(b+a_2-1)!} \Big]\mb{\vartheta}_0\Big)\nonumber\\
&+\frac{m_1^{(a_2-1)}}{(a_2-1)!}(\mb{y}_0^*\mb{\vartheta}_0) \Big(\sum_{b=1}^2\wt c_b \mb{\vartheta}_0^*\Big[  \frac{ (z\Xi)^{(b+a_1-1)} }{(b+a_1-1)!} \Big] \mb{y}_0 \Big) \bigg)\nonumber\\
&+ \sum_{s=1}^2 \Big( m_1(\mb{y}_0^*\mb{y}_s) \Big( \sum_{b=1}^2\wt c_b  \mb{\vartheta}_0^*\Big[ \frac{ (z\Xi)^{(b)} }{b!} \Big] \mb{\vartheta}_{s}\Big) +
m_1(\mb{y}_0^*\mb{\vartheta}_{s}) \bigg( \sum_{b=1}^2\wt c_b \mb{\vartheta}_0^*\Big[ \frac{ (z\Xi)^{(b)} }{b!} \Big] \mb{y}_s\Big)\bigg)
\bigg].
\end{align}
In the sequel, we prove both $\mathcal{R}_1$ and $\mathcal{R}_2$ are negligible. First, we apply \rf{weak_est_DG} and the facts which can be checked by elementary computations
\begin{align*}
\sum_{b=1}^2 \wt c_b \frac{(zm_1)^{(b)}}{b!} = \frac{y(1+d_i)}{(d_i+y) (d_i^2-y)}, \; \quad 
\sum_{b=1}^2 \wt c_b \frac{(zm_1)^{(b+1)}}{(b+1)!} = - \frac{d_i^2y}{(d_i+y)(d_i^2-y)^2}
\end{align*}
to get $$\mathcal{R}_1= O_\prec(d_i^{\frac 12} \delta_{i0}^{-\frac 12} N^{-\frac 12}).$$

Next, we turn to $\mathcal{R}_2$. If $d_i\leq K$ with sufficiently large constant $K>0$, direct applications of the isotropic local law \rf{weak_est_DG} yield that  $|\mathcal R_2|=O_\prec(\delta_{i0}^{-\frac32}N^{-\frac12})$. Hence, we focus on the case that  $d_i$ is large, i.e. $d_i>K$ with sufficiently large constant $K>0$. 
 Recall the new form of eigenvector empirical spectral distribution(VESD) of $H$ with respect to two fixed unit vectors $\bu, \bv$ from \rf{def:VESDnew},
Then for any given unit vectors $\bv, \bu$, we can derive 
\begin{align*}
&\sum_{b=1}^2\wt c_b \bv^* \Big[ \frac{ (z\Xi)^{(b)}}{b!} \Big] \bu \notag\\
&\quad= \int \Big(\frac{\wt c_1}{x-z} + \frac{\wt c_1z}{(x-z)^2}  + \frac{\wt c_2 }{(x-z)^2} + \frac {\wt c_2z}{(x-z)^3}\Big) \, {\rm{d}}  \big(F_{1N}^{\bu,\bv}(x) - F^{\bu,\bv}_1(x)\big) \nonumber\\
&\quad\asymp \int_{\lambda_--N^{-\frac 23+ \frac \tau 2}}^{\lambda_+ + N^{-\frac 23+ \frac \tau 2}} \Big(\frac{-\wt c_1}{(x-z)^2} -\frac{2 \wt c_1z+2\wt c_2}{(x-z)^3}   - \frac {3\wt c_2z}{(x-z)^4}\Big)\big(F_{1N}^{\bu,\bv}(x) - F^{\bu,\bv}_1(x)\big)   \, {\rm{d}}  x.
\end{align*}
Further by \rf{est:Kol uv}, i.e. $\sup_x| F_{1N}^{\bu,\bv}(x)- F_1^{\bu, \bv}(x)|= O_\prec(N^{-\frac 12})$ and the simple estimate 
$$\frac{-\wt c_1}{(x-z)^2} -\frac{2 \wt c_1z+2\wt c_2 }{(x-z)^3}   - \frac {3\wt c_2z}{(x-z)^4}= O(d_i^{-2})$$
which follows form the substitution of $\wt c_1, \wt c_2$ and $z\asymp d_i$, we  finally have
\begin{align} \label{2020051101}
\sum_{b=1}^2\wt c_b \bv^* \Big[ \frac{ (z\Xi)^{(b)}}{b!} \Big] \bu 
=O_\prec (d_i^{-2}N^{-\frac12}).
\end{align}

 Analogously, we have 
\begin{align} \label{2020051101*}
\sum_{b=1}^2\wt c_b \bv^*\Big[ \frac{ (z\Xi)^{(b+1)}}{(b+1)!} \Big] \bu
= O_\prec (d_i^{-3}N^{-\frac12}).
\end{align}
Plugging \rf{2020051101} back into \rf{def:mathcalR2}, we see 
\begin{align} \label{2020051102}
\mathcal{R}_2=&- \frac{(l-1)m_1}{\Delta(d_i)^2} \bigg[ \Big( \sum_{b=1}^2\wt c_b  \mb{\vartheta}_0^*\Big[\frac{(z\Xi)^{(b)}}{b!} \Big]\mb{\vartheta}_0\Big)(\mb{y}_0^*\mb{y}_0) (\frac{c_1}{2}m_1+ c_2m_1') \bigg] \nonumber\\
&- \frac{(l-1)m_1}{\Delta(d_i)^2} \bigg[ \Big( \sum_{b=1}^2\wt c_b  \mb{\vartheta}_0^*\Big[\frac{(z\Xi)^{(b)}}{b!} \Big]\mb{y}_0\Big)(\mb{y}_0^*\mb{\vartheta}_0) (\frac{c_1}{2}m_1+ c_2m_1') \bigg]\nonumber\\
&- \frac{(l-1)m_1}{\Delta(d_i)^2} \bigg[ m_1\Big( \frac{c_1}{2}\sum_{b=1}^2\wt c_b  \mb{\vartheta}_0^*\Big[\frac{(z\Xi)^{(b)}}{b!} 
\Big]\mb{\vartheta}_0 + c_2\sum_{b=1}^2\wt c_b  \mb{\vartheta}_0^*\Big[\frac{(z\Xi)^{(b+1)}}{(b+1)!} 
\Big]\mb{\vartheta}_0\Big)(\mb{y}_0^*\mb{y}_0)  \bigg] \nonumber\\
&- \frac{(l-1)m_1}{\Delta(d_i)^2} \bigg[ m_1\Big( \frac{c_1}{2}\sum_{b=1}^2\wt c_b  \mb{\vartheta}_0^*\Big[\frac{(z\Xi)^{(b)}}{b!} 
\Big]\mb{y}_0 + c_2\sum_{b=1}^2\wt c_b  \mb{\vartheta}_0^*\Big[\frac{(z\Xi)^{(b+1)}}{(b!+1)} 
\Big]\mb{\vartheta}_0\Big)(\mb{y}_0^*\mb{\vartheta}_0)\bigg]\notag\\
&+ O_\prec(N^{-\frac12}).
\end{align}
By \rf{2020051101} and the fact that $c_1m_1/2+ c_2m_1'= O_\prec(1/d_i)$, we see that the first two terms on the RHS of \rf{2020051102} are of order $O_\prec(N^{-\frac 12})$. Similarly to the estimates of \rf{2020051101} and \rf{2020051101*}, we also get 
\begin{align*}
&m_1\Big( \frac{c_1}{2}\sum_{b=1}^2\wt c_b  \mb{\vartheta}_0^*\Big[\frac{(z\Xi)^{(b)}}{b!} 
\Big]\mb{\vartheta}_0 + c_2\sum_{b=1}^2\wt c_b  \mb{\vartheta}_0^*\Big[\frac{(z\Xi)^{(b+1)}}{(b+1)!} 
\Big]\mb{\vartheta}_0\Big) \nonumber\\
& \asymp
 m_1 \bigg(\frac{c_1}{2}\int_{\lambda_--N^{-\frac 23+ \frac \tau 2}}^{\lambda_+ + N^{-\frac 23+ \frac \tau 2}} \Big(\frac{-\wt c_1}{(x-z)^2} -\frac{2 \wt c_1z+ 2\wt c_2}{(x-z)^3}   - \frac {3\wt c_2z}{(x-z)^4}\Big)\big(F_{1N}^{\bu,\bv}(x) - F^{\bu,\bv}_1(x)\big)   \, {\rm{d}}  x \nonumber\\
 & \quad + c_2 \int_{\lambda_--N^{-\frac 23+ \frac \tau 2}}^{\lambda_+ + N^{-\frac 23+ \frac \tau 2}} \Big(\frac{-2\wt c_1}{(x-z)^3} -\frac{3 \wt c_1z+ 3\wt c_2 }{(x-z)^4} - \frac {4\wt c_2z}{(x-z)^5}\Big)\big(F_{1N}^{\bu,\bv}(x) - F^{\bu,\bv}_1(x)\big)   \, {\rm{d}}  x \bigg).
\end{align*}
For $z\asymp d_i>K$, expanding $(x-z)^{-i}, i=1,\cdots,5$ at $x=0$ and plugging in the values of $c_{1,2}, \wt c_{1,2}$ defined in \rf{def:c12} and  \rf{def:wtc12}, after elementary computations, we see that 
\begin{align*}
&\frac{c_1}{2}\Big(\frac{-\wt c_1}{(x-z)^2} -\frac{2 \wt c_1z+ 2\wt c_2}{(x-z)^3}   - \frac {3\wt c_2z}{(x-z)^4}\Big) \notag\\
&+ c_2 \Big(\frac{-2\wt c_1}{(x-z)^3} -\frac{3 \wt c_1z+ 3\wt c_2}{(x-z)^4}  - \frac {4\wt c_2z}{(x-z)^5}\Big) = O(d_i^{-2}),
\end{align*}
which together with \rf{est:Kol uv} and  $m_1= O(d_i^{-1})$ leads to 
\begin{align*}
 m_1\Big( \frac{c_1}{2}\sum_{b=1}^2\wt c_b  \mb{\vartheta}_0^*\Big[\frac{(z\Xi)^{(b)}}{b!} 
\Big]\mb{\vartheta}_0 + c_2\sum_{b=1}^2\wt c_b  \mb{\vartheta}_0^*\Big[\frac{(z\Xi)^{(b+1)}}{(b+1)!} 
\Big]\mb{\vartheta}_0\Big)=  O_\prec(d_i^{-3}N^{-\frac12  }).
\end{align*}

Hence, we see the third and fourth terms on the RHS of \rf{2020051102} are crudely bounded by $O_\prec(N^{-\frac 12})$. Now, we conclude that $|\mathcal{R}_2|= O_\prec (N^{-\frac12})$ for large $d_i$ and  $|\mathcal{R}_2|=O_\prec(\delta_{i0}^{-\frac32}N^{-\frac12})$ for bounded $d_i$. This, together with $|\mathcal{R}_1|=O_\prec(d_i^{\frac 12} \delta_{i0}^{-\frac 12} N^{-\frac12})$, finally leads to 
\begin{align}
&h_0(0,1)= - \frac{(l-1)m_1}{\Delta(d_i)^2} \bigg[\sum_{a=1}^2 c_a \sum_{\begin{subarray}{c}a_1,a_2\geq 1;\\ a_1+a_2=a+1 \end{subarray}}\bigg( \frac{m_1^{(a_1-1)}}{(a_1-1)!} (\mb{y}_0^*\mb{y}_0) \Big( \sum_{b=1}^2\wt c_b \frac{(zm_1)^{(b+a_2-1)}}{(b+a_2-1)!}\Big) (\mb{\vartheta}_0^*\mb{\vartheta}_0)\nonumber\\
&\qquad\qquad\quad +\frac{m_1^{(a_2-1)}}{(a_2-1)!}(\mb{y}_0^*\mb{\vartheta}_0) \Big(\sum_{b=1}^2\wt c_b\frac{(zm_1)^{(b+a_1-1)}}{(b+a_1-1)!}\Big) (\mb{\vartheta}_0^*\mb{y}_0) \bigg)\nonumber\\
&\qquad\qquad+ \sum_{s=1}^2 \Big( m_1(\mb{y}_0^*\mb{y}_s) \Big( \sum_{b=1}^2\wt c_b \frac{(zm_1)^{(b)}}{b!}\Big)( \mb{\vartheta}_0^*\mb{\vartheta}_{s})\notag\\
&\qquad \qquad \quad +
m_1(\mb{y}_0^*\mb{\vartheta}_{s}) \Big( \sum_{b=1}^2\wt c_b \frac{(zm_1)^{(b)}}{b!}\Big)(\mb{\vartheta}_0^*\mb{y}_s)\Big)
\bigg]\mP^{l-2} +O_\prec(d_i^{\frac 32} \delta_{i0}^{-\frac 32} N^{-\frac12}).
\end{align}

For $t=1,2$, using the same arguments as for $h_0(0,1)$, we can derive 
\begin{align}
&h_t(0,1)= - \frac{(l-1)m_1}{\Delta(d_i)^2} \bigg[\sum_{a=1}^2 c_a \sum_{\begin{subarray}{c}a_1,a_2\geq 1;\\ a_1+a_2=a+1 \end{subarray}}\bigg( \frac{m_1^{(a_1-1)}}{(a_1-1)!} (\mb{y}_t^*\mb{y}_0) \Big( \frac{(zm_1)^{(a_2)}}{(a_2)!}\Big) (\mb{\vartheta}_t^*\mb{\vartheta}_0)\nonumber\\
&\qquad\qquad+\frac{m_1^{(a_2-1)}}{(a_2-1)!}(\mb{y}_t^*\mb{\vartheta}_0) \Big(\frac{(zm_1)^{(a_1)}}{(a_1)!}\Big) (\mb{\vartheta}_t^*\mb{y}_0) \bigg)\nonumber\\
&\qquad\qquad+ \sum_{s=1}^2 \Big( m_1(\mb{y}_t^*\mb{y}_s)  (zm_1)' ( \mb{\vartheta}_0^*\mb{\vartheta}_{s})+
m_1(\mb{y}_t^*\mb{\vartheta}_{s})(zm_1)'(\mb{\vartheta}_0^*\mb{y}_s)\Big)
\bigg]\mP^{l-2} \notag\\
& \qquad\qquad+O_\prec(d_i^{\frac 32} \delta_{i0}^{-\frac 32} N^{-\frac12}).
\end{align}

Next, we show  \rf{rankresth0(1,2)} and  \rf{rankrestht(1,2)}. First,  by the definition in (\ref{rankrdefhts}), we have 
 \begin{align} 
 h_{t}(1,2)
 &=\frac{\kappa_4}{2N^{\frac{3}{2}}}\frac{m_1}{\Delta(d_i)} \sum_{q,k}y_{tq}\frac{\partial (X^*T_t(\mG_1)\mb{\vartheta}_{t} )_{k}}{\partial x_{qk}} \notag\\
&\quad  \times \Big((l-1) \frac{\partial^2 \mP}{\partial x_{qk}^2}\mP^{l-2} + (l-1)(l-2)\Big(\frac{\partial \mP}{\partial x_{qk}}\Big)^2\mP^{l-3} \Big)\nonumber\\
 &=: \kappa_4(l-1)\mathcal J_{1t} +\kappa_4(l-1)(l-2) \mathcal J_{2t}, 
  \label{17092053}
 \end{align} 
In the following, we estimate $\mathcal J_{1t}$ and  $\mathcal J_{2t}$ for $t=0,1,2$. We only state the details for the case $t=0$ as proofs for the cases $t=1,2$ are analogous and even simpler according  to the the definition of $T_t(\mG_1)$ in (\ref{def:Tt(mG)}).
  
To estimate $\mathcal J_{10}$, using
 (\ref{eq:Pd2}) and \eqref{deriXG^l},  
we see
  \begin{align}
 \mathcal J_{10}&=
 \frac{ 1}{2N^{\frac{3}{2}}}\frac{m_1}{\Delta(d_i)} \sum_{s=1}^2 \wt c_s \sum_{q,k}y_{0q}\frac{\partial (X^* \mG_1^s \mb{\vartheta}_{0}  )_{k}}{\partial x_{qk}} \frac{\partial^2 \mP}{\partial x_{qk}^2}\mP^{l-2} \nonumber\\
 &=\frac{1}{N}\frac{m_1}{\Delta(d_i)^2}\sum_{s=1}^2 \wt c_s \sum_{q,k}y_{0q} \Big((\mG_1^s\mb{\vartheta}_{0} )_{q}-\sum_{a=1}^2\sum_{\begin{subarray}{c}s_1,s_2\geq 1\\ s_1+s_2=s+1 \end{subarray}}(X^*\mG_1^{s_1}\mathscr{P}_a^{qk}\mG_1^{s_2}\mb{\vartheta}_{0} )_{k} \Big) \nonumber\\
 &\quad \times \bigg[\sum_{b=1}^2 c_b \Big(\sum_{a_1,a_2=1}^2\sum_{\begin{subarray}{c}b_1,b_2,b_3\geq 1;\\ \sum_{i=1}^3b_i=b+2 \end{subarray}}\mb{y}_0^*\Big(\prod_{j=1}^2(\mG_1^{b_j}\mathscr{P}_{a_j}^{qk})\Big)\mG_1^{b_3}\mb{\vartheta}_0-\sum_{\begin{subarray}{c}b_1,b_2\geq 1;\\ b_1+b_2=b+1 \end{subarray}}\mb{y}_0^*\mG_1^{b_1}\mathscr{P}_0^{qk}\mG_1^{b_2}\mb{\vartheta}_{0}\Big)\nonumber\\
 &\quad + \sum_{t=1}^2  \Big(\sum_{a_1,a_2=1}^2 \mb{y}_t^*\Big(\prod_{j=1}^2(\mG_1\mathscr{P}_{a_j}^{qk})\Big)\mG_1\mb{\vartheta}_t - 
 \mb{y}_t^*\mG_1\mathscr{P}_0^{qk}\mG_1\mb{\vartheta}_{t}\Big) \bigg] \mP^{l-2}.
  \label{17092030}
 \end{align} 
Further, we claim that
\begin{align}
 \mathcal J_{10}&=\frac{1}{N}  \frac{m_1}{\Delta(d_i)^2}\sum_{s=1}^2 \wt c_s\sum_{q,k}y_{0q}\Big((\mG_1^s \mb{\vartheta}_{0}  )_{q}-\sum_{\begin{subarray}{c}s_1,s_2\geq 1\\ s_1+s_2=s+1 \end{subarray}}(X^*\mG_1^{s_1}\mathscr{P}_2^{qk}\mG_1^{s_2}\mb{\vartheta}_{0}  )_{k} \Big) \nonumber\\
 &\quad \times \bigg[\sum_{b=1}^2 c_b  \mb{y}_0^*\Big(\sum_{\begin{subarray}{c}b_1,b_2,b_3\geq 1\\b_1+b_2+b_3=b+2\end{subarray}}\mG_1^{b_1}\mathscr{P}_{1}^{qk}\mG_1^{b_2}\mathscr{P}_{2}^{qk}\mG_1^{b_3}-\sum_{\begin{subarray}{c}b_1,b_2\geq 1\\b_1+b_2=b+1\end{subarray}}\mG_1^{b_1}\mathscr{P}_0^{qk}\mG_1^{b_2}\Big)\mb{\vartheta}_{0}\nonumber\\
 &\quad + \sum_{t=1}^2  \mb{y}_t^*\Big(\mG_1\mathscr{P}_{1}^{qk}\mG_1\mathscr{P}_{2}^{qk}\mG_1-\mG_1\mathscr{P}_0^{qk}\mG_1\Big)\mb{\vartheta}_{t}
 \bigg] \mP^{l-2} + O_\prec(N^{-\frac12}). \label{17092031}
\end{align}
To see the reduction from (\ref{17092030}) to (\ref{17092031}), we first observe from Lemma \ref{lem.19072501*} that $|m_1|\sim d_i^{-1}$. Further notice that the terms absorbed into $O_\prec(N^{-\frac12})$ 
 always contain some quadratic forms of $X^*\mG_1^{a}$ for some $a\ge1$ and at least one quadratic form of $\mG_1^{b}$ for some $b\ge 1$. Then the isotropic local law \eqref{weak_est_XG} and Remark \ref{boundrmk*} can be applied to show that the quadratic forms of  $X^*\mG_1^{a}$ and $\mG_1^{b}$  are bounded by $O_\prec(N^{-\frac12}\delta_{i0}^{-\frac 12}d_i^{-\frac 12} (\delta_{i0}^{-2}d_i)^{a-1})$ and $O_\prec(d_i^{-1}\delta_{i0}^{-(b-1)})$, respectively,  for $a, b=1,2$. For instance,  we have 
\begin{align*}
&N^{-\frac32} \frac{m_1}{\Delta(d_i)}\sum_{q,k}y_{0q} \wt c_s \Big(\sum_{\begin{subarray}{c}s_1,s_2\geq 1\\ s_1+s_2=s+1 \end{subarray}}(X^*\mG_1^{s_1}\mathscr{P}_1\mG_1^{s_2}\mb{\vartheta}_{0})_{k} \Big)  \frac{\partial^2 \mP}{\partial x_{qk}^2}\mP^{l-2}\\
&\prec N^{-\frac32}\bigg| \frac{m_1}{\Delta(d_i)} \sum_{q,k}\sum_{\begin{subarray}{c}s_1,s_2\geq 1\\ s_1+s_2=s+1 \end{subarray}} y_{0q} \wt c_s  (X^*\mG_1^{s_1})_{kq} (X^*\mG_1^{s_2} \mb{\vartheta}_{\varphi(t)} )_{k}\frac{\partial^2 \mP}{\partial x_{qk}^2}\bigg| \\
&\prec N^{-1} \frac{|m_1|}{\Delta(d_i)}\sum_{\begin{subarray}{c}s_1,s_2\geq 1\\ s_1+s_2=s+1 \end{subarray}}  \wt c_s\Big(\sum_{q} \big| y_{tq} \big| \Big)\sum_{k} \big| (X^*\mG_1^{s_1})_{kq} \big| \big| (X^*\mG_1^{s_2} \mb{\vartheta}_{0})_k\big| \sqrt{\frac{d_i-\sqrt y}{d_i}}\\
&=O_\prec(N^{-\frac12}).
\end{align*}
In the first and second steps above, we used the  bound $|\mathcal P| \prec 1$ and (\ref{19072411}),  
 respectively. In the last step, we used the isotropic local law \eqref{weak_est_XG} and also the fact  $\sum_{q}\big|y_{tq}\big|\prec \sn$. 
 The other negligible terms can be estimated similarly. We omit the details. 

Plugging the definitions in (\ref{P012}) into (\ref{17092031}) yields
 \begin{align*}
\mathcal J_{10} &= \frac{m_1}{N\Delta(d_i)^2} \sum_{s=1}^2 \wt c_s \sum_{q,k}y_{0q}\Big((\mG_1^s\mb{\vartheta}_{0})_{q}-\sum_{\begin{subarray}{c}s_1,s_2\geq 1\\s_1+s_2=s+1\end{subarray}}(X^*\mG_1^{s_1}X)_{kk}(\mG_1^{s_2}\mb{\vartheta}_{0})_{q}\Big)\nonumber\\
  &\quad \times \bigg[\sum_{b=1}^2 c_b \Big(\sum_{\begin{subarray}{c}b_1,b_2,b_3\geq 1\\b_1+b_2+b_3=b+2\end{subarray}}( \mb{y}_0^*\mG_1^{b_1})_q (X^*\mG_1^{b_2}X)_{kk}(\mG_1^{b_3}\mb{\vartheta}_{0})_q -\sum_{\begin{subarray}{c}b_1,b_2\geq 1\\b_1+b_2=b+1\end{subarray}}(\mb{y}_{0}^*\mG_1^{b_1})_q (\mG_1^{b_2}\mb{\vartheta}_{0})_q\Big)\nonumber\\
 &\quad + \sum_{t=1}^2  \Big(( \mb{y}_t^*\mG_1)_q (X^*\mG_1X)_{kk}(\mG_1\mb{\vartheta}_{t})_q-(\mb{y}_{t}^*\mG_1)_q (\mG_1\mb{\vartheta}_{t})_q\Big)
 \bigg] \mP^{l-2} + O_\prec(N^{-\frac12}).
  \end{align*}
Again, applying the identities in Lemma \ref{lemrelation}, the isotropic local laws \eqref{weak_est_DG}, and \eqref{est_m12N}, we see 
 \begin{align*}
\mathcal J_{10}&=\frac{m_1}{\Delta(d_i)^2} \sum_{s=1}^2 \wt c_s\bigg(\frac{m_1^{(s-1)}}{(s-1)!}-\sum_{\begin{subarray}{c}s_1,s_2\geq 1\\s_1+s_2=s+1\end{subarray}}\Big(\frac{m_2^{(s_1-2)}}{(s_1-2)!}+z\frac{m_2^{(s_1-1)}}{(s_1-1)!}\Big)\frac{m_1^{(s_2-1)}}{(s_2-1)!}\bigg) \nonumber\\
 & \times \bigg(\sum_{b=1}^2 c_b \Big[\sum_{\begin{subarray}{c}b_1,b_2,b_3\geq 1\\b_1+b_2+b_3=b+2\end{subarray}}
\frac{m_1^{(b_1-1)}}{(b_1-1)!}\frac{m_1^{(b_3-1)}}{(b_3-1)!}\Big(\frac{m_2^{(b_2-2)}}{(b_2-2)!}+z\frac{m_2^{(b_2-1)}}{(b_2-1)!}\Big)-\frac{(m_1^2)^{(b-1)} }{(b-1)!}\Big]
\sum_{q} y_{0q}^2 {\vartheta}_{0q}^2 \nonumber\\
&+\sum_{t=1}^2 (m_1^2(1+zm_2)-m_1^2) \sum_{q} y_{0q} y_{tq} {\vartheta}_{0q}  {\vartheta}_{tq}
\bigg)\mP^{l-2} + O_\prec (N^{-\frac 12}),
 \end{align*}
 which by elementary calculation can be rewritten as
 \begin{align}
\mathcal J_{10}&= - \frac{m_1}{\Delta(d_i)^2} \sum_{s=1}^2 \wt c_s\bigg(\frac{(zm_2m_1)^{(s-1)}}{(s-1)!}\bigg) 
\times \bigg(\sum_{b=1}^2 c_b \frac{(zm_2m_1^2)^{(b-1)}}{(b-1)!}
 \sum_{q} y_{0q}^2 {\vartheta}_{0q}^2\nonumber\\
&\qquad +\sum_{t=1}^2 (zm_2m_1^2) \sum_{q} y_{0q} y_{tq} {\vartheta}_{0q}  {\vartheta}_{tq}
\bigg)\mP^{l-2} + O_\prec (N^{-\frac 12}).
\end{align}
 Next, we show that
 \begin{align}
 \mathcal J_{20} =\frac{m_1}{2N^{\frac{3}{2}}\Delta(d_i)} \sum_{s=1}^2 \wt c_s \sum_{q,k} y_{0q}\frac{\partial (X^*\mG_1^s \mb{\vartheta}_{0})_{k}}{\partial x_{qk}}\Big(\frac{\partial \mP}{\partial x_{qk}}\Big)^2\mP^{l-3} = O_\prec(N^{-\frac12}). \label{17092051}
 \end{align}

With \eqref{deriXG^l}, we write $ \mathcal J_{20}$ as 
\begin{align}
\mathcal J_{20}&=\frac{m_1}{2N^{\frac{3}{2}}\Delta(d_i)} \sum_{s=1}^2 \wt c_s 
\sum_{q,k}y_{0q}\bigg((\mG_1^s\mb{\vartheta}_{0})_{q}-\sum_{a=1}^2\sum_{\begin{subarray}{c} s_1,s_2\geq 1\\s_1+s_2=s+1\end{subarray}}(X^*\mG_1^{s_1}\mathscr{P}_a^{qk}\mG_1^{s_2}\mb{\vartheta}_{0})_{k}\bigg)\Big(\frac{\partial \mP}{\partial x_{qk}}\Big)^2\mP^{l-3} \nonumber\\
&=\frac{m_1}{2N^{\frac{3}{2}}\Delta(d_i)} \sum_{s=1}^2 \wt c_s \sum_{q,k}y_{0q}\bigg((\mG_1^s \mb{\vartheta}_{0})_{q}-\sum_{\begin{subarray}{c} s_1,s_2\geq 1\\s_1+s_2=s+1\end{subarray}}(X^*\mG_1^{s_1}X)_{kk}(\mG_1^{s_2}\mb{\vartheta}_{0})_{q}\bigg)\Big(\frac{\partial \mP}{\partial x_{qk}}\Big)^2\mP^{l-3} \nonumber\\
&+O_\prec(N^{-\frac12}), \label{17092052}
\end{align}
where in the last step we bounded the $a=1$ terms by  $O_\prec(N^{-\frac12})$ since 
\begin{align*}
&\bigg|\frac{m_1}{2N^{\frac{3}{2}}\Delta(d_i)} \sum_{s=1}^2 \wt c_s\sum_{q,k}y_{0q}\sum_{\begin{subarray}{c} s_1,s_2\geq 1\\s_1+s_2=s+1\end{subarray}}(X^*\mG_1^{s_1})_{kq}(X^*\mG_1^{s_2}\mb{\vartheta}_{0})_{kq}\Big(\frac{\partial \mP}{\partial x_{qk}}\Big)^2\mP^{l-3}\bigg|\\
&\prec N^{-1} \Delta(d_i)^{-1}|m_1| \sum_{s=1}^2\sum_{\begin{subarray}{c} s_1,s_2\geq 1\\s_1+s_2=s+1\end{subarray}} |\wt c_s| \sum_{k}\big|(X^*\mG_1^{s_1})_{kq}\big|\big|(X^*\mG_1^{s_2}\mb{\vartheta}_{0})_{kq}\big|\\
&\prec N^{-1}(d_i-\sqrt y)^{-\frac 12}d_i^{\frac 12}\prec  N^{-\frac 12}.
\end{align*}
Here we used the  $O_\prec(1)$ bound for both $\mP$ and $\partial\mP/\partial x_{qk}$ (\cf (\ref{19072410})) for the first step, \eqref{weak_est_XG} for the second step and \rf{asd}, \rf{Delta:asymp} for the last step.

Further, we have
\begin{align} \label{2020051201}
&\frac{m_1}{2N^{\frac{3}{2}}\Delta(d_i)} \sum_{s=1}^2 \wt c_s \sum_{q,k}y_{0q}\bigg((\mG_1^s \mb{\vartheta}_{0})_{q}-\sum_{\begin{subarray}{c} s_1,s_2\geq 1\\s_1+s_2=s+1\end{subarray}}(X^*\mG_1^{s_1}X)_{kk}(\mG_1^{s_2}\mb{\vartheta}_{0})_{q}\bigg)\Big(\frac{\partial \mP}{\partial x_{qk}}\Big)^2\mP^{l-3} \nonumber\\
&= \frac{m_1}{2N^{\frac{3}{2}}\Delta(d_i)} \sum_{q,k}y_{0q}
{\vartheta}_{0q}\Big( \wt c_1 (-zm_2m_1) + \wt c_2 (-zm_2m_1)'\Big)\Big(\frac{\partial \mP}{\partial x_{qk}}\Big)^2\mP^{l-3}  \notag\\
&\quad+ O_\prec(N^{-\frac12})
\end{align}
by using \rf{weak_est_DG}. Recall \rf{def:wtc12} and the explicit expressions of $\wt c_{1,2}$. Elementary computations indicate 
\begin{align} \label{cancel:wtc12}
\wt c_1 (-zm_2m_1) + \wt c_2 (-zm_2m_1)' =-\frac{y(d_i^2+2d_i+y)}{d_i^2 (d_i+y)}
=O(d_i^{-1}),
\end{align}
 by which we can further estimate the RHS of \rf{2020051201} as
 \begin{align} \label{2020051202}
 |m_1|N^{-\frac 12} \Delta({d_i})^{-1}\Vert \mb{y}_0\Vert\, \Vert \mb{\vartheta}_0 \Vert d_i^{-1} \Big|\frac{\partial \mP}{\partial x_{qk}} \Big| ^2=O_\prec(N^{-\frac 12}). 
 \end{align}
Combining \rf{17092052}, \rf{2020051201} and \rf{2020051202}, we proved \rf{17092051}. And this completes the estimate  of $h_0(1,2)$.

 For the case $t=1,2$, by similar arguments, we will have 
 \begin{align*}
 &h_{t}(1,2)\nonumber\\
 &=\frac{\kappa_4}{2N^{\frac{3}{2}}}\frac{m_1}{\Delta(d_i)} \sum_{q,k}y_{tq}\frac{\partial (X^*\mG_1\mb{\vartheta}_{t} )_{k}}{\partial x_{qk}}\Big((l-1) \frac{\partial^2 \mP}{\partial x_{qk}^2}\mP^{l-2} + (l-1)(l-2)\Big(\frac{\partial \mP}{\partial x_{qk}}\Big)^2\mP^{l-3} \Big)\nonumber\\
 &=- \frac{\kappa_4(l-1)}{\Delta(d_i)^2} zm_2m_1^2
 \bigg(\sum_{b=1}^2 c_b \frac{(zm_2m_1^2)^{(b-1)}}{(b-1)!}
 \sum_{q} y_{tq} y_{0q} {\vartheta}_{tq}  {\vartheta}_{0q}
 \nonumber\\
 &\qquad+\sum_{s=1}^2 zm_2m_1^2 \sum_{q} y_{tq} y_{sq} {\vartheta}_{tq}  {\vartheta}_{sq}
\bigg)\mP^{l-2}  + O_\prec (N^{-\frac 12}).
 \end{align*} 
 We omit the details here since it follows the same arguments as for $h_0(1,2)$.

In the sequel,  we prove (2) of Lemma \ref{rankrlemh_ts}, i.e.,   we show that except for (\ref{rankrest:h0(1,0)})-(\ref{rankrestht(1,2)}), all the other terms $h_{t}(\alpha_1,\alpha_2)$ with $\alpha_1+\alpha_2\le 3$ can be bounded by $O_\prec(N^{-\frac12})$.  We start with the case when $\alpha_1+\alpha_2=2$, i.e. $(\alpha_1,\alpha_2)=(2,0),(0,2),(1,1)$. 
First,  by (\ref{eq:2ndXG}), we have 
\begin{align}\label{eq:hst20}
h_{0}(2,0)&= \frac{\kappa_3 m_1}{2N\Delta(d_i)} \sum_{s=1}^2\wt c_s\sum_{q,k}y_{0q}\frac{\partial^2 (X^*\mG_1^s\mb{\vartheta}_{0})_{k}}{\partial x_{qk}^2}\mP^{l-1}\nonumber\\
&=\frac{\kappa_3m_1 }{N\Delta(d_i)}  \sum_{s=1}^2\wt c_s\sum_{q,k}y_{0q}\bigg(\sum_{a_1,a_2=1}^2\sum_{\begin{subarray}{c}s_1,s_2,s_3\geq 1\\s_1+s_2+s_3=s+2\end{subarray}}\Big(X^*\mG_1^{s_1}\mathscr{P}_{a_1}^{qk}\mG_1^{s_2}\mathscr{P}_{a_2}^{qk}\mG_1^{s_3}\mb{\vartheta}_{0}\Big)_{k}\nonumber\\
&\quad -\sum_{a=1}^2\sum_{\begin{subarray}{c}s_1,s_2\geq 1\\s_1+s_2=s+1\end{subarray}}\Big(\mG_1^{s_1}\mathscr{P}_{a}^{qk}\mG_1^{s_2}\mb{\vartheta}_{0}\Big)_{q}-\sum_{\begin{subarray}{c}s_1,s_2\geq 1\\s_1+s_2=s+1\end{subarray}}\Big(X^*\mG_1^{s_1}\mathscr{P}_0^{qk}\mG_1^{s_2}\mb{\vartheta}_{0}\Big)_{k}\bigg)\mP^{l-1}.
\end{align}
 Note that each term above contains at least one quadratic form of $X^*\mG_1^a$ as a factor,  for some $a\ge 1$. This fact eventually leads to the $O_\prec(N^{-\frac12}d_i^{- 1})$ bound for all the terms above, by the isotropic local law.  More specifically, plugging the definitions in (\ref{P012}) into (\ref{eq:hst20}) and taking the sums, we can see that  the RHS of (\ref{eq:hst20}) is a linear combination of the terms of the following forms 
\begin{align}
&N^{-1} \frac{\wt c_s m_1}{\Delta(d_i)} \sum_{q,k}y_{0q}(X^*\mG_1^{s_1})_{kq}(X^*\mG_1^{s_2})_{kq}(X^*\mG_1^{s_3}\mb{\vartheta}_{0})_k \mP^{l-1}, \label{17092071}\\
&N^{-1}\frac{ \wt c_s m_1}{\Delta(d_i)} \sum_{q,k}y_{0q}\big((X^*\mG_1^{s_1}X)_{kk}\big)^a(X^*\mG_1^{s_2})_{kq}(\mG_1^{s_3}\mb{\vartheta}_{0})_q \mP^{l-1},\label{17092072}\\ 
&N^{-1}\frac{\wt c_s m_1}{\Delta(d_i)} \sum_{q,k}y_{0q}\big((X^*\mG_1^{s_1}X )_{kk}\big)^a(\mG_1^{s_2})_{qq}(X^*\mG_1^{s_3}\mb{\vartheta}_{0})_{k} \mP^{l-1}, \qquad a=0,1. \label{17092070}
\end{align}
where $s=1,2$ and $s= s_1+s_2+s_3-2$ for \rf{17092071} and $s=a(s_1-1)+s_2+s_3-1$
 for \rf{17092072} and \rf{17092070}. 
 
It suffices to estimate \eqref{17092071}-\eqref{17092070}. First, notice that $|\mP^{l-1}|=O_\prec(1)$,  $|m_1|=O_\prec(1/d_i)$ and the quadratic forms of $(X^*\mG_1)$ and $(X^*\mG_1^2)$ can be bounded by $O_\prec(N^{-\frac12}(d_i-\sqrt y)^{-\frac 12}d_i^{-\frac 12})$ and $O_\prec(N^{-\frac12}(d_i-\sqrt y)^{-\frac 52}d_i^{\frac 12}) $ respectively.   By the Cauchy Schwarz inequality, $\sum_{q}|y_{tq}|=O_\prec (\sn)$. With the expression of $\wt c_{1,2}$ in \rf{def:wtc12},  one see that the  term in (\ref{17092071}) is bounded by $O_\prec(N^{-1}(d_i-\sqrt y)^{-1})$.  Second, combining the following bounds
\begin{align*}
&|\mP^{l-1}|=O_\prec(1), \quad |(X^*\mG_1X)_{kk}|=O_\prec(d_i^{-1}), \notag\\
& |(X^*\mG_1^2X)_{kk}|= O_\prec(d_i^{-1}(d_i-\sqrt y)^{-1}),\\
&|(X^*\mG_1)_{kq}|= O_\prec(N^{-\frac12}(d_i-\sqrt y)^{-\frac 12}d_i^{-\frac 12}), \notag\\
& |(X^*\mG_1^2)_{kq}|=O_\prec(N^{-\frac12}(d_i-\sqrt y)^{-\frac 52}d_i^{\frac 12}),
\end{align*}
together with the estimate of $\Delta(d_i)$ in \rf{Delta:asymp}, we have for $a=1$,
\begin{align}
(\ref{17092072})&=O_\prec\Big(N^{-\frac12} \wt c_s d_i^{s_2-2} (d_i-\sqrt y)^{-(s_1-1)-2(s_2-1)}\sum_{q}\big|y_{tq}(\mG_1^{s_3}\mb{\vartheta}_{\varphi(t)})_q\big|\Big)\nonumber\\
&=O_\prec(N^{-\frac12} \wt c_s d_i^{s_2-3} (d_i-\sqrt y)^{-(s-1)-(s_2-1)}) = O_\prec(d_i^{-1}N^{-\frac12}).
\end{align}
In the last two steps above, we used \rf{Importestsum1}, $\wt c_1= O(d_i)$ and $\wt c_2= O((d_i-\sqrt y)^{2})$. More specifically, if $s=1$, then $s_2=1$, $\wt c_s d_i^{s_2-3} (d_i-\sqrt y)^{-(s-1)-(s_2-1)}= O(d_i^{-1})$; if $s=2$, we have $\wt c_s d_i^{s_2-3} (d_i-\sqrt y)^{-(s-1)-(s_2-1)}=O(d_i^{-(3-s_2)}  (d_i-\sqrt y)^{2-s_2})< O(d_i^{-1})$.  The case of $a=0$ follows similarly as
\begin{align}
(\ref{17092072})=O_\prec\Big(N^{-\frac12} \wt c_s d_i^{s_2-1} (d_i-\sqrt y)^{-2(s_2-1)}\sum_{q}\big|y_{tq}(\mG_1^{s_3}\mb{\vartheta}_{\varphi(t)})_q\big|\Big)
 = O_\prec(N^{-\frac12}).
\end{align}
Third, by Cauchy-Schwarz inequality and \eqref{weak_est_DG}, we have
\begin{align}\label{eq:middleCS}
&\Big|N^{-1}\frac{\wt c_s m_1}{\Delta(d_i)} \sum_{q,k}y_{0q}\big((X^*\mG_1^{s_1}X )_{kk}\big)^a(\mG_1^{s_2})_{qq}(X^*\mG_1^{s_3}\mb{\vartheta}_{0})_{k}\Big|\nonumber\\
&\leq N^{-1}\frac{\wt c_s|m_1|}{\Delta(d_i)} \|\mathbf{y}_0\|\Big(\sum_{q}(\mG_1^{s_2})_{qq}^2\Big)^{1/2}\Big|\sum_{k}\big((X^*\mG_1^{s_1}X )_{kk}\big)^a(X^*\mG_1^{s_3}\mb{\vartheta}_{0})_{k}\Big|\nonumber\\
&\prec N^{-\frac12}\wt c_s d_i^{-2} (d_i-\sqrt y)^{-(s_2-1)}\Delta(d_i)^{-1}\Big|\sum_{k}\big((X^*\mG_1^{s_1}X )_{kk}\big)^a(X^*\mG_1^{s_3}\mb{\vartheta}_{0})_{k}\Big| \notag\\
&\prec N^{-\frac12}.
\end{align}
In the last step above,  we get the $O_\prec(d_i^{-a-\frac12+s_3-1}\delta_{i0}^{-\frac12-a(s_0-1)-2(s_3-1)})$ bound for the term $\Big|\sum_{k}\big((X^*\mG_1^{s_1}X )_{kk}\big)^a(X^*\mG_1^{s_3}\mb{\vartheta}_{0})_{k}\Big|$, similarly to  \eqref{Importestsum2}. Hence, we conclude $h_{0}(2,0)=O_\prec(N^{-\frac12})$. Similarly, for the cases of $t=1,2$, the estimates of $h_{t}(2,0)$ are reduced to those of the forms \rf{17092071}-\rf{17092070} only with the coefficients $\wt c_s$ removed. The computations turn out to be even easier; the details are omitted. Therefore, we also have $h_{t}(2,0)=O_\prec(N^{-\frac12})$ for $t=1,2$.

Next, in the case of  $(\alpha_1,\alpha_2)=(0,2)$, by the definition in (\ref{rankrdefhts}), we have 
\begin{align}
 h_{t}(0,2)
 &=(l-1)(l-2)\frac{\kappa_3 m_1}{2N\Delta(d_i)}\sum_{q,k}y_{tq}(X^*T_t(\mG_1)\mb{\vartheta}_{t})_{k}\Big(\frac{\partial\mP}{\partial x_{qk}}\Big)^2 \mP^{l-3}\nonumber\\
 &\qquad+(l-1)\frac{\kappa_3 m_1 }{2N\Delta(d_i)}\sum_{q,k}y_{tq}(X^*T_t(\mG_1)\mb{\vartheta}_{t})_{k}\frac{\partial^2\mP}{\partial x_{qk}^2} \mP^{l-2}. \label{rankrhts(0,2)}
\end{align}
We then use the formula  in \rf{eq:Pd1}. After expanding the first term in the RHS of \rf{rankrhts(0,2)}, one notices that it can be written as a linear combination of the terms of the  forms (by ignoring prefactor $m_1/\Delta(d_i)^3$)
\begin{align}
&\sum_{q}y_{tq}(\mG_1\mb{\eta}_1)_{q}(\mG_1\mb{\psi}_1)_{q}\sum_{k}(X^*T_t(\mG_1)\mb{\vartheta}_{t})_{k}(X^*\mG_1\mb{\eta}_2)_{k}(X^*\mG_1\mb{\psi}_2)_{k}\mP^{l-3},  \label{19072080} \\
& \sum_{q,k}y_{tq} (X^*T_t(\mG_1)\mb{\vartheta}_{t})_{k} \sum_{a=1}^2 c_a \sum_{\begin{subarray} {c} a_1,a_2\geq 1\\ a_1+a_2=a+1 
\end{subarray}}(\mG_1^{a_1}\mb{\eta}_{1})_{q} (X^*\mG_1^{a_2}\mb{\eta}_2)_{k}\notag\\
& \qquad\qquad\qquad\qquad\times\sum_{b=1}^2 c_b \sum_{\begin{subarray} {c} b_1,b_2\geq 1\\ b_1+b_2=b+1 
\end{subarray}}(\mG_1^{b_1}\mb{\psi}_{1})_{q} (X^*\mG_1^{b_2}\mb{\psi}_2)_{k},
  \label{19072080*} \\
& \sum_{q,k}y_{tq} (X^*T_t(\mG_1)\mb{\vartheta}_{t})_{k} \Big(\sum_{a=1}^2 c_a \sum_{\begin{subarray} {c} a_1,a_2\geq 1\\ a_1+a_2=a+1 
\end{subarray}}(\mG_1^{a_1}\mb{\eta}_{0})_{q} (X^*\mG_1^{a_2}\mb{\eta}_0)_{k}\Big)
 (\mG_1\mb{\psi}_{s})_{q} (X^*\mG_1\mb{\psi}_s)_{k},  \notag\\
 &\qquad\qquad\qquad s=1,2.
 \label{19072080**} 
\end{align}
where the vectors $\mb{\eta}_\alpha,{\mb{\psi}}_\alpha$ $(\alpha=1,2)$ take $\mb{\vartheta}_{s}$ or $\mb{y}_t$ for $s,t=0,1,2$.
We claim that  the above forms  (\ref{19072080})- \rf{19072080**} are bounded by $O_\prec(N^{-\frac12}d_i^{-\frac 72}(d_i-\sqrt y)^{-\frac 32} )$ for all choices of $\mb{\eta}_\alpha,{\mb{\psi}}_\alpha$ listed above. To see this, we first notice that  the isotropic local law  \eqref{weak_est_XG} implies
\begin{align}
\sum_{k}(X^*\mG_1\mb{\vartheta}_{t})_{k}(X^*\mG_1\mb{\eta}_2)_{k}(X^*\mG_1\mb{\psi}_2)_{k}=O_\prec(N^{-\frac 12} (d_i-\sqrt y)^{-\frac 32}d_i^{-\frac 32 } ). \label{19072081}
\end{align}
Plugging in the expressions of $\wt c_{1,2}$, we see that  if $d_i\leq K$ with sufficiently large constant $K>0$,
\begin{align*}
(X^*T_0(\mG_1)\mb{\vartheta}_{0})_{k}=\sum_{s=1}^2 \wt c_s (X^*\mG_1^s \mb{\vartheta}_{0})_{k} = O_\prec(N^{-\frac 12}  (d_i-\sqrt y)^{-\frac 12} d_i^{-\frac 12}).
\end{align*} 
If  $d_i > K$ with sufficiently large constant $K>0$, expanding $\mG_1, \mG_1^2$ around $-1/\theta(d_i)$ and $1/(\theta(d_i))^2$, respectively, we obtain 
\begin{align} \label{2020052002}
(X^*T_0(\mG_1)\mb{\vartheta}_{0})_{k}&=\sum_{s=1}^2 \wt c_s (X^*\mG_1^s \mb{\vartheta}_{0})_{k} \notag\\
&= \theta(d_i)\Big(X^*\big(\mG_1+ \theta(d_i)\mG_1^2\big)\mb{\vartheta}_{0}\Big)_{k} + O_\prec (N^{-\frac 12}  (d_i-\sqrt y)^{-\frac 12} d_i^{-\frac 12}) \nonumber\\
&= \frac{(X^*H\mb{\vartheta}_{0})_{k}}{\theta(d_i)} + O_\prec(N^{-\frac 12}d_i^{-1}) = O_\prec(N^{-\frac 12}d_i^{-1}),
\end{align}
where in the last step, we used $(X^*H^\alpha\mb{\vartheta}_{0})_{k} = O_\prec(N^{-1/2})$ for $\alpha\in \mathbb{Z}^+$ which can be proved simply by moment method.

Combining the two cases together with the isotropic local law for $(X^*\mG_1\mb{\vartheta}_{t})_{k}$, we get the estimate 
\begin{align}\label{est:XT012}
(X^*T_t(\mG_1)\mb{\vartheta}_{t})_{k}= O_\prec(N^{-\frac 12}  (d_i-\sqrt y)^{-\frac 12} d_i^{-\frac 12}), \quad \text{ for } t=0,1,2.
\end{align}
This implies 
\begin{align}  \label{2020052001}
\sum_{k}(X^*T_0(\mG_1)\mb{\vartheta}_{0})_{k}(X^*\mG_1\mb{\eta}_2)_{k}(X^*\mG_1\mb{\psi}_2)_{k}=O_\prec(N^{-\frac 12} (d_i-\sqrt y)^{-\frac 32}d_i^{-\frac 32 } ). 
\end{align}
Further, from \rf{Importestsum1}, we have 
\begin{align}
\sum_{q}y_{tq}(\mG_1\mb{\eta}_1)_{q}(\mG_1\mb{\psi}_1)_{q}=O_\prec(d_i^{-2} ). \label{19072082}
\end{align} 
Combining \rf{19072081}, \rf{2020052001},  (\ref{19072082}), and the fact $|\mP|\prec 1$, we easily get $$\rf{19072080}=O_\prec(N^{-\frac12}d_i^{-\frac 72}(d_i-\sqrt y)^{-\frac 32}) \quad \text{for $t=0,1,2$}.$$ 

For  \rf{19072080*} and \rf{19072080**}, the coefficients $c_{1,2}$ contain powers of $d_i$ thus we need to handle them more carefully in the case of large $d_i$.  Notice that if $d_i=O(1)$ then $c_{1,2}= O(1)$, hence similarly  to the estimates \rf{19072081}, \rf{2020052001} and (\ref{19072082}), we can easily derive the crude bound $O_\prec(N^{-\frac12}d_i^{-\frac 72}(d_i-\sqrt y)^{-\frac 32})$ for \rf{19072080*} and \rf{19072080*} by using the isotropic local law. In the sequel, we only state the details  for the case $d_i>K$ with sufficiently large constant $K>0$. First by plugging in the expressions of $c_{1,2}$ and recalling \rf{2020052002}, we see that  
\begin{align} \label{2020052003}
&\sum_{a=1}^2 c_a \sum_{\begin{subarray} {c} a_1,a_2\geq 1\\ a_1+a_2=a+1 
\end{subarray}}(\mG_1^{a_1}\mb{\eta}_{1})_{q} (X^*\mG_1^{a_2}\mb{\eta}_2)_{k} \nonumber\\
&= \Big(\frac {c_1}{2}  (X^*\mG_1\mb{\eta}_2)_{k} +  c_2(X^*\mG_1^{2}\mb{\eta}_2)_{k}\Big) (\mG_1\mb{\eta}_{1})_{q}
 +   \Big(\frac {c_1}{2}   (\mG_1\mb{\eta}_{1})_{q}+  c_2 (\mG_1^2\mb{\eta}_{1})_{q}\Big)   (X^*\mG_1\mb{\eta}_2)_{k} \nonumber\\
 &= O_\prec(N^{-\frac 12}d_i^{-1})m_1 \eta_{1q}  +   \Big(\frac {c_1}{2}   m_1+  c_2m_1'\Big) \eta_{1q}  O_\prec(N^{-\frac 12}d_i^{-1}) + O_\prec(N^{-\frac 12}d_i^{-2}) \nonumber\\
 &= O_\prec(N^{-\frac 12}d_i^{-2})  \eta_{1q},
\end{align}
where we use the fact that $\frac {c_1}{2}   m_1+  c_2m_1'= O(d_i^{-1})$. Then by \rf{2020052003} and \rf{est:XT012}, we obtain that 
\begin{align}
\rf{19072080*} = O_\prec (N^{-\frac 12}{ d_i^{-5}}) \Big(\sum_{q} y_{tq}\eta_{1q } \psi_{1q} \Big) = O_\prec (N^{-\frac 12 }d_i^{-5} ),
\end{align}
which follows from Cauchy-Schwarz inequality for $\sum_{q} y_{tq}\eta_{1q } \psi_{1q}$. Analogously, we also have $\rf{19072080**} = O_\prec (N^{-\frac 12 }d_i^{-5} )$. Note that we get the bound $O_\prec (N^{-\frac 12 }d_i^{-5} )$ in the large $d_i$ case, which can be combined together with the case $d_i=O(1)$ to get the crude bound $O_\prec(N^{-\frac12}d_i^{-\frac 72}(d_i-\sqrt y)^{-\frac 32})$ for \rf{19072080*} and \rf{19072080**}. Now, multiplying with the prefactor $m_1/\Delta(d_i)^3$ and using $m_1=O(d_i^{-1})$, we finally get the first term on the RHS of \rf{rankrhts(0,2)} is of order $O_\prec (N^{-1/2})$.

 Analogously,  using the formula in  \rf{eq:Pd2}, it is easy to see that the second term in the RHS of \rf{rankrhts(0,2)} is a linear combination of the terms of the following forms containing coefficients $ c_{1,2}$
  \begin{align}
&N^{-\frac12}\frac{m_1}{\Delta(d_i)^2}\sum_{b=1}^2 c_b \sum_{\begin{subarray}{c}
b_1,b_2\geq 1\\ b_1+b_2=b+1
\end{subarray}
}\Big(\sum_{q}y_{tq}(\mG_1^{b_1}\mb{\eta}_1)_q(\mG_1^{b_2}\mb{\eta}_2)_q\Big)\notag\\
&\qquad\qquad\qquad\qquad\times\Big(\sum_{k}(X^*T_t(\mG_1)\mb{\vartheta}_{t})_{k}\Big)\mP^{l-2},\label{eq:h02-1}\\
&N^{-\frac12}\frac{m_1}{\Delta(d_i)^2}\sum_{b=1}^2 c_b \sum_{\begin{subarray}{c}
b_1,b_2,b_3\geq 1\\ b_1+b_2+b_3=b+1
\end{subarray}
}\Big(\sum_{q}y_{tq}(\mG_1^{b_1}\mb{\eta}_1)_q(\mG_1^{b_2}\mb{\eta}_2)_q\Big)\notag\\
&\qquad\qquad\qquad\qquad\times\Big(\sum_{k}(X^*\mG_1^{b_3}X)_{kk}(X^*T_t(\mG_1)\mb{\vartheta}_{t})_{k}\Big)\mP^{l-2},\label{eq:h02-11}\\
&N^{-\frac12} \frac{m_1}{\Delta(d_i)^2} \sum_{b=1}^2 c_b \sum_{\begin{subarray}{c}
b_1,b_2,b_3\geq 1\\ b_1+b_2+b_3=b+1
\end{subarray}
}
\Big(\sum_{q}y_{tq}(\mG_1^{b_1}\mb{\eta}_1)_q\Big)\notag\\
&\qquad\qquad\qquad\qquad\times\Big(\sum_{k}(X^*\mG_1^{b_2})_{kq}(X^*\mG_1^{b_3}\mb{\eta}_2)_k(X^*T_t(\mG_1)\mb{\vartheta}_{t})_{k}\Big)\mP^{l-2},\label{eq:h02-2}\\
&N^{-\frac12} \frac{m_1}{\Delta(d_i)^2} \sum_{b=1}^2 c_b \sum_{\begin{subarray}{c}
b_1,b_2,b_3\geq 1\\ b_1+b_2+b_3=b+1
\end{subarray}
}
\Big(\sum_{q}y_{tq}(\mG_1^{b_1})_{qq}\Big)\notag\\
&\qquad\qquad\qquad\qquad\times\Big(\sum_{k}(X^*\mG_1^{b_2}\mb{\eta}_1)_{k}(X^*\mG_1^{b_3}\mb{\eta}_2)_k(X^*T_t(\mG_1)\mb{\vartheta}_{t})_{k}\Big)\mP^{l-2}, \label{eq:h02-3}
\end{align}
for $\mb{\eta}_0,\mb{\eta}_1, \mb{\eta}_2=\mb{\vartheta}_{1}, \mb{\vartheta}_{2},\mb{y}_1, \mb{y}_2, \mb{y}_3$, $a=0,1$ and $b_1,b_2,b_3=1$ or 2 as well as the following forms  irrelevant to $c_{1,2}$,
  \begin{align}
&N^{-\frac12}\frac{m_1}{\Delta(d_i)^2}\Big(\sum_{q}y_{tq}(\mG_1\mb{\eta}_1)_q(\mG_1\mb{\eta}_2)_q\Big)\Big(\sum_{k}(X^*T_t(\mG_1)\mb{\vartheta}_{t})_{k}\Big)\mP^{l-2},\label{eq:h02-1*}\\
&N^{-\frac12}\frac{m_1}{\Delta(d_i)^2}\Big(\sum_{q}y_{tq}(\mG_1\mb{\eta}_1)_q(\mG_1\mb{\eta}_2)_q\Big)\Big(\sum_{k}(X^*\mG_1X)_{kk}(X^*T_t(\mG_1)\mb{\vartheta}_{t})_{k}\Big)\mP^{l-2},\label{eq:h02-11*}\\
&N^{-\frac12} \frac{m_1}{\Delta(d_i)^2} \Big(\sum_{q}y_{tq}(\mG_1\mb{\eta}_1)_q\Big)\Big(\sum_{k}(X^*\mG_1)_{kq}(X^*\mG_1\mb{\eta}_2)_k(X^*T_t(\mG_1)\mb{\vartheta}_{t})_{k}\Big)\mP^{l-2},\label{eq:h02-2*}\\
&N^{-\frac12} \frac{m_1}{\Delta(d_i)^2} \Big(\sum_{q}y_{tq}(\mG_1)_{qq}\Big)\Big(\sum_{k}(X^*\mG_1\mb{\eta}_1)_{k}(X^*\mG_1\mb{\eta}_2)_k(X^*T_t(\mG_1)\mb{\vartheta}_{t})_{k}\Big)\mP^{l-2}, \label{eq:h02-3*}
\end{align}
for $\mb{\eta}_0,\mb{\eta}_1, \mb{\eta}_2=\mb{\vartheta}_{1}, \mb{\vartheta}_{2},\mb{y}_1, \mb{y}_2, \mb{y}_3$, $a=0,1$ and $b_1,b_2,b_3=1$ or 2.

 We claim that all of the above terms can be bounded crudely by $O_\prec(N^{-\frac 12} (d_i-\sqrt y)^{-\frac12} d_i^{\frac 12})$. First, recall the bound $|\mathcal{P}|\prec 1$.  The $O_\prec(N^{-\frac12}(d_i-\sqrt y)^{\frac 12}d_i^{-\frac12})$ bounds for \eqref{eq:h02-1*} and \eqref{eq:h02-11*} follow directly from $|m_1|\sim 1/d_i$, \rf{Delta:asymp}, \rf{Importestsum1} and \rf{Importestsum2}. From \eqref{Importestsum1} and the isotropic local law \eqref{weak_est_XG}, together with \rf{est:XT012}, we see that \eqref{eq:h02-2*} is  bounded by $O_\prec(N^{-1}d_i^{-\frac 12} (d_i-\sqrt y)^{-\frac 12})$. The estimate of \eqref{eq:h02-3*} is similar to that of \eqref{eq:h02-2*} by using the Cauchy-Schwarz inequality and the isotropic local law \rf{weak_est_XG}. We thus omit the details.
 
  In the sequel, we bound the terms  \rf{eq:h02-1}-\rf{eq:h02-3}. The estimates for \rf{eq:h02-11}-\rf{eq:h02-3} are the same as those of \rf{eq:h02-11*}-\rf{eq:h02-3*} by using isotropic local laws, \rf{est:XT012} and  the bounds $c_1=O(d_i), c_2= O((d_i-\sqrt y)^2), m_1=O(d_i^{-1})$. We observe that each summand for $b=1,2$ in \rf{eq:h02-11}-\rf{eq:h02-3} is crudely of order $O_\prec(N^{-\frac 12})$, hence we can conclude the $O_\prec(N^{-\frac 12})$ bounds for \rf{eq:h02-11}-\rf{eq:h02-3}.  As for \rf{eq:h02-1}, in the case that $d_i\leq K$,  using the same arguments as for  \rf{eq:h02-1*},  one gets the bound $O_\prec(N^{-\frac 12})$. We only need to show the case that  $d_i>K$, and we will observe the same cancellation as \rf{2020052003}. More specifically, for \rf{eq:h02-1}, notice that 
 \begin{align} \label{2020052403}
\sum_{b=1}^2 c_b \sum_{\begin{subarray}{c}
b_1,b_2\geq 1\\ b_1+b_2=b+1
\end{subarray}}
&\Big(\sum_{q}y_{tq}(\mG_1^{b_1}\mb{\eta}_1)_q(\mG_1^{b_2}\mb{\eta}_2)_q\Big)\notag\\
&= \sum_{q}y_{tq}\eta_{1q}\eta_{2q} \Big(c_1m_1^2+ 2c_2 m_1m_1'\Big) +O_\prec(N^{-\frac 12} d_i^{-2})\nonumber\\
&= O_\prec(d_i^{-2}).
 \end{align}
 This together with \rf{Importestsum2} leads to the bound $O_\prec(N^{-\frac 12})$ for \rf{eq:h02-1}.
  Hence, the second term in the RHS of \rf{rankrhts(0,2)} is also of order $O_\prec(N^{-\frac12})$. 
This together with the same bound for the first term in the RHS of \rf{rankrhts(0,2)} leads to $$h_{t}(0,2)=O_\prec\Big(N^{-\frac12} \Big), \text { for $t=0,1,2$}.$$

Next, we turn to  $h_{t}(1,1)$.  By the definition in (\ref{rankrdefhts}), we have 
\begin{align}
{h}_{t}(1,1)=\frac{\kappa_3m_1}{N\Delta(d_i)}\sum_{q,k}y_{tq}\frac{\partial (X^*T_t(\mG_1)\mb{\vartheta}_{t})_{k}}{\partial x_{qk}}\frac{\partial \mP}{\partial x_{qk}}\mP^{l-2}.\label{rankrhts(1,1)}
\end{align}
Using \rf{deriXG^l} and \rf{eq:Pd1}, for $t=0$, we can write
\begin{align} \label{def:h0(1,1)}
&h_{0}(1,1)\notag\\
&=-\frac{\kappa_3m_1}{\sn\Delta(d_i)^2}\sum_{q,k}y_{0q}\Big( \big(T_0(\mG_1)\mb{\vartheta}_{0} \big)_{q}-\sum_{s=1}^2 \wt c_s \sum_{a=1}^2 \sum_{\begin{subarray}{c}s_1,s_2\geq 1\\s_1+s_2=s+1\end{subarray}}
(X^*\mG_1^{s_1}\mP^{qk}_{a}\mG_1^{s_2} \mb{\vartheta}_{0} )_k \Big)
\nonumber\\
&\qquad \times \bigg( \sum_{a=1}^2 c_a \sum_{\begin{subarray}{c}a_1,a_2\geq 1;\\ a_1+a_2=a+1 \end{subarray}}\Big((\mb{y}_0^*\mG_1^{a_1})_q(X^*\mG_1^{a_2}\mb{\vartheta}_0)_k+(X^*\mG_1^{a_1}\mb{y}_0)_k(\mG_1 ^{a_2}\mb{\vartheta}_0)_q\Big)\nonumber\\
&\qquad + \sum_{t=1}^2 \Big((\mb{y}_t^*\mG_1)_q(X^*\mG_1\mb{\vartheta}_{t})_k+(X^*\mG_1\mb{y}_t)_k(\mG_1\mb{\vartheta}_{t})_q\Big) \bigg)\mP^{l-1}.
\end{align}
Note that by denoting  
\begin{align} \label{def:wtTG}
\wt T_0(\mG_1) :=\frac {c_1}{2} \mG_1 + c_2 \mG_1^2, \quad \wt T_1(\mG_1) = \mG_1,
\end{align}
we can represent the first term of ${\partial \mP}/{\partial x_{qk}}$ as 
\begin{align*}
&\sum_{a=1}^2 c_a \sum_{\begin{subarray}{c}a_1,a_2\geq 1;\\ a_1+a_2=a+1 \end{subarray}}\Big((\mb{y}_0^*\mG_1^{a_1})_q(X^*\mG_1^{a_2}\mb{\vartheta}_0)_k+(X^*\mG_1^{a_1}\mb{y}_0)_k(\mG_1 ^{a_2}\mb{\vartheta}_0)_q\Big)\nonumber\\
&=\sum_{\begin{subarray}{c}
t_1,t_2\geq 0\\ t_1+t_2=1 \end{subarray}} \Big[
(\mb{y}_0^*\wt T_{t_1}(\mG_1))_q(X^*\wt T_{t_2}(\mG_1)\mb{\vartheta}_0)_k + 
(\mb{\vartheta}_0^*\wt T_{t_1}(\mG_1))_q(X^*\wt T_{t_2}(\mG_1)\mb{y}_0)_k  \Big]. 
\end{align*}
Hence, it is easy to see that the RHS of \rf{def:h0(1,1)} is a linear combination of the terms of  the following forms
\begin{align} \label{terms:h0(1,1)}
&N^{-\frac12} \frac{m_1}{\Delta(d_i)^2}\Big(\sum_{q}y_{tq} \big(T_0(\mG_1)\mb{\vartheta}_{0} \big)_{q} \big(\wt T_{t_1}(\mG_1)\mb{\eta}_1 \big)_{q}\Big)\Big(\sum_k \big(X^*\wt T_{t_2}(\mG_1)\mb{\eta}_2 \big)_{k}\Big)\mP^{l-1},\nonumber\\
&N^{-\frac12} \frac{m_1 \wt c_s}{\Delta(d_i)^2}\Big(\sum_{q}y_{tq} \big(\mG_1^{s_1}\mb{\vartheta}_{0} \big)_{q} \big(\wt T_{t_1}(\mG_1)\mb{\eta}_1 \big)_{q}\Big)\Big(\sum_k 
(X^*\mG_1^{s_2}X)_{kk}\big(X^*\wt T_{t_2}(\mG_1)\mb{\eta}_2 \big)_{k}\Big)\mP^{l-1},\nonumber\\
&N^{-\frac12} \frac{m_1 \wt c_s}{\Delta(d_i)^2}\Big(\sum_{q}y_{tq}\big(\wt T_{t_1}(\mG_1)\mb{\eta}_1 \big)_{q}\Big)\Big(\sum_k(X^*\mG_1^{s_1})_{kq}(X^*\mG_1^{s_2}\mb{\vartheta}_{0})_k\big(X^*\wt T_{t_2}(\mG_1)\mb{\eta}_2 \big)_{k}\Big)\mP^{l-1},
\end{align}
where $t_1,t_2=0,1 $, $s_1, s_2=1,2$ satisfying $s= s_1+s_2- 1\in \{1,2\}$ and $\mb{\eta}_1,\mb{\eta}_2= \mb{y}_0, \mb{y}_1, \mb{y}_2, \mb{\vartheta}_0, \mb{\vartheta}_1, \mb{\vartheta}_2$. Recall the definitions of $\wt c_{1,2}$ and $c_{1,2}$ (c.f. (\ref{def:c12}), (\ref{def:wtc12})). We can write 
\begin{align} \label{2020052401}
\big(X^*\wt T_{0}(\mG_1)\mb{\eta}_2 \big)_{k} = \big(X^*T_{0}(\mG_1)\mb{\eta}_2 \big)_{k} + O(1)\times \big(X^*\mG_1\mb{\eta}_2 \big)_{k}. 
\end{align}
Similar estimates like \rf{Importestsum2} and \rf{est:XT012} can be deduced from replacing $T_t(\mG_1) $ by $\wt T_t(\mG_1)$. Besides, we also have 
\begin{align} \label{2020052402}
 &\big(\wt T_{0}(\mG_1)\mb{\eta}_1 \big)_{q} =   \big( T_{0}(\mG_1)\mb{\eta}_1 \big)_{q}  + O(1)\times  \big(\mG_1\mb{\eta}_1 \big)_{q} \notag\\
 &= (\wt c_1m_1 +\wt c_2 m_1')\eta_{1q} + \big( (\wt c_1 \Xi+ \wt c_2 \Xi' ) \mb{\eta}\big)_{q}  +O_\prec\big( d_i^{-1}\big) \eta_{1q}  \nonumber\\
 &=  \big( (\wt c_1 \Xi+ \wt c_2 \Xi' ) \mb{\eta}\big)_{q}  +O_\prec\big( d_i^{-1}\big) \eta_{1q}  \nonumber\\
 &= O_\prec\big( d_i^{-1}\big) \eta_{1q}
\end{align}
which can be verified by directly applying isotropic local law \rf{weak_est_DG} for the case $d_i\leq K$ or involving eigenvector empirical spectral distribution to do further analysis of $\big( (\wt c_1 \Xi+ \wt c_2 \Xi' ) \mb{\eta}\big)_{q} $ for the case $d_i>K$. The details for the case $d_i>K$ are analogous to that of \rf{2020051101} and hence we skip them. 
Then, with the above estimates, similar to the estimates of \eqref{eq:h02-1}-\eqref{eq:h02-2}, all the  terms in \rf{terms:h0(1,1)} can be bounded  by $O_\prec(N^{-\frac12})$. For the case $t=1,2$, similarly, we still see the forms in \rf{terms:h0(1,1)} with $T_0(\mG_1)$ replaced by $\mG_1$ and the coefficients $\wt c_s$ removed. These forms can be bounded by $O_\prec(N^{-\frac12})$ using the isotropic local laws \rf{weak_est_DG} and \rf{weak_est_XG}. We omit the details here. Hence, we proved $h_{t}(1,1)=O_\prec(N^{-\frac12})$. 

Next, we  consider the other cases for $\alpha_1+\alpha_3=3$ except for (\ref{rankresth0(1,2)}) and (\ref{rankrestht(1,2)}), i.e. $(\alpha_1,\alpha_2)=(3,0),(2,1),(0,3)$.  We start with the formulas  in (\ref{rankrorder3derivative}) and  (\ref{eq:Pd3}). Notice that according to the definition in (\ref{def of O}), we have $\mathcal{O}_2=\{(0,1),(0,2),(1,0),(2,0)\}$ in (\ref{rankrorder3derivative}) and  (\ref{eq:Pd3}). 

With (\ref{rankrorder3derivative}), we  can now estimate   $h_{t}(3,0)$. By the definition in \rf{rankrdefhts}, 
\begin{align}
h_{t}(3,0)=\frac{\kappa_4 m_1}{3!N^{\frac{3}{2}}\Delta(d_i)}\sum_{q,k}y_{tq}\frac{\partial^3 (X^*T_t(\mG_1)\mb{\vartheta}_{t})_{k}}{\partial x_{qk}^3}\mP^{l-1}. \label{eq:h(3,0)}
\end{align}
For $t=0$, after plugging in \eqref{rankrorder3derivative} and taking the sums, one can check that except for the following type of terms  
\begin{align}
&N^{-\frac32}\frac{m_1 \wt c_s}{\Delta(d_i)}\Big(\sum_{q}y_{0q}(\mG_1^{s_3})_{qq}(\mG_1^{s_4}\mb{\vartheta}_{0})_q\Big)\Big(\sum_k\big((X^*\mG_1^{s_1}X)_{kk}\big)^a\big((X^*\mG_1^{s_2}X)_{kk}\big)^b\Big)\mathcal P^{l-1} \notag\\
&\qquad\qquad\qquad\qquad  a,b=0,1, \label{19072101}
\end{align} 
all the other terms of (\ref{eq:h(3,0)}) contain at least  one quadratic form of $X^*\mG_1^a$ for some $a=1,2$ and one quadratic form of $\mG_1^b$ for some $b=1,2$. Here, $s_1,s_2, s_3,s_4 =1,2$ satisfy $s=s_3+s_4-1+ a(s_1-1)+ b(s_2-1)\in \{1,2\}$.  Actually, by the isotropic local law \eqref{weak_est_XG}, those terms with at least  one quadratic form of $X^*\mG_1^a$ and one quadratic form of $\mG_1^b$ can all be dominated by a crude bound $O_\prec(N^{-\frac 12})$. For instance, 
\begin{align*}
&N^{-\frac32} \frac{m_1\wt c_s}{\Delta(d_i)}\sum_{q,k} y_{0q} \big(\mG_1^{s_1}\mathscr{P}_1^{qk}\mG_1^{s_2}\mathscr{P}_1^{qk}\mG_1^{s_3} \mb{\vartheta}_{0} \big)_{q}\mP^{l-1} \\
&= N^{-\frac32}  \frac{m_1 \wt c_s}{\Delta(d_i)} \sum_{q,k} y_{0q} (\mG_1^{s_1})_{qq} (X^*\mG_1^{s_2})_{kq} (X^*\mG_1^{s_3} \mb{\vartheta}_{0})_k \mP^{l-1} \notag\\
&=O_\prec(N^{-1}(d_i-\sqrt y)^{-\frac 12}d_i^{-\frac12}),
\end{align*}
where $s= s_1+s_2+s_3-2 \in\{1,2\}$ and $s_1,s_2,s_3=1,2$.
The other terms with at least  one quadratic form of $X^*\mG_1^a$  and one quadratic form of $\mG_1^b$  can be estimated similarly. And we use a crude bound $O_\prec(N^{-1/2})$ to estimate them. We skip the details for those terms. 
Further, using \rf{Importestsum1}, Remark \ref{boundrmk*}, $\wt c_1=O(d_i)$ and $\wt c_2=O((d_i-\sqrt y)^{2})$, one can easily get the terms in (\ref{19072101}) are bounded by $O_\prec(N^{-\frac12})$. Analogously to the above statements, we can derive the crude bound  $O_\prec(N^{-\frac12})$ for the case $t=1,2$. As a consequence, we get  $h_{t}(3,0)=O_\prec (N^{-\frac 12})$.

For $h_{t}(2,1), t=0,1,2$, by the definition in (\ref{rankrdefhts}), we have
\begin{align}
h_{t}(2,1)=\frac{(l-1)\kappa_4m_1}{2N^{\frac{3}{2}}\Delta(d_i)}\sum_{q,k}y_{tq}\frac{\partial^2 (X^*T_t(\mG_1)\mb{\vartheta}_{t})_{k}}{\partial x_{qk}^2}\frac{\partial \mP}{\partial x_{qk}}\mP^{l-2}.\label{rankrhts(2,1)}
\end{align}
Using the formula in (\ref{eq:2ndXG}) and further the $O_\prec(N^{-\frac12}(d_i-\sqrt y)^{-\frac 12-2(a-1)}d_i^{-\frac 12+(a-1)})$ bound for  the quadratic forms of $X^*\mG_1^a$ for  $a= 1, 2$, the $O_\prec((d_i-\sqrt y)^{-(b-1)}d_i^{-1})$ bound for  the quadratic forms of $\mG_1^b$ for $b= 1, 2$, it is not difficult to see 
 $$ \frac{\wt c_s \partial^2 (X^*\mG_1^s\mb{\vartheta}_{t})_{k}}{\partial x_{qk}^2}=O_\prec (N^{-\frac 12}(d_i-\sqrt y)^{-\frac 12}d_i^{-\frac 12}).$$
Similarly, we can also obtain 
 \begin{align*}
 \frac{\partial^2 (X^*T_t(\mG_1)\mb{\vartheta}_{t})_{k}}{\partial x_{qk}^2}= O_\prec (N^{-\frac 12}(d_i-\sqrt y)^{-\frac 12}d_i^{-\frac 12}) \quad\text{ for } t=0,1,2.
 \end{align*}
Then, from the above bound, $|\mathcal{P}|=O_\prec(1)$, $|\partial \mP/\partial x_{qk}|=O_\prec(1)$ (\cf (\ref{19072410})) and $m_1= O(d_i^{-1})$, one can conclude that  for $t=0,1,2$, $h_{t}(2,1)=O_\prec(N^{-\frac 12} )$.

For $h_{t}(0,3)$, we write
\begin{align*}
h_{t}(0,3)&=\frac{\kappa_4 m_1}{3!N^{\frac{3}{2}}\Delta(d_i)}  \sum_{q,k}y_{tq}(X^*T_t(\mG_1)\mb{\vartheta}_{t})_{k}\frac{\partial^3 \mP^{l-1}}{\partial x_{qk}^3}\nonumber\\
&=\mathcal \kappa_4(l-1)(l-2)(l-3)\mathcal{L}_1 + \kappa_4(l-1)(l-2)\mathcal L_2 + \kappa_4(l-1)\mathcal L_3, 
\end{align*}
where
\begin{align*}
&\mathcal L_1:= \frac{m_1 }{3!N^{\frac{3}{2}}\Delta(d_i)}\sum_{q,k}y_{tq}(X^*T_t(\mG_1)\mb{\vartheta}_{t})_{k}\Big(\frac{\partial \mP}{\partial x_{qk}}\Big)^3\mP^{l-4},\nonumber\\
&\mathcal L_2:= \frac{m_1 }{2!N^{\frac{3}{2}}\Delta(d_i)}\sum_{q,k}y_{tq}(X^*T_t(\mG_1)\mb{\vartheta}_{t})_{k}\frac{\partial \mP}{\partial x_{qk}}\frac{\partial^2 \mP}{\partial x_{qk}^2}\mP^{l-3},\nonumber\\
&\mathcal L_3:=  \frac{m_1 }{3!N^{\frac{3}{2}}\Delta(d_i)}\sum_{q,k}y_{tq}(X^*T_t(\mG_1)\mb{\vartheta}_{t})_{k}\frac{\partial^3 \mP}{\partial x_{qk}^3}\mP^{l-2}.
\end{align*}

First, note that  $\mathcal L_1=\onh$, by using the facts $|\partial \mP/\partial x_{ik}|\prec 1$ (\cf (\ref{19072410})), $|m_1|\asymp d_i^{-1}$ and  $|(X^*T_t(\mG_1)\bv_i)_{k}|\prec N^{-\frac12}(d_i-\sqrt y)^{-\frac 12}d_i^{-\frac 12}$ (\cf (\ref{est:XT012})), together with the Cauchy-Schwarz inequality. 

Second, as for $\mathcal L_2$, we use the formula of $\partial^2 \mP/\partial x_{qk}^2$ in \eqref{eq:Pd2}. Observe that, by the isotropic local laws  \eqref{weak_est_DG} and \eqref{weak_est_XG},  the terms in \eqref{eq:Pd2} can be  bounded  by either $O_\prec(N^{-1/2}(d_i-\sqrt y)^{-\frac 12} d_i^{\frac 12})$ or $O_\prec(\sn (d_i-\sqrt y)^{\frac 12}d_i^{\frac 12})$.
 More precisely, the $O_\prec(N^{-1/2}(d_i-\sqrt y)^{-\frac 12} d_i^{\frac 12})$ bound  is derived for the following forms 
\begin{align*}
\frac{\sn c_b}{\Delta(d_i)} (X^*\mG_1^{b_1}\mb{\eta}_1)_k(X^*\mG_1^{b_2})_{kq}(\mG_1^{b_3}\mb{\eta}_2)_q, \qquad \frac{\sn c_b}{\Delta(d_i)} (X^*\mG_1^{b_1}\mb{\eta}_1)_k(\mG_1^{b_2})_{qq}(X^*\mG_1^{b_3}\mb{\eta}_2)_k,
\end{align*} 
and  the $O_\prec(\sn (d_i-\sqrt y)^{\frac 12}d_i^{\frac 12})$ bound is obtained for the form $$\sn  c_b \Delta(d_i)^{-1}(\mG_1^{b_1}\mb{\eta}_1)_q(X^*\mG_1^{b_2}X)_{kk}^a(\mG_1^{b_3}\mb{\eta}_2)_q.$$ Here, $\mb{\eta}_1,\mb{\eta}_2=\mb{y}_1,\mb{y}_2, \mb{y}_3, \bw_{\mathsf{I}}^0$ or $\mb{\varsigma}_{\mathsf{I}}^0$,  $a=0,1$ and $b_1,b_2,b_3=1$ or $2$ with $b_1+b_2+b_3-2=b\in \{1,2\}$.
 The contribution from the $O_\prec(N^{-1/2}(d_i-\sqrt y)^{-\frac 12} d_i^{\frac 12})$ terms can be discussed similarly to $\mathcal{L}_1$. We thus omit the details. For the contribution from the  those  $O_\prec(\sn (d_i-\sqrt y)^{\frac 12}d_i^{\frac 12})$ terms,   recalling the notations in \rf{def:wtTG}, we notice that it suffices to consider the bound of the following forms 
\begin{align}\label{eq:last3forms}
&\frac{m_1 c_b}{\sn\Delta(d_i)^3}\Big(\sum_{q}y_{tq}(\mG_1^{b_1}\mb{\eta}_1)_q(\mG_1^{b_2}\mb{\eta}_2)_q(\wt T_{t_1}(\mG_1)\mb{\eta}_{3})_q\Big)\notag\\
&\qquad\qquad\qquad\times\Big(\sum_k(X^*T_t(\mG_1) \mb{\vartheta}_{t} )_k(X^*\wt T_{t_2}(\mG_1)\mb{\eta}_4)_k (X^*\mG_1^{b_3}X)_{kk}\Big) \mP^{l-3},\nonumber\\
&\frac{m_1 }{\sn\Delta(d_i)^3}\sum_{b=1}^2 c_b
 \sum_{\begin{subarray} {c} b_1,b_2\geq 1\\ b_1+b_2=b+1\end{subarray}}
\Big(\sum_{q}y_{tq}(\mG_1^{b_1}\mb{\eta}_1)_q(\mG_1^{b_2}\mb{\eta}_2)_q(\wt T_{t_1}(\mG_1)\mb{\eta}_{3})_q\Big)\notag\\
&\qquad\qquad\qquad\times\Big(\sum_k(X^*T_t(\mG_1) \mb{\vartheta}_{t} )_k(X^*\wt T_{t_2}(\mG_1)\mb{\eta}_4)_k \Big) \mP^{l-3},\nonumber\\
&\frac{m_1 }{\sn\Delta(d_i)^3}\Big(\sum_{q}y_{tq}(\mG_1\mb{\eta}_1)_q(\mG_1\mb{\eta}_2)_q(\wt T_{t_1}(\mG_1)\mb{\eta}_{3})_q\Big)\notag\\
&\qquad\qquad\qquad\times\Big(\sum_k(X^*T_t(\mG_1) \mb{\vartheta}_{t} )_k(X^*\wt T_{t_2}(\mG_1)\mb{\eta}_4)_k\big((X^*\mG_1X)_{kk}\big)^a\Big) \mP^{l-3},
\end{align}
for $a=0,1$, $t_1,t_2=0,1$ and $b=b_1+b_2+b_3-2\in \{1,2\}$.
By \rf{2020052401}, we get similar estimates to \rf{Importestsum1} for the sum $\sum_{q}y_{tq}(\mG_1^{b_1}\mb{\eta}_1)_q(\mG_1^{b_2}\mb{\eta}_2)_q(\wt T_{t_1}(\mG_1)\mb{\eta}_{3})_q$.
This bound, together with \rf{2020052402}, \rf{est:XT012}, Remark \ref{boundrmk*} and the estimates $c_1=O(d_i), c_2=O((d_i-\sqrt y)^{-2})$, leads to the crude bound $O_\prec(N^{-\frac 12})$ for the first and last forms in \eqref{eq:last3forms}. For the second form, one can refer to the estimate \rf{2020052403}. Analogously, we can also get 
\begin{align*}
\sum_{b=1}^2 c_b
 \sum_{\begin{subarray} {c} b_1,b_2\geq 1\\ b_1+b_2=b+1\end{subarray}}
\Big(\sum_{q}y_{tq}(\mG_1^{b_1}\mb{\eta}_1)_q(\mG_1^{b_2}\mb{\eta}_2)_q(\wt T_{t_1}(\mG_1)\mb{\eta}_{3})_q\Big) = O_\prec(d_i^{-3}),
\end{align*}
by which together with \rf{2020052402}, \rf{est:XT012}, we have that the second form of (\ref{eq:last3forms})  is bounded crudely by $O_\prec(N^{-\frac 12})$.
 Hence, $\mathcal L_2$ is also bounded by $O_\prec (N^{-\frac 12 })$.

Finally, for $\mathcal L_3$, we use the bound in (\ref{eq:bdPd3}). 
Then following the same argument as for $\mathcal L_1$, one can prove that $\mathcal L_3= O_\prec(N^{-\frac12})$. 
Thus, we have $h_{t}(0,3)=O_\prec (N^{-\frac 12 } )$ for $t=0,1,2$ and the proof of (2) in Lemma \ref{rankrlemh_ts} is complete.

Before we proceed to the proof of the remainder term $\mathcal R_{t}$, let us comment that, using the same reasoning as we  previously did for $h_{t}(\alpha_1,\alpha_2)$ with $\alpha_1+\alpha_2\le 3$, $t=0,1,2$, one can also get
\begin{align}\label{eq:bdh4'}
N^{-2}\frac{|m_1|}{\Delta(d_i)}\sum_{q,k} \Big|y_{tq} \frac{\partial^4 \big((X^*T_t(\mG_1)\mb{\vartheta}_{t})_{k}\mP^{l-1}\big)}{\partial x_{qk}^4}\Big| = O_\prec(N^{-\frac12}),
\end{align}
which implies that $h_{t}(\alpha_1,\alpha_2) = O_\prec(N^{-\frac12})$ for $\alpha_1+\alpha_2 =4$. The main tools are still Lemma \ref{Importantests} and the isotropic local laws \eqref{weak_est_DG}  and \eqref{weak_est_XG}. We omit the details. The necessary formulas for the fourth derivatives of $(X^*\mG_1^s)_{kj}$ and $\mP$ are recorded in the Appendix \ref{s.derivative of G} for the readers' convenience.

In the end, we prove  (3) of Lemma \ref{rankrlemh_ts}. By Lemma \ref{cumulantexpansion}, we can bound $\mathcal{R}_{t}$ by
\begin{align}
|\mathcal{R}_{t}|\leq \sn \frac{|m_1|}{\Delta(d_i)} \sum_{q,k}\mathbb{E}\Big( & N^{-\frac{5}{2}}\sup_{|x_{qk}|\leq N^{-\frac{1}{2}+\epsilon}}\Big|  y_{tq}\frac{\partial^4 \big((X^*T_t(\mG_1)\mb{\vartheta}_{t})_{k}\mP^{l-1}\big)}{\partial x_{qk}^4}\Big| \nonumber\\
&+N^{-\wt K}\sup_{x_{qk}\in\mathbb{R}}\Big| y_{tq}\frac{\partial^4 \big((X^*T_t(\mG_1)\mb{\vartheta}_{t})_{k}\mP^{l-1}\big)}{\partial x_{qk}^4}\Big|\Big), \label{boundR_t}
\end{align}
for any sufficiently large constant $\wt K$. We evaluate the RHS of \rf{boundR_t} term by term. 

First, we claim that similarly to  \eqref{eq:bdh4'} we have 
\begin{align}
&\sn \frac{|m_1|}{\Delta(d_i)}\mathbb{E}\Big(N^{-\frac{5}{2}}\sup_{|x_{qk}|\leq N^{-\frac{1}{2}+\epsilon}}\sum_{q,k}\Big| y_{tq}\frac{\partial^4 \big((X^*T_t(\mG_1)\mb{\vartheta}_{\varphi(t)})_{k}\mP^{l-1}\big)}{\partial x_{qk}^4}\Big|\Big)\notag\\
&=O_\prec(N^{-\frac12}). \label{19072201}
\end{align}
The difference between \eqref{eq:bdh4'} and (\ref{19072201}) is that we actually consider a random matrix $\wt{X}$ with the $(q,k)$-th entry deterministic while all the other entries random.  Using a regular perturbation argument through the resolvent expansion,   one can show that replacing one random entry $x_{qk}$ in $X$ by any deterministic  number bounded by $N^{-1/2+\epsilon}$ while keeping all the other $X$ entries random will not change the isotropic local law. Thus the isotropic local law together with the deterministic bounds \rf{trivialbd_G1}-\rf{trivialbd_XG1X} leads to (\ref{19072201}).

For the second term on the RHS of \rf{boundR_t}, we simply apply the crude deterministic bounds \rf{trivialbd_G1}-\rf{trivialbd_XG1X}. By choosing $\wt{K}$ sufficiently large, we can conclude that the second term in   \rf{boundR_t} is negligible. 
  
Thus $|\mathcal{R}_{t}| = O_\prec(N^{-\frac12})$. This completes the proof of  Lemma \ref{rankrlemh_ts}.

\end{proof}
 
 \section{Proofs of theorems in Section \ref{s.application}} \label{appendix.proof}

In this section, we prove the two theorems in Section \ref{s.application}, Theorems \ref{thm_firststatdist} and \ref{thm_secondtestatistics}, and also their corollaries, Corollaries \ref{cor.082001} and \ref{cor_simulateddistribution}. We also provide  a result equivalent to the joint eigenvalue-eigenvector distribution in the multiple case, which will be needed in the proofs of some of the aforementioned theorems and corollaries.
In addition, the proofs of Theorems \ref{thm_firststatdist} and \ref{thm_secondtestatistics}  rely on Propositions \ref{prop.simples}-\ref{prop.multiple}, among which the proofs of Propositions \ref{prop.simples}, \ref{joint ev-evector multiple} and \ref{prop.multiple} are analogous to that of Theorem \ref{mainthm}, and the proof of Proposition \ref{repre.multiple.eigv} is an extension of that of Lemma \ref{lem.representation.eigv}. 
More specifically, Proposition \ref{prop.simples} is an extension of Theorem \ref{thm.simple case}  by considering  more simple spikes and it will support the proof of the all-simple spikes case for both Theorem \ref{thm_firststatdist} and Theorem  \ref{thm_secondtestatistics}. Second, Proposition \ref{prop.multiple} generalizes Theorem \ref{thm.simple case} to allow multiple spikes. Consequently,  
the case of all-equal spikes for both Theorem \ref{thm_firststatdist} and Theorem \ref{thm_secondtestatistics} will rely on Proposition \ref{prop.multiple}.  It will be seen that the proof of Theorems \ref{thm_firststatdist} and \ref{thm_secondtestatistics} will also rely on Proposition \ref{repre.multiple.eigv} which establish the expansion of the spiked eigenvalue estimator for a multiple spike.

Recall the notation $\mathds{1}_{E}$ as the indicator function of some event $E$, and also the notation $\mathcal{I}^c= \lb1,r\rb\setminus \mathcal{I}$ for any $\mathcal{I}\subset \lb1,r\rb $.
Further, we recall the notation $\mathsf{I}\equiv \mathsf{I}(i)$ as the set of indices  of the multiple $d_t$'s including $d_i$.  We remark here that throughout this section, $\mathcal{I}$ is used to denote a generic subset of $\lb1,r\rb$, which may include the indices of distinct $d_t$'s. In contrast, the notation $\mathsf{I}\equiv \mathsf{I}(i)$ is always used to denote a set of indices of a multiple $d_i$.

 Further, we raise the following notation as a rewriting of  \rf{def:v_I_varsigma_I} in certain cases: For a fixed $i\in \lb1, r\rb$, and  any unit vector $\mb{u}\in \mathrm{Span}\{\mathbf v_j \}_{j\in \lb 1,M\rb \setminus \mathsf{I}(i)}$, we set 
 \begin{align}\label{2020081001}
\mb{\varsigma}_{i}^{\mb{u}}:=\sum_{j\in \lb 1,M\rb \setminus \{i\}} \frac{d_i\sqrt{d_j +1}}{d_i -d_j} \langle \mb{u}, \bv_j \rangle \bv_j.
 \end{align}
  Here we made the convention that $\langle \mb{u}, \bv_j \rangle/(d_i-d_j)=0$, in case $\langle \mb{u}, \bv_j \rangle=d_i-d_j=0$. Note that $\mb{\varsigma}_{i}^{\mb{u}}= \mb{\varsigma}_{\mathsf{I}(i)}$ (with $\mb{w}=\mb{u}$) (c.f. \rf{def:v_I_varsigma_I}). Here we emphasize the dependence of the notation on $\mb{w}=\mb{u}$, and also 
 the dependence of $\mathsf{I}(i)$ on $i$. Such a rewriting will be mostly suitable for the discussion in this section.

Recall the projection $\rm{P}_{\mathcal{I}}$ defined in (\ref{def of PI general}).  In the sequel, we state and briefly prove the aforementioned propositions. We will consider the projections of vectors onto $ {\rm{P}}_{\mathcal{I}}$, and we are primarily interested in the case that either $\{d_t\}_{t\in \mathcal{I}}$ are all distinct and well-separated or are all equal. We call the former as {\it all-simple} case, and the latter as {\it all-equal} case. Nevertheless, many of our results hold for more general $\mathcal{I}$ in the sense that it can contain indices of both simple and multiple $d_t$'s.

First, we consider an extension of Theorem \ref{thm.simple case}. Suppose we are interested in the all-simple case, i.e., all the $d_i$'s for $i\in \mathcal{I}$ are simple, we have the following proposition.

\begin{prop}\label{prop.simples}
Suppose that Assumptions \ref{assumption} ,  \ref{supercritical},  and the setting (\ref{19071802}) hold.  In case $d_i\equiv d_i(N)\to \infty$ as $N\to\infty$ for some $i\in \mathcal{I}$, we additionally assume that $|y-1|\geq \tau_0$ for some small but fixed $\tau_0>0$. 
We further assume that all $d_i, i\in\mathcal{I}$ are simple. 
Then for each $i\in \mathcal{I}$, we have the expansion of $i$-th eigenvalue $\mu_i$  in \rf{expansion.eigv}.
Moreover, we have the following statements:

\noindent (1). For any  $\bv_i$ with $i\in \mathcal{I}$,  we have
\begin{align*}
\langle  \bv_i, {\rm{P}}_\mathcal{I} \bv_i\rangle=
&\frac{d_i^2-y}{d_i(d_i+y)}+ \frac{1}{\sqrt{N(d_i^2-y)}}  \,\Theta_{\bv_{i}}^{\bv_i} \nonumber\\
&-\frac{1}{N}\sum_{t\in  \{i\}^c}\frac{d_id_t}{(d_i-d_t)^2} (\Pi_{\bv_t}^{\bv_i})^2 
+\frac {\Vert \mb{\varsigma}_{j}^{ \mb{v}_i}\Vert^2}{N} \sum_{j\in \mathcal{I}\setminus\{i\}}\frac{1}{d_j} (\Delta_{\bv_j}^{ \mb{v}_i})^2  \nonumber\\
& 
+O_\prec \bigg({N^{-(1+\varepsilon)}}\sum_{j\in \mathcal{I}\setminus\{i\}}\frac{1}{d_j} \Vert \mb{\varsigma}_{j}^{ \mb{v}_i}\Vert^2\bigg) \nonumber\\
&+ O_\prec\bigg(N^{-\varepsilon}\Big(N^{-\frac 12} \frac{1}{\sqrt{d_i^2-y}} + \frac{1}{N}\sum_{t\in\{ i\}^c} \frac{d_id_t}{(d_i-d_t)^2}\Big)\bigg).
\end{align*}
 Further, $\{\Phi_i\}_{i\in\mathcal{I} },\{ \Theta_{\bv_{i}}^{\bv_i}\}_{i\in\mathcal{I} }, \{\Pi_{\bv_t}^{\bv_i}\}_{i\in\mathcal{I} , t\in\{ i\}^c}$, $\{\Delta_{\bv_j}^{ \mb{v}_i}\}_{i,j\in \mathcal{I}, j\neq i}$ are asymptotically jointly Gaussian distributed. Especially, 
\begin{align}
\Big( \{\Phi_i\}_{i\in \mathcal{I}}, \{\Theta_{\bv_{i}}^{\bv_i}\}_{i\in \mathcal{I}}\Big) \simeq \mathcal{N}(0, C_{\mathcal{I}}) \label{082101}
\end{align}
where $C_{\mathcal{I}}$ is defined entrywise by  the RHS of the following equations
\begin{align*}
&{\rm{cov}} (\Phi_i, \Phi_j) \doteq  \mathds{1}_{\{i=j\}} \cdot 2(1+d_i^{-1})^2 \nonumber\\
&\hspace{15ex}+ \kappa_4 (1-yd_i^{-2})^{\frac12} (1+d_i^{-1})  (1-yd_j^{-2})^{\frac12}  (1+d_j^{-1}) s_{2,2} (\bv_i,\bv_j),\nonumber\\
&{\rm{cov}}(\Theta_{\bv_{i}}^{\bv_i}, \Theta_{\bv_{j}}^{\bv_j}) \doteq \mathds{1}_{\{i=j\}} \cdot  2y \mathtt{h}(d_i)^2 (1+ y \mathtt{h}(d_i)^2)\nonumber\\
 &\hspace{15ex}+  \kappa_4 (1-yd_i^{-2})^{\frac12} \mathtt{f}(d_i) (1-yd_j^{-2})^{\frac12}  \mathtt{f}(d_j)  s_{2,2} (\bv_i,\bv_j), \nonumber\\
& {\rm{cov}}(\Phi_i, \Theta_{\bv_{j}}^{\bv_j}) \doteq  \mathds{1}_{\{i=j\}} \cdot  2y \mathtt{h}(d_i)^2 (1+d_i^{-1})  \nonumber\\
 &\hspace{15ex}+ \kappa_4 (1-yd_i^{-2})^{\frac12} (1+d_i^{-1})  (1-yd_j^{-2})^{\frac12}  \mathtt{f}(d_j)s_{2,2} (\bv_i,\bv_j),
\end{align*}
for $i,j\in \mathcal{I}$.  Here $\Phi_i$ is given in \rf{expansion.eigv}. 

\noindent(2). Let $\mathcal{J}$ be any fixed index set. Further, let $ \{\mb{u}_j\}_{j\in \mathcal{J}} \subset \mathrm{Span}\{\mathbf v_i \}_{i\in \lb 1,M\rb \setminus \mathcal{I}} $ be any  set of orthonormal vectors.  Then 
 \begin{align*}
\langle \mb{u}_j, {\rm{P}}_\mathcal{I} \mb{u}_j\rangle =\sum_{i\in \mathcal{I}}\frac{\Vert \mb{\varsigma}_{i}^{ \mb{u}_j}\Vert^2}{Nd_i} (\Delta_{\bv_i}^{ \mb{u}_j})^2  
+O_\prec \bigg({N^{-(1+\varepsilon)}} \sum_{i\in \mathcal{I}}d_i^{-1} \Vert \mb{\varsigma}_{i}^{ \mb{u}_j}\Vert^2\bigg),\, \text{ for $j\in \mathcal{J}$} 
\end{align*}
and $\Big(\{\Delta_{\bv_i}^{ \mb{u}_j}\}_{(j, i)\in \mathcal{J}\times \mathcal{I}} \Big)\simeq \mathcal{N}(0,\mb{U}_0)$. The  covariance  matrix $\mb{U}_0$ is defined by  the RHS of the following equation
\begin{align}\label{200826010}
&{\rm{cov}}(\Delta_{\bv_{i_1}}^{ \mb{u}_{j_1}}, \Delta_{\bv_{i_2}}^{ \mb{u}_{j_2}})
\doteq \mathds{1}_{\{i_1=i_2\}}\mathtt{h}(d_{i_1}) \langle \big(\mb{\varsigma}_{i_1}^{ \mb{u}_{j_1}}\big)^0,  \big(\mb{\varsigma}_{i_2}^{ \mb{u}_{j_2}}\big)^0 \rangle \nonumber\\
&\qquad+ \kappa_4\frac{\sqrt {(d_{i_1}^2-y)\mathtt{h}(d_{i_1})}}{d_{i_1}} \frac{\sqrt {(d_{i_2}^2-y)\mathtt{h}(d_{i_2})}}{d_{i_2}} s_{1,1,1,1}\Big(\bv_{i_1}, \bv_{i_2}, \big(\mb{\varsigma}_{i_1}^{ \mb{u}_{j_1}}\big)^0, \big(\mb{\varsigma}_{i_2}^{ \mb{u}_{j_2}}\big)^0 \Big), 
\end{align}
for $j_{1,2}\in \mathcal{J}, i_{1,2}\in \mathcal{I}$. Here we use the notation $\big( \mb{\varsigma}_{i}^{ \mb{u}_j}\big)^0$ as the normalized $\mb{\varsigma}_{i}^{ \mb{u}_j}$ similarly to \rf{def.nor.Varsigma}. And  we  define  $\mb{U}$ to be the matrix obtained by multiplying the entry of $\mb{U}_0$, the RHS of \rf{200826010},  by $\Vert \mb{\varsigma}_{i_1}^{ \mb{u}_{j_1}} \Vert \Vert\mb{\varsigma}_{i_2}^{ \mb{u}_{j_2}}\Vert/\sqrt{d_{i_1}d_{i_2}}$.
\end{prop} 

\begin{rmk} \label{rmk.082102} We remark here that  (\ref{082101}) and the whole statement (2) in Proposition \ref{prop.simples}
hold for general $\mathcal{I}\subset \llbracket 1, r_0\rrbracket $, with the same expression of the covariance structure $C_{\mathcal{I}}$ and $\mb{U}_0$. Here general $\mathcal{I}$ means that it can contain the indices of either simple $d_t$'s  or multiple ones,  or even both. The proof of such an extended version is nearly the same as the simple one. 
\end{rmk}

\begin{proof}[Proof of Proposition \ref{prop.simples}]  The proof is basically the same as the one 
of Theorem \ref{thm.simple case}. The only change is that now one needs to put all Green function quadratic forms corresponding to distinct simple spikes  together and consider their joint Gaussianity. It will only require one to change the quadratic forms in the  definition of  $\mathcal{P}$ in (\ref{19071950}) and go through the recursive moment estimate again. Hence, we omit the details and conclude the proof. 
\end{proof}

Next, we consider the all-equal case, i.e.,   all $d_i, i\in \mathcal{I}$ are identical. Analogously to the all-simple case, we can also estimate the multiple spike from the data. In the sequel, we present the detailed estimation of the unknown $d_i$ in the multiple case. Recall the index set $\mathsf{I}(i)$ as the collection of indices for a multiple $d_i$. In the sequel, we are primarily interested in the case $|\mathsf{I}(i)|>1$, but the results apply to the degenerated (simple) case as well. For any $t\in \mathsf{I}(i)$, according to Lemma \ref{locationeig}, we have   $|\mu_t - \theta(d_i)|\prec d_i^{\frac 12}\delta_{i0}^{\frac 12} N^{-\frac 12}$. 
We shall derive the expansion  for $\{\mu_t\}_{t\in \mathsf{I}(i)}$ and further  get the expression of $\sum_{t\in \mathsf{I}(i)} \mu_t /|\mathsf{I}(i)|$ which can be applied to estimate the unknown $\theta(d_i)$. Define 
\begin{align*}
D_{\mathsf{I}(i)}: = \frac {d_i}{1+d_i} { I}_{|\mathsf{I}(i)|} \text{ and }  V_{\mathsf{I}(i)}:= (\{\bv_t\}_{t\in \mathsf{I}(i)}),
\end{align*}
where $ I_{|\mathsf{I}(i)|}$ means an identity matrix of dimension $|\mathsf{I}(i)|$ and $V_{\mathsf{I}(i)}$  is the submatrix of $V$ obtained by selecting the columns with  indices in $\mathsf{I}(i)$.  For convenience,  for the columns in $V_{\mathsf{I}(i)}$, we keep their original  indices in $V$. With some abuse of notation, in this section, we will use the notation $(\lambda_t (V_{\mathsf{I}(i)}^* A V_{\mathsf{I}(i)})) _{t\in {\mathsf{I}(i)}}$ to represent the family of eigenvalues of  $V_{\mathsf{I}(i)}^* A V_{\mathsf{I}(i)}$, with  the  increasing order $\lambda_{t_1} (V_{\mathsf{I}(i)}^* A V_{\mathsf{I}(i)})\leq \lambda_{t_2} (V_{\mathsf{I}(i)}^* A V_{\mathsf{I}(i)})$ if  $t_1, t_2\in \mathsf{I}(i), t_1<t_2$. We recall that $\delta_{i0}:= d_i-\sqrt y$, the distance of the multiple $d_i$ to critical value $\sqrt y$. With the above notations, we have the following proposition, whose proof will be stated in the end.
\begin{prop}\label{repre.multiple.eigv}
Suppose that the assumptions in Theorem \ref{mainthm} hold. Let $i\in \llbracket 1, r_0\rrbracket$. For any $t\in \mathsf{I}(i)$,   we have
\begin{align*}
\mu_t = \theta(d_i)-  (d_i^2-y)\theta(d_i)\, \lambda_t \big( V_{\mathsf{I}(i)}^* \Xi(\theta(d_i)) V_{\mathsf{I}(i)} \big)   + O_\prec \Big(d_i^{\frac12} \delta_{i0}^{\frac12} N^{-\frac12 - \varepsilon}\Big)
\end{align*}
for some small fixed $\varepsilon>0$.
\end{prop}

Recall the representation in Lemma  \ref{weak_lem_decomp}. Similarly to Theorem \ref{mainthm}, we can also derive the joint distribution of generalized components of eigenvectors and all  entries of the matrix  $V_{\mathsf{I}(i)}^* \Xi(\theta(d_i)) V_{\mathsf{I}(i)} $ under appropriate scaling. The result is collected in the following proposition.  Set the $|\mathsf{I}(i)|\times |\mathsf{I}(i)|$ matrix
\begin{align*}
\Phi\equiv (\Phi_{st})_{s,t}:= -\sqrt{N (d_i^2-y)}\,\theta(d_i) V_{\mathsf{I}(i)}^* \Xi(\theta(d_i)) V_{\mathsf{I}(i)}. 
\end{align*}
We remark here that with certain abuse of notation $\Phi_{ii}\equiv \Phi_i$ in (\ref{090301}) in the simple case $|\mathsf{I}(i)|=1$. 
\begin{prop} \label{joint ev-evector multiple}
Under the assumptions in Theorem  \ref{mainthm},  besides the expansions for  eigenvalues in Proposition \ref{repre.multiple.eigv} and the expansion for generalized components \rf{weak_gf_decomp}, for all of the upper triangular entries of $\Phi$, i.e., $\{\Phi_{st}\}_{s,t\in {\mathsf{I}(i)}, s\leq t}$,  together with  the random terms in  \rf{weak_gf_decomp}, we have 
\begin{align*}
\Big( \big\{\{\Phi_{st}\}_{t\in {\mathsf{I}(i)} , t\geq s}\big\}_{s\in {\mathsf{I}(i)} }, \Theta_{\bw_{{\mathsf{I}(i)} }}^{\bw}, \Lambda_{\mb{\varsigma}_{{\mathsf{I}(i)} }}^{\bw}, \{\Delta_{\bv_t}^{\bw}\}_{t\in {\mathsf{I}(i)} }, \{\Pi_{\bv_j}^{\bw}\}_{j\in (\mathsf{I}(i)) ^c} \Big) \simeq \mathcal{N}(0, C_{\mathsf{I}(i)} ^\bw)
\end{align*}
where $ C_{\mathsf{I}(i)} ^\bw$ is of size $r+2+|{\mathsf{I}(i)} |^2$ and the lower right $(r+2)\times (r+2)$ corner of $ C_{\mathsf{I}(i)} ^\bw$ is given by $A_{{\mathsf{I}(i)} }^{\bw} + \kappa_4  \frac{d_i^2-y}{d_i^2} B_{{\mathsf{I}(i)} }^{\bw}$ with $A_{{\mathsf{I}(i)} }^{\bw}$ and $B_{{\mathsf{I}(i)} }^{\bw}$ defined in \eqref{def:covA} and \eqref{def:covB} respectively. And the other entires of $ C_{\mathsf{I}(i)} ^\bw$ are defined by the RHS of the following equations
\begin{align*}
&{\rm{cov}} (\Phi_{\ell_1k_1}, \Phi_{\ell_2k_2} )\doteq \mathds{1}_{\{(\ell_1, k_1)=(\ell_2,k_2)\}} (1+\mathds{1}_{\ell_1=k_1} )(1+d_i^{-1})^2 \nonumber\\
&\qquad\qquad\qquad\quad + \kappa_4 (1-yd_i^{-2})(1+d_i^{-1})^2s_{1,1,1,1} (\bv_{\ell_1}, \bv_{k_1}, \bv_{\ell_2}, \bv_{k_2}),
\nonumber\\
& {\rm{cov}} (\Phi_{\ell k}, \Theta_{\bw_{\mathsf{I}}}^{\bw} ) \doteq 2y\mathtt{h}(d_i)^2(1+d_i^{-1})\langle \bw_{{\mathsf{I}(i)} }, \bv_\ell\rangle \langle \bw_{{\mathsf{I}(i)} }, \bv_k\rangle \nonumber\\
&\qquad\qquad\qquad  +\kappa_4 (1-yd_i^{-2})(1+d_i^{-1})\mathtt{f}(d_i) s_{1,1,2}(\bv_\ell,\bv_k, \bw_{\mathsf{I}}),
\nonumber\\
& {\rm{cov}} (\Phi_{\ell k}, \Lambda_{\mb{\varsigma}_{\mathsf{I}}}^{\bw} ) \doteq \kappa_4 (1-yd_i^{-2})(1+d_i^{-1})\mathtt{g}(d_i) s_{1,1,1,1} (\bv_\ell,\bv_k, \mb{\varsigma}_{\mathsf{I}(i)}^0, \bw_{\mathsf{I}}),
\nonumber\\
& {\rm{cov}} (\Phi_{\ell k}, \Delta_{\bv_t}^{\bw} ) \doteq \kappa_4 (1-yd_i^{-2})(1+d_i^{-1})\sqrt{\mathtt{h}(d_i) }s_{1,1,1,1} (\bv_\ell,\bv_k, \bv_t, \mb{\varsigma}_{{\mathsf{I}(i)} }^0), \text{ for $t\in {\mathsf{I}(i)} $},
 \nonumber\\
& {\rm{cov}} (\Phi_{\ell k }, \Pi_{\bv_j}^{\bw} ) \doteq \kappa_4 (1-yd_i^{-2})(1+d_i^{-1})\mathtt{l}(d_i) s_{1,1,1,1} (\bv_\ell,\bv_k, \bv_j, \bw_{{\mathsf{I}(i)} }), \text{ for $j\in{\mathsf{I}(i)} ^c$},
\end{align*}
for $\ell, \ell_1, \ell_2, k, k_1, k_2\in {\mathsf{I}(i)} $ satisfying  $\ell\leq k, \ell_1\leq k_1, \ell_2\leq k_2$.
\end{prop}
\begin{proof}
Note that the entries of $ V_{\mathsf{I}(i)}^* \Xi(\theta(d_i)) V_{\mathsf{I}(i)} $ admit the form $\chi_{st}(\theta(d_i))= \bv_s^*  \Xi(\theta(d_i))  \bv_t$ for $s,t \in {\mathsf{I}(i)}$. Putting these quadratic forms together with those in Lemma  \ref{weak_lem_decomp}, and modifying the definition of $\mathcal{P}$ in (\ref{19071950}) by including those quadratic forms as entries of $ V_{\mathsf{I}(i)}^* \Xi(\theta(d_i)) V_{\mathsf{I}(i)} $, one can apply the proof strategy of Theorem \ref{mainthm} mutatis mutandis to the proof of Proposition \ref{joint ev-evector multiple}. Hence we omit the details and conclude the proof. 
\end{proof}

We can then propose an estimator of $\theta(d_i)$, that is $\sum_{t\in {\mathsf{I}(i)} } \mu_t /|{\mathsf{I}(i)}|$.
The estimator of $\theta(d_i)$ then leads to an estimator of $d_i$ via taking the inverse $\theta^{-1}$.  Applying Proposition \ref{repre.multiple.eigv}, we see that this estimator admits the following expansion 
\begin{align}
\frac{1}{|\mathsf{I}(i)|}\sum_{t\in \mathsf{I}(i)} \mu_t = \theta(d_i) - (d_i^2-y)\theta(d_i)\,
\frac {1}{|\mathsf{I}(i)|} {\rm{Tr}} \, V_{\mathsf{I}(i)}^* \Xi(\theta(d_i)) V_{\mathsf{I}(i)} +  O_\prec \Big(d_i^{\frac12} \delta_{i0}^{\frac12} N^{-\frac12 - \varepsilon}\Big). \label{082106}
\end{align} 
Therefore, to perform hypothesis testing with this estimator, we shall study the fluctuation of $\frac {1}{|\mathsf{I}(i)|} {\rm{Tr}} \, V_{\mathsf{I}(i)} ^* \Xi(\theta(d_i)) V_{\mathsf{I}(i)} $ under appropriate scaling. Note that $\frac {1}{|{\mathsf{I}(i)} |} {\rm{Tr}} \, V_{\mathsf{I}(i)} ^* \Xi(\theta(d_i)) V_{\mathsf{I}(i)} $ is simply a linear combination of the quadratic forms, i.e, 
\begin{align*}
\frac {1}{|{\mathsf{I}(i)} |} {\rm{Tr}} \, V_{\mathsf{I}(i)} ^* \Xi(\theta(d_i)) V_{\mathsf{I}(i)}= \frac {1}{|{\mathsf{I}(i)} |} \sum_{t\in \mathsf{I}(i)} \chi_{tt}(\theta(d_i)),
\end{align*}
where we recall $ \chi_{tt}(z)= \bv_t^*\Xi(z)\bv_t$. Hence,  the joint distribution of the above term and  the other quadratic forms in the expansion of eigenvector in Lemma  \ref{weak_lem_decomp} can be derived in the same manner as the proof of Theorem  \ref{mainthm}.   
 
Let  \begin{align} \label{def:Phi_I}
\Phi_{\mathsf{I}(i)} :=-  \sqrt{N (d_i^2-y)} \, \theta(d_i) \frac {1}{|{\mathsf{I}(i)}|} \sum_{t\in {\mathsf{I}(i)}} \chi_{tt}(\theta(d_i)).
\end{align}
We have the following proposition.

\begin{prop}  \label{prop.multiple}
Under the assumptions in Theorem \ref{mainthm}, we have 
\begin{align}
\frac{1}{|{\mathsf{I}(i)}|}\sum_{t\in {\mathsf{I}(i)}}  \mu_t = \theta(d_i) + \frac{\sqrt{d_i^2-y}}{\sn} \Phi_{\mathsf{I}(i)}+  O_\prec \Big(d_i^{\frac12} \delta_{i0}^{\frac12} N^{-\frac12 - \varepsilon}\Big), \label{082105}
\end{align}
 for some small fixed $\varepsilon>0$.  
 
 \noindent (1). For  any $\bv_k, k\in \mathsf{I}(i)$, we recall from (\ref{082021}) with $\mb{w}=\bv_k$, 
 \begin{align}
 \la \bv_k, {\rm{P}}_{\mathsf{I}(i)}  \bv_k\ra &=
 \frac{d_i^2-y}{d_i(d_i+y)}+ 
\frac{1}{\sqrt{N(d_i^2-y)}}  \,\Theta_{\bv_{k}}^{\bv_k}
-\frac{1 }{N}\sum_{t\in{\mathsf{I}}^c}\frac{d_id_t}{(d_i-d_t)^2} (\Pi_{\bv_t}^{\bv_k})^2 
\nonumber  \\
&+ O_\prec\bigg(N^{-\varepsilon}\Big(N^{-\frac 12} \frac{1}{\sqrt{d_i^2-y}} + \frac{1}{N}\sum_{t\in\{ i\}^c} \frac{d_id_t}{(d_i-d_t)^2}\Big)\bigg). \label{082109}
 \end{align} 
Further, $\Big(\Phi_{\mathsf{I}(i)} , \{\Theta_{\bv_{k}}^{\bv_k}\}_{k\in {\mathsf{I}(i)} }, \{\Pi_{\bv_t}^{\bv_k}\}_{k\in \mathsf{I}, t\in (\mathsf{I}(i))^c}\Big)$ is asymptotically Gaussian distributed. In particular, 
 \begin{align}
 \Big(\Phi_{\mathsf{I}(i)} , \, \{\Theta_{\bv_{k}}^{\bv_k}\}_{k\in {\mathsf{I}(i)} }\Big)\simeq \mathcal{N}(0, C_{{\mathsf{I}(i)} ,e}) \label{082103}
 \end{align} 
 where $C_{{\mathsf{I}(i)} ,e}$ is defined by the RHS of the following equations
 \begin{align*}
 &{\rm{var}} (\Phi_{\mathsf{I}(i)})\doteq \frac{2(1+d_i^{-1})^2}{|{\mathsf{I}(i)}|^{2}}  + \kappa_4\, \frac{(1-yd_i^{-2})(1+d_i^{-1})^2}{|{\mathsf{I}(i)} |^{2} } \sum_{k_1,k_2\in {\mathsf{I}(i)} } s_{2,2} (\bv_{k_1}, \bv_{k_2}),\nonumber\\
&{\rm{cov}} (\Theta_{\bv_{k_1}}^{\bv_{k_1}}, \Theta_{\bv_{k_2}}^{\bv_{k_2}}) \doteq  \mathds{1}_{\{k_1=k_2\}} \cdot  2y \mathtt{h}(d_i)^2 (1+ y \mathtt{h}(d_i)^2) +  \kappa_4 (1-yd_i^{-2}) \mathtt{f}(d_i)^2  s_{2,2} (\bv_{k_1},\bv_{k_2}), \nonumber\\
& \mathrm{cov}(\Phi_{\mathsf{I}(i)} , \Theta_{\bv_k}^{\bv_k})\doteq 
\frac{2y \mathtt h(d_i)^2 (1+d_i^{-1}) }{ |{\mathsf{I}(i)}|}
+\kappa_4\frac{(1-yd_i^{-2})(1+d_i^{-1}) \mathtt f(d_i)} {|{\mathsf{I}(i)}|} \sum_{t\in {\mathsf{I}(i)}} s_{2,2}(\bv_t, \bv_k) 
 \end{align*}
 for $k,k_1,k_2\in {\mathsf{I}(i)}$.
 
 \noindent(2). Let $\mathcal{J}$ be  any fixed index set. Let  $\{ \mb{u}_k\}_{k\in \mathcal{J}}\subset \mathrm{Span}\{\mathbf v_j \}_{j\in \lb 1,M\rb \setminus {\mathsf{I}(i)}} $ be any given family of orthonormal vectors. Then according to the second case in Remark \ref{rmk_inside}, i.e., (\ref{0820100}), we have 
 \begin{align*}
 \la \mb{u}_k, {\rm{P}}_{\mathsf{I}(i)} \mb{u}_k\ra =
 \frac{ \Vert \mb{\varsigma}_{{\mathsf{I}(i)}}^{ \mb{u}_k}\Vert^2}{Nd_i} \sum_{t\in {\mathsf{I}(i)}}(\Delta_{\bv_t}^{ \mb{u}_k})^2  
+O_\prec \bigg({N^{-(1+\varepsilon)}}d_i^{-1} \Vert \mb{\varsigma}_{{\mathsf{I}(i)}}^{ \mb{u}_k}\Vert^2\bigg),
 \end{align*} 
 with $\{\Delta_{\bv_t}^{ \mb{u}_k}\}_{(k,t)\in \mathcal{J}\times {\mathsf{I}(i)}}\simeq \mathcal{N}(0, \mb{U}_m)$. 
{
 Here $\mb{U}_m $ is a symmetric matrix of dimension $|\mathcal{J}| |\mathsf{I}(i)|$ that can be defined by
 the RHS of the following equation
\begin{align*}
{\rm{cov}} (\Delta_{\bv_{t_1}}^{ \mb{u}_{k_1}}, \Delta_{\bv_{t_2}}^{ \mb{u}_{k_2}}) \doteq & \mathds{1}_{\{t_1=t_2\}}\mathtt{h}(d_i) 
\langle \big(\mb{\varsigma}_{t_1}^{\mb{u}_{k_1}}\big)^0, \big(\mb{\varsigma}_{t_2}^{\mb{u}_{k_2}}\big)^0  \rangle \notag\\
&+ \kappa_4\frac{ {(d_i^2-y)}}{d_i^2 }\mathtt{h}(d_i) s_{1,1,1,1} \Big(\bv_{t_1}, \bv_{t_2}, \big(\mb{\varsigma}_{t_1}^{\mb{u}_{k_1}} \big)^0, \big(\mb{\varsigma}_{t_2}^{\mb{u}_{k_2}}\big)^0 \Big),
\end{align*}
for $k_{1,2}\in\mathcal{J}, t_{1,2}\in {\mathsf{I}(i)}$. Here $\big( \mb{\varsigma}_{{\mathsf{I}(i)}}^{ \mb{u}_k}\big)^0$  represents the normalized $\mb{\varsigma}_{{\mathsf{I}(i)}}^{ \mb{u}_k}$ similarly to \rf{def.nor.Varsigma}.
 }
\end{prop}

\begin{proof} First,  (\ref{082105}) is an obvious conclusion of (\ref{082106})-(\ref{def:Phi_I}). Then, (\ref{082109}) is simply a rewriting of (\ref{082021}) with $\mb{w}=\bv_k$. Observe that $\Phi_{\mathsf{I}(i)}= |{\mathsf{I}(i)}|^{-1} \sum_{t\in {\mathsf{I}(i)}}\Phi_{tt}$. In addition, according Remark \ref{rmk.082102}, we can also write up the joint distribution of $\{\Phi_{tt}\}_{t\in \mathsf{I}(i)}$ and $\{\Theta_{\bv_{k}}^{\bv_k}\}_{k\in {\mathsf{I}(i)} }$, which admits the same form as (\ref{082101}) by choosing $\mathcal{I}=\mathsf{I}(i)$.  Then the joint distribution in (\ref{082103}) can be regarded as a corollary of such a more general joint distribution. 
The whole statement (2) above is also analogous to its counterpart, Proposition \ref{prop.simples} (2); see Remark \ref{rmk.082102}. The details can be checked similarly to Theorem \ref{mainthm}, and are thus omitted. Hence, we conclude the proof. 
\end{proof}

With the  above propositions and remarks, we are now ready to prove Theorems \ref{thm_firststatdist} and \ref{thm_secondtestatistics}.

\begin{proof} [Proof of Theorem \ref{thm_firststatdist} ] 
Under the null  hypothesis of \rf{eq_spectralprojection}, without loss of generality, we can choose  $\mb{u}_i= \bv_i$ for $i\in \mathcal{I}$, since our statistic $\sn \mathcal{T}$ is independent of the choice of the basis of $Z_0$. To ease the understanding of the reader,  
we will first show the proof in all-simple case, and then we discuss the all-equal case, and finally we conclude the extension to more general case. 

First, assume that all $d_i, i\in\mathcal{I}$ are simple. We assume that $\mb{\xi}_i$ is the eigenvector corresponding to  $d_i$ and the direction $\bv_i$.    By definition of $\mathcal{T}$ in \rf{eq_teststat}, we have the following derivation 
\begin{align}\label{2020081601}
\sn \mathcal{T}&= \sn \sum_{i \in \mathcal{I}}\langle \mb{v}_i, {\rm{P}}_{\mathcal{I}} \bm{v}_i  \rangle-\vartheta(\widehat d_i))\nonumber\\
&= \sn \sum_{i \in \mathcal{I}}\Big(|\langle \mb{\xi}_i, \bm{v}_i  \rangle|^2-\vartheta( d_i) + \vartheta( d_i)-\vartheta(\widehat d_i)\Big)+ \sn \sum_{i \in \mathcal{I}}\sum_{j \in \mathcal{I}\setminus\{i\}}|\langle \mb{\xi}_j, \bm{v}_i  \rangle|^2\nonumber\\
&= \sum_{i \in \mathcal{I}} \frac{1}{\sqrt{d_i^2-y}} \Theta_{\bv_i}^{\bv_i} -  \sn \sum_{i\in \mathcal{I}} \vartheta'(d_i)(\widehat d_i- d_i) + O_\prec\Big(N^{-\varepsilon}\sum_{i \in \mathcal{I}} \frac{1}{\sqrt{d_i^2-y}}\Big) \nonumber\\
&= \sum_{i \in \mathcal{I}} \frac{1}{\sqrt{d_i^2-y}} \Theta_{\bv_i}^{\bv_i} -  \sum_{i\in \mathcal{I}} \vartheta'(d_i)(\theta^{-1})'(\theta(d_i)) \sqrt{d_i^2-y} \, \Phi_i\notag\\
&\qquad+ O_\prec\Big(N^{-\varepsilon}\sum_{i \in \mathcal{I}} \frac{1}{\sqrt{d_i^2-y}}\Big),
\end{align}
for some small constant $\varepsilon>0$. Here in the third step, we used  Remark \ref{rmk_inside} and Assumption \ref{asm.082021} to absorb the $\chi^2$-terms into the error, and also the term $\sn \sum_{i\in \mathcal{I},  j \in \mathcal{I}\setminus\{i\}}|\langle \mb{\xi}_j, \bm{v}_i  \rangle|^2$ which possesses the same bound that can be seen from Proposition \ref{prop.simples} and the definition of $\mb{\varsigma}_{i}^{\mb{u}}$ in (\ref{2020081001}).  More precisely, from statement (2) of Proposition \ref{prop.simples}, we see that for each $i\in \mathcal{I}$, 
\begin{align*}
\sn \sum_{  j \in \mathcal{I}\setminus\{i\}}|\langle \mb{\xi}_j, \bm{v}_i  \rangle|^2 &= O_\prec \Big( \sum_{j \in \mathcal{I}\setminus\{i\}} \frac{\Vert \mb{\varsigma}_j^{\bv_i}\Vert^2 }{\sn d_j} \big( \Delta_{\bu_j}^{\bv_i}\big)^2\Big)\nonumber\\
&= O_\prec \Big(  N^{-\frac12}\sum_{j \in \mathcal{I}, t\in \{j\}^c } \frac{d_jd_t}{(d_j-d_t)^2}\Big)
\end{align*}
where in the last step, we plugged in the definition of $\mb{\varsigma}_{i}^{\mb{u}}$. Then by Assumption \ref{asm.082021}, we can also absorb the above term into the error.

 By elementary computation, we further obtain 
\begin{align*}
\vartheta'(d_i)= \frac{y(d_i^2+2d_i+y)}{d_i^2(d_i+y)^2}, \quad (\theta^{-1})'(\theta(d_i))= (1- yd_i^{-2})^{-1},
\end{align*}
so that the coefficient of $\Phi_i$  in the last line of (\ref{2020081601}) is $$-\frac{y(d_i^2+2d_i+y)}{(d_i+y)^2(d_i^2-y)^{\frac12}}.$$
Therefore, we can apply the result in Proposition \ref{prop.simples}  to finish the proof by the fact that  it is a linear combination of asymptotically Gaussian random variables.

Next, we prove (\ref{eq_statone}) for the all-equal case, i.e. $d_i=d_e$ for all $i\in \mathcal{I}$, and $\mathcal{I}=\mathsf{I}(i)$. Recall the definition of $\mathcal{T}$ in \rf{eq_teststat}. We have 
\begin{align} \label{2020081601*}
 \sn {\mathcal{T}} &= \sn \sum_{t\in \mathcal{I}}\big(\la\bv_t, {\rm{P}}_{\mathcal{I}} \bv_t\ra - \vartheta(d_t) +\vartheta(d_t) - \vartheta(\widehat d_t)\big) 
 \nonumber\\
 &= (d_e^2-y)^{-\frac12} \sum_{t\in \mathcal{I}} \Theta_{\bv_t}^{\bv_t}  
 - \sn \, |\mathcal{I}|\vartheta'(d_e) (\theta^{-1})'(\theta(d_e)) \Big(\frac{1}{|\mathcal{I}|} \sum_{i\in \mathcal{I}} \mu_i- \theta(d_e)\Big) \notag\\
 &\qquad + O_\prec\Big(N^{-\varepsilon}\frac{1}{\sqrt{d_e^2-y}}\Big) \nonumber\\
 &=  (d_e^2-y)^{-\frac12} \sum_{t\in \mathcal{I}} \Theta_{\bv_t}^{\bv_t}  
 - |\mathcal{I}|\vartheta'(d_e) (\theta^{-1})'(\theta(d_e)) (d_e^2-y)^{\frac12} \Phi_{\mathcal{I}}\notag\\
 &\qquad + O_\prec\Big(N^{-\varepsilon}\frac{1}{\sqrt{d_e^2-y}}\Big)
 \end{align}
 for some small constant $\varepsilon>0$.
 Note that by \rf{def:Phi_I} and $d_i=d_e$ for all $i\in \mathcal{I}$, we can rewrite $\Phi_{\mathsf{I}(i)}= \frac{1}{|{\mathsf{I}(i)}|} \sum_{t\in \mathcal{I}} \Phi_{tt} $ by setting $\Phi_{tt}:= -\sqrt{N(d_t^2-y)} \theta(d_t) \chi_{tt}(\theta(d_t))$ which is consistent to the defined $\Phi_i\equiv \Phi_{ii}$ in the simple  case (see Theorem \ref{thm.simple case} and its proof). Hence, we can represent $\sn \mathcal{T}$ by 
 \begin{align} \label{2020081701}
  \sn {\mathcal{T}}=  &\sum_{t\in \mathcal{I}} (d_t^2-y)^{-\frac12} \Theta_{\bv_t}^{\bv_t}  
 -  \sum_{t\in \mathcal{I}}\vartheta'(d_t) (\theta^{-1})'(\theta(d_t)) (d_t^2-y)^{\frac12} \Phi_{tt} \notag\\
 &+ O_\prec\Big(N^{-\varepsilon}\sum_{t \in \mathcal{I}} \frac{1}{\sqrt{d_t^2-y}}\Big)
 \end{align}
 which coincides with the  representation in \rf{2020081601}. And the asymptotic distribution is reduced to the joint distribution of $\{\Theta_{\bv_t}^{\bv_t} \}_{t\in \mathcal{I}}, \{\Phi_{tt}\}_{t\in \mathcal{I}}$ which is asymptotically Gaussian,  sharing the same covariance structure as all-simple case, as we mentioned in Remark \ref{rmk.082102}. Since here $d_i=d_e$ for all $i\in \mathcal{I}$, we can get the explicit expression of the  variance of $\sn \mathcal{T}$ (c.f.\rf{def.V1(d_I)}) by applying (1) of  Proposition \ref{prop.multiple}.
 
 In the end, we claim the result for the general case where $\mathcal{I}= \bigcup_{k=1}^\ell \mathcal{I}_k$ such that all the spikes with indices from the same subset $\mathcal{I}_k$ are equal and distinct otherwise with  distance satisfying Assumptions \ref{supercritical} and \ref{asm.082021} . Especially, we do not further require that each subset is a singleton or not.  We then do the decomposition
 \begin{align*}
  \sn {\mathcal{T}}& = \sn \sum_{k=1}^ \ell \sum_{t\in \mathcal{I}_k}\big(\la\bv_t, {\rm{P}}_{\mathcal{I}} \bv_t\ra - \vartheta(d_t) +\vartheta(d_t) - \vartheta(\widehat d_t)\big)  \nonumber\\
  &=  \sum_{k=1}^ \ell\Big(\sn  \sum_{t\in \mathcal{I}_k}\big(\la\bv_t, {\rm{P}}_{\mathcal{I}_k} \bv_t\ra - \vartheta(d_t) +\vartheta(d_t) - \vartheta(\widehat d_t)\big) \Big) \notag\\
  &\qquad+  O_\prec\Big(N^{-\varepsilon}\sum_{t \in \mathcal{I}} \frac{1}{\sqrt{d_t^2-y}}\Big).
 \end{align*}
 Applying the discussion for the all-equal case to each summand above, we can finish the proof analogously. 
 The details are omitted. Hence, we conclude the proof of Theorem \ref{thm_firststatdist}. 
\end{proof}

\begin{proof} [Proof of Theorem  \ref{thm_secondtestatistics}]
Recall the definition of the statistic \rf{eq_finalstat}. We can write
\begin{align*}
N\mathbb{T}_{2}=N\sum_{j \in \mathcal{J}} \langle\mb{u}_j, {\rm{P}}_\mathcal{I} \bm{u}_j \rangle&=    \sum_{k=1}^ \ell \sum_{j \in \mathcal{J}} N \langle\mb{u}_j, {\rm{P}}_{\mathcal{I}_k} \bm{u}_j\rangle \notag\\
&=\sum_{i\in\mathcal{I}, j\in \mathcal{J}}\frac{\Vert \mb{\varsigma}_{i}^{\bu_j}\Vert^2}{d_i} (\Delta_{\bv_i}^{\bu_j})^2 + O_\prec\Big(N^{-\varepsilon}\sum_{i \in \mathcal{I}} \frac{1}{\sqrt{d_i^2-y}}\Big),
\end{align*}
for some small constant $\varepsilon>0$ by using  statement  (2) of Proposition \ref{prop.multiple} for each summand in the third step. And the asymptotic distribution immediately follows from jointly CLT of $\{\Delta_{\bv_i}^{\bu_j}\}_{i\in \mathcal{I}, j\in \mathcal{J}}$. Then combining the results in
Proposition \ref{prop.simples} and Proposition \ref{prop.multiple} further with Remark \ref{rmk.082102}, we can then finish the proof.

\end{proof}

 In the end, we prove Proposition \ref{repre.multiple.eigv}, by adapting the proof of Proposition 4.5 of \cite{KY14}. We emphasize that the discussion in \cite{KY14} was done for deformed Wigner matrices, where the strength of the deformation is required to be bounded above by a constant. Here we need to adapt the discussion to our model with possibly diverging $d_i$'s. Hence, we need to keep tracking the dependence of the error terms on $d_i$ more carefully. Nevertheless, apart from the size of $d_i$, the discussion is essentially the same as that of Proposition 4.5 of \cite{KY14}.
 \begin{proof} [Proof of  Proposition \ref{repre.multiple.eigv}  ]
Similarly to the proof of Lemma \ref{lem.representation.eigv},  $\mu_t$ is an eigenvalue outside the spectrum of $H$ if and only if $z=\mu_t$ solves the equation ${\rm{det}} (D^{-1}+ zV^*\mG_1(z)V)=0$ or equivalently  $D^{-1}+ \mu_tV^*\mG_1(\mu_t)V$ has a zero eigenvalue. 
By Lemma \ref{locationeig}, we see that for $t\in {\mathsf{I}(i)}$, $|\mu_t-\theta(d_i)|= O_\prec(d_i^{1/2} \delta_{i0}^{1/2}N^{-1/2})$.  Without loss of generality, up to a permutation, we assume the following decomposition
\begin{align} \label{note:2020080301}
D^{-1}= D_{\mathsf{I}(i)}^{-1} \oplus D_{{\mathsf{I}(i)}^c}^{-1}, \quad V= (V_{\mathsf{I}(i)}, V_{{\mathsf{I}(i)}^c}). 
\end{align}
Here we recall ${\mathsf{I}(i)}^c=\llbracket 1, r\rrbracket\setminus {\mathsf{I}(i)}$. 
 For simplicity, we set the following neighborhood of  $\theta(d_i)$
\begin{align}
\mathcal{D}(\theta(d_i)) &= \big(\theta(d_i)- d_i^{1/2}\delta_{i0}^{1/2}N^{-1/2 + \delta}, \theta(d_i)+ d_i^{1/2}\delta_{i0}^{1/2}N^{-1/2 + \delta} \big) \notag\\
&=: (\theta^-, \theta^+ ) 
\end{align}
 for some sufficiently small $\delta$. We will identify those $\mu_t$'s by analysing the behaviour of eigenvalues of $D^{-1}+ xV^*\mG_1(x)V$ as $x$ varies. Our discussion in the sequel will rely on the fact  that the eigenvalues of $D^{-1}+ xV^*\mG_1(x)V$ should not attain zero simultaneously as $x$ varies in $\mathcal{D}(\theta(d_i)) $ with high probability, so that we can exactly find all the $\mu_t$ for $t\in {\mathsf{I}(i)}$. In order to see this, similarly to the counterpart in \cite{KY14},  
we introduce an additional randomness, by adding some small perturbation $\iota \Delta$ where $\Delta$ is a Hermitian matrix with entries i.i.d. and has an absolutely continuous law supported in the unit disc. And the scalar $\iota\equiv \iota(N)>0$ can be chosen arbitrarily small, say $e^{-N}$. We now turn to  study the behaviours of eigenvalues of $D^{-1}+ xV^*\mG_1(x)V+ \iota \Delta$ instead. 

 First, for $x\in \mathcal{D}(\theta(d_i)) $ for some sufficiently small $\delta$, we define two matrices $\mathcal{A}^{\iota}(x)$ and $\wt {\mathcal{A}}^{\iota}(x)$  by 
\begin{align*}
&\mathcal{A}^{\iota}(x) := D^{-1}+ xV^*\mG_1(x)V+ \iota \Delta- xm_1(x){\rm{I}}_r, \notag\\
& \wt {\mathcal{A}}^{\iota}(x):= \mathcal{A}_{\mathsf{I}(i)}^{\iota}(x) \oplus \mathcal{A}_{{{\mathsf{I}(i)}^c}}
^{\iota}(x).
\end{align*}
Here we use the notation $\mathcal{A}^{\iota}_{{I}}(x)$ to represent a submatrix of $\mathcal{A}^{\iota}(x)$ by taking columns  and rows  both from $I$. In particular, $\mathcal{A}_{\mathsf{I}(i)}^{\iota}(x)= D_{\mathsf{I}(i)}^{-1}+ xV^*_{\mathsf{I}(i)} \mG_1(x)V_{\mathsf{I}(i)} + \iota \Delta_{\mathsf{I}(i)} - xm_1(x) {\rm{I}}_{|{\mathsf{I}(i)}|}$, where $\Delta_{\mathsf{I}(i)}$ is the submatrix of $\Delta$ with row and column indices both from $I$.

By definition, we see that $\wt {\mathcal{A}}^{\iota}(x)$ is obtained from ${\mathcal{A}}^{\iota}(x)$ by taking the two  block matrices on diagonal.  Using isotropic local law \rf{est.DG} and taking $\iota$ small enough, say $e^{-N}$, we can easily get 
\begin{align} \label{2020073101}
\Vert {\mathcal{A}}^{\iota}(x)- \wt {\mathcal{A}}^{\iota}(x)\Vert_{\text{op}} \prec  d_i^{-\frac12} \delta_{i0}^{-\frac12} N^{-\frac12}. 
\end{align}
Next, we claim that the spectrum of the two block matrices in $\wt{\mathcal{A}}^{\iota}(x)$ are well separated, i.e. we will establish  a lower bound on the spectral gap 
\begin{align*}
{\rm{dist}} \Big( \text{Spec}\big(\mathcal{A}_{\mathsf{I}(i)}^{\iota}(x)\big), \text{Spec}\big(\mathcal{A}_{{{\mathsf{I}(i)}^c}}^{\iota}(x)\big)\Big).
\end{align*}
This can be easily achieved by using triangle inequality and isotropic local law \rf{est.DG}, by choosing $\iota$ sufficiently small, 
\begin{align*}
{\rm{dist}} \Big( \text{Spec}\big(\mathcal{A}_{{\mathsf{I}(i)}^c_l}^{\iota}(x)\big), \text{Spec}\big(\mathcal{A}_{\mathsf{I}(i)}^{\iota}(x)\big)\Big) \geq {\rm{dist } } \big( D_{\mathsf{I}(i)}^{-1}, D_{{{\mathsf{I}(i)}^c}}^{-1} \big)  - d_i^{-\frac12} \delta_{i0}^{-\frac12} N^{-\frac12+\frac{\varepsilon}{2}}
\end{align*}
with high probability. 
Then by non-overlapping condition \rf{asd2}, we see that $ {\rm{dist } } \big( D_{\mathsf{I}(i)}^{-1}, D_{{{\mathsf{I}(i)}^c}}^{-1} \big) > d_i^{-\frac12} \delta_{i0}^{-\frac12} N^{-\frac12 + \varepsilon}$ for some small fixed $\varepsilon>0$. This implies with high probability, 
\begin{align}\label{2020073102}
{\rm{dist}} \Big(\text{Spec}\big(\mathcal{A}_{\mathsf{I}(i)}^{\iota}(x)\big),\text{Spec}\big(\mathcal{A}_{{{\mathsf{I}(i)}^c}}^{\iota}(x)\big)\Big)> d_i^{-\frac12} \delta_{i0}^{-\frac12} N^{-\frac12 + \varepsilon} 
\end{align}
by slightly modifying the value of $\varepsilon$. 

Further, we define $ \{\wt a_t^\iota(x)\}_{t\in {\mathsf{I}(i)}}$ (resp. $ \{\wt a_j^\iota(x)\}_{j\in {\mathsf{I}(i)}^c}$) the eigenvalues of $\wt{\mathcal{A}}_{{I}}^{\iota}(x)$  (resp. $\wt{\mathcal{A}}_{{{\mathsf{I}(i)}^c}}^{\iota}(x)$) in an increasing order. And they are actually the eigenvalues of $\wt{\mathcal{A}}^{\iota}(x)$ due to its block structure. Correspondingly, we denote by $a_k^\iota(x)$ the eigenvalues of $\mathcal{A}^\iota(x)$.  Note that for $x\in \mathcal{D}(\theta(d_i))$, by isotropic local law \rf{est.DG}
\begin{align*}
&\wt a_t^\iota(x) =  1+ d_i^{-1} +O_\prec(d_i^{-1/2} \delta_{i0}^{-1/2}N^{-1/2}), \text{ for } t\in {\mathsf{I}(i)}\nonumber\\
&\wt a_j^\iota(x) =  1+ d_j^{-1} +O_\prec(d_i^{-1/2} \delta_{i0}^{-1/2}N^{-1/2}), \text{ for } j\in {\mathsf{I}(i)}^c.
\end{align*}
This with \rf{2020073101} and non-overlapping condition \rf{asd2}, implies that for $j\in {\mathsf{I}(i)}^c$,
$$ a_j^\iota(x) + xm_1(x) \asymp d_j^{-1}-d_i^{-1}.$$
Hence, $a_j^\iota(x) =- xm_1(x)$ get no solution on $\mathcal{D}(\theta(d_i))$ for $j\in {\mathsf{I}(i)}^c$.

We then further set   $H^{\iota}:= \Sigma_{\iota}^{1/2} XX^* \Sigma_{\iota}^{1/2}$ where we define that $\Sigma_{\iota}=  I+ V ({\rm{\diag}}(\{1/d_t\}_{t\in {\mathsf{I}(i)}}, \{1/d_j\}_{j\in {\mathsf{I}(i)}^c} )+ \iota\Delta)^{-1}V^*$ with $V= (V_{\mathsf{I}(i)}, V_{{\mathsf{I}(i)}^c})$.
Note that since $\iota $ can be choosing arbitrarily small, then by Lemma \ref{locationeig}, we get that $H^{\iota}$ has with high probability exactly $|{\mathsf{I}(i)}|$ eigenvalues $\{\mu^\iota_t\}_{t\in {\mathsf{I}(i)}}$ in $ \mathcal{D}(\theta(d_i))$ and $\mu_t^{\iota} -\theta(d_i) = O_\prec(d_i^{1/2} \delta_{i0}^{1/2}N^{-1/2})$ further with  $\mu_t^\iota m_1(\mu_t^\iota)= -(1+1/d_i) +O_\prec(d_i^{-1/2} \delta_{i0}^{-1/2}N^{-1/2}) $. Here tentatively we  do not specify the ordering of the eigenvalues $\{\mu^\iota_t\}_{t\in {\mathsf{I}(i)}}$. Later we will see how to  pin down each $\mu_t^\iota$ for $t\in {\mathsf{I}(i)}$.
Recall the  notations in  \rf{note:2020080301}. By elementary computation,
 ${\rm{det}}(D^{-1}+ \mu_t^\iota V^*\mG_1(\mu_t^\iota)V+ \iota \Delta)=0$ which implies that $\mathcal{A}^\iota(\mu_t^\iota)$ has an eigenvalue $-\mu_t^\iota m_1(\mu_t^\iota)$. 
 
  From all the above arguments, we see that all those $\mu_t^\iota$ can  be a solution to the equation $a_t^\iota(x)=-xm_1(x)$ only when $t\in {\mathsf{I}(i)}$ and $x\in \mathcal{D}(\theta(d_i))$.  
We can further claim that for all $x\in \mathcal{D}(\theta(d_i)) $,  $a_t^{\iota}(x)= -xm_1(x)$ for at most one $t\in {\mathsf{I}(i)}$ (a.s.)
by referring to  the proof of Proposition 4.5 in \cite{KY14}, thanks to the continuously distributed perturbation term $\iota\Delta$.
Next, we can claim that $a_t^{\iota}(x)= -xm_1(x)$ indeed has  solution in $\mathcal{D}(\theta(d_i))$ for each $t\in {\mathsf{I}(i)}$ by the continuity of $a_t^{\iota}(x), -xm_1(x)$ and  showing that 
\begin{align}
a_t^{\iota}(\theta^-)< -\theta^-m_1(\theta^-), \quad a_t^{\iota}(\theta^+) > -\theta^+m_1(\theta^+).
\end{align}
We will show the detail for the first inequality above and the second one can be done analogously.  By isotropic local law \rf{est.DG}, we have 
\begin{align}
a_t^{\iota}(\theta^-) + \theta^-m_1(\theta^-)& = \wt a_t^{\iota}(\theta^-)+ \theta^-m_1(\theta^-) + O_\prec(d_i^{-1/2} \delta_{i0}^{-1/2}N^{-1/2})\nonumber\\
& = -\theta(d_i)m_1(\theta(d_i)) +\theta^-m_1(\theta^-)+ O_\prec(d_i^{-1/2} \delta_{i0}^{-1/2}N^{-1/2}) \nonumber\\
&= \frac{\theta^- -\theta(d_i)}{d_i^2-y} + O_\prec(d_i^{-1/2} \delta_{i0}^{-1/2}N^{-1/2}) \nonumber\\
&= - O_\prec(d_i^{-1/2} \delta_{i0}^{-1/2}N^{-1/2+\delta}) <0.
\end{align}

The above arguments state that  with high probability each equation  $a_t^{\iota}(x)= -xm_1(x)$ identifies at least one distinct $\mu_t^\iota$ for $t\in {\mathsf{I}(i)}$, and totally, we have $|{\mathsf{I}(i)}|$ eigenvalues of $H^\iota$  sitting in $\mathcal{D}(\theta(d_i))$. Thus with high probability, we shall have a one-to-one correspondence between  $\mu_t^\iota$ and solution of  $a_t^{\iota}(x)= -xm_1(x)$ in $\mathcal{D}(\theta(d_i))$ for $t\in {\mathsf{I}(i)}$ so that  
 $-\mu_t^\iota m_1(\mu_t^\iota) = a_{ t }^\iota( \mu_t^\iota)$.
 Next, by isotropic local law \rf{est.DG}, one can also check $\Vert \partial \mathcal{A}^\iota(x)\Vert_{op}\prec d_i^{1/2} \delta_{i0}^{-5/2}N^{-1/2}$  which together  with $\mu_t^{\iota} -\theta(d_i) = O_\prec(d_i^{1/2} \delta_{i0}^{1/2}N^{-1/2})$ and $\rf{asd}$ further leads to 
 \begin{align} \label{2020080302}
  a_{ t }^\iota( \mu_t^\iota)-  a_{ t }^\iota( \theta(d_i)) = O_\prec(d_i^{-\frac12}\delta_{i0}^{-\frac 12}N^{-\frac12-\varepsilon})
 \end{align}
 for some small fixed $\varepsilon>0$.
 
 Furthermore, by using perturbation theory we can claim that 
 \begin{align} \label{2020080303}
  a_{ t }^\iota( \theta(d_i)) = \wt a_{ t }^\iota( \theta(d_i)) + O_\prec(d_i^{-\frac12}\delta_{i0}^{-\frac 12}N^{-\frac12-\varepsilon}),
 \end{align}
which can be checked by applying Proposition A.1 of \cite{KY14}  that 
\begin{align*}
\big| a_{ t }^\iota( \theta(d_i)) - \wt a_{ t }^\iota( \theta(d_i)) \big| 
&\prec \frac{\Vert \mathcal{A}^\iota - \wt{\mathcal{A}}^\iota\Vert_{\text{op}}^2}{{\rm{dist}} \Big( \sigma\big(\mathcal{A}_{\mathsf{I}(i)}^{\iota}(x)\big), \sigma\big(\mathcal{A}_{{{\mathsf{I}(i)}^c}}^{\iota}(x)\big)\Big) - \Vert \mathcal{A}^\iota - \wt{\mathcal{A}}^\iota\Vert_{\text{op}}  } \nonumber\\
&= O_\prec(d_i^{-\frac12}\delta_{i0}^{-\frac 12}N^{-\frac12-\varepsilon}).
\end{align*}
Combining \rf{2020080302}, \rf{2020080303} and  $-\mu_t^\iota m_1(\mu_t^\iota) = a_{ t }^\iota( \mu_t^\iota)$, we eventually arrive at 
\begin{align*}
-\mu_t^\iota m_1(\mu_t^\iota)  &= \wt a_{ t }^\iota( \theta(d_i)) + O_\prec(d_i^{-\frac12}\delta_{i0}^{-\frac 12}N^{-\frac12-\varepsilon}) \nonumber\\
&= \lambda_ t \Big(D_{\mathsf{I}(i)}^{-1}+ \theta(d_i)V^*_{\mathsf{I}(i)} \mG_1(\theta(d_i))V_{\mathsf{I}(i)} + \iota \Delta_{\mathsf{I}(i)} - \theta(d_i)m_1(\theta(d_i)) {\rm{I}}_{|{\mathsf{I}(i)}|}\Big)\notag\\
&\quad +  O_\prec(d_i^{-\frac12}\delta_{i0}^{-\frac 12}N^{-\frac12-\varepsilon}).
\end{align*}
Here we emphasize that we use the notation $\lambda_t(A_{\mathsf{I}(i)})$ for a $|{\mathsf{I}(i)}|\times |{\mathsf{I}(i)}|$ Hermitian matrix $A_{\mathsf{I}(i)}$ indexed by $|{\mathsf{I}(i)}|$ to denote its eigenvalues that admit the ordering  $\lambda_{t_1}(A_{\mathsf{I}(i)})\leq \lambda_{t_2}(A_{\mathsf{I}(i)})$ for any $t_1, t_2\in {\mathsf{I}(i)}, t_1\leq t_2$. Using the identity $1+ d_i^{-1}= - \theta(d_i)m_1(\theta(d_i))$ and expanding $\mu_t^\iota m_1(\mu_t^\iota)  $ around $ \theta(d_i)m_1(\theta(d_i))$,  we get 
\begin{align*}
\mu_t^\iota - \theta(d_i)&= -\frac{1}{d_i^2-y} \lambda_t\Big( \theta(d_i) V_{\mathsf{I}(i)}^* \Xi(\theta(d_i))V_{\mathsf{I}(i)} + \iota \Delta_{\mathsf{I}(i)}\Big)
\notag\\
&\quad +  O_\prec(d_i^{-\frac12}\delta_{i0}^{-\frac 12}N^{-\frac12-\varepsilon}).
\end{align*}
In the end, let $\iota\to 0$, we ultimately proved  Proposition   \ref{repre.multiple.eigv}.
\end{proof}
 
\begin{proof}[Proofs of Corollaries \ref{cor.082001} and  \ref{cor_simulateddistribution}] The conclusions in Corollaries \ref{cor.082001} and  \ref{cor_simulateddistribution} follow from Theorems \ref{thm_firststatdist}  and \ref{thm_secondtestatistics} and the fact that $\widehat{\bm{d}}_{\mathcal{I}}$ (resp. $\widehat{\bm{d}}$)  is a consistent estimator of 
$\bm{d}_{\mathcal{I}}$ (resp. $\bm{d}$) after appropriate scaling. 
\end{proof}
 
 \section{Additional simulations} \label{sec_additionalsimulation} 
In this section, we provide additional simulation results regarding the two-point random variable $\frac{1}{3} \delta_{\sqrt{2}}+\frac{2}{3} \delta_{-\frac{1}{\sqrt{2}}}$.

\begin{table}[!ht]
\setlength\arrayrulewidth{1.5pt}
\renewcommand{\arraystretch}{1.4}
 \captionsetup{width=1\linewidth}
 \addtolength{\tabcolsep}{-0.5pt} 
\begin{center}
\scriptsize
\hspace*{-0.8cm}
{
\begin{tabular}{clclclclclclc|c|c|c|c}
\toprule
 & \multicolumn{5}{c}{$N=200$}                                                                                                                       & \multicolumn{5}{c}{$N=500$}                                                                                                                       \\ \hline
Method   & $d=2$& $d=5$  & $d=10$  & $d=50$  & $ d=100$  & $d=2$ & $d=5$ & $d=10$ & $d=50$ & $d=100$ \\ \hline
$\texttt{Fr-bootstrap}$ & 0.045 &  0.049  &  0.053  & 0.054  & 0.063 & 0.045 &  0.047 &  0.053  & 0.061 & 0.063  \\
$\texttt{Fr-Bayes}$ & {0.053} & {\bf 0.059} & 0.063  & 0.072  & 0.087 & 0.052 & {\bf 0.062}  &  {\bf 0.073}& {0.084}  & {0.094} \\
$\texttt{En-bootstrap}$& 0.038  & 0.042  &  0.047 & 0.052 & 0.061 & 0.049 &  0.051 & 0.054 &  0.062 & 0.068  \\
$\texttt{En-Bayes}$ &  0.053 &  0.058 & {\bf 0.073}  &  0.082 & {\bf 0.093}  & 0.058  & 0.059 & 0.068 & 0.085 & {0.096}  \\
$\texttt{Fr-Datadriven}$& 0.042  & 0.043  & 0.046  & 0.051 & 0.057 &  0.045 & 0.046 & 0.051  & 0.049  & 0.059   \\
$\texttt{HPV-LeCam}$& 0.394 & 0.419 & 0.424  & 0.39 & 0.4   & 0.395 & 0.387 & 0.423 &  0.392 & 0.412 \\
$\texttt{Sp-bootstrap}$	& 0.048   & 0.051 & 0.047 & 0.056 & 0.063  & 0.047   & 0.039 & 0.053  &  0.058 & 0.071\\
$\texttt{Sp-Bayes}$&  {\bf 0.061}   & {\bf 0.059}    & {\bf 0.073}  & {\bf 0.094}  & {\bf 0.097}   &  {\bf 0.059} & {0.054} & 0.069 & {\bf 0.089} &  {\bf 0.099} \\
$\texttt{Fr-Adaptive}$ & {\bf 0.11}  & {\bf 0.102} & {\bf 0.105} &  {\bf 0.099} & {\bf 0.093}  & {\bf 0.102}  & {\bf 0.095} & {\bf 0.102} & {\bf 0.094}  & {\bf 0.098} \\ 
 \bottomrule
\end{tabular}
}
\end{center}
\caption{ Scenario I: simulated type I error rates under the nominal level $0.1$ for $y=0.1$.  We report our results based on 2,000 Monte-Carlo simulations with two-point random variables. We highlighted the two most accurate methods for each value of $d.$   
} \label{table_singutp}
\end{table}

\begin{table}[!ht]
\setlength\arrayrulewidth{1.5pt}
\renewcommand{\arraystretch}{1.4}
 \captionsetup{width=1\linewidth}
 \addtolength{\tabcolsep}{-0.5pt} 
\begin{center}
\scriptsize
\hspace*{-0.8cm}
{
\begin{tabular}{clclclclclclc|c|c|c|c}
\toprule
 & \multicolumn{5}{c}{$N=200$}                                                                                                                       & \multicolumn{5}{c}{$N=500$}                                                                                                                       \\ \hline
Method   & $d=2$& $d=5$  & $d=10$  & $d=50$  & $ d=100$  & $d=2$ & $d=5$ & $d=10$ & $d=50$ & $d=100$ \\ \hline
$\texttt{Fr-bootstrap}$ & 0.052   & 0.047   & 0.047   & 0.046 & 0.054 &  0.045 &  0.039 &  0.042  & 0.049  &  0.052 \\
$\texttt{Fr-Bayes}$ &  0.048 & 0.054  & 0.061  & 0.069  & {\bf 0.088} & {\bf 0.055} & 0.048  & 0.046 & 0.062  & 0.084 \\
$\texttt{En-bootstrap}$& 0.046  & 0.049  &  0.051 & 0.046 & 0.048 & {0.052} &  0.039 & 0.049 &  0.055 & 0.062  \\
$\texttt{En-Bayes}$ & {\bf 0.059 } &  {\bf 0.062} & {\bf 0.067}  &  0.072 & {0.086}  & 0.054  & {\bf 0.067} & {\bf 0.072} & {\bf 0.086} & {\bf 0.088}  \\
$\texttt{Fr-Datadriven}$& 0.039  & 0.042  & 0.046  & 0.043 & 0.052 &  0.045 & 0.052 & 0.043  & 0.052  & 0.058   \\
$\texttt{HPV-LeCam}$& 0.783 & 0.779 & 0.813  & 0.793 & 0.754   & 0.789 & 0.832 & 0.789 &  0.787 & 0.813 \\
$\texttt{Sp-bootstrap}$	&  {\bf 0.059} & 0.057  & 0.049 & 0.054 &  0.053 & 0.049  & {0.051} & 0.058  & 0.061 & 0.059 \\
$\texttt{Sp-Bayes}$& 0.057   &  0.059    & {0.066}  & {\bf 0.081}  & 0.085  & 0.049 & 0.051 &  {0.063} &  {0.075} & {0.084} \\
$\texttt{Fr-Adaptive}$ & {\bf 0.108}   & {\bf 0.092}  & {\bf 0.103}  & {\bf 0.095}  &  {\bf 0.096}  & {\bf 0.11}  & {\bf 0.097}  & {\bf 0.097}  & {\bf 0.102}  & {\bf 0.104}  \\ 
 \bottomrule
\end{tabular}
}
\end{center}
\caption{ Scenario I: simulated type I error rates under the nominal level $0.1$ for $y=1$.  We report our results based on 2,000 Monte-Carlo simulations with two-point random variables. We highlighted the two most accurate methods for each value of $d.$  
} \label{table_singu1tp}
\end{table}

\begin{table}[!ht]
\setlength\arrayrulewidth{1.5pt}
\renewcommand{\arraystretch}{1.4}
 \captionsetup{width=1\linewidth}
 \addtolength{\tabcolsep}{-0.5pt} 
\begin{center}
\scriptsize
\hspace*{-0.8cm}
{
\begin{tabular}{clclclclclclc|c|c|c|c}
\toprule
 & \multicolumn{5}{c}{$N=200$}                                                                                                                       & \multicolumn{5}{c}{$N=500$}                                                                                                                       \\ \hline
Method   & $d=2$& $d=5$  & $d=10$  & $d=50$  & $ d=100$  & $d=2$ & $d=5$ & $d=10$ & $d=50$ & $d=100$ \\ \hline
$\texttt{Fr-bootstrap}$ & 0.031  & 0.042  & 0.039   & 0.043 & 0.052 &  0.035&  0.043 &  0.045 & 0.051 & 0.049  \\
$\texttt{Fr-Bayes}$ &  {\bf 0.043} & 0.047  & 0.049  & 0.052  & 0.059 & 0.041 &  0.045  & 0.039 & 0.052  & {\bf 0.068} \\
$\texttt{En-bootstrap}$& 0.039 & 0.046  & {\bf 0.052}  & 0.057 & 0.059 & 0.041 & 0.045  & 0.039  &0.057   & 0.052   \\
$\texttt{En-Bayes}$ & 0.038  & {0.043}  & {\bf 0.052}   &  {\bf 0.061}  &  {\bf 0.068} & 0.041 &  0.049 &  0.043 & {0.057}  & 0.064   \\
$\texttt{Fr-Datadriven}$& 0.034  & 0.038   & 0.029  & 0.042 & 0.048 &  0.029  & 0.035  &  0.044  & 0.051   & 0.059  \\
$\texttt{HPV-LeCam}$& 0.912  & 0.895 & 0.923   & 0.944 &  0.957  & 0.933 & 0.897 & 0.962 &  0.932 & 0.955  \\
$\texttt{Sp-bootstrap}$	& 0.041  & 0.043  & 0.049 & 0.052  & 0.058   & 0.039  & 0.048  & 0.042 & 0.052 & 0.061 \\
$\texttt{Sp-Bayes}$& {0.039}   & {\bf 0.048}   & {\bf 0.052}   &{\bf 0.061}   & 0.059   &  {\bf 0.048} & {\bf 0.053} & {\bf 0.049} & {\bf 0.061} & {\bf 0.068} \\
$\texttt{Fr-Adaptive}$ & {\bf 0.102}  & {\bf 0.107} & {\bf 0.099}  &{\bf 0.103}  & {\bf 0.095}  &  {\bf 0.102} & {\bf 0.096} & {\bf 0.103} & {\bf 0.094}  & {\bf 0.093}  \\ 
 \bottomrule
\end{tabular}
}
\end{center}
\caption{ Scenario I: simulated type I error rates under the nominal level $0.1$ for $y=10$.  We report our results based on 2,000 Monte-Carlo simulations with two-point random variables. We highlighted the two most accurate methods for each value of $d.$ 
} \label{table_singu2tp}
\end{table}

\begin{table}[!ht]
\setlength\arrayrulewidth{1.5pt}
\renewcommand{\arraystretch}{1.4}
 \captionsetup{width=1\linewidth}
 \addtolength{\tabcolsep}{-0.5pt} 
\begin{center}
\scriptsize
\hspace*{-0.8cm}
{
\begin{tabular}{clclclclclclc|c|c|c|c}
\toprule
 & \multicolumn{5}{c}{$N=200$}                                                                                                                       & \multicolumn{5}{c}{$N=500$}                                                                                                                       \\ \hline
Method   & $d=2$& $d=5$  & $d=10$  & $d=50$  & $ d=100$  & $d=2$ & $d=5$ & $d=10$ & $d=50$ & $d=100$ \\ \hline
$\texttt{Fr-bootstrap}$ &  0.045 &  0.053   &  0.058   &  0.065 & 0.063  & 0.039 & 0.046 & 0.048  & 0.055 & 0.062  \\
$\texttt{Fr-Bayes}$ &  0.055 & {\bf 0.064} & 0.068   & {\bf 0.089}  & {\bf 0.095} & 0.062 & {\bf 0.075}  & {\bf 0.085} &  {\bf 0.092} & {\bf 0.095} \\
$\texttt{En-bootstrap}$&  0.042  & {0.056}  & 0.059  &  0.063 & 0.059  & 0.041 &  0.049  &  0.058 &  0.057  & 0.059   \\
$\texttt{En-Bayes}$ &{\bf 0.059}  & 0.06   & 0.062  & 0.085 &  {0.094}  & {\bf 0.064} & {0.067} & 0.073 & 0.085 & {0.089}  \\
$\texttt{Fr-Datadriven}$& 0.034 & 0.033  & 0.042 & 0.049 & 0.052& 0.039  & 0.045  & 0.048 & 0.052 &  0.055   \\
$\texttt{HPV-LeCam}$& 0.384 & 0.379 & 0.363  & 0.397 & 0.333   & 0.354 & 0.321 & 0.343 &  0.352 & 0.393 \\
$\texttt{Sp-bootstrap}$	& {0.047}  & 0.048  &  0.053 &  0.065 & 0.068  &  0.039  & 0.043  &  0.053  & 0.058 & 0.065 \\
$\texttt{Sp-Bayes}$&  {0.057}  & 0.063     & {\bf 0.069}  & 0.088  & 0.092  &  0.063 & {0.074} & {0.082} & {0.089} & {\bf 0.095} \\
$\texttt{Fr-Adaptive}$ & {\bf 0.091} & {\bf 0.103} &  {\bf 0.104} &  {\bf 0.095} & {\bf 0.103}  & {\bf 0.103}  & {\bf 0.103} & {\bf 0.093} & {\bf 0.096}  & {\bf 0.098} \\ 
 \bottomrule
\end{tabular}
}
\end{center}
\caption{ Scenario II: simulated type I error rates under the nominal level $0.1$ for $y=0.1$.  We report our results based on 2,000 Monte-Carlo simulations with two-point random variables. We highlighted the two most accurate methods for each value of $d.$   
} \label{table_singutpdenegate}
\end{table}

\begin{table}[!ht]
\setlength\arrayrulewidth{1.5pt}
\renewcommand{\arraystretch}{1.4}
 \captionsetup{width=1\linewidth}
 \addtolength{\tabcolsep}{-0.5pt} 
\begin{center}
\scriptsize
\hspace*{-0.8cm}
{
\begin{tabular}{clclclclclclc|c|c|c|c}
\toprule
 & \multicolumn{5}{c}{$N=200$}                                                                                                                       & \multicolumn{5}{c}{$N=500$}                                                                                                                       \\ \hline
Method   & $d=2$& $d=5$  & $d=10$  & $d=50$  & $ d=100$  & $d=2$ & $d=5$ & $d=10$ & $d=50$ & $d=100$ \\ \hline
$\texttt{Fr-bootstrap}$ & 0.037   & 0.042  &  0.048   & 0.053  & 0.061  &  0.042 & 0.048 & 0.051  & 0.057  & 0.059  \\
$\texttt{Fr-Bayes}$ & {0.047}  &  0.049  & {0.056}  & {\bf 0.064} &  {\bf 0.073} & {\bf 0.051}  & {0.049}  & {\bf 0.062} &  {\bf 0.071}  & {\bf 0.078} \\
$\texttt{En-bootstrap}$&  0.039 & 0.041  & 0.049 &  0.058 & {0.062} & {0.048} & 0.041  & 0.053 &  0.059 & 0.058   \\
$\texttt{En-Bayes}$  & 0.052  &{0.049}  &   {\bf 0.058} & 0.062 & 0.064  & 0.049  &  0.053 & {0.059} & 0.061 & 0.063  \\
$\texttt{Fr-Datadriven}$& 0.031  & 0.029  & 0.042  & 0.046 & 0.051 & 0.047  & 0.052  &  0.042& 0.051 & 0.049   \\
$\texttt{HPV-LeCam}$& 0.812 & 0.795 & 0.773  & 0.862 & 0.767   & 0.815 & 0.833 & 0.799 &  0.787 & 0.813 \\
$\texttt{Sp-bootstrap}$	&  0.042 & {\bf 0.053} & {\bf 0.058} & {\bf 0.064} & 0.059  &  {\bf 0.051}  & {\bf 0.058} & 0.047  & 0.058 & 0.059 \\
$\texttt{Sp-Bayes}$&   {\bf 0.056} & 0.049    & 0.052  & {0.058}   & 0.063  & 0.042 & 0.051 & 0.057 & 0.063 & 0.059 \\
$\texttt{Fr-Adaptive}$ & {\bf 0.11} & {\bf 0.11} & {\bf 0.104}  & {\bf 0.097}  &  {\bf 0.095}  & {\bf 0.102}  & {\bf 0.092} & {\bf 0.094} & {\bf 0.107}  & {\bf 0.103} \\ 
 \bottomrule
\end{tabular}
}
\end{center}
\caption{ Scenario II: simulated type I error rates under the nominal level $0.1$ for $y=1$.  We report our results based on 2,000 Monte-Carlo simulations with two-point random variables.  We highlighted the two most accurate methods for each value of $d.$  
} \label{table_singu1tpdenegrate}
\end{table}

\begin{table}[!ht]
\setlength\arrayrulewidth{1.5pt}
\renewcommand{\arraystretch}{1.4}
 \captionsetup{width=1\linewidth}
 \addtolength{\tabcolsep}{-0.5pt} 
\begin{center}
\scriptsize
\hspace*{-0.8cm}
{
\begin{tabular}{clclclclclclc|c|c|c|c}
\toprule
 & \multicolumn{5}{c}{$N=200$}                                                                                                                       & \multicolumn{5}{c}{$N=500$}                                                                                                                       \\ \hline
Method   & $d=2$& $d=5$  & $d=10$  & $d=50$  & $ d=100$  & $d=2$ & $d=5$ & $d=10$ & $d=50$ & $d=100$ \\ \hline
$\texttt{Fr-bootstrap}$ & 0.035   & 0.037  & 0.051    & 0.049 & 0.059 & 0.036 & {0.045} & 0.054  & 0.057  & 0.055 \\
$\texttt{Fr-Bayes}$ & {\bf 0.049} & 0.047  & {\bf 0.055}  & {0.059}  & {0.062} & 0.043 & 0.051  & 0.057 & {0.053}  &{\bf 0.062} \\
$\texttt{En-bootstrap}$& 0.041  & 0.042  &  0.053 & 0.052 & 0.049 & {\bf 0.048} &  0.053 & 0.049 &  0.055 & {\bf 0.062}  \\
$\texttt{En-Bayes}$ &  0.048 &  0.052 & 0.049  &  0.053 & {0.061}  & 0.043  & {\bf 0.056} & {\bf 0.061} & 0.058 & 0.059  \\
$\texttt{Fr-Datadriven}$& 0.034  & 0.041  & 0.039  & 0.052 & 0.049 &  0.046 & 0.037 & 0.049  & 0.039  & 0.046   \\
$\texttt{HPV-LeCam}$& 0.894 & 0.913 & 0.952  & 0.934 & 0.933   & 0.898 & 0.913 & 0.922 &  0.935 & 0.911 \\
$\texttt{Sp-bootstrap}$	& 0.039   & 0.041  &  0.047 & 0.051 & 0.053  & 0.042   & 0.049 & 0.052  & 0.049 & 0.057 \\
$\texttt{Sp-Bayes}$&  {\bf 0.049}  & {\bf 0.059}    &  {0.054} & {\bf 0.062}  & {\bf 0.064} & 0.045 & 0.053 & 0.056 & {\bf 0.061} & 0.059\\
$\texttt{Fr-Adaptive}$ & {\bf 0.097}  & {\bf 0.102} & {\bf 0.094}  & {\bf 0.101} & {\bf 0.103}  & {\bf 0.094} & {\bf 0.103} & {\bf 0.103} &{\bf 0.105}  & {\bf 0.104}  \\ 
 \bottomrule
\end{tabular}
}
\end{center}
\caption{ Scenario II: simulated type I error rates under the nominal level $0.1$ for $y=10$.  We report our results based on 2,000 Monte-Carlo simulations with two-point random variables.  We highlighted the two most accurate methods for each value of $d.$ 
} \label{table_singu2tpdegenerate}
\end{table}

\begin{figure}[!ht]
\hspace*{-1.4cm}
\begin{subfigure}{0.45\textwidth}
\includegraphics[width=6.5cm,height=4.8cm]{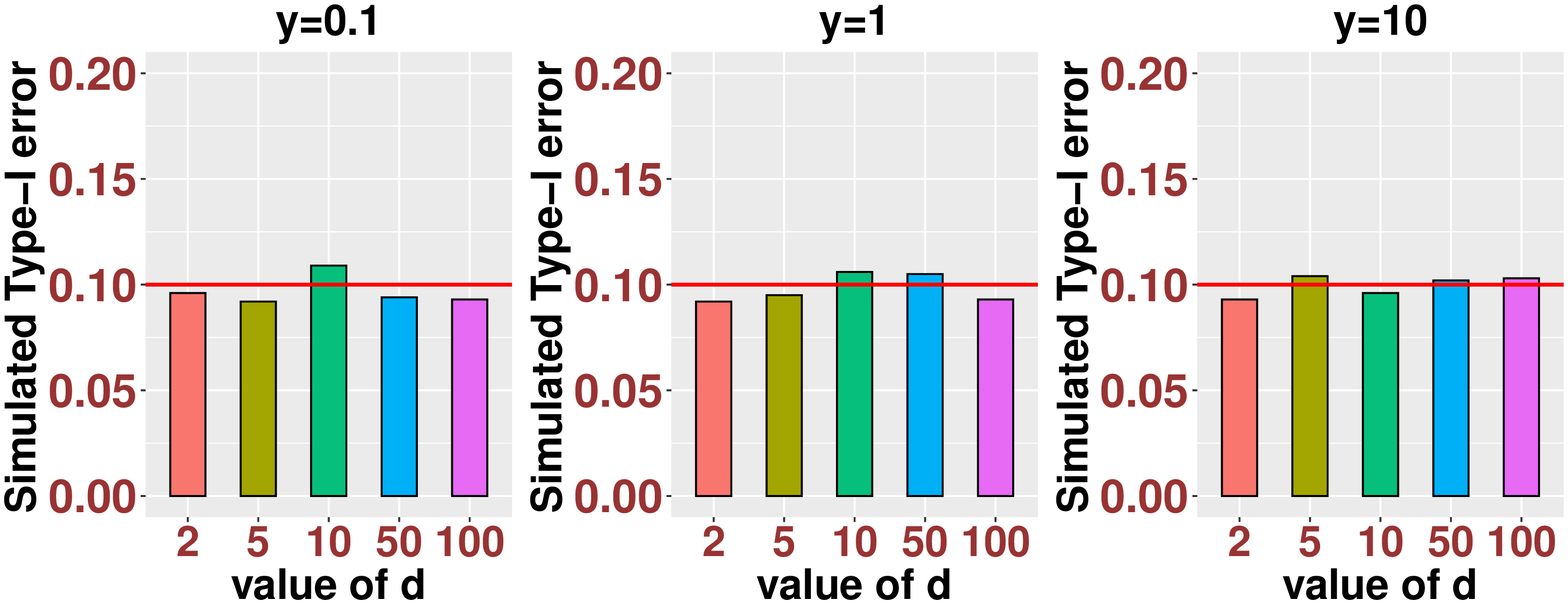}
\caption{$N=200.$}\label{subfig_nullorthogonaltypei200tp}
\end{subfigure}
\hspace{1cm}
\begin{subfigure}{0.45\textwidth}
\includegraphics[width=6.5cm,height=4.8cm]{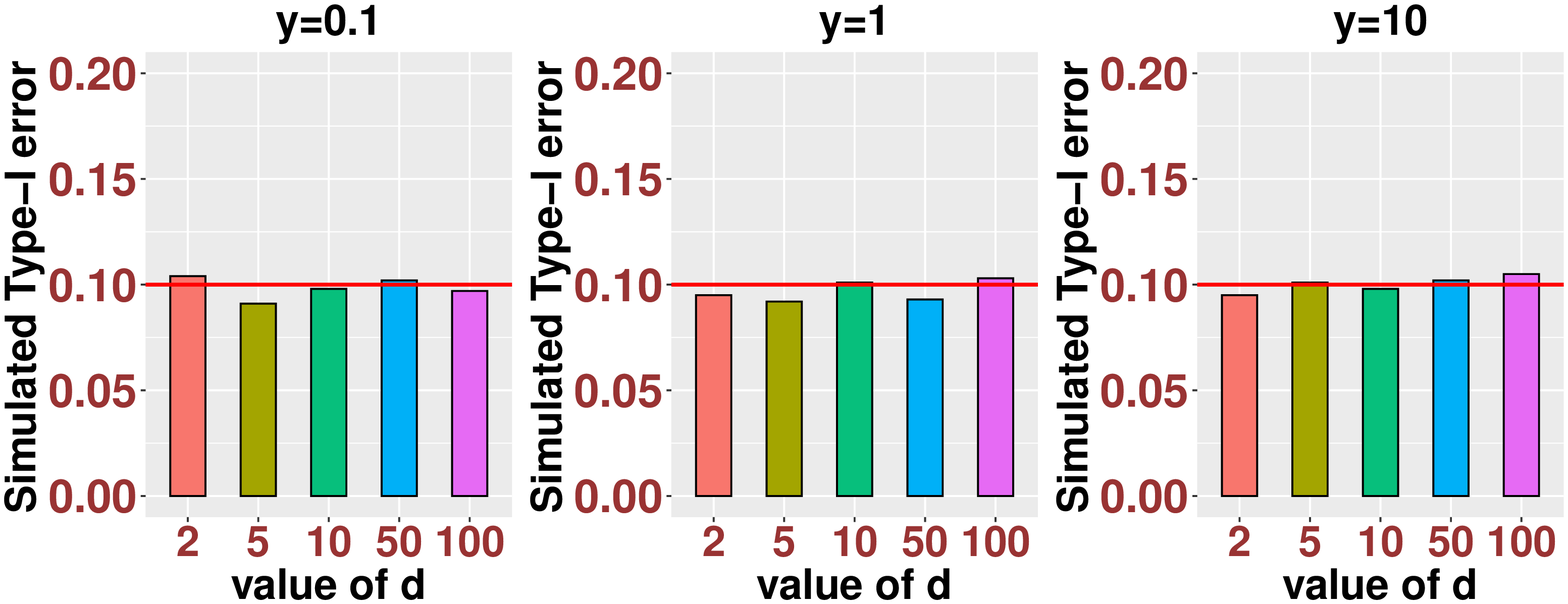}
\caption{$N=500.$}\label{subfig_nullorthogonaltypei500tp}
\end{subfigure}
\caption{{ \small Scenario A: simulated type I error rates for (\ref{eq_orthogobalteset}) using (\ref{eq_finalstat}). We report our results based on 2,000 Monte-Carlo simulations with two-point variables. The critical values are generated using Corollary \ref{cor_simulateddistribution}. } }
\label{fig_figorttypeitp}
\end{figure}

\begin{figure}[!ht]
\hspace*{-1.4cm}
\begin{subfigure}{0.45\textwidth}
\includegraphics[width=6.5cm,height=4.8cm]{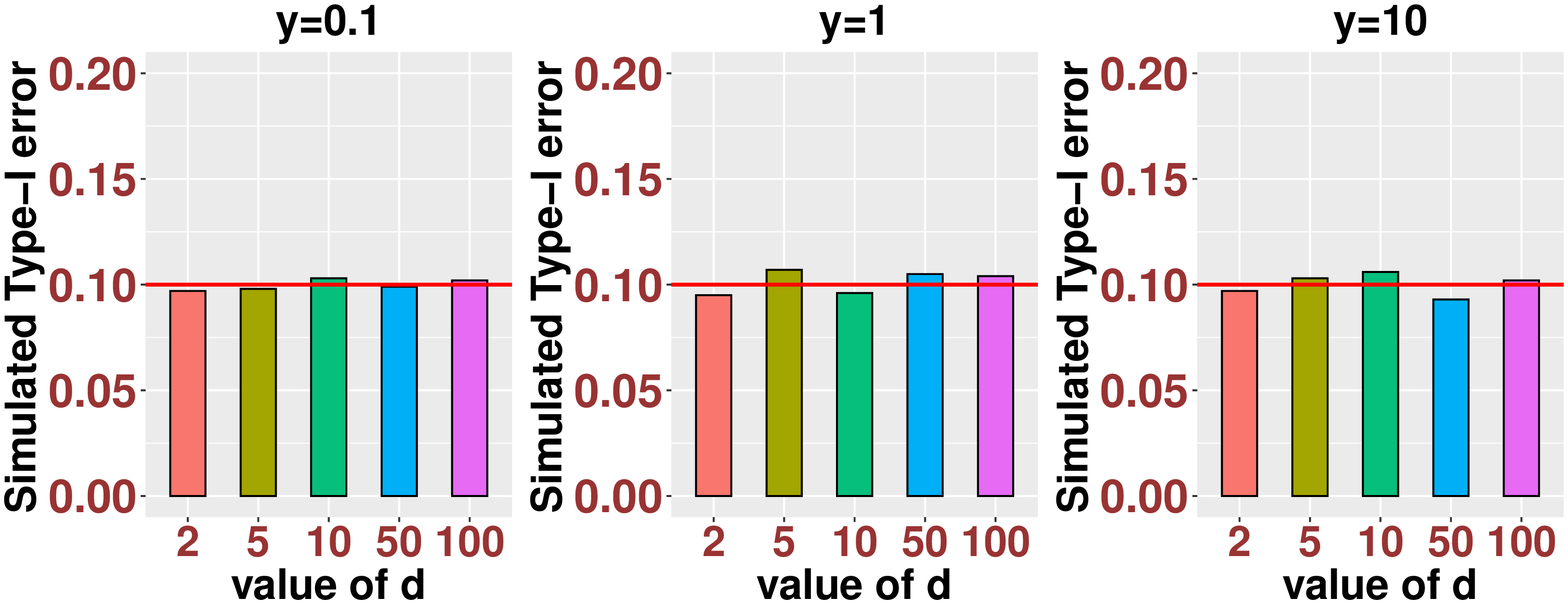}
\caption{$N=200.$}\label{subfig_nullorthogonaltypeii500tp}
\end{subfigure}
\hspace{1cm}
\begin{subfigure}{0.45\textwidth}
\includegraphics[width=6.5cm,height=4.8cm]{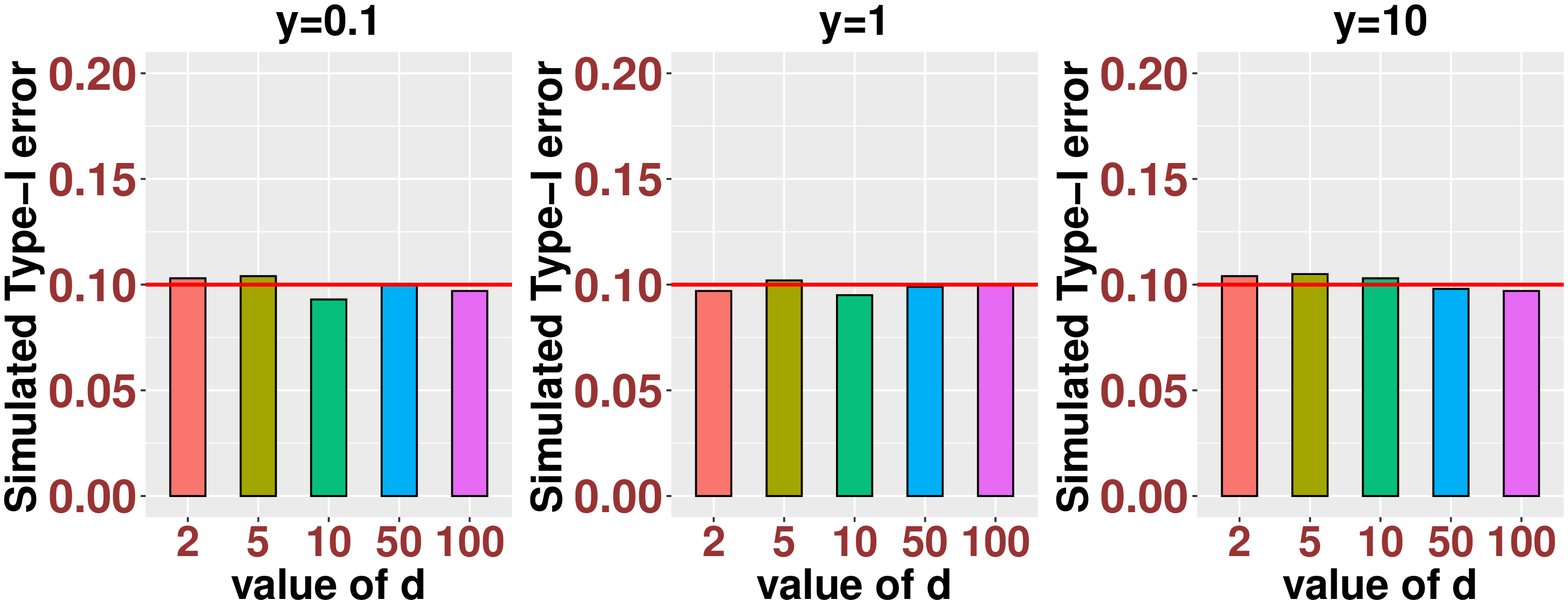}
\caption{$N=500.$}\label{subfig_nullorthogonaltypeii500tp}
\end{subfigure}
\caption{{ \small Scenario B: simulated type I error rates for (\ref{eq_orthogobalteset}). We report our results based on 2,000 Monte-Carlo simulations with two-point variables. The critical values are generated using Corollary \ref{cor_simulateddistribution}. }}
\label{fig_figorttypeiitp}
\end{figure}

\begin{figure}[!ht]
\hspace*{-2.0cm}
\begin{subfigure}{0.3\textwidth}
\includegraphics[width=5.8cm,height=5cm]{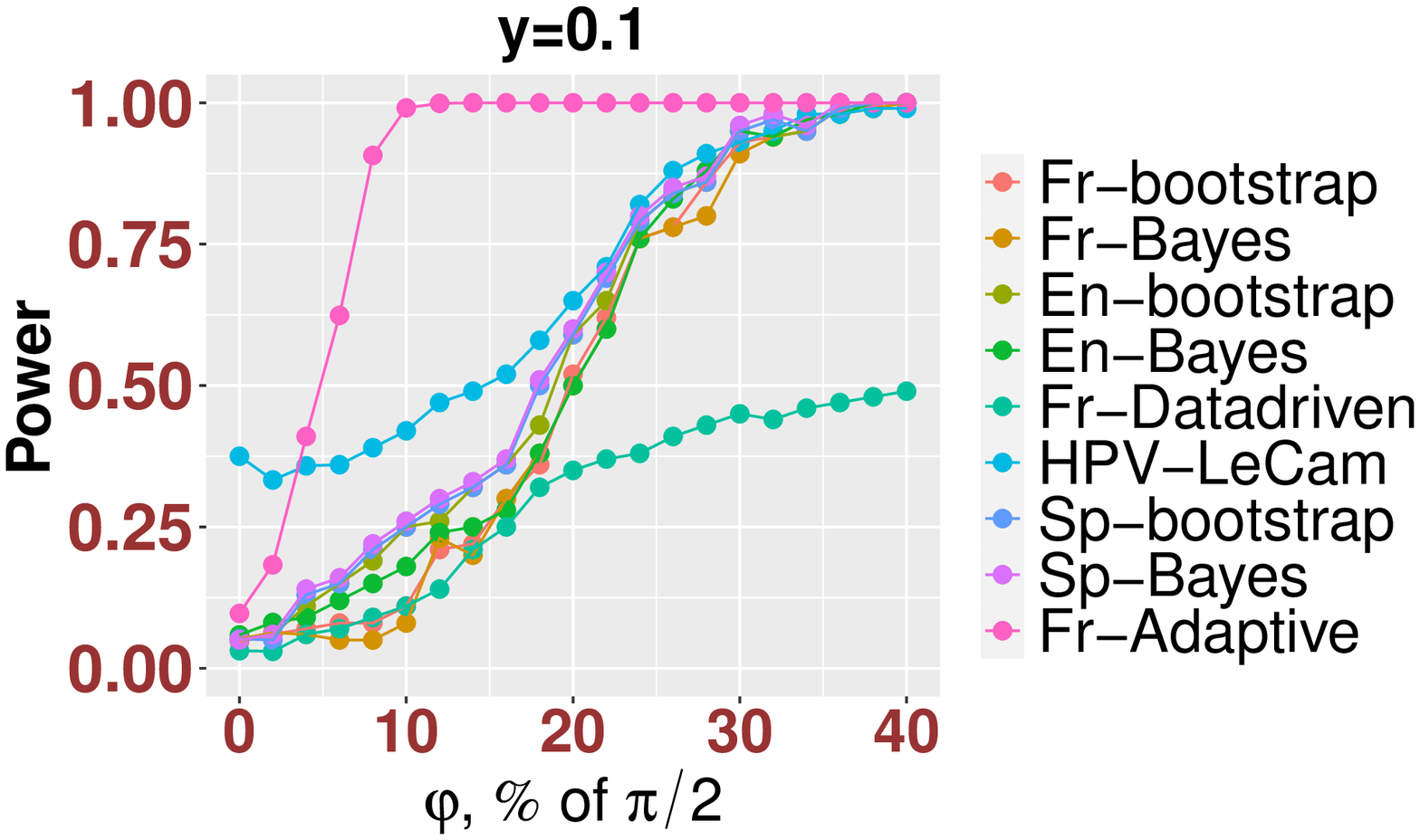}
\end{subfigure}
\begin{subfigure}{0.3\textwidth}
\includegraphics[width=5.8cm,height=5cm]{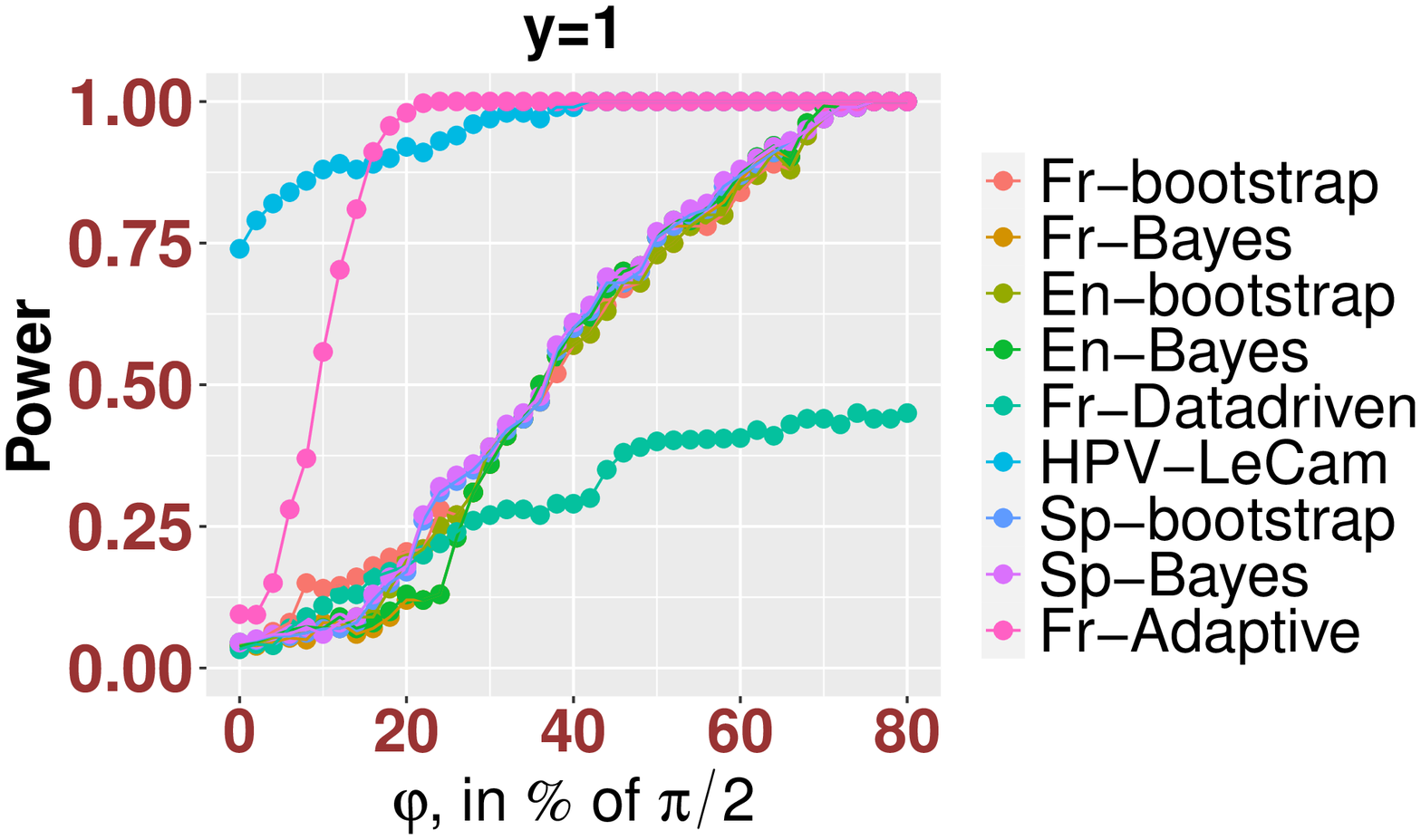}
\end{subfigure}
\begin{subfigure}{0.3\textwidth}
\includegraphics[width=5.8cm,height=5cm]{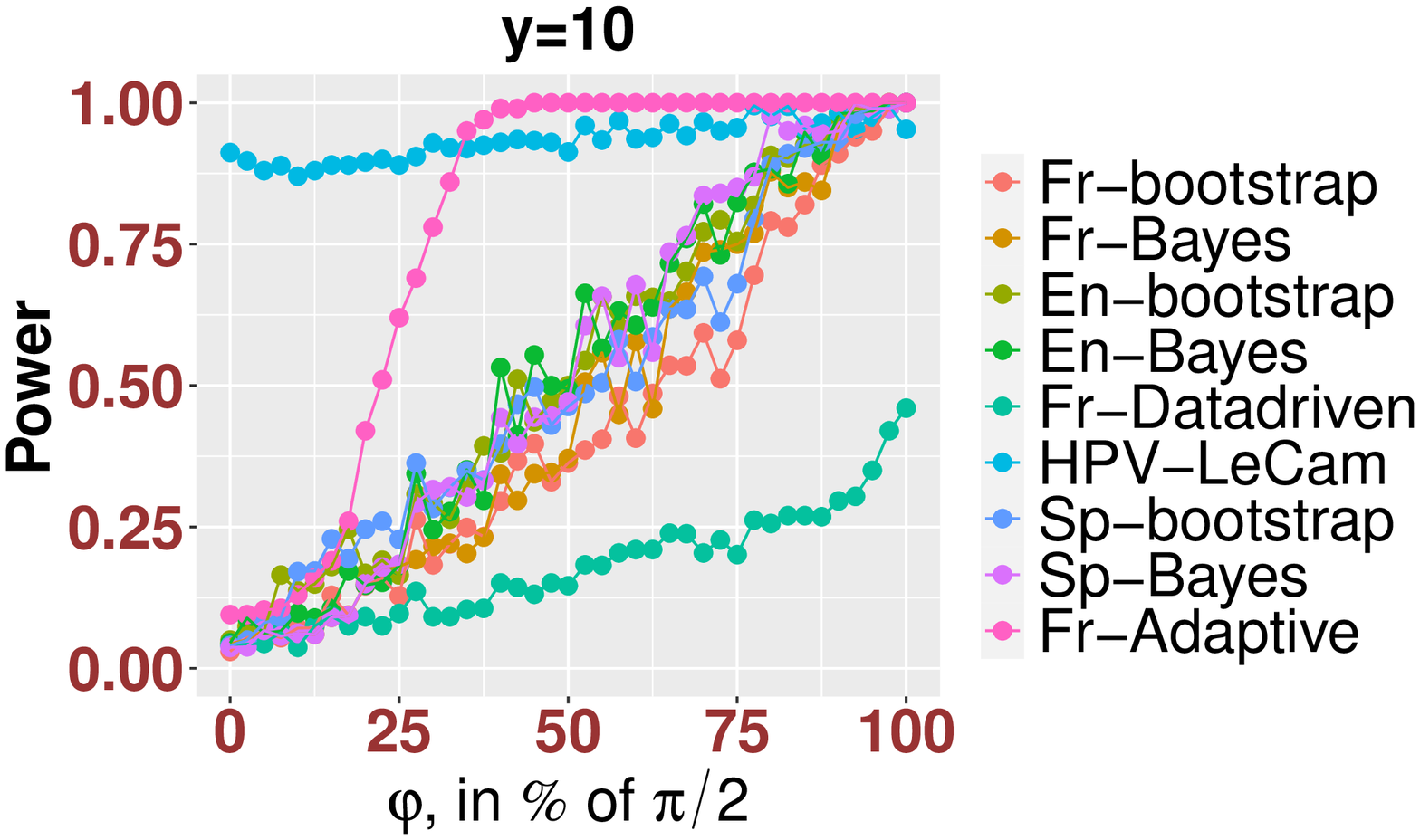}
\end{subfigure}
\caption{{ \small Comparison of power for Scenario I. We choose $d=5$ and use two-point random variables. We report our results under the nominal level 0.1 based on $2,000$ simulations.  Here $N=500.$ }}
\label{fig_powertp1}
\end{figure}

\begin{figure}[!ht]
\hspace*{-2.0cm}
\begin{subfigure}{0.3\textwidth}
\includegraphics[width=5.8cm,height=5cm]{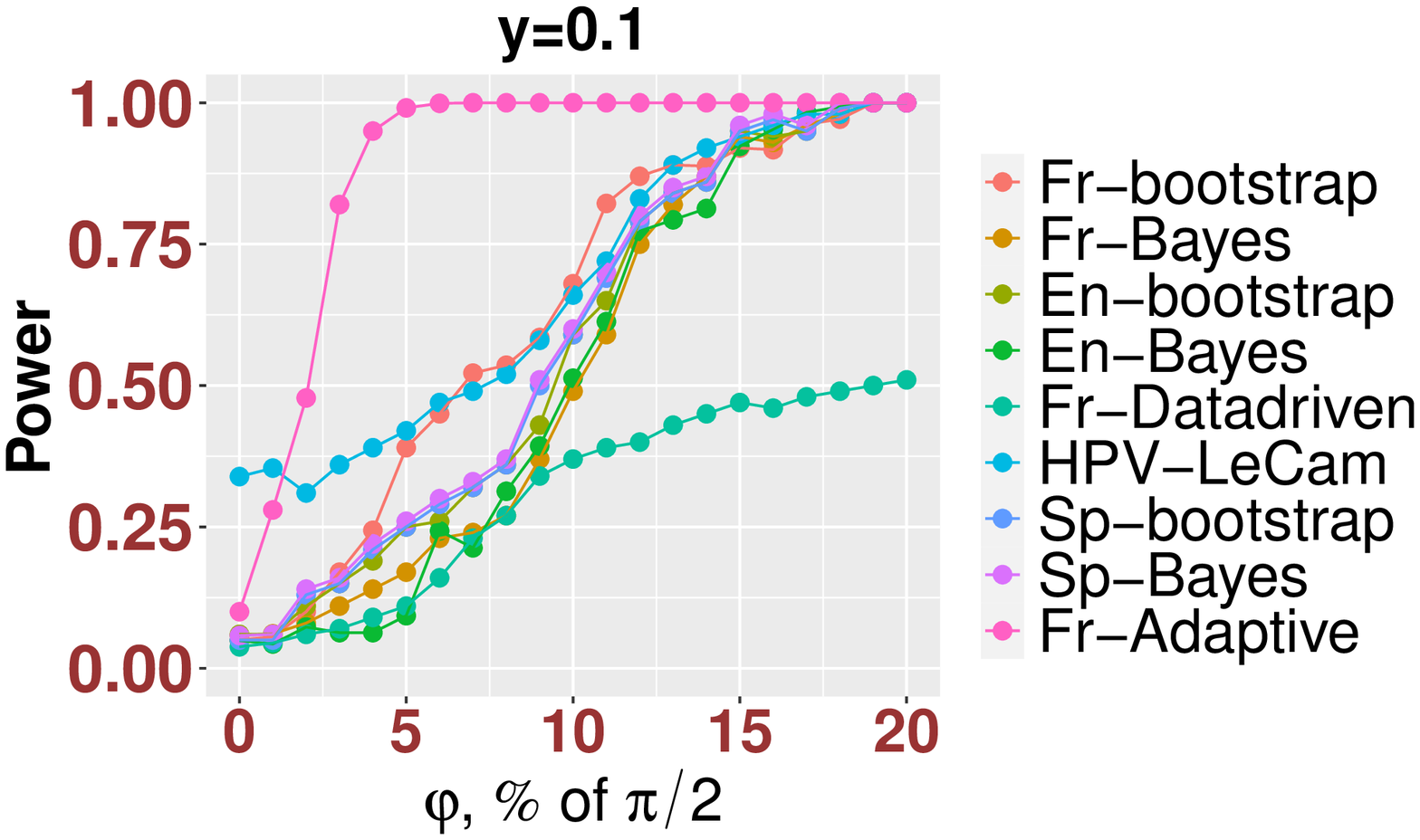}
\end{subfigure}
\begin{subfigure}{0.3\textwidth}
\includegraphics[width=5.8cm,height=5cm]{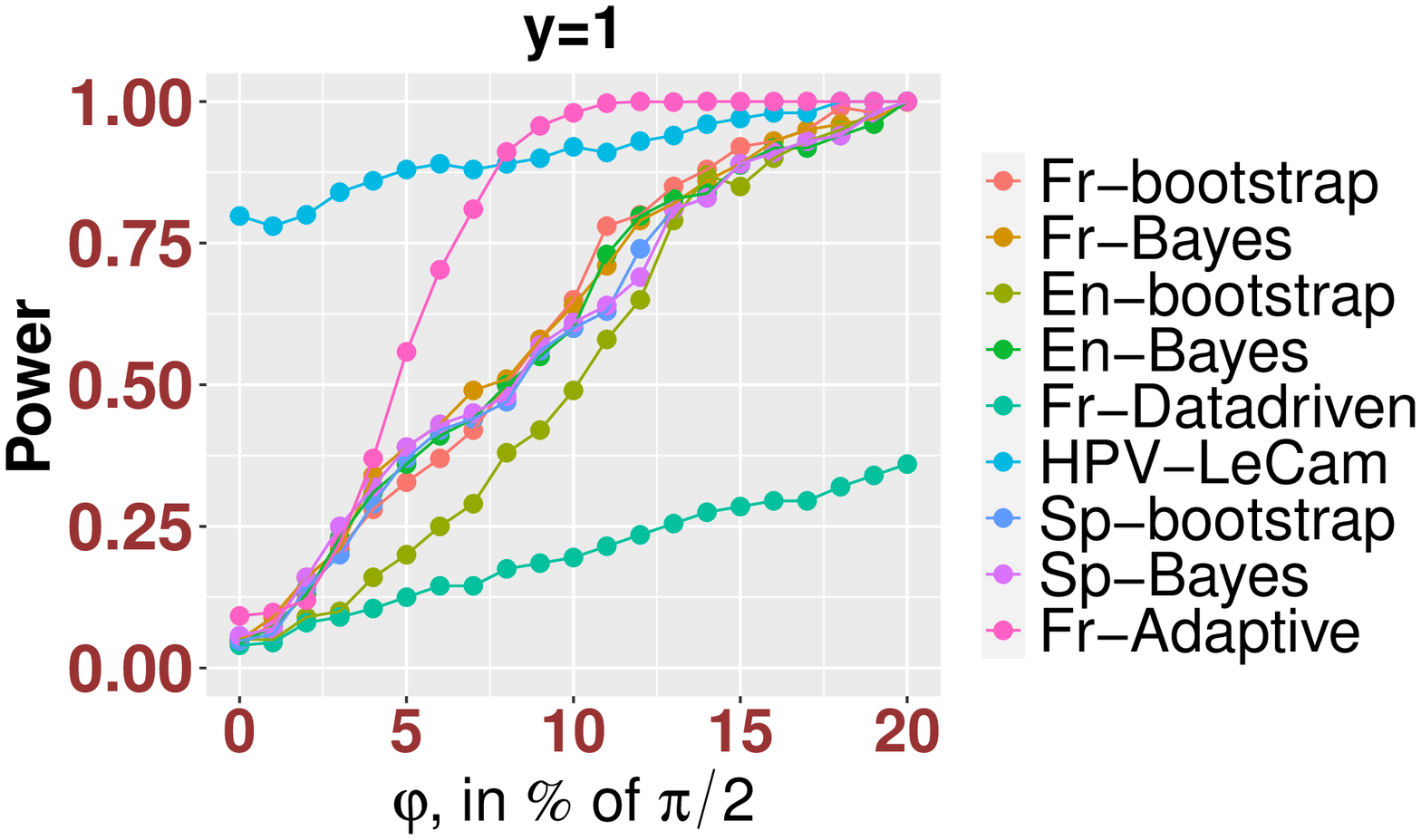}
\end{subfigure}
\begin{subfigure}{0.3\textwidth}
\includegraphics[width=5.8cm,height=5cm]{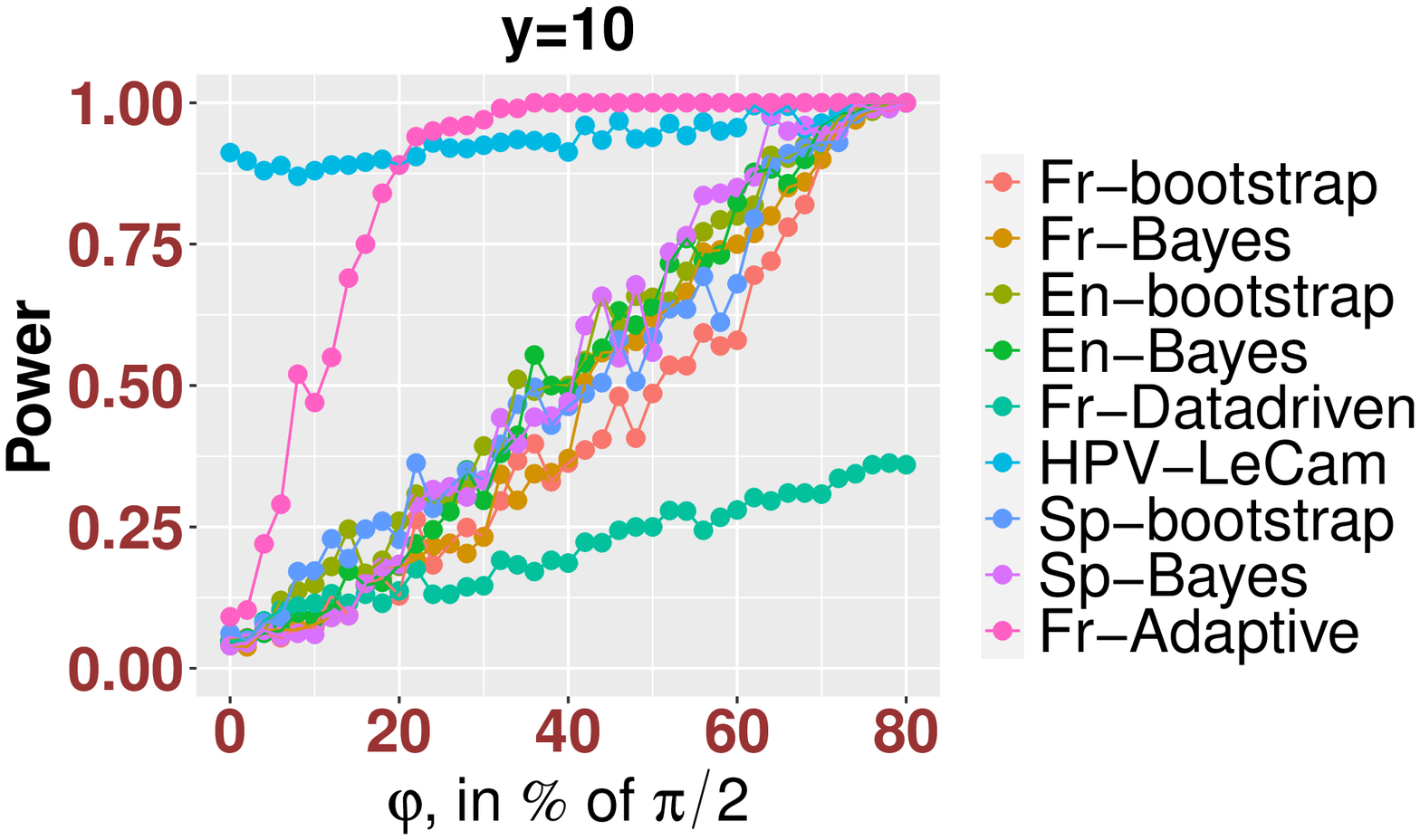}
\end{subfigure}
\caption{{ \small Comparison of power for Scenario I. We choose $d=50$ and use two-point random variables. We report our results under the nominal level 0.1 based on $2,000$ simulations. Here $N=500.$ }}
\label{fig_powertp150}
\end{figure}

\begin{figure}[!ht]
\hspace*{-2.2cm}
\begin{subfigure}{0.3\textwidth}
\includegraphics[width=5.8cm,height=5cm]{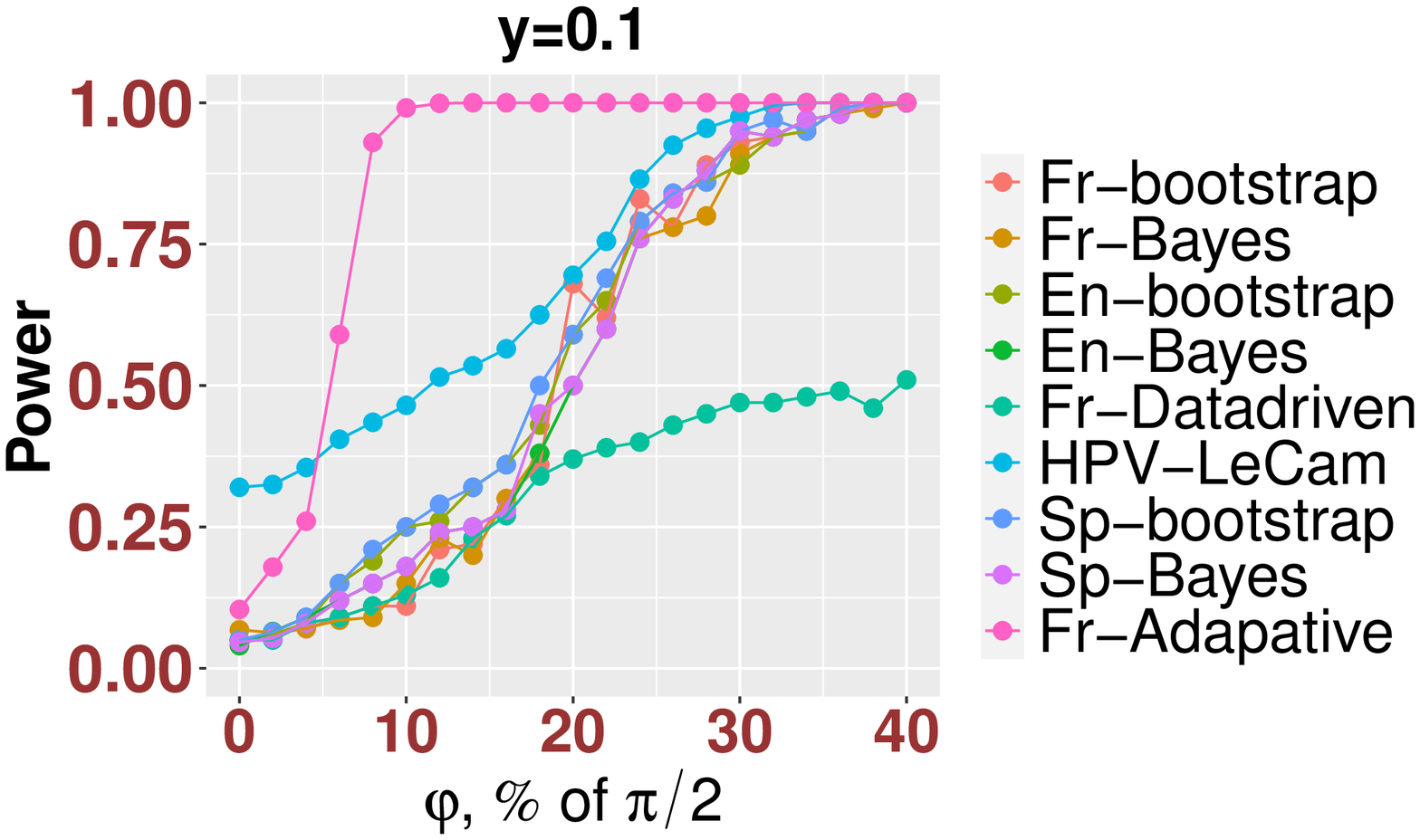}
\end{subfigure}
\begin{subfigure}{0.3\textwidth}
\includegraphics[width=5.8cm,height=5cm]{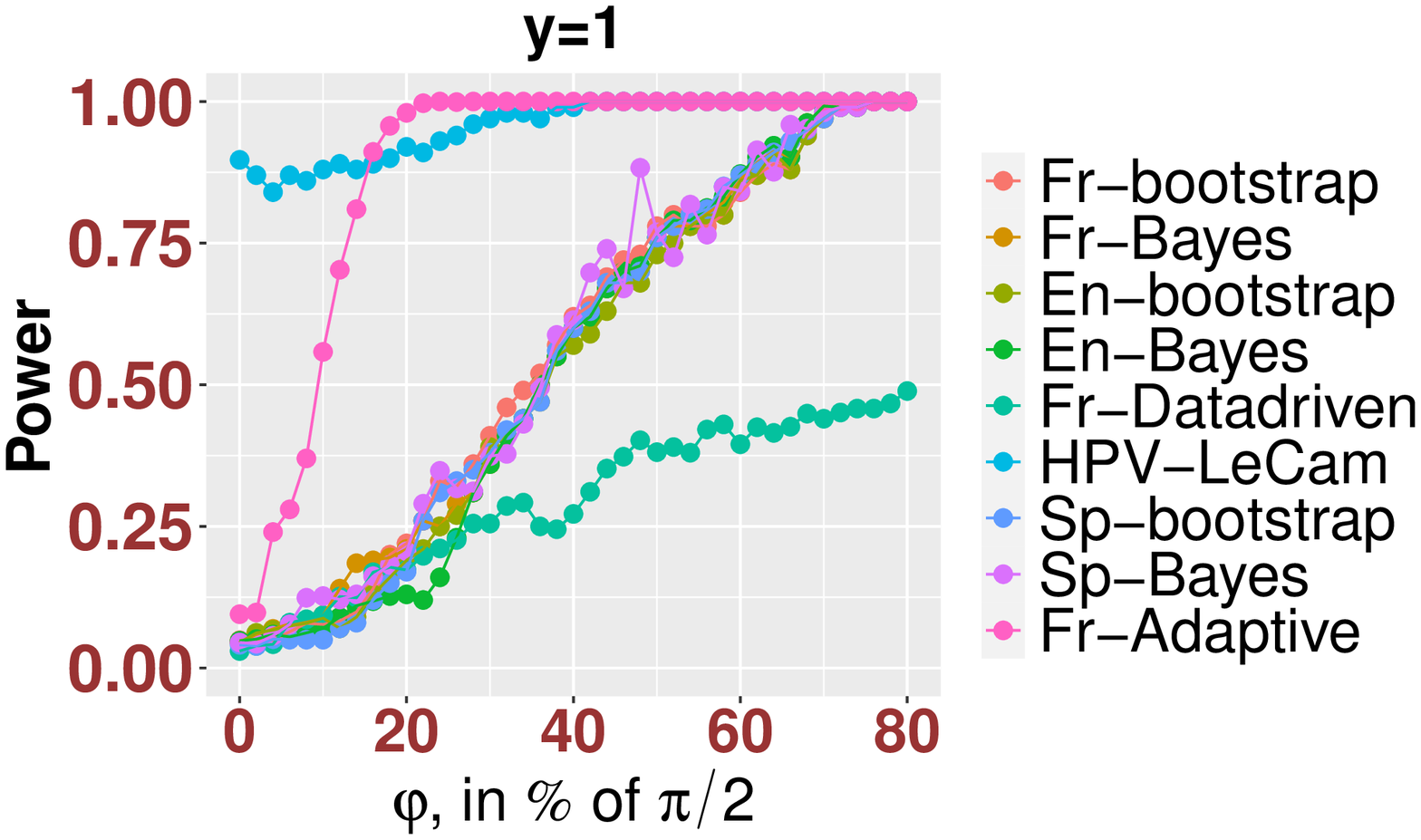}
\end{subfigure}
\begin{subfigure}{0.3\textwidth}
\includegraphics[width=5.8cm,height=5cm]{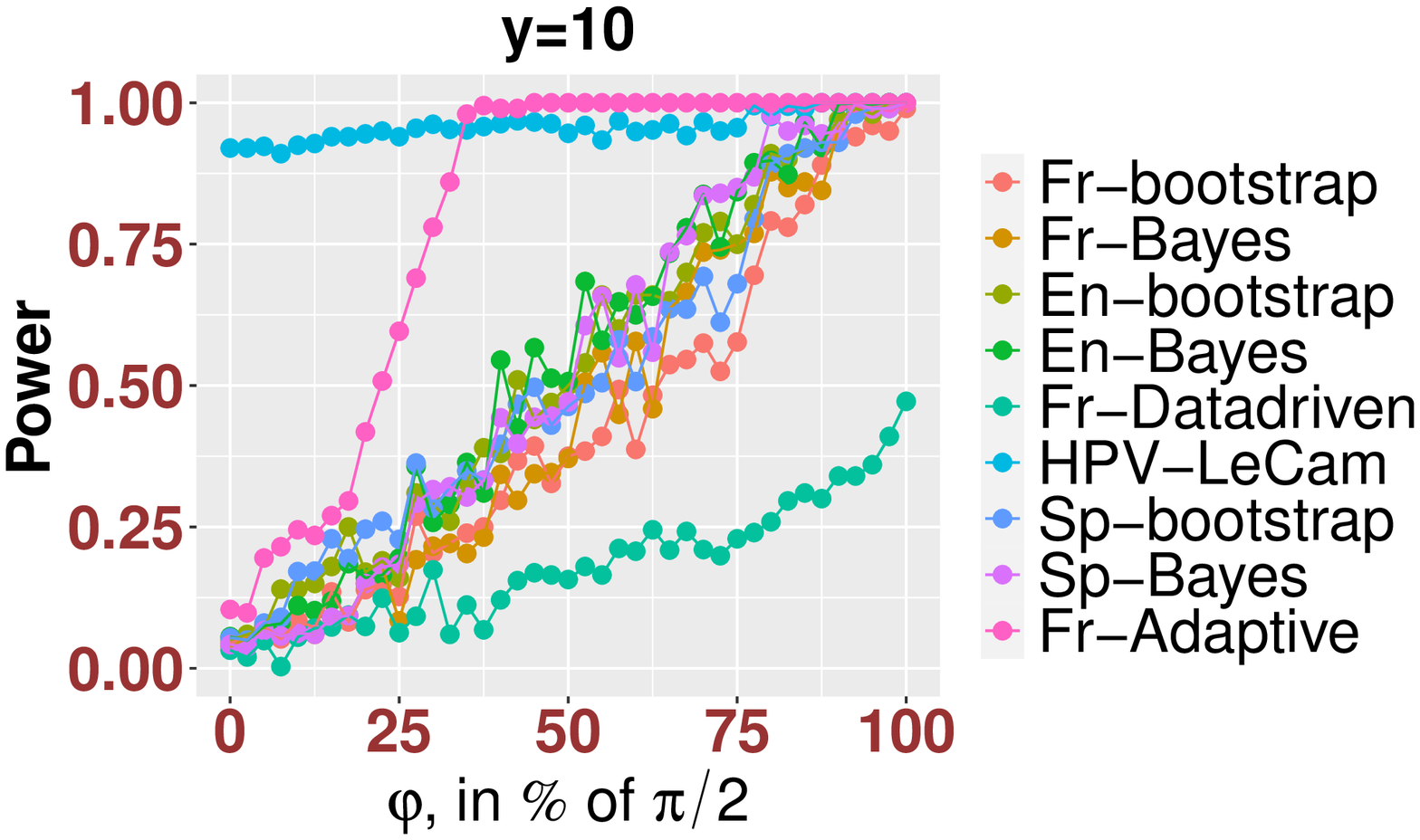}
\end{subfigure}
\caption{{ \small  Comparison of power for Scenario II. We choose $d=5$ and use two-point random variables. We report our results under the nominal level 0.1 based on $2,000$ simulations.  Here $N=500.$ }  }
\label{fig_powertp2}
\end{figure}

\begin{figure}[!ht]
\hspace*{-2.2cm}
\begin{subfigure}{0.3\textwidth}
\includegraphics[width=5.8cm,height=5cm]{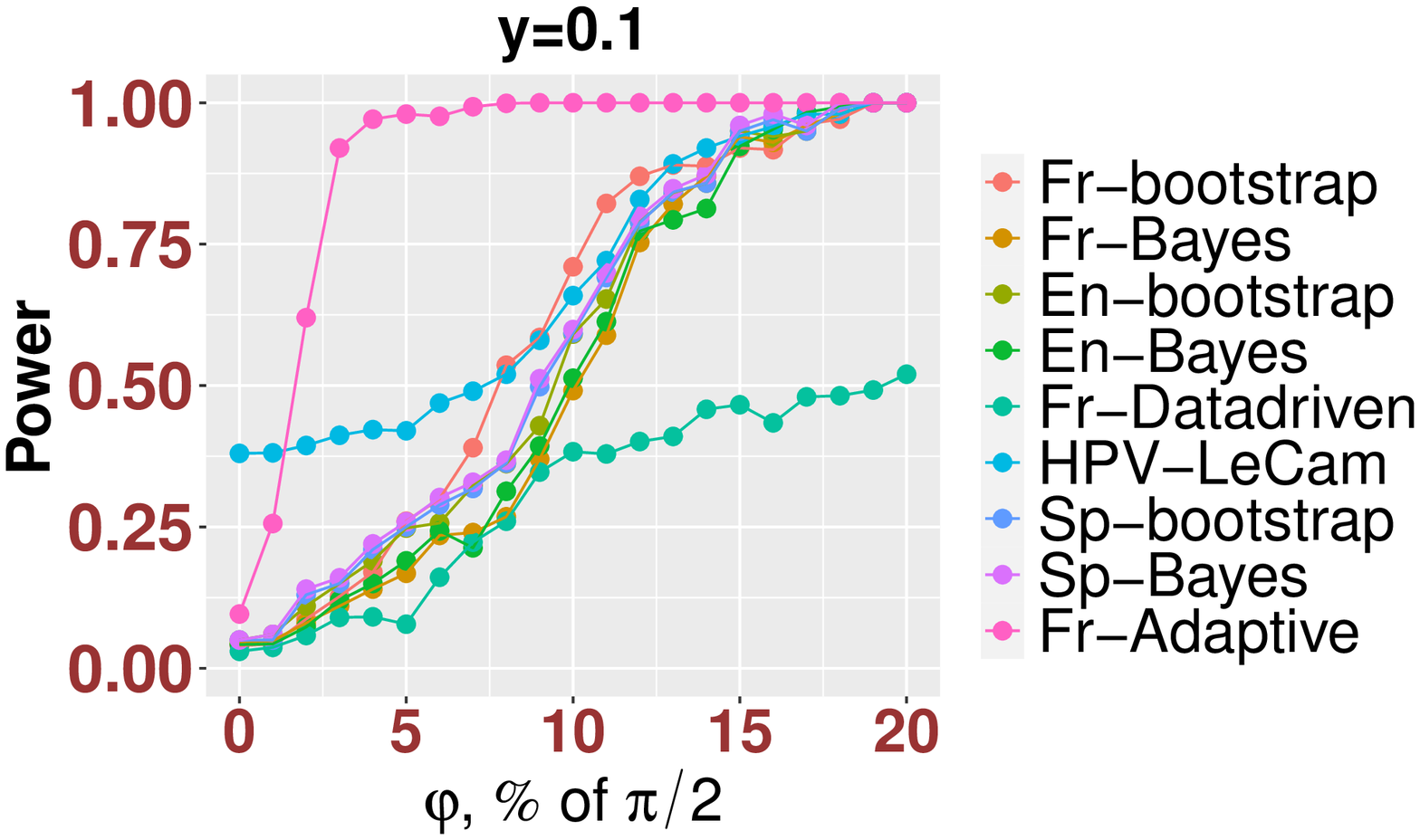}
\end{subfigure}
\begin{subfigure}{0.3\textwidth}
\includegraphics[width=5.8cm,height=5cm]{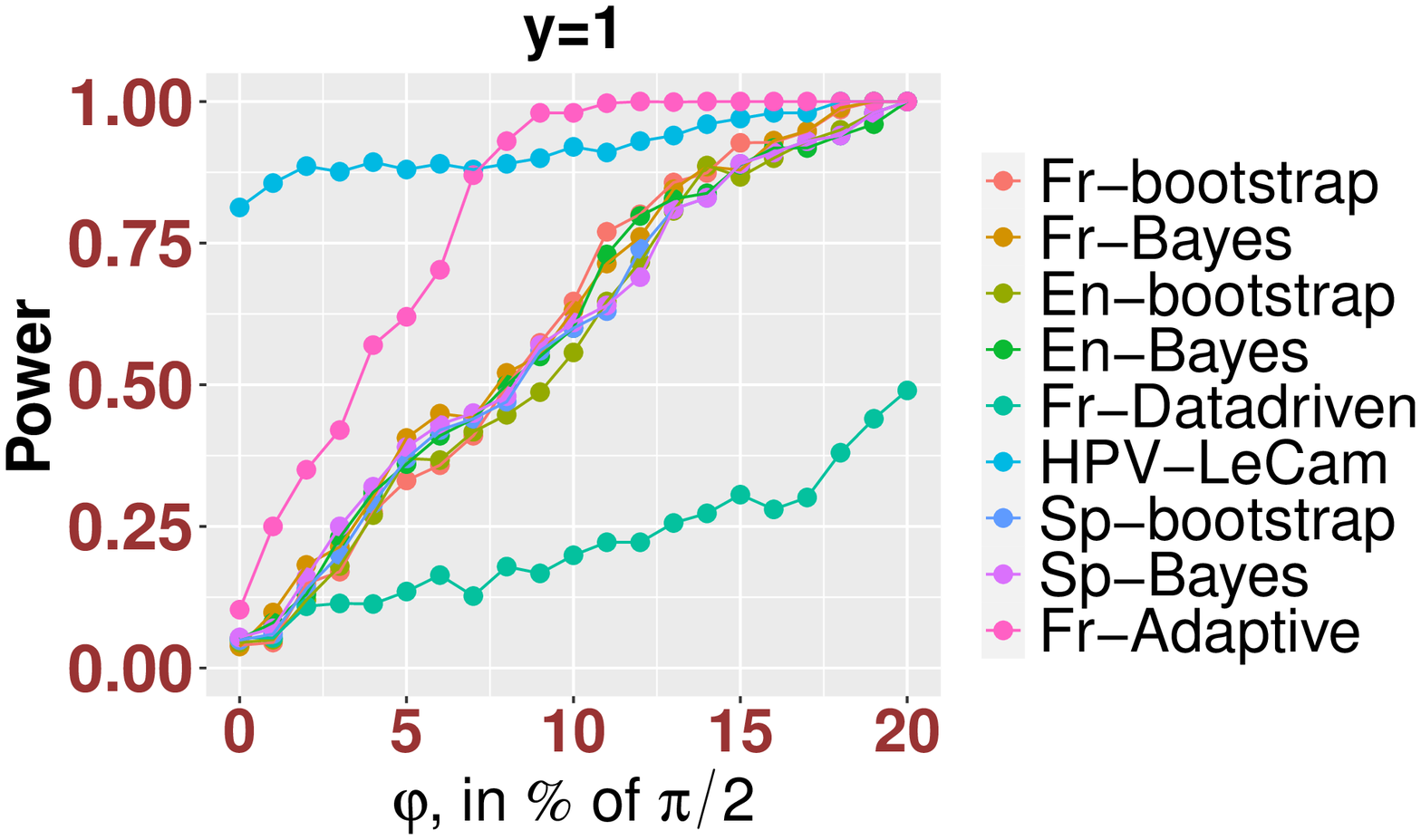}
\end{subfigure}
\begin{subfigure}{0.3\textwidth}
\includegraphics[width=5.8cm,height=5cm]{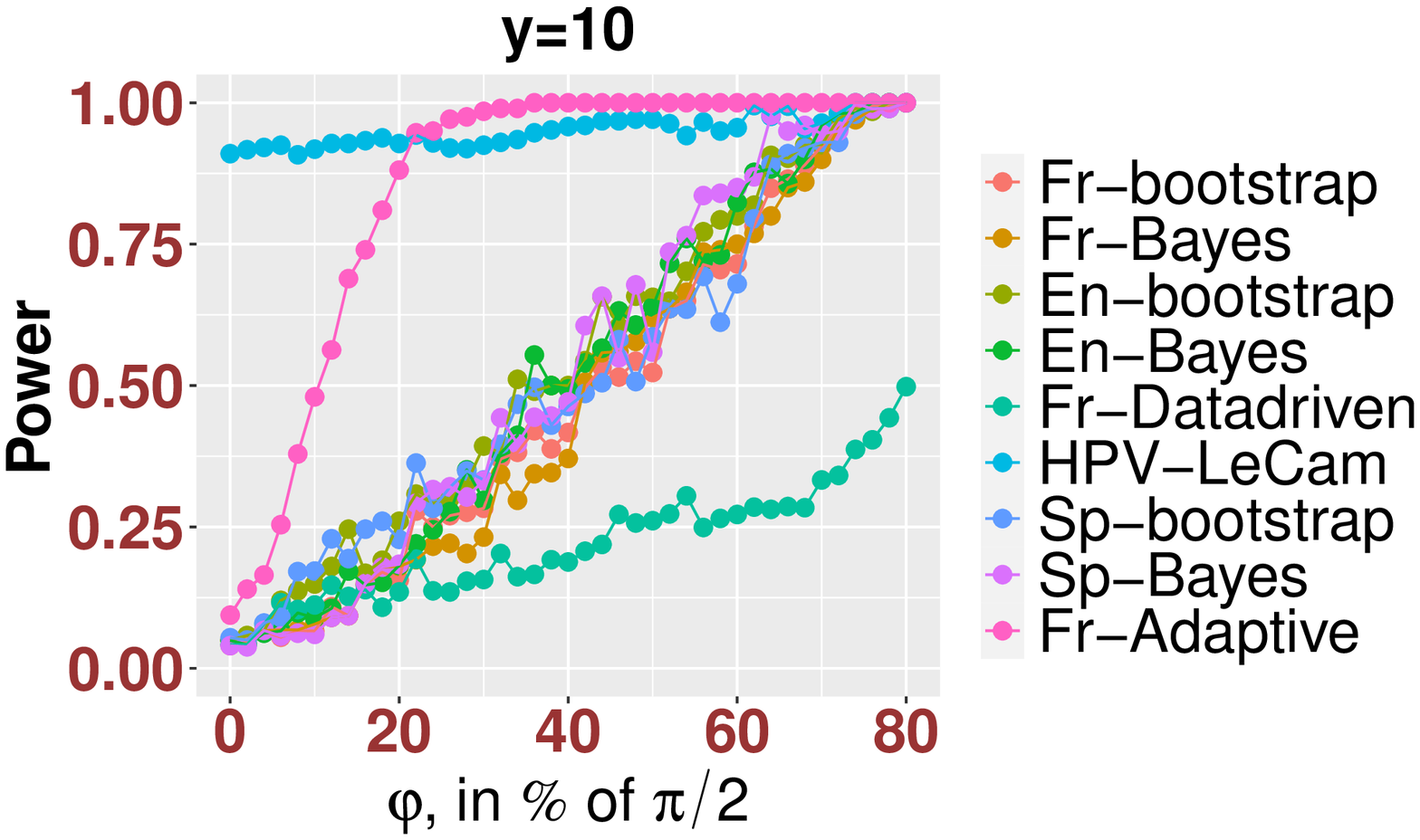}
\end{subfigure}
\caption{{ \small Comparison of power for Scenario II. We choose $d=50$ and use two-point random variables. We report our results under the nominal level 0.1 based on $2,000$ simulations.  Here $N=500.$ } }
\label{fig_powertp250}
\end{figure}

\begin{figure}[!ht]
\hspace*{-1.0cm}
\begin{subfigure}{0.45\textwidth}
\includegraphics[width=6.8cm,height=5cm]{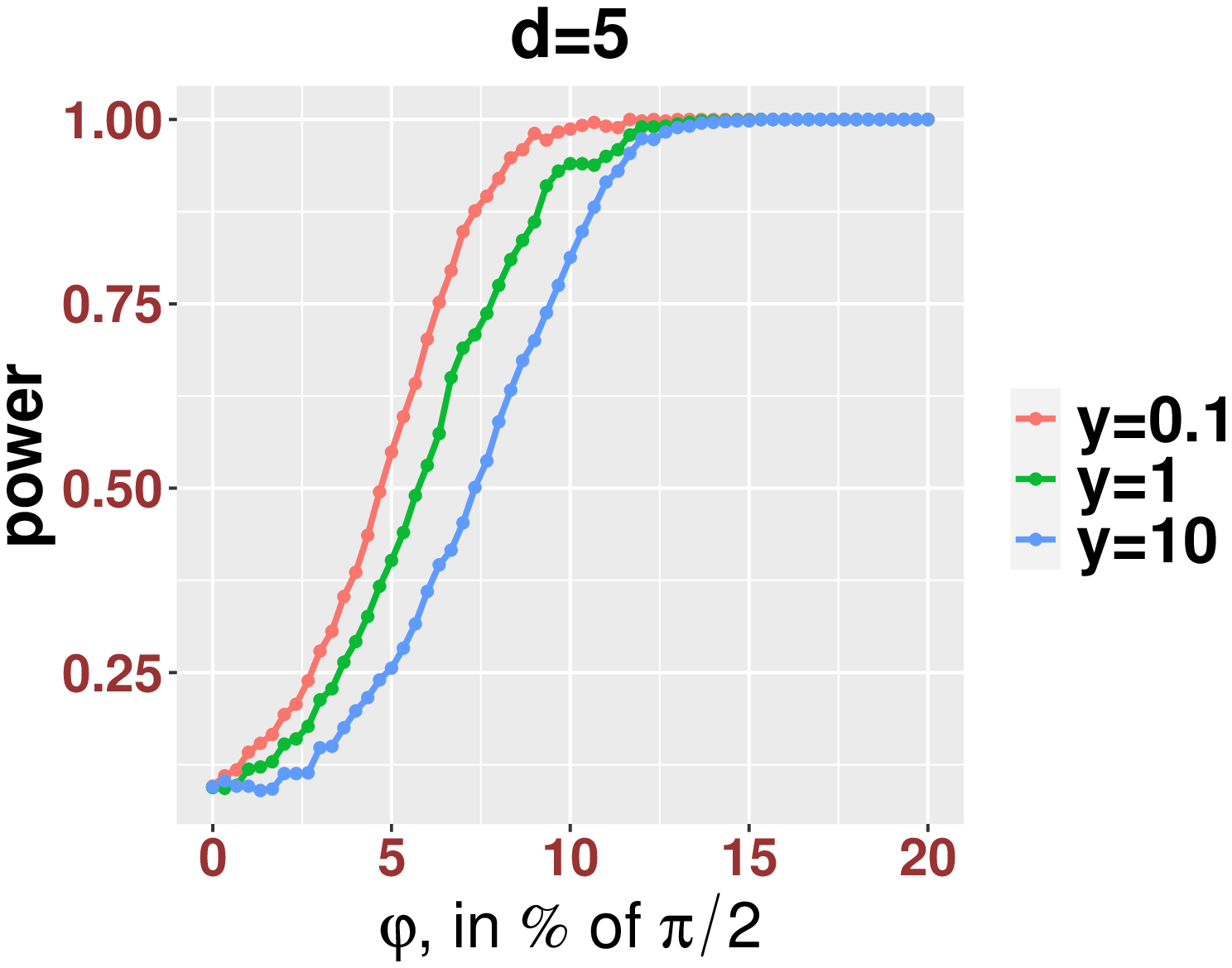}
\caption{$d=5.$}
\end{subfigure}
\hspace{1cm}
\begin{subfigure}{0.45\textwidth}
\includegraphics[width=6.8cm,height=5cm]{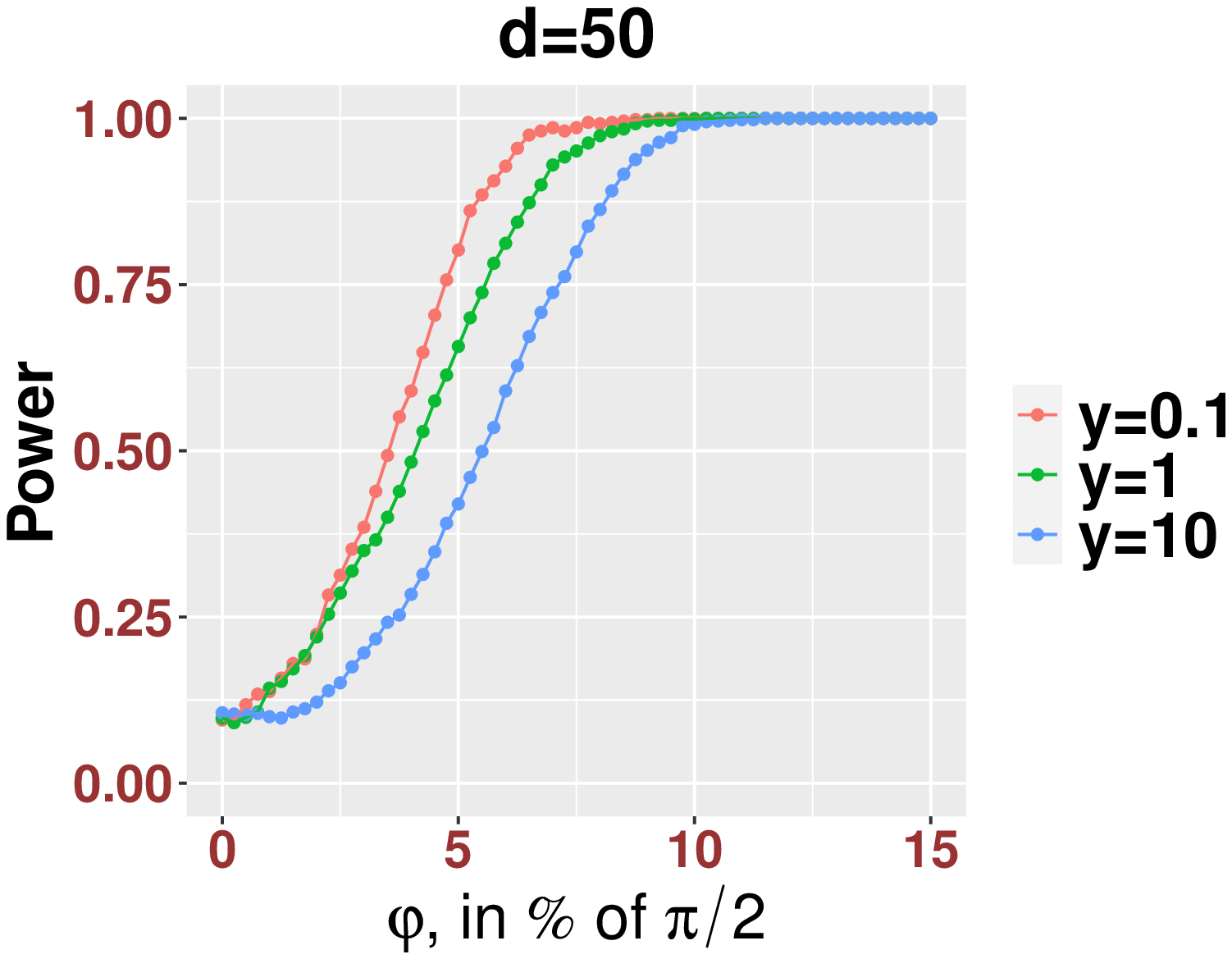}
\caption{$d=50.$}
\end{subfigure}
\caption{{\small Power of Scenario A for two-point random variables using (\ref{eq_finalstat}).  We report our results under the nominal level 0.1 based on $2,000$ simulations.  Here $N=500$ and the critical values are generated using Corollary \ref{cor_simulateddistribution}. }}
\label{fig_figorttypeipowertp}
\end{figure}

\begin{figure}[!ht]
\hspace*{-1.0cm}
\begin{subfigure}{0.45\textwidth}
\includegraphics[width=6.8cm,height=4.8cm]{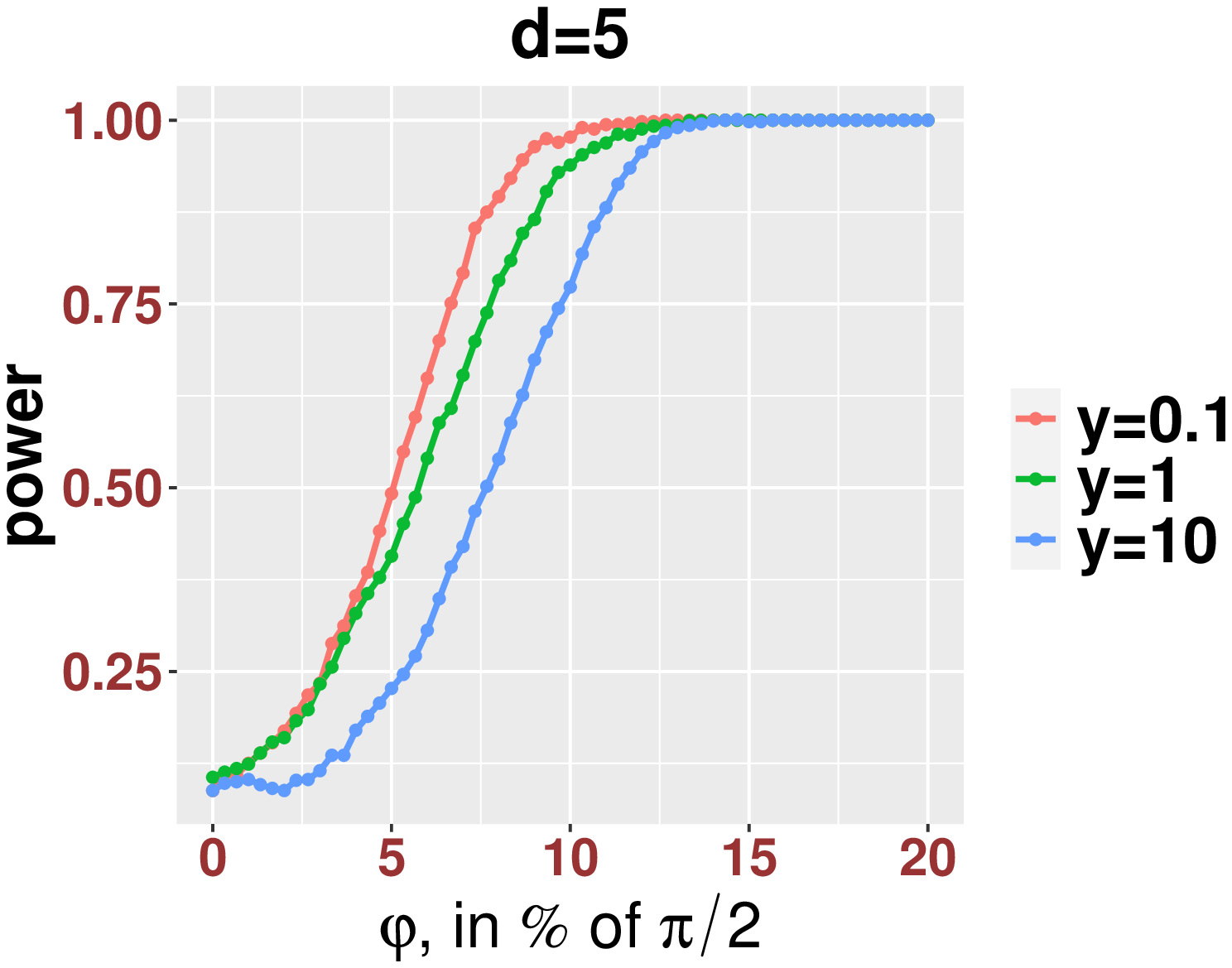}
\caption{$d=5.$}
\end{subfigure}
\hspace{1cm}
\begin{subfigure}{0.45\textwidth}
\includegraphics[width=6.8cm,height=4.8cm]{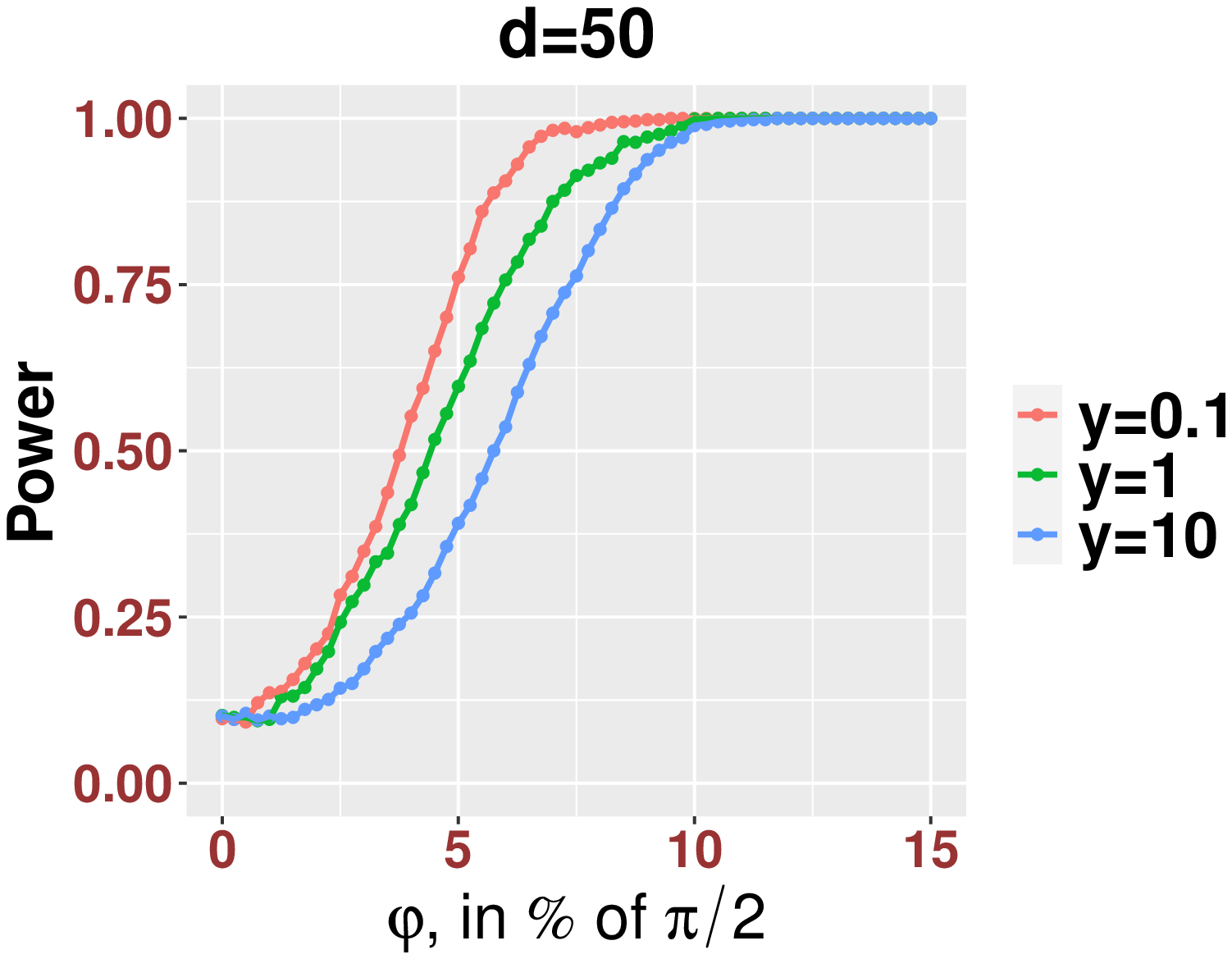}
\caption{$d=50.$}
\end{subfigure}
\caption{{ \small Power of Scenario B for two-point random variables using (\ref{eq_finalstat}).  We report our results under the nominal level 0.1 based on $2,000$ simulations. Here $N=500$ and the critical values are generated using Corollary \ref{cor_simulateddistribution}. }  }
\label{fig_figorttypeiipowertp}
\end{figure}

\end{appendix}
 

\bibliographystyle{unsrt}

\end{document}